\documentclass[10pt]{amsart}
\usepackage{amssymb}
\usepackage{dsfont}
\usepackage{amscd}
\usepackage[mathscr]{euscript} 
\usepackage[all]{xy}
\usepackage{url}
\usepackage{stmaryrd}
\usepackage{comment}
\usepackage[retainorgcmds]{IEEEtrantools}
\usepackage[hidelinks]{hyperref}
\usepackage{graphicx}
\usepackage{mathtools}
\usepackage[dvipsnames]{xcolor}
\usepackage{centernot}
\usepackage{marvosym}
\usepackage{tikz}
\usepackage{tikz-cd}
\usepackage{xspace}
\usepackage{bbm}
\usepackage[title]{appendix}
\usepackage{adjustbox}

\usepackage{tkz-fct}
\usepackage{soul}
\usepackage{lineno}
\usepackage[T1]{fontenc}
\usepackage[utf8]{inputenc}

\usepackage[english]{babel}
\usepackage[autostyle]{csquotes}
\usepackage[shortlabels]{enumitem}
\usepackage[backend=bibtex, style=alphabetic]{biblatex}
\allowdisplaybreaks 

\DeclareMathOperator{\Th}{Th}

\title[]{Convolution semigroups for automorphism dynamics}


\usepackage{ulem}
\usepackage{todonotes}
\setuptodonotes{noinline}

\colorlet{KyleColor}{green!80!yellow}
\colorlet{PierreColor}{cyan}
\colorlet{IlijasColor}{red}

\author[K. Gannon]{Kyle Gannon$^{\ast}$}
\thanks{$^{\ast}$Supported by the Fundamental Research Funds for the Central Universities, Peking University, grant no. 7100604835.}
\address{$^{\ast}$ Beijing International Center for Mathematical Research (BICMR) \\ Peking University \\ Beijing, China.}
\email{kgannon@bicmr.pku.edu.cn}

\author[D. M. HOFFMANN]{Daniel Max Hoffmann$^{\dagger}$}
\thanks{$^{\dagger}$SDG. The second author is supported by the National Science Centre (Narodowe Centrum Nauki, Poland) 
grant no. 2021/43/B/ST1/00405.}
\address{$^{\dagger}$
Instytut Matematyki\\
Uniwersytet Warszawski\\
Warszawa\\
Poland}
\email{daniel.max.hoffmann@gmail.com}
\urladdr{{https://sites.google.com/site/danielmaxhoffmann/}}

\author[K. Krupi\'{n}ski]{Krzysztof Krupi\'{n}ski$^{\ddagger}$}
\address{$^{\ddagger}$Instytut Matematyczny Uniwersytetu Wroc{\l}awskiego, pl. Grunwaldzki 2, 50-384 Wroc{\l}aw, Poland}
\email{Krzysztof.Krupinski@math.uni.wroc.pl}

\date{\today}

\DeclareMathOperator{\E}{E}

 \DeclareMathOperator{\aut}{Aut} \DeclareMathOperator{\id}{id}
\DeclareMathOperator{\cl}{cl}
\DeclareMathOperator{\stab}{Stab}

 \DeclareMathOperator{\gal}{Gal}

\DeclareMathOperator{\df}{def}
\DeclareMathOperator{\Df}{Def} 
 
\DeclareMathOperator{\tp}{tp}

\DeclareMathOperator{\inv}{inv}

\DeclareMathOperator{\gen}{gen}

\DeclareMathOperator{\res}{res}

\DeclareMathOperator{\sym}{Sym}

\DeclareMathOperator{\Sym}{Sym}

\DeclareMathOperator{\Ord}{Ord}

\DeclareMathOperator{\Rr}{\mathbb{R}}

\newtheorem{theorem}{Theorem}[section]
\newtheorem{proposition}[theorem]{Proposition}
\newtheorem{lemma}[theorem]{Lemma}
\newtheorem{cor}[theorem]{Corollary}
\newtheorem{fact}[theorem]{Fact}
\theoremstyle{definition}
\newtheorem{definition}[theorem]{Definition}
\newtheorem{example}[theorem]{Example}
\newtheorem{remark}[theorem]{Remark}
\newtheorem{question}[theorem]{Question}
\newtheorem{conj}[theorem]{Conjecture}

\theoremstyle{remark}
\newtheorem*{theorem*}{Theorem}
\newtheorem*{cor*}{Corollary}
\newtheorem*{conj*}{Conjecture}

\theoremstyle{definition}
\theoremstyle{definition}

\theoremstyle{definition}

\theoremstyle{remark}

\newtheorem{clm}{Claim}
\newtheorem*{clm*}{Claim}

\AtEndEnvironment{proof}{\setcounter{clm}{0}}
\newenvironment{clmproof}[1][\proofname]{\proof[#1]}{\endproof}

\newtheorem{innercustomgeneric}{\customgenericname}
\providecommand{\customgenericname}{}
\newcommand{\newcustomtheorem}[2]{%
  \newenvironment{#1}[1]
  {%
   \ifdefined\crefalias\crefalias{innercustomgeneric}{#2}\fi
   \renewcommand\customgenericname{#2}%
   \renewcommand\theinnercustomgeneric{##1}%
   \innercustomgeneric
  }
  {\endinnercustomgeneric}%
  \ifdefined\crefname\crefname{#2}{#2}{#2s}\fi
}

\newcustomtheorem{customthm}{Theorem}
\newcustomtheorem{customConj}{\bf{Conjecture}}

\makeatletter
\providecommand*{\cupdot}{%
  \mathbin{%
    \mathpalette\@cupdot{}%
  }%
}
\newcommand*{\@cupdot}[2]{%
  \ooalign{%
    $\m@th#1\cup$\cr
    \sbox0{$#1\cup$}%
    \dimen@=\ht0 %
    \sbox0{$\m@th#1\cdot$}%
    \advance\dimen@ by -\ht0 %
    \dimen@=.5\dimen@
    \hidewidth\raise\dimen@\box0\hidewidth
  }%
}

\providecommand*{\bigcupdot}{%
  \mathop{%
    \vphantom{\bigcup}%
    \mathpalette\@bigcupdot{}%
  }%
}
\newcommand*{\@bigcupdot}[2]{%
  \ooalign{%
    $\m@th#1\bigcup$\cr
    \sbox0{$#1\bigcup$}%
    \dimen@=\ht0 %
    \advance\dimen@ by -\dp0 %
    \sbox0{\scalebox{2}{$\m@th#1\cdot$}}%
    \advance\dimen@ by -\ht0 %
    \dimen@=.5\dimen@
    \hidewidth\raise\dimen@\box0\hidewidth
  }%
}
\makeatother

\def\Ind#1#2{#1\setbox0=\hbox{$#1x$}\kern\wd0\hbox to 0pt{\hss$#1\mid$\hss}
\lower.9\ht0\hbox to 0pt{\hss$#1\smile$\hss}\kern\wd0}

\def\ind{\mathop{\mathpalette\Ind{}}}

\def\notind#1#2{#1\setbox0=\hbox{$#1x$}\kern\wd0
\hbox to 0pt{\mathchardef\nn=12854\hss$#1\nn$\kern1.4\wd0\hss}
\hbox to 0pt{\hss$#1\mid$\hss}\lower.9\ht0 \hbox to 0pt{\hss$#1\smile$\hss}\kern\wd0}

\def\Gen{\operatorname{Gen}}
\def\Stab{\operatorname{Stab}}
\def\Mlt{\operatorname{Mlt}}

\def\Ima{\operatorname{Im}}
\def\Def{\operatorname{Def}}

\newcommand{\FC}{\mathfrak{C}}

\newcommand{\CL}{{\mathcal L}}
\newcommand{\CM}{{\mathcal M}}

\newcommand{\CF}{{\mathcal F}}

\newcommand{\fs}{\mathrm{fs}}
\newcommand{\sfs}{\mathrm{sfs}}
\newcommand{\conv}{\mathrm{conv}}
\newcommand{\Av}{\mathrm{Av}}
\newcommand{\supp}{\mathrm{supp}}

\newcommand{\autf}{\mathrm{Autf}}
\newcommand{\KP}{\mathrm{KP}}

\begin{document}

\begin{abstract}
Initially motivated by Hrushovski’s paper on definability patterns,
we obtain homeomorphisms between Ellis semigroups related to natural actions of the automorphism groups of first order structures
and certain collections of types and Keisler measures. Thus, we can transfer the semigroup operation from these Ellis semigroups to the corresponding collections of types and Keisler measures. By generalizing this transferred product, we obtain a new convolution operation for invariant types and measures in arbitrary first-order theories. We develop its general theory and prove several correspondence theorems between idempotent measures and closed subgroups of the automorphism group 
of a sufficiently large (so-called monster) model with respect to the {\em relatively definable topology}. Via the affine sort construction, we demonstrate that this new notion of convolution encodes the standard definable convolution operation over definable groups.
\end{abstract}

\maketitle

\tableofcontents

\section{Introduction}
\subsection{Background}
We introduce and develop a theory of random automorphisms\footnote{Here, the term \emph{random automorphism} corresponds to certain measures over certain type spaces which allow one to coherently randomly sample automorphisms (with respect to the Morley product). This naming convention is similar to the one used for \emph{Invariant Random Subgroups} (IRSs) where the adjective \emph{random} implies that the object of interest is a measure.} and their convolution product over sufficiently big models of arbitrary first-order theories.  
In the realm of logic, this theory has deep connections with \emph{definable convolution}, but is also inherently connected to classical harmonic analysis, measured group theory, and topological dynamics. 
We demonstrate that the spaces of random automorphisms over combinatorially tame structures (e.g. stable, NIP) form \emph{well-behaved} semigroups and develop 
the theory of the relevant dynamics.
While this study is of intrinsic interest, it also lays the foundation for potential applications and connections to the theory of random walks, particularly convolution random walks. Model theory often finds its applications to other areas of mathematics by examining specific mathematical structures through the lens of structurally significant subsets. In practice, this allows model theorists to construct tame definable analogues of complex classical structures. In future work, we intend to further this theory by constructing definable analogues to the Martin and Poisson boundaries – objects that traditionally describe the behavior of random walks at infinity – as well as a definable analogue of harmonic analysis. This paper serves as a foundational first step towards these broader goals. 

Keisler measures are finitely additive probability measures on definable sets or, equivalently, regular Borel probability measures on the corresponding spaces of types. They were introduced by Keisler in \cite{Keisler1} and have played a major role in model theory and its applications since their appearance in the proof of Pillay's conjecture in \cite{HPP1}. For example, see \cite{Anand_Udi2011,HruPiSi13,ArtemPierre, Yaa, Hru, MaWa, HruKruPi, HruKruPi2, Wild, ConGanHan23}. In particular, deep results were obtained for definably amenable groups, i.e., definable groups admitting a left invariant Keisler measure (e.g. see \cite{Anand_Udi2011,ArtemPierre}). In \cite{HruKruPi} and earlier in 
\cite{KrPi,KruNewSim}, some ideas and theorems around definable amenability were adapted to the context of arbitrary theories, where one does not have any definable group around and amenability concerns the group of automorphisms of a sufficiently saturated model of the theory in question.

Keisler measures are strongly related to topological dynamics in model theory. Methods from topological dynamics were introduced to model theory by Newelski in \cite{New2009,newelski12}. 
Since then the topic has been broadened and deepened by many authors, e.g. see 
\cite{Pil13,ArtemPierre, KruPil_Generalized_Bohr, BYT,KruNewSim,KPRz}. This led not only to essentially new results in model theory (e.g. on the complexity of strong types in \cite{KPRz, KrRz}) but also to applications to additive combinatorics in \cite{KrPi2}. Important interactions between Keisler measures and topological dynamics were studied in \cite{ArtemPierre,HruKruPi,HruKruPi2}. Topological dynamics also plays an essential role in the theory of the definable convolution product developed in \cite{Artem_Kyle,Artem_Kyle2, CGK}.
This theory of convolution treats classification of idempotents as one of its central problems.
We remark that the convolution product of Keisler measures on a definable group is a natural counterpart of the classical convolution product of regular Borel probability measures on a locally compact group (which is recalled below).

In this paper, we define a variant of the convolution product for global Borel-definable Keisler measures over arbitrary first-order theories and develop a theory of this product. Then, via the ``adding an affine sort'' construction, we note that our various results imply the corresponding results already known for definable groups. 
In fact, our convolution product generalizes the one  in the definable group setting, and our results actually generalize important fragments of the theory developed in  \cite{Artem_Kyle,Artem_Kyle2, CGK} from the context of definable groups to the 
context of arbitrary first-order theories. 
This is consistent with the general phenomenon that various aspects of the model theory of definable groups can be adapted to the context of arbitrary theories where the automorphism groups of their sufficiently saturated models are playing the role of definable groups. In some situations, generalizing aspects of the theory of definable groups to automorphism groups is quite straightforward, but in other situations even formulating the correct definitions and statements in 
the automorphism group context is a nontrivial task.

\subsection{New convolution product}
Let us recall the classical version of the convolution product of measures. If $G$ is a locally compact group and $\mathcal{M}(G)$ is the space of regular Borel probability measures on $G$, one  extends group multiplication on $G$ to  \emph{convolution} $\ast$ on $\mathcal{M}(G)$: if $\mu, \nu \in \mathcal{M}(G)$ and $B$ is a Borel subset of $G$, then
\begin{equation*}
    (\mu * \nu) (B) = \int_{G} \int_{G} \mathbf{1}_{B}(x \cdot y) d\mu(x)d\nu(y).  
\end{equation*}
In \cite{Artem_Kyle,Artem_Kyle2, CGK}, a counterpart of the above product was defined and studied for Keisler measures on definable groups.
 In this paper, we develop a convolution product for Keisler measures over arbitrary theories. This new product is natural but not so obvious; we give an intuitive explanation below.

Let $M$ be a strongly $\aleph_0$-homogeneous first-order structure and let $\FC \succeq M$ be a monster model, i.e. a $\kappa$-saturated and strongly $\kappa$-homogeneous elementary extension of $M$ for a strong limit cardinal $\kappa>|M| + |L|$, where $L$ is the language of $M$. (For many purposes, it is enough to assume that $\FC \succeq M$ is just $|M|^+$-saturated and strongly $|M|^+$-homogeneous.) 
Let $\bar m$ be an enumeration of $M$ and $S^{\textrm{fs}}_{\bar m}(\FC,M)$ be the space of complete types over $\FC$ which both concentrate on $\tp(\bar m/\emptyset)$ and are finitely satisfiable in $M$. Let $\bar x$ be a tuple of variables corresponding to $\bar m$. 

The starting point for us was Proposition 3.14 of \cite{Udi_def_patterns}, which yields a correspondence (if fact, a homeomorhism) between the semigroup of endomorphisms of $S_{\bar x}(M)$ with respect to the so-called ``definability patterns structure'' and the space $S^{\textrm{fs}}_{\bar m}(\FC,M)$. 
This allows one to induce a left topological semigroup structure on  $S^{\textrm{fs}}_{\bar m}(\FC,M)$. Restricting the domain of Hrushovski's correspondence to the Ellis semigroup of the flow $(S_{\bar x}(M), \aut(M))$, we computed that the target of the correspondence becomes the space of all types $p \in S_{\bar m}(\FC)$ such that for every formula $\varphi(\bar x;\bar b) \in p$ there exists $\bar a \in M^{\bar x}$ with $\bar a \equiv \bar m$ and $\models \varphi(\bar a;\bar b)$. We call such types {\em strongly finitely satisfiable} (or {\em sfs}) {\em in $M$}  and denote them by $S^{\textrm{sfs}}_{\bar m}(\FC,M)$. 

A natural next step was to extend this restricted correspondence to the context of measures. So let $\mathfrak{M}^{\textrm{sfs}}_{\bar m}(\FC,M)$ be the space of Keisler measures over $\FC$ such that the support of the measure is contained in $S^{\textrm{sfs}}_{\bar m}(\FC,M)$. 
The first main result of this paper yields, under NIP, a correspondence between the Ellis semigroup of the flow $(\mathfrak{M}_{\bar x}(M),\conv(\aut(M))$ (where $\conv(\aut(M))$ denotes the convex hull of $\aut(M)$) and the space $\mathfrak{M}^{\textrm{sfs}}_{\bar m}(\FC,M)$. In fact, we obtain such correspondence not only for $\aut(M)$ but also 
for suitable subgroups of $\aut(M)$ (see the beginning of Subsection \ref{sec:analogon}) and the corresponding subspaces of $\mathfrak{M}^{\textrm{sfs}}_{\bar m}(\FC,M)$. Our correspondence restricted to the Ellis semigroup of the flow $(S_{\bar x}(M), \aut(M))$ coincides with the restricted Hrushovski's correspondence on types mentioned in the previous paragraph. Using our correspondence, we induce a left continuous semigroup operation $*$ on $\mathfrak{M}^{\textrm{sfs}}_{\bar m}(\FC,M)$ (cf. Definition \ref{def:star.product.definition.0}) which can be given explicitly in terms of integrals (cf. Proposition  \ref{prop:star.product.1} and Definition \ref{def:star.definition}); the operation $*$ restricted to types can be given by an explicit formula (cf. Proposition \ref{prop:formula.for.star.on.types}). We often refer to this convolution product simply as the $*$-product.

Next, we generalize the $*$-product to global $M$-invariant Keisler measures which are Borel-definable over $M$, without the assumption that $M$ is strongly $\aleph_0$-homogeneous (cf. Definition \ref{def:star.definition}). 
The main context of interest is NIP, in which case a result by Hrushovki and Pillay demonstrates that all global $M$-invariant Keisler measures are Borel-definable over $M$, making harmonic analysis in this setting feasible.  We show that the $*$-product is affine in each coordinate, and left continuous under NIP. However, due to the local nature of the definition, proving associativity of the $\ast$-product is a very non-trivial task. The following question remains open.

\begin{question}\label{question: associativity of *}
Is $*$ associative in NIP theories?
\end{question}

\noindent
While this question is left open in full generality, we give positive answers for large classes of measures. In particular, the convolution product is associative on:
\begin{enumerate}
\item $S_{\bar m}^{\textrm{inv}}(\FC,M)$, i.e. complete types over $\FC$ concentrated on $\tp(\bar m/\emptyset)$ which are invariant over $M$  (even without NIP and without any definability assumptions on the types in question);
\item $\mathfrak{M}^{\textrm{def}}_{\bar m}(\FC,M)$, i.e. Keisler measures over $\FC$ concentrated on $\tp(\bar m/\emptyset)$ which are {\em definable over $M$} (with no NIP assumption);
\item $\mathfrak{M}^{\textrm{fs}}_{\bar m}(\FC,M)$, i.e. Keisler measures over $\FC$ concentrated on $\tp(\bar m/\emptyset)$ which are {\em finitely satisfiable in $M$} (under NIP); note that this space contains $\mathfrak{M}^{\textrm{sfs}}_{\bar m}(\FC,M)$ considered above;
\item the entire $\mathfrak{M}^{\textrm{inv}}_{\bar m}(\FC,M)$ when $T$ and $M$ are countable (under NIP), 
i.e. the space of Keisler measures over $\FC$ concentrated on $\tp(\bar{m}/\emptyset)$ and {\em invariant over $M$}.
\end{enumerate}

One could try to extend Hrushovski's definability patterns approach from types to measures.
In particular, one can hope to reprove item (3) above via a Hrushovski-style correspondence between the semigroup of endomorphisms with respect to a suitably defined definability patterns structure on $\mathfrak{M}_{\bar{x}}(M)$ and the space $\mathfrak{M}^{\textrm{fs}}_{\bar m}(\FC,M)$, and then try to extend it to $\mathfrak{M}^{\inv}_{\bar m}(\FC,M)$ in place of $\mathfrak{M}^{\textrm{fs}}_{\bar m}(\FC,M)$ by adjusting the definability patterns structure.
However, in this paper we decided to focus on convolution products and postpone developing the theory of definability patterns for measures to a future paper.

One argument that our notion of convolution product is natural is the aforementioned correspondence for sfs measures through which the obtained semigroup $(\mathfrak{M}^{\textrm{sfs}}_{\bar m}(\FC,M),*)$ is isomorphic to the Ellis semigroup of the corresponding flow. 
Another point is that elaborating on the ``adding an affine sort'' construction, we can recover the various results from \cite{Artem_Kyle,Artem_Kyle2,CGK} about the convolution product of Keisler measures on definable groups from the corresponding results for our convolution product (cf. Section \ref{sec: affine sort}).

To summarize, in the first part of the paper we develop the basic theory of the convolution product of invariant Keisler measures.

\subsection{Correspondence for idempotent measures}
In the second part of the paper, our main goal is to classify idempotent (in the sense of our convolution product) {\em frequency interpretation measures} (in short {\em fim}; see Definition \ref{def:fim}) in the most general possible settings.
A measure $\mu$ is \emph{idempotent} if $\mu * \mu = \mu$. A classical line of work established a correspondence between compact subgroups of a locally compact group $G$ and idempotent measures in $\mathcal{M}(G)$, in progressively broader contexts 
\cite{kawada1940probability,Wendel,Rudin,Glicksberg2,Cohen} culminating in the following:

\begin{fact}\cite[Theorem A.4.1]{Pym62}\label{fact: classical Pym} Let $G$ be a locally compact group and $\mu \in \mathcal{M}(G)$. Then the following are equivalent: 
\begin{enumerate}
    \item $\mu$ is idempotent; 
    \item the support $\supp(\mu)$ of $\mu$ is a compact subgroup of $G$ and $\mu|_{\supp(\mu)}$ is the normalized Haar measure on $\supp(\mu)$. 
\end{enumerate}
\end{fact}

\noindent
One of the main motivating problems in \cite{Artem_Kyle,Artem_Kyle2, CGK} was the following counterpart of the above result for definable groups. 
By $\mathfrak{M}^{\textrm{inv}}_G(\FC,M)$, we denote the space of Keisler measures over $\FC$ concentrated on $G$ and {\em invariant over $M$}.

\begin{conj}\label{conjecture: main conjecture for definable groups}
Let $G=G(\FC)$ be a definable group and $\mu \in \mathfrak{M}^{\textrm{inv}}_G(\FC,M)$ be fim over $M$. We know that then the right stabilizer $\stab(\mu)$ of $\mu$ is type-definable over $M$.  Then the following are conjectured to be equivalent:
\begin{enumerate}
\item $\mu$ is idempotent (with respect to definable convolution);

\item $\mu$ is the unique right $G$-invariant (equivalently, the unique left $G$-invariant) Keisler measure concentrated on $\stab(\mu)$.
\end{enumerate}
In particular, there is a correspondence between idempotent fim measures in $\mathfrak{M}^{\textrm{inv}}_G(\FC,M)$ and type-definable over $M$ fim subgroups of $G(\FC)$.
\end{conj}

\noindent
In  \cite{Artem_Kyle,Artem_Kyle2, CGK}, this conjecture was confirmed in several interesting situations: $G$ is definable in a stable theory; $G$ is definable in a NIP theory and the idempotent measure in question is $G^{00}$-invariant; $G$ abelian; $\mu$ is a type and the ambient theory is rosy; $\mu$ is a stable type.

In the present paper, we formulate a generalization of the above conjecture to the context of theories (with no definable group involved). In the context of theories, type-definable subgroups of the monster model are replaced with 
relatively type-definable subgroups of the automorphism group. Both here and in the whole paper, relatively type-definable subsets of $\aut(\FC)$ play a key role. They were defined in \cite{HruKruPi}, extending the notion of relatively definable subsets of $\aut(\FC)$ defined in \cite{KPRz}. 
Recall that for a short tuple $\bar a$ in $\FC$, a {\em relatively $\bar a$-type-definable over $B$ subset of $\aut(\FC)$} is a subset of the form $\{\sigma \in \aut(\FC)\;\colon\;\models \pi(\sigma(\bar a);\bar b)\}$ for some partial type $\pi(\bar x;\bar y)$ without parameters and a short tuple $\bar b$ in $B$. 
Here, we consider the case of $B=M$ and $\bar a =\bar m$ an enumeration of $M$, and our main focus lies on the relatively $\bar{m}$-type-definable over $M$ subgroups of $\aut(\FC)$, i.e. on the subgroups of the form
$$G_{\pi,\FC}:=\{\sigma \in \aut(\FC)\;\colon\;\models \pi(\sigma(\bar m); \bar m)\}$$ 
(where without loss of generality the partial type $\pi(\bar{x};\bar{y})$ contains ``$\bar{x}\equiv_{\emptyset}\bar{y}$''). Such groups are 
precisely the closed subgroups of $\aut(\FC)$ in the {\em relatively definable over $M$ topology} (see the paragraph after Definition \ref{definition: relatively type-definable sets}).

By $\mathfrak{M}^{\textrm{inv}}_{\pi(\bar x;\bar m)}(\FC,M)$ we denote the space of global Keisler measures which are invariant over $M$ and concentrated on $\pi(\bar x;\bar m)$. Note that $\aut(\FC)$ acts naturally on the left on $\mathfrak{M}_{\bar x}(\FC)$ and we show that the stabilizer $\stab(\mu)$ is relatively $\bar m$-type-definable over $M$ for any $\mu \in \mathfrak{M}^{\textrm{def}}_{\bar x}(\FC,M)$ (cf. Lemma \ref{lemma: rel. type-definability}). For the notion of a {\em fim} relatively type-definable subgroup of $\aut(\FC)$ see Definition \ref{def: fim subgroups}. Now, we have all ingredients needed to formulate our main conjecture:

\newtheorem*{conj:0.7}{\bf{Conjecture} \ref{conjecture: main conjecture}}
\begin{conj:0.7}
\label{conjecture: main conjecture in the introduction}
Let $\mu \in \mathfrak{M}^{\inv}_{\bar m}(\FC,M)$ be fim over $M$. 
We know that $\stab(\mu)=G_{\pi,\FC}$ for some partial type $\pi(\bar x;\bar y)\vdash\bar{x}\equiv_{\emptyset}\bar{y}$.
Then it is conjectured that the following conditions are equivalent: 
\begin{enumerate}
\item $\mu$ is an idempotent (with respect to the \emph{new} convolution product);
\item $\mu$ is the unique (left) $G_{\pi,\FC}$-invariant measure in $\mathfrak{M}^{\inv}_{\pi(\bar{x};\bar{m})}(\FC,M)$.
\end{enumerate}
In particular, there is a correspondence between idempotent fim measures in $\mathfrak{M}^{\inv}_{\bar m}(\FC,M)$ and relatively $\bar m$-type-definable over $M$ fim subgroups of $\aut(\FC)$.
\end{conj:0.7}

\noindent
In this paper, we prove the above conjecture in all situations in which Conjecture \ref{conjecture: main conjecture for definable groups} was proved, except the abelian case (note that groups of automorphisms are typically nonabelian).

The main result of the second part of the paper is Corollary \ref{cor: 0.7 for stable}, i.e. confirmation of Conjecture \ref{conjecture: main conjecture} for stable theories. 
This requires proving a counterpart for stable theories of a theorem of Newelski on stable groups 
\cite[Theorem 2.3]{New89}. Namely, we prove Theorem \ref{theorem: counterpart of Newelski's theorem} in which, roughly speaking, we describe the smallest relatively $\bar m$-type-definable over $M$ subgroup of $\aut(\FC)$ containing a given relatively $\bar m$-invariant over $M$ subset of $\aut(\FC)$ (i.e. a subset of the form $\{ \sigma \in \aut(\FC)\;\colon\; \tp (\sigma(\bar m)/M)) \in P\}$ for some $P \subseteq S_{\bar x}(M)$). This in turn requires developing a counterpart of some fundamentals of stable group theory for the group of automorphisms of a monster model of any stable theory. Having Theorem \ref{theorem: counterpart of Newelski's theorem} among the main tools, it still requires quite a lot of work to prove Conjecture \ref{conjecture: main conjecture} for stable theories.

Elaborating on the ``adding an affine sort'' construction, we show that
Conjecture \ref{conjecture: main conjecture} implies Conjecture \ref{conjecture: main conjecture for definable groups}. So our results confirming Conjecture \ref{conjecture: main conjecture} in the various situations imply the corresponding results from \cite{Artem_Kyle,Artem_Kyle2, CGK}.

\subsection{Structure of text}
In the preliminaries, Section \ref{sec: preliminaries}, we recall the relevant definitions and facts concerning Keisler measures and relatively type-definable subsets of $\aut(\FC)$. In particular, we observe that the stabilizer of any global Keisler measure definable over $M$ is relative $\bar m$-type-definable over $M$, and an analogous statement for invariant measures. We also introduce the notion of generically stable and fim relatively type-definable subgroups of $\aut(\FC)$.

In Section \ref{sec: injectivity}, using approximation of measures by types in NIP theories, we show in several situations that the restriction maps from Ellis semigroups of some flows of measures to the corresponding spaces of types are injective or even homeomorphisms. This will be important for transfering the product from the Ellis semigroup to the appropriate space of types in Section \ref{sec:star.product}. 

Section \ref{sec:star.product} is one of the main parts of the paper. 
Here, we establish the aforementioned correspondence (homeomorphism) for NIP theories  between the Ellis semigroup of the flow $(\mathfrak{M}_{\bar x}(M),\conv(\aut(M)))$ and the space $\mathfrak{M}^{\textrm{sfs}}_{\bar m}(\FC,M)$, which induces a left topological semigroup structure on  $\mathfrak{M}^{\textrm{sfs}}_{\bar m}(\FC,M)$ (in Subsection \ref{sec:analogon}). 
We then extend the definition of our product to a very wide class of global $M$-invariant Keisler measures, showing
that it yields left topological semigroup structures on important subspaces of $\mathfrak{M}^{\textrm{inv}}_{\bar m}(\FC,M)$ (in Subsection \ref{subsec:star.definitions}).

In Section \ref{sec: affine sort}, we elaborate on the ``adding an affine sort'' construction. We essentially show that if we are given a definable subgroup $G$ in a structure $M$, then we may construct a new structure $\bar{M}$ so that the group $G(\FC)$ lives in a particular relatively type-definable subgroup of the group of automorphisms of $\bar \FC \succeq \bar M$ in such a way 
that there are homeomorphisms between the spaces of auxiliary objects (i.e., the relevant spaces of types and measures). Since both spaces admit a convolution product, we show that they are isomorphic as topological semigroups (for measures under the assumption of NIP). This connection allows us to deduce the results for definable groups from the corresponding results that we prove for general theories.

Besides defining and analyzing fundamental properties of the convolution product, the second main contribution of this paper is, starting from Section \ref{sec:fim idempotent}, focused around classification of idempotent fim measures (including idempotent generically stable types as a special case). 

In Subsection \ref{subsec:6.1}, we study fim relatively type-definable subgroups of $\aut(\FC)$. We observe that the implication $(2)\Rightarrow(1)$ in Conjecture \ref{conjecture: main conjecture} always holds, and then we show uniqueness in the second point of Conjecture \ref{conjecture: main conjecture}, i.e. uniqueness 
of left invariant measures in some general contexts, e.g. in all NIP theories. 
In Subsection \ref{subsec:6.2}, we observe that the appropriate proofs from \cite{CGK} can be easily adapted to show Conjecture \ref{conjecture: main conjecture} for types in rosy theories and for stable types. 
In Subsection \ref{subsec:6.3}, we prove Conjecture \ref{conjecture: main conjecture} in NIP theories for measures which are additionally invariant under the Kim-Pillay strong automorphisms (moreover, under this assumption, one can drop the assumption that the measure is fim, but then uniqueness in item (2) of Conjecture \ref{conjecture: main conjecture} should also be removed). This result again implies the corresponding result for definable groups.

In Section \ref{sec:group chunk}, we first develop an analog of basic stable group theory for the group of automorphisms of the monster model $\FC$ of any stable theory. Then, working in the developed context, we prove Theorem \ref{theorem: counterpart of Newelski's theorem}, a counterpart of \cite[Theorem 2.3]{New89}. 
This theorem may have many interesting applications, e.g. in the spirit of the results from \cite{New89}. 
In this paper, we apply Theorem \ref{theorem: counterpart of Newelski's theorem} to prove Conjecture \ref{conjecture: main conjecture} for stable theories, which is done in Subsection \ref{subsec:0.7 for stable}. In earlier subsections of Section \ref{sec: 0.7 for stable}, we prove some preparatory results. In particular, we show a strong uniqueness result for left or right invariant Keisler measures in stable theories.

\subsection{Acknowledgments} 
We would like to thank Ludomir Newelski for correcting the proof of Theorem 2.3 from \cite{New89}.

\section{Preliminaries}\label{sec: preliminaries}

\subsection{Miscellaneous}
In this paper, compact spaces are Hausdorff by definition.
If $r, s, \epsilon \in \mathbb{R}$ with $\epsilon > 0$, we write $r \approx_{\epsilon} s$ to mean $|r - s| < \epsilon$. 

\begin{definition} If $V$ is a real vector space and $X \subseteq V$, we let $\conv(X)$ be the collection of finite convex combinations of elements from $X$, i.e., 
\begin{equation*} 
\conv(X) := \left\{ \sum_{i \leq n} r_i x_i : n \in \mathbb{N}_{\geq 1}, r_i \in [0,1], \sum_{i \leq n} r_i = 1, x_i \in X \right\}.
\end{equation*} 
\end{definition}

In general if $X$ is a compact space, we let $\mathcal{M}(X)$ denote the space of regular Borel probability measures on $X$. We recall that $\mathcal{M}(X)$ is a subset of a Banach space, namely the space of all signed regular Borel probability measures on $X$. Thus it makes sense to discuss $\conv(Y)$ when $Y \subseteq \mathcal{M}(X)$. If $x \in X$, we let $\delta_{x}$ denote the Dirac measure at $x$, i.e. $\delta_{x}(B) = 1$ if and only if $x \in B$. 

If $X$ and $Y$ are compact spaces, $f: X \to Y$ is a continuous map, and $\mu \in \mathcal{M}(X)$, then the {\em pushforward} of $\mu$, denoted $f_{*}(\mu)$, is an element of $\mathcal{M}(Y)$ such that for any Borel set $B$, $f_{*}(\mu)(B) = \mu(f^{-1}[B])$.

\subsection{Ellis semigroups}\label{subsection: Ellis semigroups} We recall some basic conventions related to Ellis semigroups. 
For more details consult \cite{Auslander}. We remark that we are interested in Ellis semigroups of families of maps which are not necessary homeomorphisms. 

Let $(X,S,\pi)$ be a triple where $X$ is a compact (Hausdorff) space, $S$ is a semigroup, and $\pi$ is an action of $S$ on $X$ by continuous maps, i.e., the map $\pi \colon S \times X \to X$ is a semigroup action and for each $s \in S$, $\pi(s,-) \colon X \to X$ is a continuous map 
(note that we do not require $\pi$ to be jointly continuous). Typically, we will write $(X,S,\pi)$ as $(X,S)$ when the action $\pi$ is understood. 
One can consider the composition semigroup $(\{\pi(s,-) \colon s \in S\},\circ\})$ as a subset of $X^{X}$. 
If we equip $X^{X}$ with the standard product topology (also known as the topology of pointwise convergence), then the {\em Ellis semigroup} of the action $(X,S,\pi)$ is precisely the topological closure of $\{\pi(s,-) \colon s \in S\}$ in $X^{X}$.
It is a compact {\em left topological semigroup} (i.e., $\circ$ is left continuous), which will be denoted by $\E(X,S)$.

In this paper, we will often be concerned with \emph{affine actions} on spaces of measures. Suppose that $G$ is a (discrete) group which acts on a compact Hausdorff space $X$ by homeomorphisms via $\pi$. 
We embed $G$ in the Banach space of measures on $\mathcal{P}(G)$ via $g \to \delta_{g}$ and let
\begin{equation*} 
\conv(G) := \conv\left( \{\delta_{g} : g \in G \} \right).
\end{equation*} 
In practice, we often identify $\delta_{g}$ with $g$ in this setting. The action of $G$ extends to an action on $\mathcal{M}(X)$ via the pushforward operation. 
Moreover, this action extends linearly from $G$ to an action $\tilde \pi \colon \conv(G) \times \mathcal{M}(X) \to \mathcal{M}(X)$ given by 
\begin{equation*} 
\tilde{\pi} \left(\sum_{i=1}^{n} r_i \delta_{g_i},\mu\right) = \sum_{i=1}^{n} r_i (\pi_{g_i})_{*}(\mu),
\end{equation*} 
where $\pi_{g_i}(-) := \pi(g_i,-)$. We are often interested in $\E(\mathcal{M}(X), \conv(G),\tilde{\pi})$.

\subsection{Types and Keisler measures} We recall some basic notation and fix some model-theoretic conventions. Throughout the text, we let $\CL$ be a language and $T$ be a complete first-order $\CL$-theory with infinite models. We let $\FC$ be a {\em monster model} of $T$, i.e., $\FC$ is $\kappa$-saturated and strongly $\kappa$-homoegeneous for some \emph{large enough} $\kappa$. 
A {\em small subset} of $\FC$ is a subset of cardinality less than $\kappa$; a {\em short tuple} is a tuple of length less than $\kappa$. We use the symbols $x,y,z$ to denote singletons of variables and use $\bar{x},\bar{y},\bar{z}$ to denote 
short tuples of variables. We emphasize that these tuples are possibly infinite. If $\bar{x}$ is a tuple of variables and $A$ is a subset of $\FC$, then an $\mathcal{L}_{\bar{x}}(A)$-formula is a formula whose free variables are among $\bar{x}$ and whose parameters are from $A$. We often write $\mathcal{L}_{\bar{x}}(A)$-formulas like `$\varphi(\bar{x};\bar{a})$', even though formally, there are only finitely many free variables occurring in $\varphi(\bar{x};\bar{a})$. An $\mathcal{L}_{\bar{x}}$-formula is an $\mathcal{L}_{\bar{x}}(\emptyset)$-formula, and when there is no possibility of confusion, we simply use the term $\mathcal{L}$-formula. For any subset $A \subseteq \FC$, we let $S_{\bar{x}}(A)$ denote the space of types in variables $\bar{x}$ over parameters from $A$. 
A Keisler measure (in variables $\bar{x}$ over parameters from $A$) is a finitely additive probability measure on the $A$-definable subsets of $\FC$. 
Equivalently, a Keisler measure is a finitely additive probability measure on $\mathcal{L}_{\bar{x}}(A)$ modulo logical equivalence, i.e., a Keisler measure gives the same value to formulas which define the same subsets of the monster. We identify definable sets with the formulas which define them. We let $\mathfrak{M}_{\bar{x}}(A)$ denote the space of Keisler measures in variables $\bar{x}$ over parameters from $A$. We remark that there is a one-to-one correspondence between $\mathfrak{M}_{\bar{x}}(A)$ and regular Borel probability measures on $S_{\bar{x}}(A)$, 
namely $\mathcal{M}(S_{\bar{x}}(A))$. 
We often freely identify a Keisler measure with its corresponding regular Borel probability measure. 

We let $M$ denote a small elementary submodel of $\FC$ and let $M \preceq N \preceq \mathfrak{C}$. For any such $N$, we let $\aut(N)$ denote the automorphism group of $N$. Throughout this article, we will mostly be concerned with $\aut(\FC)$, but here we make some more general statements. The group $\aut(N)$ naturally acts on $S_{\bar{x}}(N)$ by permuting parameters, i.e., 
if $\sigma \in \aut(N)$ and $p \in S_{\bar{x}}(N)$, then  $\sigma \cdot p = \{\varphi(\bar{x};\sigma(\bar{b})): \varphi(\bar{x};\bar{b}) \in p\}$. This group action can be naturally extended to an action on the space of Keisler measures $\mathfrak{M}_{\bar{x}}(N)$ by considering the pushforward of each automorphism. 
More explicitly, for any $\sigma \in \aut(N)$ and $\mu \in \mathfrak{M}_{\bar{x}}(N)$, we let $ \sigma\cdot\mu:= \sigma_\ast(\mu)$.
In other words, for any $\mathcal{L}_{\bar{x}}(N)$-formula  $\varphi(\bar{x};\bar{n})$,
$$(\sigma\cdot\mu)\big([\varphi(\bar{x};\bar{n})]\big)=\mu\big([\varphi(\bar{x};\sigma^{-1}(\bar{n}))] \big).$$
For $\mu \in \mathfrak{M}_{\bar{x}}(N)$ we consider the stabilizer
\begin{equation*} 
\stab(\mu) := \{\sigma \in \aut(N) : \sigma \cdot \mu = \mu\}. 
\end{equation*} 
Furthermore, as described in Subsection \ref{subsection: Ellis semigroups}, we may also consider the action of $\conv(\aut(N))$ on $\mathfrak{M}_{\bar{x}}(N)$ by extending the action linearly. More explicitly, if $\sum_{i \leq n} r_i \delta_{\sigma_i} \in \conv(\aut(N))$, $\mu \in \mathfrak{M}_{\bar{x}}(N)$, and $\varphi(\bar{x};\bar{b})$ is a $\mathcal{L}_{\bar{x}}(N)$-formula, then
\begin{equation*} 
\left( \left( \sum_{i \leq n} r_i \delta_{\sigma_i} \right) \cdot \mu \right)(\varphi(\bar{x};\bar{b})) = \sum_{i \leq n} r_i \mu(\varphi(\bar{x};\sigma_{i}^{-1}(\bar{b}))). 
\end{equation*} 
As convention, if $\sigma \in \aut(N)$ and $\lambda \in \conv(\aut(N))$, we may also write $\sigma \cdot \mu$ as $\sigma(\mu)$ or $\lambda \cdot \mu$ as $\lambda(\mu)$ without confusion. 

The topology on $\mathfrak{M}_{\bar{x}}(N)$ is the induced topology from the space $[0,1]^{\mathcal{L}_{\bar{x}}(N)}$ endowed with the product topology. A basic open set is of the form,
\begin{equation*}
\bigcap_{i=1}^{n} \{ \mu \in \mathfrak{M}_{\bar{x}}(N) : r_i < \mu(\varphi_i(\bar{x})) < s_i\}.
\end{equation*} 
where $r_1,\ldots.,r_n,s_1,\ldots.,s_n$ are real numbers and $\varphi_1(\bar{x}),\ldots,\varphi_n(\bar{x})$ are $\mathcal{L}_{\bar{x}}(N)$-formulas.

\vspace{.1in}

We now recall some basic definitions and properties about Keisler measures. Almost all of the following definitions were originally defined for types and extended to the context of measures. 

\begin{definition}\label{cheat} Let $\mu \in \mathfrak{M}_{\bar{x}}(\FC)$ and $M \preceq \mathfrak{C}$ be small. We emphasize that the tuple $\bar{x}$ is possibly infinite.
\begin{enumerate}
    \item 
    The measure $\mu$ is $M$-invariant if for every $\mathcal{L}$-formula $\varphi(\bar{x};\bar{y})$, for any $\bar{a}, \bar{b} \in \FC^{\bar{y}}$ such that $\bar{a} \equiv_{M} \bar{b}$ we have that $\mu(\varphi(\bar{x};\bar{a})) = \mu(\varphi(\bar{x};\bar{b}))$. We let $\mathfrak{M}_{\bar{x}}^{\inv}(\FC,M)$ be the collection of measures in $\mathfrak{M}_{\bar{x}}(\FC)$ which are $M$-invariant. We recall that this set is a closed subset of $\mathfrak{M}_{\bar{x}}(\FC)$. 
    \item Let $\mu \in \mathfrak{M}^{\inv}_{\bar{x}}(\FC,M)$. Then for any $\mathcal{L}(M)$-formula $\varphi(\bar{x};\bar{y})$, we define the map $F_{\mu,M}^{\varphi}: S_{\bar{y}}(M) \to [0,1]$ via $F_{\mu,M}^{\varphi}(q) = \mu(\varphi(\bar{x};\bar{b}))$ where $\bar{b} \models q$. We remark that this map is well-defined since $\mu$ is $M$-invariant. 
    \item Let $\mu \in \mathfrak{M}^{\inv}_{\bar{x}}(\FC,M)$. We say that $\mu$ is Borel-definable (over $M$) if for every $\mathcal{L}$-formula $\varphi(\bar{x};\bar{y})$, the map $F_{\mu,M}^{\varphi}$ is a Borel function. 
    \item Let $\mu \in \mathfrak{M}^{\inv}_{\bar{x}}(\FC,M)$. We say that $\mu$ is definable (over $M$) if for every $\mathcal{L}$-formula $\varphi(\bar{x};\bar{y})$, the map $F_{\mu,M}^{\varphi}$ is a continuous function. We let $\mathfrak{M}_{\bar{x}}^{\df}(\FC,M)$ denote the collection of global measures which are $M$-definable. 
    \item Let $\mu \in \mathfrak{M}^{\inv}_{\bar{x}}(\FC,M)$. We say that $\mu$ is finitely satisfiable over $M$ if for every $\mathcal{L}(\FC)$-formula $\varphi(\bar{x};\bar{b})$, if $\mu(\varphi(\bar{x};\bar{b})) > 0$, then there exists some $\bar{a} \in M^{\bar{x}}$ such that $\FC \models \varphi(\bar{a};\bar{b})$. We let $\mathfrak{M}_{\bar{x}}^{\fs}(\FC,M)$ be the collection of measures in $\mathfrak{M}_{\bar{x}}(\FC)$ which are finitely satisfiable in $M$. We recall that this set is a closed subset of $\mathfrak{M}_{\bar{x}}(\FC)$. 
    \item Let $\mu \in \mathfrak{M}^{\inv}_{\bar{x}}(\FC,M)$ and $\nu \in \mathfrak{M}_{\bar{y}}(\FC)$. Suppose that $\mu$ is Borel-definable. We define the Morley product of $\mu$ with $\nu$, denoted $\mu \otimes \nu$, as follows: For any $\mathcal{L}(\FC)$-formula $\varphi(\bar{x};\bar{y})$, we have that 
    \begin{equation*}
        (\mu_{\bar{x}} \otimes \nu_{\bar{y}})(\varphi(\bar{x};\bar{y})) = \int_{S_{\bar{y}}(M_0)} F_{\mu_{\bar{x}},M_0}^{\varphi(\bar{x};\bar{y})} d\nu_{\bar{y}}|_{M_0},
    \end{equation*}
    where $M_0$ is any small model containing $M$ and all the parameters occurring in $\varphi$. 
The measure $\nu|_{M_0}$ is the regular Borel probability measure on $S_{\bar{y}}(M_0)$ corresponding to the restriction of $\nu$ to $\mathcal{L}_{\bar{y}}(M_0)$. We remark that this product is well-defined. In practice, we often drop the $M_0$ from the notation when there is no possibility of confusion, e.g.\ $F_{\mu}^{\varphi(\bar{x};\bar{y})}$ is used instead of $F_{\mu,M_0}^{\varphi(\bar{x};\bar{y})}$ and $\nu$ is used instead of $\nu|_{M_0}$. We will also sometimes write $\mu_{\bar{x}}$ as $\mu$ and $\nu_{\bar{y}}$ as $\nu$ again when there is no possibility of confusion. As convention, we will often use the order of the variables in the formula to encode whether or not the function $F_{\mu}^{\varphi(\bar{x};\bar{y})}$ is a fiber function in $\bar{x}$ or $\bar{y}$. 
Namely, $F_{\mu}^{\varphi(\bar{x};\bar{y})}$ implies that $\mu$ is a measure in variables $\bar{x}$. We will often write $F_{\mu}^{\varphi^{\textrm{opp}}(\bar{y};\bar{x})}$ (which implies that $\mu$ is in variables $\bar{y}$), where $\varphi^{\textrm{opp}}(\bar{y};\bar{x})$ is the formula $\varphi(\bar x;\bar y)$ but with switched roles of variables.

    \item If $\mu \in \mathfrak{M}^{\inv}_{\bar{x}}(\FC,M)$ is a Borel-definable Keisler measure, then we define the iterated Morley products as follows:
   \begin{enumerate} 
    \item $\mu^{(1)}_{\bar{x}_1} = \mu_{\bar{x}}$. 
   \item $\mu^{(n+1)}_{\bar{x}_1,\ldots,\bar{x}_{n+1}} = \mu_{\bar{x}_{n+1}} \otimes \mu_{\bar{x}_1,\ldots,\bar{x}_n}$. 
    \item $\mu^{(\omega)} = \bigcup_{n = 1}^{\omega} \mu^{(n)}_{\bar{x}_1,\ldots,\bar{x}_n}$. 
    \end{enumerate} 
\end{enumerate}
\end{definition}

\begin{remark}\label{remark:Morley_convention} Suppose that $\mu \in \mathfrak{M}^{\inv}_{\bar{x}}(\mathfrak{C},M)$, $\mu$ is Borel-definable over $M$, $\nu$ and $\nu'$ are measures in $\mathfrak{M}_{\bar{y}}(\mathfrak{C})$ such that $\nu|_{M} = \nu'|_{M}$ and $\varphi(\bar{x};\bar{y},\bar{z})$ is an $\mathcal{L}_{\bar{x},\bar{y},\bar{z}}$-formula. Then for any $\bar{b} \in M^{\bar{z}}$, it follows that
\begin{equation*} 
(\mu \otimes \nu)(\varphi(\bar{x};\bar{y},\bar{b})) = (\mu \otimes \nu')(\varphi(\bar{x};\bar{y},\bar{b})).
\end{equation*} 
Therefore, if $\mu$ is as above and $\nu_0$ is a measure in $\mathfrak{M}_{\bar{y}}(M)$, 
we can define the Morley product of $\mu$ and $\nu_0$ as a measure in $\mathfrak{M}_{\bar{x},\bar{y}}(M)$ via, 
\begin{equation*} 
(\mu \otimes \nu_0)(\varphi(\bar{x};\bar{y},\bar{b})) := (\mu \otimes \nu)(\varphi(\bar{x};\bar{y},\bar{b})),
\end{equation*} 
where $\nu$ is any global extension of $\nu_0$. 
\end{remark} 

\begin{remark} Let $p$ be a type in $S_{\bar{x}}(\FC)$. We say that $p$ is invariant over $M$ [definable over $M$, Borel-definable over $M$, finitely satisfiable in $M$] if and only if the corresponding Dirac measure $\delta_{p}$ has such property over $M$. We use $S^{\inv}_{\bar{x}}(\FC,M)$ and $S^{\fs}_{\bar{x}}(\FC,M)$ to denote the space of global $M$-invariant types and the space of global types which are finitely satisfiable in $M$, respectively. 
\end{remark}

\begin{remark} 
The Morley product for types is more general than the Morley product for measures. If $p \in S_{\bar{x}}^{\inv}(\FC,M)$ and $q \in S_{\bar{y}}(\FC)$ are any types, then a formula $\varphi(\bar{x};\bar{y},\bar{b}) \in p \otimes q$ if and only if $\varphi(\bar{x};\bar{a},\bar{b}) \in p$ where $\bar{a} \models q|_{M\bar{b}}$. A variant of Remark \ref{remark:Morley_convention} applies as well. 
\end{remark} 

\begin{definition}\label{definition: definition of p} Suppose that $p \in S_{\bar y}^{\inv}(\FC,M)$ and $\varphi(\bar{x};\bar{y})$ is an $\mathcal{L}$-formula. Then we let the \emph{definition} of $p$ be denoted as,
\begin{equation*} 
d_{p}^{\varphi} := \{\bar{b} \in \FC^{\bar{x}} : \varphi(\bar{b};\bar{y}) \in p\}.  
\end{equation*} 
We remark that when $p$ is Borel-definable over $M$, then the set 
\begin{equation*}
D_{p,M}^{\varphi} := \{ q  \in S_{\bar{x}}(M) : \textrm{there exists } \bar{b} \in \FC^{\bar{x}} \text{ such that } \bar{b} \models q \text{ and } \varphi(\bar{b};\bar{y}) \in p\} 
\end{equation*} 
is a Borel subset of $S_{\bar{x}}(M)$. 
\end{definition} 

We now define \emph{average measures} as well as the support of a measure.

\begin{definition} Suppose that $A \subseteq \FC$ and let $\bar{p} = (p_1,\dots,p_k)$ be a sequence of types in $S_{\bar{x}}(A)$. 
The average measure $\Av(\bar{p})$ in $\mathfrak{M}_{\bar{x}}(A)$ is given by
\begin{equation*} 
\Av(\bar{p})(\varphi(\bar{x})) := \frac{|\{i \leq k : \varphi(\bar{x}) \in p_i \}| }{k}
\end{equation*}
for any $\mathcal{L}_{\bar{x}}(A)$-formula $\varphi(\bar{x})$.
If $p_1,\dots,p_k$ are realized types, say $p_i = \tp(a_i/A)$ where $a_i \in A^{|\bar{x}|}$, we often write $\Av(\bar{p}) = \Av(\bar{a})$ for $\bar{a} := (a_1,\dots,a_k)$. 
\end{definition} 

\begin{definition} Suppose that $A \subseteq \FC$ and $\mu \in \mathfrak{M}_{\bar{x}}(A)$. We let $\supp(\mu)$ denote the support of $\mu$. In other words,
\begin{equation*} 
\supp(\mu) := \{p \in S_{\bar{x}}(A): \mu(\varphi(x)) > 0 \textrm{ for all } \varphi(x) \in p\}. 
\end{equation*} 
\end{definition} 

We remark that for any Keisler measure $\mu$, the support $\supp(\mu)$ is always a non-empty closed subset of $S_{\bar{x}}(A)$. 

The general theory of Keisler measures in the NIP context is well understood. 
The following facts are some general results concerning measures and the Morley product in NIP theories. 

\begin{fact}\label{fact:MP} Suppose that $T$ is NIP and $\mu \in \mathfrak{M}_{\bar{x}}(\FC)$. 
\begin{enumerate} 
\item Then $\mu \in \mathfrak{M}_{\bar{x}}^{\inv}(\FC,M)$ if and only if $\mu$ is Borel-definable over $M$
(see \cite[Corollary 4.9]{Anand_Udi2011}).
\item For any finitely many $\CL$-formulas $\varphi_1(\bar{x};\bar{y}),\dots, \varphi_n(\bar{x};\bar{y})$ and $\epsilon > 0$, there exist types $\bar{p} = (p_1,\dots,p_t)$ in $\supp(\mu)$ such that for every $k \leq n$ 
\begin{equation*} 
\sup_{\bar b \in \FC^{\bar{y}}} |\mu(\varphi_k(\bar{x};\bar b)) - \Av(\bar{p})(\varphi_k(\bar{x};\bar b))| < \epsilon.
\end{equation*} 
(For $n=1$ this is \cite[Proposition 7.11]{Guide_NIP}; for bigger $n$ one can construct a single formula $\varphi(\bar x;\bar y')$ whose instances are (up to equivalence) precisely all instances of the formulas $\varphi_1(\bar x,\bar y),\dots, \varphi_n(\bar x,\bar y)$, and use the case $n=1$.)

\item Fix $\nu \in \mathfrak{M}_{\bar{y}}^{\inv}(\FC,M)$. Then the map $- \otimes \nu_{\bar{y}} : \mathfrak{M}_{\bar{x}}^{\inv}(\FC,M) \to \mathfrak{M}_{\bar{x}\bar{y}}^{\inv}(\FC,M)$ is continuous (see \cite[Theorem 6.3]{Artem_Kyle}). 
\item The Morley product is associative on triples of invariant measures (see \cite[Theorem 2.2]{Gabe_Kyle21}). 
\item Without the NIP assumption, the Morley product is associative on triples of definable measures (see \cite[Proposition 2.6]{ConGann20}).
\end{enumerate} 
\end{fact}

The term \emph{invariantly supported measure} comes from \cite{Artem_Kyle}. The following is essentially Lemma 2.10 from \cite{Artem_Kyle}, but the result is really just a reformulation of a theorem of Hrushovski and Pillay. 

\begin{fact}\label{prop:NIP} Suppose that $T$ is NIP. Let $\mu \in \mathfrak{M}_{\bar{x}}(\FC)$. Then $\mu$ is $M$-invariant if and only if $\mu$ is invariantly supported over $M$, i.e., for every $p \in \supp(\mu)$, $p$ is $M$-invariant. 
\end{fact} 

The following fact is standard and follows directly from the definitions.

\begin{fact} Let $\mu \in \mathfrak{M}_{\bar{x}}^{\inv}(\FC,M)$. The following are equivalent: 
\begin{enumerate}
\item $\mu$ is $M$-definable. 
\item For any $\mathcal{L}$-formula $\varphi(\bar{x};\bar{y})$ and any closed subset $C$ of $[0,1]$, the set 
\begin{equation*} 
\{\bar{b} \in \FC^{\bar{y}}: \mu(\varphi(\bar{x};\bar{b})) \in C\}, 
\end{equation*} 
is an $M$-type definable subset of $\FC^{\bar{y}}$. 
\end{enumerate} 
\end{fact}

We now come to the class of \emph{fim measures}. This class of Keisler measures, which were identified by Hrushovski, Pillay, and Simon in \cite{HruPiSi13}, are quite important (in this paper as well as elsewhere). They are the \emph{tamest} kind of Keisler measures and are often useful in applications.

\begin{definition}\label{def:fim}
    Let $\mu \in \mathfrak{M}_{\bar{x}}^{\inv}(\FC,M)$ be Borel-definable. 
We say that $\mu$ is a {\em fim measure} over $M$ (a {\em frequency interpretation measure} over $M$) if for any finite $\bar{x}' \subseteq \bar{x}$ and $\CL$-formula $\varphi(\bar{x}';\bar{y})$ there exists a sequence of $\mathcal{L}(M)$-formulas $(\theta_n(\bar{x}_1,\ldots, \bar{x}_n))_{1\leqslant n<\omega}$ such that $|\bar{x}_i| = |\bar{x}'|$ and:
    \begin{enumerate}
        \item for any $\epsilon>0$, there exists some integer $n_\epsilon$ such that 
        for every $n \geqslant n_\epsilon$ and every $\bar{a} = (\bar{a}_1,\dots,\bar{a}_n) \in \FC^{(\bar{x}_1,\ldots,\bar{x}_n)}$
        with $\models\theta_n(\bar{a})$ we have
        $$\sup_{b\in \FC^{\bar{y}}}|\Av(\bar{a})(\varphi(\bar{x}';\bar{b}))-\mu(\varphi(\bar{x}';\bar{b}))|<\epsilon,$$
        \item $\lim_{n\to\infty}\mu^{(n)}(\theta_n(\bar{x}_1,\ldots,\bar{x}_n))=1$.
    \end{enumerate}
We say that a fim measure $\mu \in \mathfrak{M}_{\bar{x}}^{\inv}(\FC,M)$ is \emph{super-fim} over $M$ if $\mu^{(n)}$ is fim for every $n \geq 1$. 
\end{definition}

It is easy to check that a fim measure over $M$ is definable over $M$.
Fim has an analogue for types. The property is called \emph{generic stability} and was defined originally by Pillay and Tanovi\'{c} \cite{PiTa}. See \cite[Proposition 3.2]{ConGann20} for a proof of the equivalences. 

\begin{definition}\label{definition: generically stable} Suppose that $p \in S_{\bar{x}}^{\inv}(\FC,M)$. We say that $p$ is {\em generically stable} over $M$ if any of the following equivalent conditions hold:
\begin{enumerate} 
\item For any Morley sequence $(\bar{a}_i)_{i < \omega}$ in $p$ over $M$, $\lim_{i \to \infty} \tp(\bar{a}_i/\FC) = p$. 
\item For any Morley sequence $(\bar{a}_i)_{i < \kappa}$ in $p$ over $M$ and any $\mathcal{L}_{\bar{x}}(\FC)$-formula $\theta(\bar{x})$, $\{i < \kappa : \models \theta(\bar{a}_i)\}$ is finite or cofinite. 
\item The measure $\delta_{p}$ is a fim measures. 
\end{enumerate} 
Additionally, we say that $p$ is \emph{super-generically stable} if for every $n \geq 1$, $p^{(n)}$ is generically stable.
\end{definition}

\begin{question}\label{remark: fim implies super-fim?} It is known that if $T$ is NIP, then any fim measure is super-fim \cite{HruPiSi13}. It is also known that if $T$ is NTP2, then any generically stable type is super-generically stable \cite{ConGanHan23}. In general, whether or not all fim measures are super-fim (or all generically stable types are super-generically stable) is wide open. 
There have been many attempts to give a proof or to find a counter-example to this problem, but all of them contained serious mistakes.
So we ask the question again: Suppose that $\mu$ is a global measure which is fim over $M$. Does this imply that $\mu^{(n)}$ is fim over $M$? If $p$ is a global type which is generically stable over $M$, does this imply that $p^{(n)}$ is generically stable over $M$?
\end{question}

We now describe Keisler measures which concentrate on the type of an enumeration of our small model $M$. 
We first choose an enumeration $\bar{m}$ of $M$ and let $\bar{x}$ be a tuple of variables corresponding to $\bar{m}$; let $\bar{y}$ be another tuple of variables corresponding to $\bar{m}$. Consider a partial type $\pi(\bar{x};\bar{y})$ over $\emptyset$. Then $[\pi(\bar{x};\bar{m})] = \{p \in S_{\bar{x}}(\FC): p \vdash \pi(\bar{x};\bar{m})\}$ is a closed subset of $S_{\bar{x}}(\FC)$. 

\begin{definition} We define the following collections of types and measures: 
\begin{enumerate}
    \item $S_{\pi(\bar{x};\bar{m})}(\FC) := \{p \in S_{\bar{x}}(\FC): p \vdash \pi(\bar{x};\bar{m})\}$. 
    \item For $\dagger \in \{\fs,\inv,\df\}$, we let $S^{\dagger}_{\pi(\bar{x};\bar{m})}(\FC,M) : = S^{\dagger}_{\bar{x}}(\FC,M) \cap S_{\pi(\bar{x};\bar{m})}(\FC)$. 
    \item $\mathfrak{M}_{\pi(\bar{x};\bar{m})}(\FC):=\{\mu\in\mathfrak{M}_{\bar{x}}(\FC)\;\colon\;\mu([\pi(\bar{x};\bar{m})])=1\}$. 
    \item For $\dagger \in \{\fs,\inv,\df\}$, we let $\mathfrak{M}^{\dagger}_{\pi(\bar{x};\bar{m})}(\FC,M) : = \mathfrak{M}^{\dagger}_{\bar{x}}(\FC,M) \cap \mathfrak{M}_{\pi(\bar{x};\bar{m})}(\FC)$. 
    \item If $\pi(\bar{x};\bar{y})$ is the type ``$\bar{x}\equiv_{\emptyset}\bar{y}$'', we denote $S_{\pi(\bar{x};\bar{m})}(\FC)$ as $S_{\bar{m}}(\FC)$ and $\mathfrak{M}_{\pi(\bar{x};\bar{m})}(\FC)$ as $\mathfrak{M}_{\bar{m}}(\FC)$. These are the spaces of types and measures which extend/concentrate on the partial type $\tp(\bar{m}/\emptyset)$. 
    \item For $\dagger \in \{\inv,\fs,\df\}$, we let $S^{\dagger}_{\bar{m}}(\FC,M) := S^{\dagger}_{\bar{x}}(\FC,M) \cap S_{\bar{m}}(\FC)$ and $\mathfrak{M}_{\bar{m}}^{\dagger}(\FC,M):=\mathfrak{M}_{\bar{x}}^{\dagger}(\FC,M)\cap \mathfrak{M}_{\bar{m}}(\FC)$. 
    \end{enumerate}
\end{definition}

Finally, we will also be concerned with a new important family of types and measures in this paper, i.e., the collection of \emph{strongly finitely satisfiable} types and measures. These types and measures naturally arise from the study of dynamical systems in infinitely many variables. Consider the definitions below:

\begin{definition}\label{def:sfs} 
For a partial type $\pi(\bar{x};\bar{y})$ over $\emptyset$, we introduce the space $S_{\pi(\bar{x};\bar{m})}^{\sfs}(\FC,M)$ of global types which extend $\pi(\bar{x};\bar{m})$ and are {\em $\pi$-strongly finitely satisfiable in $M$}:
\begin{equation*}
S_{\pi(\bar{x};\bar{m})}^{\sfs}(\FC,M) := \{p \in S_{\pi(\bar{x};\bar{m})}(\FC): \varphi(\bar{x};\bar{b}) \in p \Rightarrow (\exists \bar{a} \in M^{|\bar{x}|})(\models \pi(\bar{a},\bar{m}) \wedge \varphi(\bar{a};\bar{b})) \}. 
\end{equation*}
Likewise, we introduce the space  $\mathfrak{M}_{\pi(\bar{x};\bar{m})}^{\sfs}(\FC,M)$ of global measures which concentrate on $\pi(\bar{x};\bar{m})$ and are {\em $\pi$-strongly finitely satisfiable in $M$}:
\begin{equation*} 
\mathfrak{M}_{\pi(\bar{x};\bar{m})}^{\sfs}(\FC,M) := \{\mu \in \mathfrak{M}_{\pi(\bar{x};\bar{m})}(\FC): \mu(\varphi(\bar{x};\bar{b})) > 0 \Rightarrow (\exists \bar{a} \in M^{|\bar{x}|})(\models \pi(\bar{a},\bar{m}) \wedge \varphi(\bar{a};\bar{b})) \}. 
\end{equation*} 
In the case where $\pi(\bar{x};\bar{y})$ is the type ``$\bar{x} \equiv_{\emptyset} \bar{y}$", we write $S_{\pi(\bar{x};\bar{m})}^{\sfs}(\FC,M)$ as $S_{\bar{m}}^{\sfs}(\FC,M)$ and $\mathfrak{M}_{\pi(\bar{x};\bar{m})}^{\sfs}(\FC,M)$ as $\mathfrak{M}_{\bar{m}}^{\sfs}(\FC,M)$. We refer to these spaces simply as 
\emph{the space of global types strongly finitely satisfiable in $M$} and \emph{the space of global measures strongly finitely satisfiable in $M$},
respectively.
\end{definition} 

Note that $S_{\pi(\bar{x};\bar{m})}^{\sfs}(\FC,M)$ and  $\mathfrak{M}_{\pi(\bar{x};\bar{m})}^{\sfs}(\FC,M)$ are closed subsets of $S_{\bar{x}}(\FC)$ and $\mathfrak{M}_{\bar{x}}(\FC)$, respectively. 

The following fact is elementary and left to the reader. 

\begin{fact}\label{fact: support of sfs measures} Fix a partial type $\pi(\bar{x};\bar{y})$ over $\emptyset$ 
and a measure $\mu \in \mathfrak{M}_{\bar x}(\FC)$. The following are equivalent: 
\begin{enumerate} 
\item The measure $\mu$ is in  $\mathfrak{M}^{\sfs}_{\pi(\bar{x};\bar{m})}(\FC,M)$. 
\item For every $p \in \supp(\mu)$, $p \in S^{\sfs}_{\pi(\bar{x};\bar{m})}(\FC,M)$. 
\end{enumerate} 
\end{fact}

\subsection{Relatively type-definable subgroups} 
Let $M \prec \FC \prec \FC'$ and let $A\subseteq \FC$ be small and $\bar{\alpha}$ be a short tuple of elements from $\FC$, where $\FC'$ is a bigger monster model in which $\FC$ is small. The following definitions come from \cite{KPRz,HruKruPi}.

\begin{definition}\label{definition: relatively type-definable sets}
    By a {\em relatively $\bar{\alpha}$-definable over $A$} subset of $\aut(\FC)$
    we mean a subset of the form 
    $$G_{\varphi,\bar{\alpha},\bar{a}}=\{\sigma\in\aut(\FC)\;\colon\;\models\varphi(\sigma(\bar{\alpha});\bar{a})\},$$
    where $\varphi(\bar{x};\bar{y})$ is an $\CL$-formula and $\bar{a}$ is a tuple from $A$.

 By a {\em relatively $\bar{\alpha}$-type-definable over $A$} subset of $\aut(\FC)$
    we mean a subset of the form
    $$G_{\pi,\bar{\alpha},\bar{a}}=\{\sigma\in\aut(\FC)\;\colon\;\models\pi(\sigma(\bar{\alpha});\bar{a})\},$$
    where $\pi(\bar{x};\bar{y})$ is a partial type over $\emptyset$ and $\bar{a}$ is a tuple from $A$.
By a {\em relatively type-definable} subset of $\aut(\FC)$ we mean a subset which is relatively $\bar{\alpha}$-type-definable over $A$ for some short $\bar \alpha$ and small $A$ in $\FC$.
\end{definition}

Let $\bar{m}$ be an enumeration of $M$ and let $\bar{c}$ be an enumeration of $\FC$.
Moreover, let $\bar{x}$ and $\bar{y}$ be two distinct tuples of variables, each of them corresponding to the tuple $\bar{m}$.
One can consider the topology on $\aut(\FC)$ given by the basic open sets being the  relatively $\bar{m}$-definable over $M$ subsets of $\aut(\FC)$,
and call it the \emph{relatively definable over $M$ topology} on $\aut(\FC)$. It is clear that the closed sets in this topology are precisely the relatively $\bar m$-type-definable over $M$ subsets of $\aut(\FC)$. 

\begin{definition}\label{definition: notation G pi,C}
    Let $\pi(\bar{x};\bar{y})$ be a partial type over $\emptyset$. 
    We define
    \begin{enumerate}
        \item $G_{\pi,\FC}:=G_{\pi,\bar m,\bar m}=\{\sigma\in\aut(\FC)\;\colon\; \FC\models\pi(\sigma(\bar{m});\bar{m})\}$, and $G_{\pi,M}:=\{\sigma\in\aut(M)\;\colon\; M\models\pi(\sigma(\bar{m});\bar{m})\}$,
        \item $\widetilde{G}_{\pi,\FC}:=S_{\bar{c}}(\FC)\,\cap\,[\pi(\bar{x};\bar{m})]$, where $S_{\bar{c}}(\FC)$ is the collection of complete types over $\FC$ in variables corresponding to  $\bar{c}$ which extend $\tp(\bar{c}/\emptyset)$.
\item  $\widetilde{G}_{\pi,M}:=S_{\bar{m}}(M)\,\cap\,[\pi(\bar{x};\bar{m})]$, where $S_{\bar{m}}(M)$ is the collection of complete types over $M$ in variables $\bar x$ (corresponding to  $\bar{m}$) which extend $\tp(\bar{m}/\emptyset)$.
    \end{enumerate} 
\end{definition}

\begin{remark}\label{remark: two descriptions of the type space}
We have
$$\widetilde{G}_{\pi,\FC} =\cl(G_{\pi,\FC}\cdot\tp(\bar{c}/\FC))=
\cl\big(\{\tp\big(\sigma(\bar{c})/\FC\big)\;\colon\;\sigma\in G_{\pi,\FC}\} \big).$$
\end{remark}

\begin{proof}
($\supseteq$) Consider any $p \in \cl(G_{\pi,\FC}\cdot\tp(\bar{c}/\FC))$ and $\varphi(\bar x;\bar m) \in \pi(\bar x;\bar m)$. 
Suppose for a contradiction that $\neg \varphi(\bar x;\bar m) \in p$. 
Then there is $\sigma \in G_{\pi,\FC}$ such that 
$\sigma(\bar c)\models \neg \varphi(\bar{x};\bar m)$. 
But this means that $\models \neg \varphi(\sigma(\bar m);\bar m)$ which implies that $\sigma \notin G_{\pi,\FC}$, a contradiction.

($\subseteq$) Consider a type $p \in \widetilde{G}_{\pi,\FC}$ 
and a formula $\varphi(\bar x';\bar a) \in p$, 
where $\bar{x}'$ is some finite tuple of variables and $\bar a$ is a finite tuple contained in $\bar c$.
Take $\bar b_0 \in p(\FC')$. 
Let $\bar b$ and $\bar b'$  be the subtuples of $\bar b_0$ corresponding to $\bar x$ and $\bar x'$, respectively. 
Then $\models \pi(\bar b;\bar m) \wedge \varphi(\bar b';\bar a)$ 
and $\bar b \bar b' \equiv_{\emptyset} \bar m \bar c'$,
where $\bar c'$ is the finite subtuple of $\bar c$ corresponding to the variables $\bar x'$. 
We can find $\bar d\bar d'$ in $\FC$ so that $\bar d \bar d' \equiv_{\bar m \bar a} \bar b \bar b'$. 

Since  $\bar d \bar d' \equiv_{\emptyset}  \bar m \bar c'$, there exists $\sigma \in \aut(\FC)$ such that $\sigma( \bar m \bar c') = \bar d \bar d'$. 
As, $\bar d \equiv_{\bar m} \bar b$ and $\models \pi(\bar b;\bar m)$,
we get $\models \pi(\bar d;\bar m)$, and so $\sigma \in G_{\pi,\FC}$.
On the other hand, since $\bar d' \equiv_{\bar a} \bar b'$ and $\models \varphi(\bar b';\bar a)$, we get $\models \varphi(\bar d';\bar a)$, i.e. $\models \varphi(\sigma(\bar c);\bar a)$. 
We conclude that $\tp(\sigma(\bar c)/\FC) \in [\varphi(\bar x';\bar a)]$ with $\sigma\in G_{\pi,\FC}$.

Because, the neighborhood $[\varphi(\bar x';\bar a)]$ was chosen arbitrarily, we conclude that the type $p$ is in the desired closure.
\end{proof}

\begin{remark}
Let $r \colon S_{\bar c}(\FC) \to S_{\bar m}(M)$ be the restriction map (to the variables $\bar x$ and parameters $M$). 
Then $r^{-1}[\widetilde{G}_{\pi,M}] = \widetilde{G}_{\pi,\FC}$ 
and $r[ \widetilde{G}_{\pi,\FC}]=\widetilde{G}_{\pi,M}$.
\end{remark}

Bearing in mind Definitions \ref{definition: relatively type-definable sets} and \ref{definition: notation G pi,C}, whenever $G_{\pi,\FC}$ is a group, we say that $G_{\pi,\FC}$ is a {\em relatively $\bar m$-type-definable over $M$ subgroup} of $\aut(\FC)$ (i.e  a subgroup closed in the relatively definable over $M$ topology).

\begin{remark}\label{remark: group in every model}
The property that $G_{\pi,\FC}$ is a group does not depend on the choice of the monster model $\FC$ in which $M$ is small.
Moreover, if $G_{\pi,\FC}\leqslant\aut(\FC)$, then $G_{\pi,M}\leqslant\aut(M)$.
\end{remark}

\begin{proof}
It follows from the observation that the property that $G_{\pi,\FC}$ is a group is equivalent to the conjunction of the following three conditions:
\begin{enumerate}
\item 
$\models \pi(\bar m; \bar m)$;

\item 
the type $\exists \bar y (\pi(\bar y;\bar m) \wedge  \bar x  \bar m \equiv \bar m \bar y)$ is equivalent to the type $\pi(\bar x;\bar m)$;

\item 
the type $\exists \bar y \exists \bar z (\pi(\bar y;\bar m) \wedge \pi(\bar z;\bar m) \wedge \bar m \equiv \bar y \wedge \bar m \bar y \equiv \bar z \bar x)$  is equivalent to the type $\pi(\bar x;\bar m)$. \qedhere
\end{enumerate}
\end{proof}

\begin{remark}\label{remark: G_pi=G_pi^opp}
If $G_{\pi,\FC}$ is a group, then it coincides with $G_{\pi^{\textrm{opp}},\FC}:= \{ \sigma \in \aut(\FC): \FC \models \pi(\bar m;\sigma(\bar m))\}$.
\end{remark}

We now give an important class of subgroups of $\aut(\FC)$. 
These subgroups are ones which admit \emph{tame invariant measures}. Their counterparts in the definable group setting have been studied quite extensively.

\begin{definition}\label{def: fim subgroups}
Let $G \leqslant \aut(\FC)$ be relatively $\bar m$-type-definable over $M$,
i.e. $G=\{\sigma\in\aut(\FC)\;\colon\;\models \pi(\sigma(\bar{m});\bar{m})\}$ for some partial type $\pi(\bar{x};\bar{y})$ over $\emptyset$
which can be assumed to imply ``$\bar x \equiv \bar y$''. 
\begin{enumerate}
\item $G$ is {\em (left) generically stable (over $M$)} if there exists a (left) $G$-invariant generically stable type in 
$S^{\inv}_{\pi(\bar{x};\bar{m})}(\FC,M)$.

\item $G$ is {\em (left) fim (over $M$)} if there exists a (left) $G$-invariant fim measure in $\mathfrak{M}^{\inv}_{\pi(\bar{x};\bar{m})}(\FC,M)$.
\end{enumerate}
\end{definition}

\begin{lemma}\label{lemma:stab.mu.subgroup.pi}
    If $\pi(\bar{x};\bar{y})\vdash \bar x \equiv \bar y$ is a partial type such that $G_{\pi,\FC}$ is a subgroup of $\aut(\FC)$ and $\mu\in \mathfrak{M}_{\pi(\bar{x};\bar{m})}^{\inv}(\FC,M)$,
    then $\stab(\mu)\leqslant G_{\pi,\FC}$.
\end{lemma}

\begin{proof}
    Take $\tau\in\stab(\mu)$ and its extension $\tau' \in\aut(\FC')$. 
    Note that $\supp(\mu)=\supp(\tau_\ast(\mu))=\tau[\supp(\mu)]$.
    Moreover, we have that $\supp(\mu)\subseteq[\pi(\bar{x};\bar{m})]$.
    
    Let $p(\bar{x})\in\supp(\mu)$.
    Then $\pi(\bar{x};\bar{m})\subseteq p(\bar{x})$ and so there exists $\sigma'\in\aut(\FC')$ such that $\sigma'(\bar{m})\models p$. We have that, $\tau'\sigma'(\bar{m})\models \tau(p)\in\supp(\mu)$ 
    and so $\models\pi(\tau'\sigma'(\bar{m});\bar{m})$.
    As $\sigma',\tau'\sigma'\in G_{\pi,\FC'}$ and $G_{\pi,\FC'}$ 
    is a subgroup (by Remark \ref{remark: group in every model}), we obtain $\tau'\in G_{\pi,\FC'}$ and finally $\tau\in G_{\pi,\FC}$.
\end{proof}

The next lemma demonstrates how definable measures are associated to relatively type definable subgroups of the automorphisms group via the stabilizer.

\begin{lemma}\label{lemma: rel. type-definability}
Assume that $\mu\in \mathfrak{M}^{\df}_{\bar{x}}(\FC,M)$.
Then $\stab(\mu)$ is relatively $\bar m$-type-definable over $M$.
\end{lemma}

\begin{proof}
The proof follows the idea from the proof of Proposition 5.3 (and Definition 5.2) from \cite{Artem_Kyle}, thus we will only sketch the proof and point out the correct formulas for the action of $\aut(\FC)$.
By Fact 2.3 from \cite{Artem_Kyle}, for every $\CL$-formula $\varphi(\bar{x};z)$ (over $\emptyset$) and every $n\in\mathbb{N}_{>0}$ there exist formulas $\Phi^{\varphi,\frac{1}{n}}_{i}(\bar{y};z)$ where $i\in I_n:=\{0,\frac{1}{n},\ldots,\frac{n-1}{n},1\}$ such that
$$\FC^{z}\subseteq \bigcup\limits_{i\in I_n}\Phi^{\varphi,\frac{1}{n}}_{i}(\bar{m};\FC),$$
and $\models\Phi^{\varphi,\frac{1}{n}}_i(\bar{m};b)$ implies $\mu(\varphi(\bar{x};b))\approx_{\frac{1}{n}}i$.
We set 
$$\Phi^{\varphi,\frac{1}{n}}_{\geqslant i}(\bar{y};z):=\bigvee\limits_{j\in I_n,\,j\geqslant i}\Phi^{\varphi,\frac{1}{n}}_j(\bar{y};z),$$
\begin{align*}
    \rho(\bar{m};\bar{y}) :=
    \bigwedge\limits_{\varphi(\bar{x};z)\in\CL}
\bigwedge\limits_{\substack{n\in\mathbb{N}_{>0} \\ i\in I_n, i \geq \frac{3}{n}}}
\Big((\forall z)\big(\Phi^{\varphi,\frac{1}{n}}_{\geqslant i}(\bar{m};z)&\rightarrow\Phi^{\varphi,\frac{1}{n}}_{\geqslant(i-\frac{2}{n})}(\bar{y};z) \big)  \\
&  \wedge (\forall z)\big(\Phi^{\varphi,\frac{1}{n}}_{\geqslant i}(\bar{y};z)\rightarrow\Phi^{\varphi,\frac{1}{n}}_{\geqslant(i-\frac{2}{n})}(\bar{m};z)\big)  \Big). 
\end{align*}
By adapting the argument of Proposition 5.3 from \cite{Artem_Kyle},
one obtains that
$\stab(\mu)=\{\sigma\in\aut(N)\;\colon\;\models\rho(\bar{m};\sigma(\bar{m}))\}$.
\end{proof}

Finally, we connect invariant measures with invariant subgroups. 
In Definition \ref{cheat}, we recalled what it means that a measure is invariant over a small model; the same definition applies over an arbitrary subset of $\FC$.
In the next lemma, $\bar x$ does not correspond to $\bar m$ anymore.

\begin{lemma}\label{lemma: rel.inv}
Let $A$ be a small subset of $\FC$, $\mu \in \mathfrak{M}^{\inv}_{\bar x}(\FC)$ be $A$-invariant, and $\bar{a}\in \FC^{\bar{x}}$ be an enumeration of $A$.
Then $\stab(\mu)$ is relatively  $\bar a$-invariant over $A$,
i.e. 
$$\stab(\mu)=\{\sigma\in\aut(\FC)\;\colon\;\models\rho(\bar{a};\sigma(\bar{a}))\},$$
where $\rho(\bar{x};\bar{y})$ is a disjunction of (possibly infinitely many) complete types over $\emptyset$.
\end{lemma}

\begin{proof}
We first show that for every $\tau,\tau' \in \aut(\FC)$  such that $\tau(\bar a) \equiv_{A} \tau'(\bar a)$ we have that if $\tau \in \stab(\mu)$, then $\tau' \in \stab(\mu)$.
Fix $\tau$ and $\tau'$ and suppose  $\tau \in \stab(\mu)$. Then there is $\sigma \in \aut(\FC/A)$  such that $\tau'(\bar a)=\sigma(\tau(\bar a))$. And so $\zeta:=\tau^{-1} \sigma^{-1}\tau' \in \aut(\FC/A)$ and $\tau'=\sigma \tau \zeta$. Since $\mu$ is invariant over $A$, we get $\tau'(\mu) = \sigma \tau \zeta(\mu)=\sigma(\tau(\mu))=\sigma(\mu)=\mu$, i.e. $\tau' \in \stab(\mu)$.

Having the above, we set $\theta_{\tau}(\bar{x};\bar{y}):=\tp(\bar{a},\tau(\bar{a})/\emptyset)$ and
$$\rho(\bar{x};\bar{y}):=\bigvee\limits_{\tau\in\stab(\mu)}\theta_{\tau}(\bar{x};\bar{y}),$$
which gives us the desired disjunction of $\emptyset$-types.
\end{proof}
By a {\em relatively definable subset} of $G_{\pi,\FC}$ we mean a set of the form 
$\{\sigma \in G_{\pi,\FC}\;\colon\;\models \varphi(\sigma(\bar n);\bar a)\}$,
where $\varphi(\bar x;\bar a)$ is a formula with parameters $\bar a$ from $\FC$ and $\bar n$ is a tuple from $\FC^{\bar x}$. 
When $G_{\pi,\FC}$ is a subgroup of $\aut(\FC)$, it acts on itself by left and right translations. The action by left translations will be denoted by $\cdot$ and by right translations by $\cdot_r$. 
These actions induce actions on relatively definable subsets of $G_{\pi,\FC}$.

\begin{remark}\label{remark: trivial remark}
The actions $\cdot$ and $\cdot_r$ commute, that is $(\sigma \cdot g) \cdot_r \tau = \sigma \cdot (g \cdot_r \tau)$. 
\end{remark}

\subsection{Convolution product for definable groups}\label{subsection: convolution for definable groups}

Let $G$ be a group definable in a small model $M \prec \FC$. 
By $S^{\textrm{fs}}_G(\FC,M)$ [resp. $S^{\textrm{inv}}_G(\FC,M)$] we denote the space of global complete types concentrated on $G$ and finitely satisfiable in $M$ [resp. $M$-invariant].

The following operation $*$ was first defined by Newelski \cite{New2009} on $S^{\textrm{fs}}_G(\FC,M)$ (as the central object in his work on topological dynamics in model theory), and later extended to $S^{\textrm{inv}}_G(\FC,M)$ by Chernikov and Gannon \cite{Artem_Kyle}.

\begin{definition}
Let $p,q \in S^{\textrm{inv}}_G(\FC,M)$. Then $p*q:= \tp(a\cdot b/\FC)$ for some/any $(a,b) \models p \otimes q$ in a larger monster
model.
\end{definition}

The following fact is well-known (e.g. see \cite[Fact 3.11]{Artem_Kyle}).

\begin{fact}
 $(S^{\textrm{inv}}_G(\FC,M),*)$ is a compact left topological semigroup, and  $S^{\textrm{fs}}_G(\FC,M)$ is a closed sub-semigroup.
\end{fact}

In fact, in \cite{Artem_Kyle}, the authors extended the context to measures. Namely, let $\mathfrak{M}^{\textrm{fs}}_G(\FC,M)$ [resp. $\mathfrak{M}^{\textrm{inv}}_G(\FC,M)$] be the space of global Keisler measures concentrated on $G$ and finitely satisfiable in $M$ [resp. $M$-invariant].

\begin{definition}\label{definition: convolutionfor definable groups}
Let $\mu_x,\nu_y \in \mathfrak{M}^{\textrm{inv}}_G(\FC,M)$ and $\mu_x$ be Borel-definable over $M$. Then 
$$\mu_x*\nu_y(\varphi(x;\bar c)):= (\mu_x \otimes \nu_y) (\varphi(x \cdot y,\bar c)) = \int_{S_G(M_0)} F^{\varphi'}_{\mu,M_0} d\nu|_{M_0},$$ where $\varphi'(x,y,\bar c):=\varphi(x \cdot y;\bar c)$ and $M_0 \prec \FC$ contains $M$ and $\bar c$.
\end{definition}

The next fact follows from  Section 6 in \cite{Artem_Kyle}.

\begin{fact}\label{fact: definable convolution is a semigroup} (NIP)
 $(\mathfrak{M}^{\textrm{inv}}_G(\FC,M),*)$ is a compact left topological semigroup, and  $\mathfrak{M}^{\textrm{fs}}_G(\FC,M)$ is a closed sub-semigroup. Moreover, $(S^{\textrm{inv}}_G(\FC,M),*)$ is a closed sub-semigroup of $(\mathfrak{M}^{\textrm{inv}}_G(\FC,M),*)$, and $S^{\textrm{fs}}_G(\FC,M)$ is a closed sub-semigroup of $(\mathfrak{M}^{\textrm{fs}}_G(\FC,M),*)$.
\end{fact}

\subsection{Strong types and Galois groups}\label{subsection: Galois groups}

Lascar, Kim-Pillay, and Shelah strong types together with the associated Galois groups play an essential role in model theory, and also in applications (mostly through the strongly related notions of model-theoretic connected compoents of definable groups). Among fundamental papers on strong types and the associated Galois groups are 
\cite{Lascar, Lascar-Pillay, Casanovas-Lascar-Pillay-Ziegler, Newelski}. See also \cite{KPRz,KrRz,rzepecki2018} for a more involved research studying connections with topological dynamics and descriptive set theory. The notions and results that we recall below can be found in the above papers. A good exposition is given in Section 2.5 of \cite{rzepecki2018}.

As usual, let $\FC$ be a monster mode of $T$, $M \preceq \FC$ a small submodel, and $\FC' \succeq \FC$ a bigger monster model in which $\FC$ is small. By a {\em bounded equivalence relation} on $\FC^{\bar z}$ (where $\bar z$ is a short tuple of variables) we mean an equivalence relation with a small number of classes (i.e., less than the fixed degree of saturation of $\FC$).

\begin{fact}
There exists a finest bounded invariant equivalence relation on $\FC^{\bar z}$, denoted by $\equiv_{ \mathrm{Ls}}$. The classes of $\equiv_{ \mathrm{Ls}}$ are called {\em Lascar strong types}.  Moreover, $\equiv_{ \mathrm{Ls}}$ is precisely the transitive closure of the relation of having the same type over a model $N \preceq \FC$ where $N$ varies. In particular, $\equiv_{ \mathrm{Ls}}$ has at most $2^{|T| + |\bar{z}|}$ classes.

Similarly, there exists a finest 0-type-definable bounded equivalence relation on $\FC^{\bar z}$, denoted by $\equiv_{ \mathrm{KP}}$. The classes of $\equiv_{ \mathrm{KP}}$ are called {\em Kim-Pillay strong types}. We clearly have that $\equiv_{ \mathrm{Ls}} \;\subseteq \;\equiv_{ \mathrm{KP}}$, and so $\equiv_{ \mathrm{KP}}$ has at most $2^{|T| + |\bar{z}|}$ classes.
\end{fact}

Let $E$ be $\equiv_{ \mathrm{Ls}}$ or $\equiv_{\mathrm{KP}}$. The quotient $\FC^{\bar z}/E$  is equipped with the {\em logic topology} in which a subset of $\FC^{\bar z}/E$ is closed if its premiage under the quotient map $\pi \colon \FC^{\bar z} \to \FC^{\bar z}/E$ is type-definable (over parameters). By the above fact, the quotient map $\pi$ factors through the type space $S_{\bar z}(M)$:

\begin{center}
			\begin{tikzcd}
			\FC^{\bar z} \ar[dr,two heads,"t"]\ar[rr,two heads,"\pi"] & & \FC^{\bar z}/E\\
			&S_{\bar z}(M) \ar[ur,two heads,"\hat{\pi}"] & 
			\end{tikzcd}
		\end{center}
where $t(\bar a):=\tp(\bar a/M)$. The logic topology on $\FC^{\bar z}/E$ coincides with the quotient topology induced by $\hat{\pi}$. It is quasi-compact (not necessarily Hausdorff) when $E=\;\equiv_{ \mathrm{Ls}}$ and compact when  $E=\;\equiv_{ \mathrm{KP}}$.

$\equiv_{ \mathrm{KP}}$ is the set of realizations in $\FC$ of a partial type over $\emptyset$, and $\equiv_{ \mathrm{Ls}}$ is the union of the sets of realizations of some partial types over $\emptyset$. If we compute these sets of realizations in another monster model, we also obtain  $\equiv_{ \mathrm{KP}}$ and $\equiv_{ \mathrm{Ls}}$, respectively, computed in this new model. Any representatives of all the $E$-classes (in $\FC$) are representatives of all the classes of $E^{\FC'}$ (i.e., $E$ computed in $\FC'$).  So $\FC'/E^{\FC'}$ can be naturally identified with $\FC/E$ as a topological space, i.e. the quotient $\FC/E$ does not depend on the choice of the monster model $\FC$.

The group $\autf_{\mathrm{L}}(\FC)$ of {\em Lascar strong automorphisms} is defined as the group generated by the subgroups of $\aut(\FC)$  fixing pointwise some elementary substructures of $\FC$, i.e. $\langle\{ \aut(\FC/N): N \preceq \FC\} \rangle$. This group coincides with the pointwise stabilizer of all Lascar strong types (i.e., classes of $\equiv_{ \mathrm{Ls}}$) on all possible $\FC^{\bar z}$, and for every $\bar z$ the relation $\equiv_{ \mathrm{Ls}}$ on $\FC^{\bar z}$ is precisely the orbit equivalence relation of  $\autf_{\mathrm{L}}(\FC)$. Moreover, for any $\bar n$ enumerating a small model $N \preceq \FC$, an automorphism $f$ of $\FC$ belongs to $\autf_{\mathrm{L}}(\FC)$ if and only if $f$ preserves the $\equiv_{ \mathrm{Ls}}$-class of $\bar n$.

The group $\autf_{\mathrm{KP}}(\FC)$ of {\em Kim-Pillay strong automorphisms} is defined as the pointwise stabilizer of all Kim-Pillay strong types (i.e. classes of $\equiv_{ \mathrm{KP}}$) on all possible $\FC^{\bar z}$. Then for every $\bar z$ the relation $\equiv_{ \mathrm{KP}}$ on $\FC^{\bar z}$ is precisely the orbit equivalence relation of  $\autf_{\mathrm{KP}}(\FC)$. As above, for any $\bar n$ enumerating a small model $N \preceq \FC$, an automorphism $f$ of $\FC$ belongs to $\autf_{\mathrm{KP}}(\FC)$ if and only if $f$ preserves the $\equiv_{ \mathrm{KP}}$-class of $\bar n$. We clearly have that $\autf_{\mathrm{L}}(\FC) \leq \autf_{\mathrm{KP}}(\FC)$ are both normal subgroups of $\aut(\FC)$.

\begin{definition}
The {\em Lascar Galois group} of $T$, denoted by $\gal_{\mathrm{L}}(T)$, is the quotient group $\aut(\FC)/\autf_{\mathrm{L}}(\FC)$. 
The {\em Kim-Pillay Galois group} of $T$, denoted by $\gal_{\mathrm{KP}}(T)$, is the quotient group $\aut(\FC)/\autf_{\mathrm{KP}}(\FC)$. 
\end{definition}

The quotient maps $\mathfrak{p}_{\mathrm{L}}^{\FC} \colon \aut(\FC) \to \gal_{\mathrm{L}}(T)$ and $\mathfrak{p}_{\mathrm{KP}}^{\FC} \colon \aut(\FC) \to \gal_{\mathrm{KP}}(T)$ factor through any type space $S_{\bar m}(N)$: 

\begin{center}
\begin{figure}[ht]
\centering
\begin{minipage}{.5\textwidth}
            \begin{tikzcd}
\aut(\FC) \ar[dr,two heads,"t"]\ar[rr,two heads,"\mathfrak{p}_{\mathrm{L}}^{\FC}"] & & \gal_{\mathrm{L}}(T)\\
&S_{\bar m}(N) \ar[ur,"\hat{\mathfrak{p}}_{\mathrm{L}}^{\bar{m},N,\FC}",two heads] &
\end{tikzcd}
\end{minipage}%
\begin{minipage}{.5\textwidth}
            \begin{tikzcd}
\aut(\FC) \ar[dr,two heads,"t"]\ar[rr,two heads,"\mathfrak{p}_{\mathrm{KP}}^{\FC}"] & & \gal_{\mathrm{KP}}(T)\\
&S_{\bar m}(N) \ar[ur,two heads,"\hat{\mathfrak{p}}_{\mathrm{KP}}^{\bar{m},N,\FC}"] &
\end{tikzcd}
\end{minipage}
\end{figure}
\end{center}
\vspace{-9mm}
where $\bar m$ is an enumeration of $M$, $t(\sigma):=\tp(\sigma(\bar m)/N)$, and $N \preceq \FC$ is small. The {\em logic topology} on $\gal_{\mathrm{L}}(T)$ and on $\gal_{\mathrm{KP}}(T)$ is the quotient topology induced by 
$\hat{\mathfrak{p}}_{\mathrm{L}}^{\bar{m},N,\FC}$ and $\hat{\mathfrak{p}}_{\mathrm{KP}}^{\bar{m},N,\FC}$, respectively.
This topology does not depend on the choice of $M$, $N$, and $\bar m$, and turns $\gal_{\mathrm{L}}(T)$ into a quasi-compact topological group, and $\gal_{\mathrm{KP}}(T)$ into a compact topological group. As topological spaces, $\gal_{\mathrm{L}}(T)$ and $\gal_{\mathrm{KP}}(T)$ can be identified with $\{\bar a \in \FC^{\bar m}: \bar a \equiv \bar m\}/\!\!\equiv_{ \mathrm{Ls}}$ and $\{\bar a  \in \FC^{\bar m}: \bar a\equiv \bar m\}/\!\!\equiv_{ \mathrm{KP}}$, respectively.

As topological groups, both  $\gal_{\mathrm{L}}(T)$ and on $\gal_{\mathrm{KP}}(T)$ do not depend on the choice of $\FC$. This is witnessed by the isomoprhism 
$$\mathfrak{r}_*^{\FC'}\colon  \aut(\FC')/\autf_{*}(\FC') \to  \aut(\FC)/\autf_{*}(\FC),$$ 
where $*$ denotes ``L'' or ``KP'', given by:
$\mathfrak{r}_*^{\FC'}(\sigma'/\autf_{*}(\FC')) := \sigma/\autf_{*}(\FC)$ for any $\sigma \in \aut(\FC)$ such that $\sigma(\bar m) \equiv_M \sigma'(\bar m)$ (equivalently, it is enough to require that $\sigma(\bar m) \equiv_{\mathrm{Ls}} \sigma'(\bar m)$). 

The following notation will be used in Section \ref{subsec:6.3}. Set $\mathfrak{r}_{\FC'}:=\mathfrak{r}_{\mathrm{KP}}^{\FC'}$, and define $\rho_{\FC} \colon S_{\bar m}(\FC) \to \gal_{\mathrm{KP}}(T)$ as the composition 
$\inv \circ \; \mathfrak{r}_{\FC'} \circ \hat{\mathfrak{p}}_{\mathrm{KP}}^{\bar{m},\FC,\FC'}$, where $\inv(g):=g^{-1}$. Explicitly, $\rho_{\FC}(p) =\sigma^{-1}/\autf_{\KP}(\FC)$ for any/some $\sigma \in \aut(\FC)$ such that $\sigma(\bar m) \models p|_M$ (equivalently, it is enough to require that $\sigma(\bar m)$ is $\equiv_{\mathrm{KP}}$-equivalent to some/any realization of $p$). The map $\rho_{\FC}$ is a continuous surjection. Notice that the following diagram commutes:

\begin{center}
\begin{tikzcd}
    S_{\bar{m}}(\FC') \ar[r,two heads,"\rho_{\FC'}"] \ar[d,two heads,"\res_{\FC}"] &  \aut(\FC')/\autf_{\KP}(\FC') \ar[d,two heads,"\mathfrak{r}_{\FC'}"] \\
    S_{\bar{m}}(\FC) \ar[r,two heads,"\rho_{\FC}"] & \aut(\FC)/\autf_{\KP}(\FC),
\end{tikzcd}
\end{center}
where $\res_{\FC}: S_{\bar m}(\FC') \to S_{\bar m}(\FC)$ is the restriction map.

\subsection{Model-theoretic connected components}
Model-theoretic connected components play a major role in model theory and applications (e.g. to additive combinatorics). Since in this paper we only mention them in Fact \ref{rem: affine Lascar and KP} and in Section \ref{subsec:6.3} in order to explain that some results involving them which were obtained in \cite{Artem_Kyle2} follow from the corresponding results obtained in this paper, we will not give here a historical background and references.

Let $M \preceq \FC$ be as always.
Let $G$ be a $\emptyset$-definable group in $M$. 

\begin{definition}
Let $A \subseteq \FC$ be a small set of parameters.
\begin{enumerate}
\item $G(\FC)^{00}_A$ is the  smallest $A$-type-definable subgroup of $G(\FC)$ of bounded index.
\item  $G(\FC)^{000}_A$ is the  smallest $A$-invariant subgroup of $G(\FC)$ of bounded index.
\end{enumerate}
\end{definition}

The existence of these components is clear by the existence of $\equiv_{\KP}$ and $\equiv_{\mathrm{Ls}}$ (with parameters from $A$ added to the language). When $A=\emptyset$, we will skip it writing  $G(\FC)^{00}$ and  $G(\FC)^{000}$. In the NIP context, the components $G(\FC)^{00}_A$  and $G(\FC)^{000}_A$  do not depend on the choice of $A$ \cite{Shelah,Gismatullin}.

The quotients $G(\FC)/G(\FC)^{00}_A$ and $G(\FC)/G(\FC)^{000}_A$ equipped with the {\em logic topology} defined by saying the a subset is closed if its preimage under the quotient map is type-definable, make these groups a compact and quasi-compact topological group, respectively (e.g., see \cite[Proposition 3.5]{GisNew08}).

The components $G(\FC)^{00}_A$ and $G(\FC)^{00}_A$ are strongly related to the equivalence relations $\equiv_{\mathrm{Ls}}$ and $\equiv_{\KP}$ via ``adding an affine sort construction'' (see Fact \ref{rem: affine Lascar and KP} below, and \cite{GisNew08} for more details).

\section{Injectivity results under NIP}\label{sec: injectivity}
We first prove that some restriction maps between Ellis semigroups are injective. Using this result, we prove that certain Ellis semigroups are isomorphic. For example, we prove that Ellis semigroups of $(S_{\bar{x}}(M),\aut(M))$ and $(\mathfrak{M}_{\bar{x}}(M),\aut(M))$ are isomorphic when the underlying theory is NIP. We remark that the Ellis semigroup of the dynamical system $(\mathfrak{M}_{\bar{x}}(M),\conv(\aut(M))$ is not isomorphic to the previous two and is of an inherently more different flavor (cf. Subsection \ref{sec:analogon}). 

Let $M\models T$ and $\bar{x}$ be any tuple of variables. Throughout this section, all of our integrals will take place over the space $S_{\bar{x}}(M)$, and so 
we will drop the subspace notation, i.e., we will write $\int_{S_{\bar{x}}(M)} f d\mu$ simply as $\int f d\mu$.  Let  $\varphi_1(\bar{x};\bar{y}),\ldots,\varphi_m(\bar{x};\bar{y})$ be  $\CL_{\bar{x};\bar{y}}$-formulas. We define the collection of functions,
$$\CF_{\varphi_1,\ldots,\varphi_m}:=\conv\{\mathbf{1}_{[\varphi_i(\bar{x};\bar{b})]} \colon \bar{b} \in M^{\bar{y}}, i \leq m \}. $$
Let $Y$ be a set, $f:S_{\bar{x}}(M)\to Y$, and $\sigma\in\aut(M)$.
Then we let $\sigma\cdot f:=f\circ\sigma^{-1}$.
Similarly, if $\lambda\in\conv(\aut(M))$ and $\lambda=\sum\limits_{l\leqslant N}\alpha_l\delta_{\sigma_l}$, we let
$$\lambda\cdot f = \left( \sum\limits_{l\leqslant N}\alpha_l\delta_{\sigma_l} \right)\cdot f :=
\sum\limits_{l\leqslant N}\alpha_l(\sigma_l\cdot f)= \sum\limits_{l\leqslant N}\alpha_l\, \left(f\circ\sigma_l^{-1} \right).$$
We also define $\lambda^{-1}:=\sum\limits_{l\leqslant N}\alpha_l\delta_{\sigma_l^{-1}}$.

\begin{remark}
    Let $\lambda\in\conv(\aut(M))$, $f\in \CF_{\varphi_1,\ldots,\varphi_m}$, and $\mu\in\mathfrak{M}_{\bar{x}}(M)$. Then
    $$\int (\lambda\cdot f)d\mu=\int f\,d(\lambda^{-1} \cdot \mu).$$
\end{remark}

\begin{proof}
Follows directly from the definitions.  
\end{proof}

\begin{lemma}\label{better.NIP}
    Assume that $T$ is NIP, $\mu\in\mathfrak{M}_{\bar{x}}(M)$, $f_1,\ldots,f_n\in\CF_{\varphi_1,\ldots,\varphi_m}$, and $\epsilon>0$. Then there exists $\bar{p}=(p_1,\ldots,p_t)\in \supp(\mu)^{<\omega} \subseteq S_{\bar{x}}(M)^{<\omega}$ such that for every $\lambda \in \conv(\aut(M))$ and $k \leq n$, 
    $$ \left|\int(\lambda\cdot f_k)d\mu\;-\;\int(\lambda\cdot f_k)d\Av(\bar{p}) \right|<\epsilon.$$
\end{lemma}

\begin{proof}
By Fact \ref{fact:MP}(2), there exist $\bar{p}=(p_1,\ldots,p_t)\in \supp(\mu)^{<\omega}$ such that for every $b\in M^{\bar{y}}$ and $k \leq n$ we have $$ \left|\mu(\varphi_k(\bar{x};\bar{b}))\;-\;\Av(\bar{p})(\varphi_k(\bar{x};\bar{b})) \right|<\epsilon.$$
    Apparently this is enough.
    Take $\lambda=\sum\limits_{l\leqslant N}\alpha_l\delta_{\sigma_l}$ and
    $f_k=\sum\limits_{s\leqslant N_k}\gamma^k_s\mathbf{1}_{[\varphi_{k,s}(\bar{x};\bar{b}_{k,s})]}$. Then,
    \begin{align*}
    \int(\lambda\cdot f_k)\,d\mu &= \int\Big( \sum\limits_{l\leqslant N}\alpha_l\delta_{\sigma_l} \Big)\cdot\Big( \sum\limits_{s\leqslant N_k}\gamma^k_s\mathbf{1}_{[\varphi_{k,s}(\bar{x};\bar{b}_{k,s})]} \Big)\,d\mu \\
    &= \sum\limits_{l,s}\alpha_l\gamma^k_s\,\mu\big(\varphi_{k,s}(\bar{x};\sigma_l(\bar{b}_{k,s})) \big) \\
    &\approx_{\Big(\sum\limits_{l,s}\alpha_l\gamma^k_s\epsilon=\epsilon\Big)} 
    \sum\limits_{l,s}\alpha_l\gamma^k_s\,\Av(\bar{p})\big(\varphi_{k,s}(\bar{x};\sigma(\bar{b}_{k,s})) \big) \\
    &= \int\Big( \sum\limits_{l\leqslant N}\alpha_l\delta_{\sigma_l} \Big)\cdot\Big( \sum\limits_{s\leqslant N_k}\gamma^k_s\mathbf{1}_{[\varphi_{k,s}(\bar{x};\bar{b}_{k,s})]} \Big)\,d\Av(\bar{p}) \\
    &= \int (\lambda\cdot f_k)\,d\Av(\bar{p}). \qedhere 
    \end{align*}  
\end{proof}

\begin{lemma}\label{lemma:injectivity1}
    Assume that $T$ is NIP. Then the map, 
\begin{equation*} \Theta:\E(\mathfrak{M}_{\bar{x}}(M),\conv(\aut(M))\to \mathfrak{M}_{\bar{x}}(M)^{S_{\bar{x}}(M)}, 
\end{equation*}  given by $\Theta(\eta):=\eta|_{S_{\bar{x}}(M)}$ is injective. 
(Here, we identify types with their corresponding Dirac measures.)
\end{lemma}

\begin{proof}
    Consider $\eta_1,\eta_2\in \E(\mathfrak{M}_{\bar{x}}(M),\conv(\aut(M)))$ such that $\eta_1\neq\eta_2$. Then there exists some $\mu\in\mathfrak{M}_{\bar{x}}(M)$, $\mathcal{L}_{\bar{x}}(M)$-formula $\varphi(\bar{x};\bar{b})$, and $\epsilon>0$ such that
    $$|\eta_2(\mu)(\varphi(\bar{x};\bar{b})) - \eta_1(\mu)(\varphi(\bar{x};\bar{b}))| > \epsilon,$$
    or in other words, 
    \begin{equation*}
    \left| \int f d\eta_2(\mu) - \int f d\eta_1(\mu) \right| > \epsilon,
    \end{equation*} 
    where $f:=\mathbf{1}_{[\varphi(\bar{x};\bar{b})]}$. By Lemma \ref{better.NIP}, there exist $\bar{p}=(p_1,\ldots,p_t)\in S_{\bar{x}}(M)^{<\omega}$
    such that for every $\lambda \in \conv(\aut(M))$ we have, 
    \begin{equation}\tag{$\ast$}\label{eq.star1}
        \left|\int (\lambda\cdot f)\,d\mu\;-\;\int(\lambda\cdot f)\,d\Av(\bar{p}) \right|<\delta:=\frac{\epsilon}{12}. 
    \end{equation}
    Since  $\eta_1,\eta_2\in \E(\mathfrak{M}_{\bar{x}}(M),\conv(\aut(M)))$, there exist nets
    $(\lambda^1_i\cdot-)_{i\in I}$ and $(\lambda^2_j\cdot-)_{j\in J}$ of elements from 
$\conv(\aut(M))$ such that $\lim_{i \in I} (\lambda^1_i\cdot-) = \eta_1$, $\lim_{j \in J}(\lambda^2_j\cdot-)_{j\in J} = \eta_2$. And so, there exists some $i\in I$ such that for $\lambda_1:=(\lambda^1_i)^{-1}$, we have,
    \begin{equation}\tag{i}\label{eq.first}
        \delta> \left|\int f\,d\eta_1(\mu)-\int  f\,d(\lambda_1^{-1}\cdot\mu)\right|=\left|\int f\,d\eta_1(\mu)-\int (\lambda_1\cdot f)\,d\mu\right|,
    \end{equation}
    and for every $k \leq t$,
    \begin{equation}\tag{ii}\label{eq.second}
        \delta>\left|\int f\,d\eta_1(\delta_{p_k})-\int f\,d(\lambda_1^{-1}\cdot\delta_{p_k}) \right|=\left|\int f\,d\eta_1(\delta_{p_k})-\int (\lambda_1\cdot f)\,d\delta_{p_k}\right|.
    \end{equation}
    Similarly, there exists $j\in J$ such that for $\lambda_2:=(\lambda^2_j)^{-1}$, we have, 
    \begin{equation}\tag{iii}\label{eq.third}
        \delta> \left|\int f\,d\eta_2(\mu)-\int (\lambda_2\cdot f)\,d\mu \right|,
    \end{equation}
    and for every $k\leqslant t$, 
    \begin{equation}\tag{iv}\label{eq.fourth}
        \delta> \left|\int f\,d\eta_2(\delta_{p_k})-\int (\lambda_2\cdot f)\,d\delta_{p_k}\right|. 
    \end{equation}
\noindent 
Now, by the choice of $\epsilon$ together with inequalities (i) and (iii), we get
$$\left|\int (\lambda_2\cdot f)\,d\mu\;-\;\int (\lambda_1\cdot f)\,d\mu \right|>\epsilon-2\delta.$$
    By applying ($\ast$) twice, one derives the following inequalities: 
    \begin{align*}\epsilon-4\delta&< \left|\int(\lambda_2\cdot f)\,d\Av(\bar{p})\;-\;\int(\lambda_1\cdot f)\,d\Av(\bar{p}) \right| \\ & \leqslant\frac{1}{t}\sum\limits_{k\leqslant t}\left|\int(\lambda_2\cdot f)\,d\delta_{p_k}-\int (\lambda_1\cdot f)\,d\delta_{p_k}\right|.
    \end{align*}
    By (ii) and (iv), we approximate
    $$\frac{1}{t}\sum\limits_{k\leqslant t} \left|\int(\lambda_2\cdot f)\,d\delta_{p_k}-\int (\lambda_1\cdot f)\,d\delta_{p_k} \right|\approx_{2\delta}
    \frac{1}{t}\sum\limits_{k\leqslant t} \left|\int f\,d\eta_2(\delta_{p_k})-\int f\,d\eta_1(\delta_{p_k})\right|.$$
    Hence,
    $$0<\epsilon-6\delta<  \frac{1}{t}\sum\limits_{k\leqslant t} \left|\int f\,d\eta_2(\delta_{p_k})-\int f\,d\eta_1(\delta_{p_k})\right|,$$
    and so there exists $k\leqslant t$ such that, 
    $$0< \left|\int f\,d\eta_2(\delta_{p_k})\;-\;\int f\,d\eta_1(\delta_{p_k}) \right|.$$
    Hence $\eta_1(\delta_{p_k})\neq \eta_2(\delta_{p_k})$ and so $\eta_1|_{S_{\bar{x}}(M)}\neq \eta_2|_{S_{\bar{x}}(M)}$.
\end{proof}

\begin{lemma}\label{injectivity2}
    Assume that $T$ is NIP and let $M\preceq N$. Let $\bar{m}$ be a tuple of variables enumerating $M$. Then, 
    $$\Theta':\E(\mathfrak{M}_{\bar{m}}(N),\conv(\aut(N))\to \mathfrak{M}_{\bar{m}}(N)^{S_{\bar{m}}(N)},$$
    given by $\Theta(\eta):=\eta|_{S_{\bar{m}}(N)}$ is injective. 
(Here, we again identify types with their corresponding Dirac measures.)
\end{lemma}

\begin{proof}
A similar proof to the proof of Lemma \ref{lemma:injectivity1} works here as well. 
The only extra observation which is needed is that the types $p_1,\dots,p_t$ can be found in $\supp(\mu) =S_{\bar{m}}(N)$ (which is the case by Lemma \ref{better.NIP}).
\end{proof}

\begin{proposition}\label{proposition:NIP.isomorphisms}
    Assume that $T$ is NIP and let $M\preceq N$ (possibly $M=N$). Let $\bar{m}$ be a tuple of variables enumerating $M$, $\bar{x}$ be a small tuple of variables, and $H\leqslant\aut(M)$. 
    Then the following maps are isomorphisms of topological semigroups:
    \begin{enumerate}
        \item $|_{S_{\bar{m}}(N)}:\E(\mathfrak{M}_{\bar{m}}(N),\aut(N))\to \E(S_{\bar{m}}(N),\aut(N))$,

        \item $|_{S_{\bar{x}}(M)}:\E(\mathfrak{M}_{\bar{x}}(M),\aut(M))\to \E(S_{\bar{x}}(M),\aut(M))$.

        \item $|_{S_{\bar{x}}(M)}:\E(\mathfrak{M}_{\bar{x}}(M),H)\to \E(S_{\bar{x}}(M),H)$.
    \end{enumerate}
\end{proposition}

\begin{proof}
The fact that these are semigroup homomorphisms is easy.
Injectivity of all these homomorphisms  follows from Lemmas \ref{lemma:injectivity1} and \ref{injectivity2}. We only need to check that these restrictions are surjective. For example, suppose that $\eta\in \E(S_{\bar{x}}(M),H)$ and let $(h_i \cdot -)_{i \in I}$ be a net of elements in $H$ which converges to $\eta$. 
    Now consider the net $(\delta_{h_i} \cdot - )_{i\in I}$ of elements in $\E(\mathfrak{M}_{\bar{x}}(M), H)$. 
    By compactness, this net admits a convergent subnet, say $(\delta_{h_j} \cdot - )_{j\in J}$, which converges to say $\eta'$. We claim that $\eta'|_{S_{\bar{x}}(M)}=\eta$.
\end{proof}

\section{New convolution operation}\label{sec:star.product}
This is the first main section of the paper. Here, we do the following:
\begin{enumerate}
\item In Section \ref{sec:analogon},  we give an explicit homeomorphism between a natural Ellis semigroup and a special collection of Keisler measures in the NIP setting. 
In particular, we show that if $M$ satisfies some homogeneity conditions, then the Ellis semigroup of $(\mathfrak{M}_{\bar{x}}(M),\conv(\aut(M)))$ is naturally homeomorphic to the space $\mathfrak{M}^{\sfs}_{\bar{m}}(\FC,M)$ (see Definition \ref{def:sfs}). 
This homeomorphism allows us to construct a \emph{new product} on $\mathfrak{M}^{\sfs}_{\bar{m}}(\FC,M)$.
\item In Section  \ref{subsec:star.definitions}, we extend this newly identified product from strongly finitely satisfiable Keisler measures to a much larger class of measures (even outside of the NIP context). We are also able to remove the homogeneity condition on $M$ and work over arbitrary small models. We then develop the fundamentals of the theory of this product. We prove a family of preservation results (e.g., the product of definable measures remains definable) and show 
that this product is associative for many \emph{practical} classes of measures. 
\end{enumerate} 

\subsection{Semigroup structures in the style of Hrushovski-Newelski}\label{sec:analogon}
As usual, $M \models T$, $\bar{m} = (m_{\alpha})_{\alpha < \gamma}$ is an enumeration of $M$ and $\mathfrak{C}$ is a monster model of $T$  where $M \prec \mathfrak{C}$. Let $\bar{x}$ and $\bar{y}$ be tuples of variables corresponding to the enumeration $\bar{m}$. Let $\pi(\bar{x};\bar{y})$ be a partial type over $\emptyset$. Recall that
\begin{equation*} G_{\pi,M}:=\{\sigma\in\aut(M)\;\colon\;\models\pi(\sigma(\bar{m});\bar{m})\}. 
\end{equation*} 
We begin by defining a certain homogeneity condition.

\begin{definition}\label{def:group-like}
    Let $\pi(\bar{x};\bar{y})$ be a partial type over $\emptyset$. We say that $\pi$ is \emph{group-like} over $M$  or that the pair $(M,\pi)$ is \emph{group-like} if the following three properties hold: 
    \begin{enumerate}
    \item $\pi(\bar{x};\bar{y}) \vdash \bar{x}\equiv_{\emptyset}\bar{y}$. 
    \item $G_{\pi,M}$ is a subgroup of $\aut(M)$. 
    \item For every $\bar{a} \in M^{\bar{y}}$ and every finite subtuple $(a_{\alpha_1},\dots,a_{\alpha_{k}}) \subseteq \bar{a}$, if $\models \pi(\bar{m};\bar{a})$, then there exists an automorphism $\sigma \in G_{\pi,M}$ such that for each $i \leq k$, $\sigma(m_{\alpha_{i}}) = a_{\alpha_i}$.
    \end{enumerate} 
\end{definition}

We remark that if $\pi(\bar{x};\bar{y})$ is ``$\bar{x} \equiv_{\emptyset} \bar{y}$'', then condition (3) in Definition \ref{def:group-like} is equivalent to asserting that $M$ is 
strongly $\aleph_0$-homogeneous. 

\begin{example} 
Suppose that $M$ is strongly $\aleph_0$-homogeneous.
\begin{enumerate} 
\item If $\pi(\bar{x};\bar{y})$ is ``$\bar{x}\equiv_\emptyset\bar{y}$'', then $\pi(\bar{x};\bar{y})$ is group-like over $M$. 
\item Partial types $\pi(\bar x,\bar y;\bar x',\bar y')$ and  $\pi^{\textrm{opp}}(\bar x',\bar y';\bar x,\bar y)$ from the \emph{affine sort construction} in Section \ref{sec: affine sort} are group-like over $\bar M$ (using the notation from  Section \ref{sec: affine sort}).
\end{enumerate} 
\end{example}
\noindent

We now consider a particular Ellis semigroup related to the automorphism group of a first-order structure. 
We define a continuous map $\Phi$ from this Ellis semigroup to the collection of strongly finitely satisfiable Keisler measures (see Definition \ref{def:sfs}). Assuming NIP, we also define a map $\Psi$ which goes in the opposite direction and turns out to be the inverse of $\Phi$.  We begin with the definition of $\Phi$.

\begin{definition} Suppose that $\pi(\bar{x};\bar{y})$ is group-like over $M$. 
We define the map 
$$\Phi:\E(\mathfrak{M}_{\bar{x}}(M),\conv(G_{\pi,M}))
\longrightarrow\mathfrak{M}^{\sfs}_{\pi(\bar{m};\bar{y})}(\mathfrak{C},M)$$
by $\Phi(\eta):=\mu_\eta$, where for any tuple $\bar{n} \in \mathfrak{C}^{\bar{x}}$ and $\mathcal{L}_{\bar{x};\bar{y}}$-formula $\varphi(\bar{x};\bar{y})$ we define

$$\mu_\eta\big(\varphi(\bar{n}; \bar y)\big):= \left( \eta(\delta_{\tp(\bar{n}/M)})\right)\big(\varphi(\bar{x}; \bar m) \big).$$
\end{definition} 

We will show that the map above is well-defined, meaning that the value of $\mu_\eta\big(\varphi(\bar{n}; \bar y)\big)$ depends only on the equivalence class of the formula $\varphi(\bar{n}; \bar y)$ and that $\mu_\eta \in \mathfrak{M}^{\sfs}_{\pi(\bar{m};\bar{y})}(\mathfrak{C},M)$. Let us start from the following lemma.

\begin{lemma}\label{lemma:pi-inv} If $\eta\in \E(\mathfrak{M}_{\bar{x}}(M),\conv(G_{\pi,M}))$, then $\eta(\delta_{\tp(\bar m/M)})\in \mathfrak{M}_{\pi(\bar{x};\bar{m})}(M)$.
\end{lemma} 

\begin{proof}
First, consider an arbitrary $\tau \in G_{\pi,M}$. Then  $$\tau(\delta_{\tp(\bar m/M)})\big([\pi(\bar{x};\bar{m})]\big)=\delta_{\tp(\bar m/M)}\big(\tau^{-1}[\pi(\bar{x};\bar{m})]\big)=
    \delta_{\tp(\bar m/M)}\big([\pi(\bar{x};\tau^{-1}(\bar{m}))]\big)=1,$$
because $\models \pi(\tau(\bar m);\bar m)$. By linearity, this implies that $\eta(\delta_{\tp(\bar m/M)})\big([\pi(\bar{x};\bar{m})]\big) =1$ for every $\eta \in \conv(G_{\pi,M})$. Thus, in full generality the conclusion follows from the fact that each $\eta \in \E(\mathfrak{M}_{\bar{x}}(M),\conv(G_{\pi,M}))$ is the limit of a net $(\eta_i \cdot -)_{i \in I}$ for some $\eta_i \in \conv(G_{\pi,M})$.
\end{proof}

\begin{lemma} The map $\Phi$ is well-defined.
\end{lemma}
    
\begin{proof}
    Let $\eta\in \E(\mathfrak{M}_{\bar{x}}(M),\conv(G_{\pi,M}))$.
    Fix tuples $\bar{n},\bar{n}'\in \FC^{\bar{x}}$ and $\mathcal{L}_{\bar{x};\bar{y}}$-formulas $\varphi(\bar{x};\bar{y}),\psi(\bar{x};\bar{y})$, such that 
    $$ \FC \models (\forall \bar{y})(\varphi(\bar{n};\bar{y}) \longleftrightarrow \psi(\bar{n}';\bar{y})).$$
    We first show that, 
    \begin{equation*}\mu_\eta\big(\varphi(\bar{n};\bar{y})\big)=\big( \eta(\delta_{\tp(\bar{n}/M)})\big)\big(\varphi(\bar{x};\bar{m})\big) 
    = \big( \eta(\delta_{\tp(\bar{n}'/M)})\big)\big(\psi(\bar{x};\bar{m})\big)=
    \mu_\eta\big(\psi(\bar{n}';\bar{y})\big).
    \end{equation*}
    We prove the statement in cases. First, if $\eta=\sigma \cdot - $ for some $\sigma \in G_{\pi,M}$, 
    \begin{align*}
    \sigma(\delta_{\tp(\bar{n}/M)})\big(\varphi(\bar{x};\bar{m})\big) &=  \delta_{\tp(\bar{n}/M)}\big(\varphi(\bar{x};\sigma^{-1}(\bar{m}))\big)=1 \\
    &\iff  \FC \models\varphi(\bar{n};\sigma^{-1}(\bar{m})\big) \\
    &\iff  \FC \models \psi(\bar{n}';\sigma^{-1}(\bar{m})\big) \\
    &\iff  1=\delta_{\tp(\bar{n}'/M)}\big(\psi(\bar{x};\sigma^{-1}(\bar{m}))\big)=\sigma(\delta_{\tp(\bar{n}'/M)})\big(\psi(\bar{x};\bar{m})\big).
    \end{align*}
    The case of $\eta\in\conv(G_{\pi,M})$ follows by linearity. 
    
    Finally, if $(\lambda_i\cdot-)_{i\in I}\to\eta$ and $\lambda_i\in\conv(G_{\pi,M})$ for each $i\in I$, then
    \begin{align*}
    \eta(\delta_{\tp(\bar{n}/M)})\big(\varphi(\bar{x};\bar{m})\big) &= \lim\limits_{i\in I}(\lambda_i(\delta_{\tp(\bar{n}/M)}))\big(\varphi(\bar{x};\bar{m})\big) \\
    &= \lim\limits_{i\in I}(\lambda_i(\delta_{\tp(\bar{n}'/M)}))\big(\psi(\bar{x};\bar{m})\big)\\
    &= \eta(\delta_{\tp(\bar{n}'/M)})\big(\psi(\bar{x};\bar{m})\big).
    \end{align*}
 
   Next, we show that $\Phi(\eta)=\mu_\eta$ is a Keisler measure. 
By the above observation, for every inconsistent formula $\varphi(\bar n;\bar y)$ we have 
$$\eta_\eta(\varphi(\bar n;\bar y))=\mu_\eta (m_0=y_0 \wedge m_0 \ne y_0) = \eta(\delta_{\tp(\bar{m}/M)})(x_0=m_0 \wedge x_0 \ne m_0) =0.$$
For additivity, consider any formulas $\varphi(\bar{n};\bar{y})$ and $\psi(\bar{n};\bar{y})$ such that $\varphi(\bar{n};\bar{y}) \wedge \psi(\bar{n};\bar{y})$ is inconsistent. Then, as we showed that the value of $\mu_\eta$ on any inconsistent formula is $0$, we have
\begin{align*}
\mu_\eta(\varphi(\bar{n};\bar{y}) \vee \psi(\bar{n};\bar{y})) &= \mu_\eta(\varphi(\bar{n};\bar{y}) \vee \psi(\bar{n};\bar{y})) + \mu_\eta(\varphi(\bar{n};\bar{y}) \wedge \psi(\bar{n};\bar{y})) \\
&=  \eta(\delta_{\tp(\bar{n}/M)})\big(\varphi(\bar{x};\bar{m}) \vee \psi(\bar{x};\bar{m}) \big) + \eta(\delta_{\tp(\bar{n}/M)})\big(\varphi(\bar{x};\bar{m}) \wedge \psi(\bar{x};\bar{m}) \big) \\
&=  \eta(\delta_{\tp(\bar{n}/M)})\big(\varphi(\bar{x};\bar{m})\big) +  \eta(\delta_{\tp(\bar{n}/M)})\big(\psi(\bar{x};\bar{m})\big) \\
&= \mu_\eta(\varphi(\bar{n};\bar{y})) + \mu_\eta(\psi(\bar{n};\bar{y})).
\end{align*}

We still need to check that:
\begin{enumerate}
\item $\Phi(\eta)$ concentrates on $\pi(\bar{m};\bar{y})$,
\item $\Phi(\eta)$ is $\pi$-strongly finitely satisfiable in $M$. 
\end{enumerate}

Notice that (1) follows directly from Lemma \ref{lemma:pi-inv}.
Hence, it remains to prove (2). 
    
    Suppose that $\mu_\eta\big(\varphi(\bar{n};\bar{y})\big) > 0$.
    Towards a contradiction, suppose that there is no $\bar{a}\in M^{\bar{y}}$
    such that $\models \pi(\bar{m};\bar{a})\,\wedge\,\varphi(\bar{n};\bar{a})$.
Using property (2) in Definition \ref{def:group-like}, we conclude that there is no $\sigma\in G_{\pi,M}$ such that
    $\mathfrak{C} \models\varphi(\bar{n};\sigma(\bar{m}))$.
    So for every $\sigma\in G_{\pi,M}$, $\sigma(\delta_{\tp(\bar{n}/M)})(\varphi(\bar{x};\bar{m})) = 0$. By linearly, for every $\lambda \in \conv(G_{\pi,M})$, $\lambda(\delta_{\tp(\bar{n}/M)})(\varphi(\bar{x};\bar{m})) = 0$. Since $\eta$ is a limit of a net of elements from $\conv(G_{\pi,M})$, we conclude that $\eta(\delta_{\tp(\bar{n}/M)})(\varphi(\bar{x};\bar{m})) = 0$ and thus $\mu_{\eta}(\varphi(\bar{n};\bar{y})) = 0$, a contradiction. Hence, $\mu_\eta\in\mathfrak{M}^{\sfs}_{\pi(\bar{m};\bar{y})}(\FC, M)$.
\end{proof}

\begin{remark}
    The map $\Phi$ is continuous. 
\end{remark}

\begin{proof}
It follows from the equality
\begin{equation*} 
\Phi^{-1} \left[\bigcap_{i=1}^{n} \{\mu : r_i < \mu(\varphi_i(\bar{b}_i;\bar{y})) < s_i\} \right] = \bigcap_{i=1}^{n} \{\eta: r_i < \eta (\delta_{\tp(\bar{b}_i/M)})\big(\varphi_i(\bar{x};\bar{m})\big) < s_i\}. 
\end{equation*} 
\end{proof}

We now construct an inverse for the map $\Phi$ in the NIP setting. We remark that the NIP assumption is needed to have that all global measures which are are finitely satisfiable over $M$ are also Borel-definable over $M$ 
(see Fact \ref{fact:MP}(1)). This allows us to use the Morley product in the following definition. 

\begin{definition} Suppose that $T$ is NIP and $\pi(\bar{x};\bar{y})$ is group-like over $M$. We define the map
$$\Psi:\mathfrak{M}^{\sfs}_{\pi(\bar{m};\bar{y})}(\FC,M)\longrightarrow \E(\mathfrak{M}_{\bar{x}}(M),\conv(G_{\pi,M}))$$
by $\Psi(\mu):=\eta_\mu$, where
$$\eta_\mu(\nu)\big(\varphi(\bar{x};\bar{m})\big):=(\mu_{\bar{y}} \otimes\nu_{\bar{x}})\big(\varphi(\bar{x}; \bar{y})\big)$$
for each $\nu\in\mathfrak{M}_{\bar{x}}(M)$ and $\varphi(\bar{x};\bar{m})\in \mathcal{L}_{\bar{x}}(M)$. (See Remark \ref{remark:Morley_convention} for our convention concerning Morley products between global and non-global measures.) 
\end{definition}

\begin{lemma}\label{lemma:Psi}
    Assuming $T$ is NIP, the map $\Psi$ is well-defined.
\end{lemma}

\begin{proof}
We assume NIP so that the Morley product in the definition of $\Psi$ is well-defined. Fix $\mu$ in $\mathfrak{M}^{\sfs}_{\pi(\bar{m};\bar{y})}(\FC,M)$, $\nu$ in $\mathfrak{M}_{\bar{x}}(M)$, and $\CL$-formulas $\varphi(\bar{x};\bar{y}),\psi(\bar{x};\bar{y})$.
Assume that $M\models (\forall\bar{x})\big(\varphi(\bar{x};\bar{m})\,\leftrightarrow\,\psi(\bar{x};\bar{m}) \big)$.
We need to show that
$$\Psi(\mu)(\nu)\big(\varphi(\bar{x};\bar{m})\big) = \Psi(\mu)(\nu)\big(\psi(\bar{x};\bar{m})\big),$$
in other words, we need $(\mu_{\bar{y}} \otimes\nu_{\bar{x}} )\big(\varphi(\bar{x};\bar{y})\big)=(\mu_{\bar{y}} \otimes\nu_{\bar{x}})\big( \psi(\bar{x};\bar{y})\big)$.
Note that $(\forall\bar{x})\big(\varphi(\bar{x};\bar{y})\,\leftrightarrow\,\psi(\bar{x};\bar{y})\big)\in\tp(\bar{m}/\emptyset)$, 
so $\pi(\bar{m};\bar{y})\vdash (\forall\bar{x})\big(\varphi(\bar{x};\bar{y})\,\leftrightarrow\,\psi(\bar{x};\bar{y})\big)$,
and therefore for every $q(\bar{y})\in S_{\pi(\bar{m};\bar{y})}(\FC)$ and every $\bar{b}\in \FC^{\bar{x}}$ we have that $q\in[\varphi(\bar{b};\bar{y})]$ if and only if $q\in[\psi(\bar{b};\bar{y})]$.
Thus
$$\mu\big( \varphi(\bar{b};\bar{y})\big) = \mu\big( [\varphi(\bar{b};\bar{y})]\cap[\pi(\bar{m};\bar{y})]\big) = \mu\big(\psi(\bar{b};\bar{y})\cap[\pi(\bar{m};\bar{y})]\big) = \mu\big(\psi(\bar{b};\bar{y})\big).$$
Using the above, we obtain
\begin{align*}
(\mu_{\bar{y}} \otimes\nu_{\bar{x}})\big(\varphi(\bar{x};\bar{y})\big) 
&= \int_{S_{\bar{x}}(M)} F_{\mu_{\bar{y}} }^{\varphi^{\textrm{opp}}(\bar{y};\bar{x})} d\nu_{\bar{x}}\\ 
&= \int_{S_{\bar{x}}(M)} F_{\mu_{\bar{y}}}^{\psi^{\textrm{opp}}(\bar{y};\bar{x})} d\nu_{\bar{x}}\\ 
&= (\mu_{\bar{y}} \otimes\nu_{\bar{x}})\big(\psi(\bar{x};\bar{y})\big).
\end{align*}

We now show that $\eta_\mu\in \E(\mathfrak{M}_{\bar{x}}(M),\conv(G_{\pi,M}))$.
By Proposition 2.11 from \cite{Artem_Kyle}, we know that $\mu$ is in the topological closure of $\conv(M^{\bar{y}})$. 
Actually, from the proof of the aforementioned proposition, the hypothesis that $\mu$ is $\pi$-strongly finitely satisfiable in $M$, and property (3) in Definition \ref{def:group-like}, it follows that
$$\mu=\lim\limits_{i\in I}\lambda_i=\lim\limits_{i\in I} \sum\limits_{j=1}^{k_i}\alpha^i_j\delta_{\sigma^i_j(\bar{m})},$$
for some $\sigma^i_j\in G_{\pi,M}$, where each $\lambda_i = \sum\limits_{j=1}^{k_i}\alpha^i_j\delta_{\sigma^i_j(\bar{m})} \in \conv(M^{\bar{y}})$.

Consider the net $\left(\sum\limits_{j=1}^{k_i}\alpha^i_j(\sigma^i_j)^{-1}\cdot- \right)_{i\in I}$ in $\E(\mathfrak{M}_{\bar{x}}(M),\conv(G_{\pi,M}))$. 
By compactness, passing to a convergent subnet, without loss of generality we may assume that
$\left(\sum\limits_{j=1}^{k_i}\alpha^i_j(\sigma^i_j)^{-1}\cdot- \right)_{i\in I}$ converges to some $\eta$ in $\E(\mathfrak{M}_{\bar{x}}(M),\conv(G_{\pi,M}))$.
Since the Morley product is left-continuous in NIP (see Fact \ref{fact:MP}), 
it follows that for every  $\nu$ in $\mathfrak{M}_{\bar{x}}(M)$ and every $\mathcal{L}_{\bar{x}}(M)$-formula $\varphi(\bar{x};\bar{m})$,
\begin{align*}
\eta_\mu(\nu)\big(\varphi(\bar{x};\bar{m})\big) &= (\mu_{\bar{y}} \otimes\nu_{\bar{x}} )\big(\varphi(\bar{x};\bar{y})\big) \\
&= \left( \Big( \lim\limits_{i\in I}\sum\limits_{j=1}^{k_i}\alpha^i_j\delta_{\sigma^i_j(\bar{m})} \Big)_{\bar{y}} \otimes \nu_{\bar{x}} \right)\big(\varphi(\bar{x};\bar{y})\big) \\
&= \lim\limits_{i\in I}\left( \Big( \sum\limits_{j=1}^{k_i}\alpha^i_j\delta_{\sigma^i_j(\bar{m})} \Big)_{\bar{y}} \otimes \nu_{\bar{x}} \right)\big(\varphi(\bar{x};\bar{y})\big) \\
&= \lim\limits_{i\in I}\left( \Big( \sum\limits_{j=1}^{k_i}\alpha^i_j(\sigma^i_j)^{-1} \Big) \cdot \nu \right)\big(\varphi(\bar{x};\bar{m})\big) \\
&= \eta(\nu)\big(\varphi(\bar{x};\bar{m})\big).
\end{align*}
We see that $\eta_\mu=\eta\in \E(\mathfrak{M}_{\bar{x}}(M),\conv(G_{\pi,M}))$.
\end{proof}

\begin{theorem}
    Assume that $T$ is NIP and $\pi(\bar{x};\bar{y})$ is group-like over $M$. 
    Then $\Phi\Psi=\id$, $\Psi\Phi=\id$, and both maps $\Phi$ and $\Psi$ are homeomorphisms.
\end{theorem}

\begin{proof}
First, we check that $\Phi\Psi=\id$.
    Fix $\mu\in \mathfrak{M}^{\sfs}_{\pi(\bar{m};\bar{y})}(\FC,M)$. We need to show that $\Phi\Psi(\mu)=\mu$.
    So consider any $\CL_{\bar{y}}(\FC)$-formula $\varphi(\bar{n};\bar{y})$, and compute:
    \begin{align*}
    \Phi\Psi(\mu)\big(\varphi(\bar{n};\bar{y})\big) &=   \Psi(\mu) \left(\delta_{\tp(\bar{n}/M)} \right) ( \varphi(\bar{x};\bar{m})) \\
    &=\eta_\mu(\delta_{\tp(\bar{n}/M)}) ( \varphi(\bar{x};\bar{m})) \\
    &= \left(\mu_{\bar{y}} \otimes (\delta_{\tp(\bar{n}/M)})_{\bar{x}} \right) (\varphi(\bar{x};\bar{y})) \\
    &= \mu (\varphi(\bar{n};\bar{y})).
    \end{align*}

    We now argue that $\Psi\Phi=\id$. Fix an arbitrary $\gamma \in \E(\mathfrak{M}_{\bar{x}}(M),\conv(G_{\pi,M}))$, $\nu\in\mathfrak{M}_{\bar{x}}(M)$, and $\varphi(\bar{x};\bar{m})\in\CL_{\bar{x}}(M)$. 
It suffices to show that 
    $$\Psi\Phi(\gamma)(\nu)\big(\varphi(\bar{x};\bar{m})\big)=\gamma(\nu)\big(\varphi(\bar{x};\bar{m})\big).$$
    By applying the definitions on the left hand side, we obtain 
    \begin{align*}
    \Psi\Phi(\gamma)(\nu)\big(\varphi(\bar{x};\bar{m})\big) &= \big(\Phi(\gamma)_{\bar{y}}  \otimes \nu_{\bar{x}} \big)\big(\varphi(\bar{x};\bar{y})\big) \\
    &=  \int_{S_{\bar{x}}(M)} F_{\Phi(\gamma)_{\bar{y}}}^{\varphi^{\textrm{opp}}(\bar{y};\bar{x})} d\nu_{\bar{x}} \\
    &= \int_{q\in S_{\bar{x}}(M)} \big(\gamma(\delta_q)\big)\big(\varphi(\bar{x};\bar{m})\big)d\nu(\bar{x})=(\spadesuit).
    \end{align*}
    We now have three cases depending on $\gamma$. Step one; we compute $(\spadesuit)$ when $\gamma = \sigma \cdot - $ for some $\sigma\in G_{\pi,M}$. Note, 

    \begin{align*}
    (\spadesuit) &= \int\limits_{q\in S_{\bar{x}}(M)} (\sigma_\ast\delta_q)\big(\varphi(\bar{x};\bar{m})\big)d\nu(\bar{x}) = \int\limits_{q\in S_{\bar{x}}(M)} \delta_q\big(\varphi(\bar{x};\sigma^{-1}(\bar{m}))\big)d\nu(\bar{x}) \\
    &= \int\limits_{ S_{\bar{x}}(M)} \mathbf{1}_{[\varphi(\bar{x};\sigma^{-1}(\bar{m}))]}d\nu(\bar{x}) \\
    &= \nu\big( \varphi(\bar{x};\sigma^{-1}(\bar{m})) \big) = (\sigma_\ast \nu)\big( \varphi(\bar{x};\bar{m}) \big)=\gamma(\nu)\big( \varphi(\bar{x};\bar{m}) \big).
    \end{align*}
    Step two; we compute $(\spadesuit)$ when $\gamma$ is a linear combination of elements from $G_{\pi,M}$. Indeed, if $\gamma = \left( \sum\limits_{i\leqslant k}\alpha_i\sigma_i \right) \cdot - $ where each $\sigma_i \in G_{\pi,M}$ and positive real numbers $\alpha_1,\dots,\alpha_k$ such that $\sum_{i\leq k} \alpha_i = 1$, then 
    \begin{align*}
    (\spadesuit) &= \sum\limits_{i\leqslant k}\alpha_i\int\limits_{q\in S_{\bar{x}}(M)}(\sigma_{i\ast}\delta_q)\big( \varphi(\bar{x};\bar{m})\big)d\nu(\bar{x}) \overset{(\dagger)}{=} \sum\limits_{i\leqslant k}\alpha_i (\sigma_{i\ast}\nu)\big(\varphi(\bar{x};\bar{m})\big) =  \\
    &= \left( \left(\sum\limits_{i\leqslant k}\alpha_i\sigma_i \right)\cdot\nu\right)\big(\varphi(\bar{x};\bar{m})\big)=\gamma(\nu)\big(\varphi(\bar{x};\bar{m})\big),
    \end{align*}
where equation $(\dagger)$ follows from the computation in step one. Step three; we assume that $\gamma$ is a limit of a net $(\lambda_i\cdot-)_{i\in I}$, where $\lambda_i\in\conv(G_{\pi,M})$ for each $i\in I$. By continuity of $\Phi$ and left-continuity of the Morely product in NIP theories (Fact \ref{fact:MP}), 
    \begin{align*}
    \Psi(\Phi(\gamma))(\nu)\big(\varphi(\bar{x};\bar{m})\big) &= \eta_{\Phi(\gamma)}(\nu)(\varphi(\bar{x};\bar{m})) \\ 
    &= (\Phi(\gamma)_{\bar{y}} \otimes \nu_{\bar{x}})(\varphi(\bar{x};\bar{y})) \\ 
    &= \left(\Phi \left(\lim_{i \in I} \lambda_i \cdot - \right)_{\bar{y}} \otimes \nu_{\bar{x}} \right)(\varphi(\bar{x};\bar{y})) \\ 
    &= \left(\lim\limits_{i\in I}\Phi\big(\lambda_i\cdot-\big)_{\bar{y}} \otimes \nu_{\bar{x}} \right) (\varphi(\bar{x};\bar{y}))\\
    &= \lim\limits_{i\in I}\left(\Phi\big(\lambda_i\cdot-\big)_{\bar{y}} \otimes \nu_{\bar{x}} \right)(\varphi(\bar{x};\bar{y})) \\
    &\overset{}{=} \lim\limits_{i\in I} \Big( \Psi\Phi(\lambda_i\cdot-)(\nu)\left(\varphi(\bar{x};\bar{m})\right)\Big) \\
    &\overset{(\ddagger)}{=} 
\lim\limits_{i\in I} \Big( (\lambda_i (\nu))(\varphi(\bar{x};\bar{m}))\Big) \\
    &= \gamma(\nu)\big(\varphi(\bar{x};\bar{m})\big),
    \end{align*}
where equation $(\ddagger)$ follows from the computation in step two. 

Thus, $\Psi=\Phi^{-1}$ and $\Phi$ is a continuous bijection between two compact (Hausdorff) spaces. Hence,  both $\Phi$ and $\Psi$ are homeomorphisms.
\end{proof}

Let us note that when $T$ is NIP and $\pi(\bar{x};\bar{y})$ is group-like over $M$, we have the following commutative diagram 
(where $\approx$ means homeomorphism):

$$\xymatrix{& \E(\mathfrak{M}_{\bar{x}}(M),\conv(G_{\pi,M})) \ar[rr]^-{\approx}_-{\Phi}& & \mathfrak{M}^{\sfs}_{\pi(\bar{m};\bar{y})}(\FC,M)  \\
\E(\mathfrak{M}_{\bar{x}}(M),G_{\pi,M}) \ar[ur]_-{\subseteq}^-{\text{closed}} &   & & \\
& E:=\E(S_{\bar{x}}(M),G_{\pi,M}) \ar[rr]^-{\approx}_-{\Phi|_E} 
\ar[ul]^-{\Gamma}_-{\cong}
\ar[uu]
& & S^{\sfs}_{\pi(\bar{m};\bar{y})}(\FC,M) \ar[uu]_-{p\mapsto\delta_p}^-{
\substack{\text{homeomorphic} \\ \text{embedding}}}}$$

In the above diagram, the map $\Gamma$ is the inverse of the canonical restriction map from $\E(\mathfrak{M}_{\bar{x}}(M),\conv(\aut(M)))$ to $E$. 
We recall that this map is an isomorphism of topological semigroups by Proposition \ref{proposition:NIP.isomorphisms}(3). With a slight abuse of notation, we identify $E$ with its isomorphic copy in $\E(\mathfrak{M}_{\bar{x}}(M),\conv(G_{\pi,M}))$ to be able to compute $\Phi|_E$. 
Since $\Phi$ takes values in $\mathfrak{M}^{\sfs}_{\pi(\bar{m};\bar{y})}(\FC,M)$, using the definition of $\Phi$, one easily gets that the image of $\Phi|_{E}$ is contained in $S_{\pi(\bar{m};\bar{y})}^{\sfs}(\FC,M)$ (see Remark \ref{remark: formula for Phi|_E}).
We will show that $\Phi|_{E}$ is a homeomorphism. But before that let us give an explicit formula for $\Phi|_E$.

\begin{remark}\label{remark: formula for Phi|_E}
$\Phi|_E$ takes values in $S^{\sfs}_{\pi(\bar{m};\bar{y})}(\FC,M)$ and is explicitly given by the formula
$\Phi|_E(\eta)=\{\varphi(\bar{n};\bar{y})\;\colon\;\varphi(\bar{x};\bar{m})\in \eta\big(\tp(\bar{n}/M)\big)\}$ for each $\eta \in E$.
\end{remark}

\begin{proof}
Directly from the definition of $\Phi$ we get that
    \begin{align*}
     \Phi(\Gamma(\eta))\big(\varphi(\bar{n};\bar{y})\big)=(\Gamma(\eta)(\delta_{\tp(\bar{n}/M)}))\big(\varphi(\bar{x};\bar{m})\big) = \delta_{\eta(\tp(\bar{n}/M))}\big(\varphi(\bar{x};\bar{m})\big) \in \{0,1\},
    \end{align*}
which is equal to $1$ if and only if $\varphi(\bar{x};\bar{m})\in \eta\big(\tp(\bar{n}/M)\big)$.
Thus, 
$$\Phi|_E(\eta) = \Phi(\Gamma(\eta)) \in \mathfrak{M}^{\sfs}_{\pi(\bar{m};\bar{y})}(\FC,M) \cap S_{\bar{y}}(\FC)\subseteq S^{\sfs}_{\pi(\bar{m};\bar{y})}(\FC,M),$$ and  $\varphi(\bar{n};\bar{y}) \in  \Phi|_E(\eta)$ if and only if $\varphi(\bar{x};\bar{m})\in \eta\big(\tp(\bar{n}/M)\big)$.
\end{proof}

\begin{proposition}
In the above diagram, we indeed have that
$$\Phi[E]=S^{\sfs}_{\pi(\bar{m};\bar{y})}(\FC,M),$$
and so $\Phi|_E$ is a homeomorphism.
\end{proposition}

\begin{proof}
We only need to argue that 
$\Phi[E]=S^{\sfs}_{\pi(\bar{m};\bar{y})}(\FC,M)$.
Consider any $q \in S_{\pi(\bar{m};\bar{y})}^{\sfs}(\FC,M)$. Choose a net $(p_i)_{i \in I}$ of types realized in $M$ and converging to $q$ each of which satisfies the partial type $\pi(\bar{m};\bar{y})$. Since $\pi(\bar{x};\bar{y})$ is group-like over $M$, we may find those types $p_i$ to be of the form $\tp(\sigma_i(\bar{m})/\FC)$ for some $\sigma_i\in G_{\pi,M}$. 
Passing to a convergent subnet, we may assume that the net $(\sigma_i^{-1} \cdot -)_{i \in I}$ of elements of $E$ converges to some $\eta \in E$. 
Then, by Remark \ref{remark: formula for Phi|_E}, $\varphi(\bar n;\bar y) \in (\Phi|_E)(\eta)$ if and only if there exists $i_0 \in I$ such that $\varphi(\bar{x};\bar{m})\in \sigma_i^{-1}(\tp(\bar{n}/M))$ for all $i>i_0$. This is equivalent to the existence of $i_0 \in I$ such that $\varphi(\bar{x};\sigma_i(\bar{m}))\in \tp(\bar{n}/M)$ for all $i>i_0$, which in turn is equivalent to the condition $\varphi(\bar n; \bar y) \in q$. Thus, $\Phi|_{E}(\eta)=q$.
\end{proof}

Let us also give an explicit formula for $(\Phi|_E)^{-1} = \Psi|_{S^{\sfs}_{\pi(\bar{m};\bar{y})}(\FC,M)}$.

\begin{remark}\label{remark: formula for Phi|_E -1}
$(\Phi|_E)^{-1}(p)(r)= \{\varphi(\bar{x};\bar{m})\;\colon\;(\forall \bar{m}'\models r)\big(\,\varphi(\bar{m}';\bar{y})\in p\,\big)\}$ for each $p \in S^{\sfs}_{\pi(\bar{m};\bar{y})}(\FC,M)$ and $r \in S_{\bar x}(M)$.
\end{remark}

\begin{proof}
Pick $\bar m' \in  \FC^{\bar{x}}$ so that $r=\tp(\bar{m}'/M)$, and consider any formula $\varphi(\bar{x};\bar{m})\in\CL_{\bar{x}}(M)$.
 Then
    \begin{align*}
    \Psi(p)(r)\big(\varphi(\bar{x};\bar{m})\big) &= (\delta_p\otimes\delta_r)\big(\varphi(\bar{x};\bar{y})\big) \\
    &= \int\limits_{\substack{q\in S_{\bar{x}}(M) \\ \bar{b}\in q(\FC)}} \delta_p\big(\varphi(\bar{b};\bar{y})\big)d\delta_r= \delta_p\big(\varphi(\bar{m}';\bar{y})\big).
    \end{align*}
Since $\varphi(\bar{x};\bar{m}) \in (\Phi|_E)^{-1}(r)$ if and only if $\Psi(p)(r)\big(\varphi(\bar{x};\bar{m})\big)=1$, we conclude that 
\begin{equation*}
(\Phi|_E)^{-1}(p)(r)=\{\varphi(\bar{x};\bar{m})\;\colon\;(\forall \bar{m}'\models r)\big(\,\varphi(\bar{m}';\bar{y})\in p\,\big)\}.\qedhere
\end{equation*}
\end{proof}

In \cite[Proposition 3.14]{Udi_def_patterns}, the author found a correspondence between the endomorphism group $\textrm{End}(S_{\bar x}(M))$ of $S_{\bar x}(M)$ with respect to  the {\em definability patterns structure on $S_{\bar x}(M)$} and the space $S^{\fs}_{\bar m}(\FC,M)$. Remarks \ref{remark: formula for Phi|_E} and \ref{remark: formula for Phi|_E -1} applied to the type $\pi(\bar x;\bar y):= (\bar x \equiv_\emptyset \bar y)$ in the context of a strongly $\aleph_0$-homogeneous model $M$ recover the restriction of Hrushovski's correspondence to the Ellis semigroup $E$ (which is a subsemigroup of $\textrm{End}(S_{\bar x}(M))$) with the target space of the restricted correspondence being the subspace $S^{\sfs}_{\bar{m}}(\FC,M)$ of $S^{\fs}_{\bar{m}}(\FC,M)$.

\subsubsection{Transferring product} 
The purpose of this short section is to transfer the semigroup operation in $\E(\mathfrak{M}_{\bar{x}}(M),\conv(G_{\pi,M}))$ via $\Phi$ and $\Psi$ to a new semigroup operation on $\mathfrak{M}^{\sfs}_{\pi(\bar{m};\bar{y})}(\FC,M)$, when $T$ is NIP and $\pi(\bar{x};\bar{y})$ is group-like over $M$. An explicit formula for this new semigroup operation will be derived later in Proposition \ref{prop:star.product.1}.

\begin{definition}\label{def:star.product.definition.0}
    Suppose that $T$ is NIP and $\pi(\bar{x};\bar{y})$ is group-like over $M$. Let $\mu$ and $\nu$ be Keisler measures in $\mathfrak{M}^{\sfs}_{\pi(\bar{m};\bar{y})}(\FC,M)$. We define $\mu\ast\nu\in \mathfrak{M}^{\sfs}_{\pi(\bar{m};\bar{y})}(\FC,M)$ via
    $$\mu\ast\nu:=\Phi(\Psi(\mu)\circ \Psi(\nu)).$$
\end{definition}

\begin{theorem}\label{thm: Newelski's analogon} Suppose that $T$ is NIP and $\pi(\bar{x};\bar{y})$ is group-like over $M$.
Then $(\mathfrak{M}^{\sfs}_{\pi(\bar{m};\bar{y})}(\FC,M),\ast)$ is a left topological semigroup (i.e.,  $*$ is left-continuous). Moreover, $(S^{\sfs}_{\pi(\bar{m};\bar{y})}(\FC,M),\ast)$ is isomorphic to a closed sub-semigroup of $(\mathfrak{M}^{\sfs}_{\pi(\bar{m};\bar{y})}(\FC,M),\ast)$, and we have the following isomorphisms of topological semigroups 
(vertical maps are formally embeddings, but after identifications we can assume that they are inclusions):
    $$\xymatrixcolsep{5pc}\xymatrix{\E(\mathfrak{M}_{\bar{x}}(M),\conv(G_{\pi,M})) \ar[r]^-{\cong}_-{\Phi} &  \big( \mathfrak{M}^{\sfs}_{\pi(\bar{m};\bar{y})}(\FC,M),\ast\big) \\
    \E(S_{\bar{x}}(M),G_{\pi,M}) \ar[r]^-{\cong}_-{\Phi}
    \ar[u]^-{\leqslant}& \big( S^{\sfs}_{\pi(\bar{m};\bar{y})}(\FC,M),\ast\big)\ar[u]^-{\leqslant}}.$$
\end{theorem}

\begin{theorem}\label{thm: mother star space}
    Suppose that $T$ is NIP,
and  let $\pi_1(\bar{x};\bar{y})\subseteq\pi_2(\bar{x};\bar{y})$ be partial types group-like over $M$.
    Then naturally $G_{\pi_2,M}\leqslant G_{\pi_1,M}$, and so
    $\E(\mathfrak{M}_{\bar{x}}(M),\conv(G_{\pi_2,M}))\subseteq \E(\mathfrak{M}_{\bar{x}}(M),\conv(G_{\pi_1,M}))$ and $\mathfrak{M}^{\sfs}_{\pi_2(\bar{m};\bar{y})}(\FC,M) \subseteq 
    \mathfrak{M}^{\sfs}_{\pi_1(\bar{m};\bar{y})}(\FC,M)$ (and similarly for types). 
    Moreover, the maps $\Phi$ and $\Psi$ defined for $\pi_2$ are the restrictions of the corresponding maps defined for $\pi_1$, the following diagram commutes, and each restriction of $\Phi$ in it is an isomorphism of topological semigroups:
    \ \\
    \begin{center}
    \adjustbox{scale=0.75,center}{
    \begin{tikzcd}
        \E(\mathfrak{M}_{\bar{x}}(M),\conv(G_{\pi_1,M})) \arrow[rr, "\Phi"]  & & \big(\mathfrak{M}^{\sfs}_{\pi_1(\bar{m};\bar{y})}(\FC,M),\,\ast\big) \arrow[from=dd, "\leqslant", crossing over, near start, swap] & \\
        & \E(S_{\bar{x}}(M),G_{\pi_1,M}) \arrow[ul, "\leqslant"]  \arrow[rr, "\Phi", crossing over, near start] & & \big(S^{\sfs}_{\pi_1(\bar{m};\bar{y})}(\FC,M),\,\ast\big) \arrow[ul, "\leqslant", swap]  \\
        \E(\mathfrak{M}_{\bar{x}}(M),\conv(G_{\pi_2,M})) \arrow[uu, "\leqslant"] \arrow[rr, "\Phi", near end]
         & & \big(\mathfrak{M}^{\sfs}_{\pi_2(\bar{m};\bar{y})}(\FC,M),\,\ast\big)
         & \\
        & \E(S_{\bar{x}}(M),G_{\pi_2,M}) 
        \arrow[uu, "\leqslant", crossing over, near start, swap] 
        \arrow[ul, "\leqslant"] \arrow[rr, "\Phi"] 
        & & \big(S^{\sfs}_{\pi_2(\bar{m};\bar{y})}(\FC,M),\,\ast\big) \arrow[ul, "\leqslant", swap] \arrow[uu, "\leqslant", swap].
    \end{tikzcd}
}
\end{center}
\end{theorem}

\begin{proof}
    By examining the definitions.
\end{proof}

By Theorem \ref{thm: mother star space},  we see that all the semigroups of the form $\mathfrak{M}^{\sfs}_{\pi(\bar{m};\bar{y})}(\FC,M)$ are subsemigroups of the semigroup $\mathfrak{M}^{\sfs}_{\bar m}(\FC,M)$ (i.e., the one obtained for $\pi(\bar x;\bar y)$ equal to $\bar x \equiv \bar y$).

\subsection{New semigroup of Keisler measures}\label{subsec:star.definitions}
In this section, we first find an explicit formula for the $*$ operation from Definition \ref{def:star.product.definition.0}. Next, we extend it to a much larger class of measures and develop the fundamentals of the theory of the obtained convolution product.

\subsubsection{Generalizing the $\ast$-product}\label{subsec: star new definition} 
Here, we let $M$ be an arbitrary small elementary substructure of the monster model $\FC$
and $\FC'\succ \FC$ be a monster model in which $\FC$ is small. 
Throughout the section, types in $S_{\bar{m}}(\FC)$ and measures in $\mathfrak{M}_{\bar{m}}(\FC)$ will be in variables $\bar{y}$. We do this so that one can easily see how the $*$-product can be generalized without too much variable confusion. For each type $p(\bar{y})\in S_{\bar{m}}(\FC)$ there exists $\sigma\in\aut(\FC')$ such that $p(\bar{y})=\tp(\sigma(\bar{m})/\FC)$, so we can always present such types in that form. 
By $\bar x$ we will denote another tuple of variables corresponding to $\bar m$. We begin by defining some auxiliary maps which will allow us to define the product.

\begin{definition}
    Take $\bar{b}\in \FC^{\bar{x}}$ and consider the following map
    $$h_{\bar{b}}:S_{\bar{m}}(\FC)\to S_{\bar{x}}(M),$$
    defined by:
    $$h_{\bar{b}}:\tp(\sigma(\bar{m})/\FC)\mapsto\tp(\sigma^{-1}(\bar{b})/M).$$
\end{definition}

We show that $h_{\bar{b}}$ is well-defined. Indeed, suppose that $\sigma(\bar{m})\equiv_\FC\tau(\bar{m})$. 
Then there exists $\gamma \in\aut(\FC'/\FC)$ such that $\gamma \sigma(\bar{m})=\tau(\bar{m})$, i.e. $\tau^{-1}\gamma \sigma(\bar{m})=\bar{m}$ and so $\tau^{-1}\gamma \sigma\in\aut(\FC'/M)$. Since $\gamma$ fixes $\FC$ pointwise,
$$\tau^{-1}(\bar{b})=(\tau^{-1}\gamma\sigma)\sigma^{-1}(\bar{b}).$$
Hence $\tau^{-1}(\bar{b})\equiv_M\sigma^{-1}(\bar{b})$, and so we conclude that $h_{\bar{b}}$ is well-defined.

\begin{lemma}\label{prop:compute} Let $\bar{b}\in \FC^{\bar{x}}$ and 
$\varphi(\bar{x};\bar{m})\in\CL_{\bar{x}}(M)$. Then 
$$h_{\bar{b}}^{-1}[\theta(\bar{x};\bar{m})]=[\theta(\bar{b};\bar{y})].$$

As consequence,
\begin{enumerate} 
\item The map $h_{\bar{b}}: S_{\bar{m}}(\FC) \to S_{\bar{x}}(M)$ is continuous. 
\item For any $\mu \in \mathfrak{M}_{\bar{y}}(\FC)$, we have that $((h_{\bar{b}})_{*}\mu)(\theta(\bar{x};\bar{m})) = \mu(\theta(\bar{b};\bar{y}))$. 
\end{enumerate}  
\end{lemma}

\begin{proof} Follows directly from the definitions. 
\end{proof} 

\begin{definition}
    Let $\bar{b}\in \FC^{\bar{x}}$, $\varphi(\bar{x};\bar{y})\in\CL_{\bar{x};\bar{y}}(\emptyset)$ 
    and $\mu\in\mathfrak{M}_{\bar{y}}^{\inv}(\FC,M)$.
    We define the map
    $$G^{\varphi(\bar{b};\bar{y})}_{\mu}:S_{\bar{m}}(\FC)\to[0,1]$$
    by
    $$G^{\varphi(\bar{b};\bar{y})}_{\mu}\big(\tp(\sigma(\bar{m})/\FC)\big):= \hat{\mu}\big(\varphi(\sigma^{-1}(\bar{b});\bar{y})\big),$$
    where $\hat{\mu}$ is the unique extension of $\mu$ to a measure in $\mathfrak{M}_{\bar{y}}^{\inv}(\FC',M)$. 
\end{definition}

The above map $G^{\varphi(\bar{b};\bar{y})}_{\mu}$ is well-defined. Indeed, if $\sigma(\bar{m})\equiv_\FC \tau(\bar{m})$, then we have $\sigma^{-1}(\bar{b})\equiv_M\tau^{-1}(\bar{b})$ 
as $h_{\bar b}$ is well-defined. 
Since $\hat{\mu}$ is $M$-invariant, we conclude that $\hat{\mu}\big(\varphi(\sigma^{-1}(\bar{b});\bar{y})\big)=\hat{\mu}\big(\varphi(\tau^{-1}(\bar{b});\bar{y})\big)$. 

However, the previous definition is a little cumbersome since it involves extending the measures to a larger model. Our next observation show that these maps decompose into ones we are already familiar with. 

\begin{proposition}\label{prop:Borel_G} Let $\bar{b}\in \FC^{\bar{x}}$, $\varphi(\bar{x};\bar{y})\in\CL_{\bar{x};\bar{y}}(\emptyset)$ 
    and $\mu\in\mathfrak{M}_{\bar{y}}^{\inv}(\FC,M)$. Then 
    \begin{equation*} 
    G^{\varphi(\bar{b};\bar{y})}_{\mu} = F_{\mu,M}^{\varphi^{\textrm{opp}}(\bar{y};\bar{x})} \circ h_{\bar{b}}. 
    \end{equation*} 
    As consequence, if $\mu$ is Borel-definable over $M$, then $G_{\mu}^{\varphi(\bar{b};\bar{y})}$ is the composition of a Borel function with a continuous function and hence Borel. 
\end{proposition} 

\begin{proof} Follows directly from the definitions. 
\end{proof}

We are ready to derive an explicit formula for the $\ast$-product from Definition \ref{def:star.product.definition.0}.

\begin{proposition}\label{prop:star.product.1}
    Assume $T$ is NIP and $\pi(\bar x;\bar y)$ is group-like over $M$.
    Let $\mu, \nu \in\mathfrak{M}_{\pi(\bar{m};\bar y)}^{\sfs}(\FC,M)$, $\bar{b}\in \FC^{\bar{x}}$, and $\varphi(\bar{x};\bar{y})$ be an $\CL_{\bar{x},\bar{y}}$-formula.
    For the $\ast$-product from Definition \ref{def:star.product.definition.0}, we have
    $$(\mu \ast \nu )\big(\varphi(\bar{b};\bar{y})\big)=\int_{S_{\bar{m}}(\FC)}G^{\varphi(\bar{b};\bar{y})}_{\mu}d\nu. $$
\end{proposition}

\begin{proof}
Since $T$ is NIP, $\mu$ is Borel-definable over $M$ (see Fact \ref{fact:MP}). By Proposition \ref{prop:Borel_G}, the map $G_{\mu}^{\varphi(\bar{b};\bar{y})}$ is Borel. Thus, the right hand side of the equation in question is well-defined.

Observe that for any $\mathcal{L}_{\bar{x},\bar{y}}$-formula $\theta(\bar{x};\bar{y})$, we have
\begin{align*}
\eta_{\nu}(\delta_{\tp(\bar{b}/M)})\big(\theta(\bar{x};\bar{m})\big) &=
(\nu \otimes \delta_{\tp(\bar{b}/M)})\big(\theta(\bar{x};\bar{y})\big) \\ &=
\nu\big(\theta(\bar{b};\bar{y})\big)=
\big((h_{\bar{b}})_{\ast} (\nu) \big)\big(\theta(\bar{x};\bar{m})\big).
\end{align*}
Hence, $\eta_{\nu}(\delta_{\tp(\bar{b}/M)}) = (h_{\bar{b}})_{\ast} (\nu) $. By the definition of the $*$-product, we have the following:
\begin{align*}
    (\mu \ast \nu )\big(\varphi(\bar{b};\bar{y})\big) &= \Phi(\Psi(\mu) \circ \Psi(\nu))(\varphi(\bar{b};\bar{y})) \\ 
    &= (\eta_{\mu} \circ \eta_{\nu})(\delta_{\tp(\bar{b}/M)})(\varphi(\bar x;\bar m))\\
    &= \eta_{\mu}((  \eta_{\nu}(\delta_{\tp(\bar{b}/M)}))(\varphi(\bar x;\bar m))\\
    &=\Big(\mu_{\bar{y}} \otimes\big( \eta_{\nu}(\delta_{\tp(\bar{b}/M)}) \big)_{\bar{x}}\Big)\big(\varphi(\bar{x};\bar{y})\big)  \\
    &= \int_{S_{\bar{x}}(M)} F^{\varphi^{\textrm{opp}}(\bar{y};\bar{x})}_{\mu_{\bar{y}}}d \left( \eta_{\nu}(\delta_{\tp(\bar{b}/M)}) \right)_{\bar{x}} \\
    &= \int_{S_{\bar{x}}(M)} F^{\varphi^{\textrm{opp}}(\bar{y};\bar{x})}_{\mu_{\bar{y}}}d  \left( (h_{\bar{b}})_{\ast}(\nu) \right)_{\bar{x}} \\
    &= \int_{S_{\bar{m}}(\FC)} \Big(F^{\varphi^{\textrm{opp}}(\bar{y};\bar{x})}_{\mu_{\bar{y}} }\circ h_{\bar{b}}\Big) d\nu_{\bar{y}} = \int_{S_{\bar{m}}(\FC)} G^{\varphi(\bar{b};\bar{y})}_{\mu} d\nu.\qedhere
    \end{align*}
\end{proof}

Using Proposition \ref{prop:star.product.1}, we may \emph{definitionally extend} the $\ast$-product to a much larger class of measures. The following definition does just that; it extends the $*$-product to relatively general pairs of measures over arbitrary theories. 

\begin{definition}\label{def:star.definition}
    Let $\mu \in\mathfrak{M}_{\bar{y}}^{\inv}(\FC,M)$ be Borel-definable over $M$ and $\nu \in\mathfrak{M}_{\bar{m}}(\FC)$. 
(Notice that both $\mu$ and $\nu$ are in variables $\bar{y}$.)
We define the convolution product $\mu \ast \nu \in\mathfrak{M}_{\bar{y}}(\FC)$ as follows: For any  $\bar{b}\in \FC^{\bar{x}}$ and $\CL_{\bar{x};\bar{y}}$-formula $\varphi(\bar{x};\bar{y})$, we define
$$(\mu \ast \nu )\big(\varphi(\bar{b};\bar{y})\big):= \Big(\mu_{\bar{y}} \otimes \big((h_{\bar{b}})_{\ast} (\nu) \big)_{\bar{x}} \Big)\big(\varphi(\bar{x};\bar{y})\big).$$
\end{definition}

\begin{remark} By the computation in the proof of Proposition \ref{prop:star.product.1}, it follows that
\begin{align*} 
(\mu \ast \nu )\big(\varphi(\bar{b};\bar{y})\big) &= \int_{S_{\bar{x}}(M)} F_{\mu_{\bar{y}}}^{\varphi^{\textrm{opp}}(\bar{y};\bar{x})} d \Big((h_{\bar{b}})_* (\nu) \Big)_{\bar{x}} \\ 
 &= \int_{S_{\bar{m}}(\FC)} \left( F_{\mu_{\bar{y}}}^{\varphi^{\textrm{opp}}(\bar{y};\bar{x})} \circ h_{\bar{b}}\right) d\nu_{\bar{y}} = \int_{S_{\bar{m}}(\FC)} G_{\mu}^{\varphi(\bar{b};\bar{y})} d\nu. 
\end{align*} 
We will often use the equalities above without comment. 
\end{remark}

By Proposition \ref{prop:star.product.1}, the $\ast$-product from Definition \ref{def:star.definition} extends the $\ast$-product from 
Definition \ref{def:star.product.definition.0}. 

It is easy to check that given $\mu$ and $\nu$ as above, $(\mu \ast \nu )\big(\varphi(\bar{b};\bar{y})\big)$ depends only on the equivalence class of $\varphi(\bar{b};\bar{y})$ and that $\mu \ast \nu$ is a Keisler measure.
We now prove a family of preservation results.

\begin{lemma}\label{lemma:def.Borel.def}
    Let $\mu \in\mathfrak{M}_{\bar{y}}^{\inv}(\FC,M)$ be Borel-definable over $M$ and let $\nu \in\mathfrak{M}_{\bar{m}}(\FC)$. The following statements are true.
    \begin{enumerate}
        \item If $\mu \in\mathfrak{M}_{\bar{y}}^{\inv}(\FC,M)$ and $\nu \in\mathfrak{M}_{\bar{m}}^{\inv}(\FC,M)$, then $\mu \ast \nu\in\mathfrak{M}_{\bar{y}}^{\inv}(\FC,M)$. 

        \item If $\mu \in\mathfrak{M}_{\bar{m}}^{\inv}(\FC,M)$ and $\nu\in\mathfrak{M}_{\bar{m}}(\FC)$, then $\mu\ast\nu\in\mathfrak{M}_{\bar{m}}(\FC)$. 

        \item If $\mu, \nu \in \mathfrak{M}_{\bar{m}}^{\inv}(\FC,M)$, then $\mu\ast\nu\in\mathfrak{M}_{\bar{m}}^{\inv}(\FC,M)$.

        \item If $\mu$ is definable over $M$ and $\nu$ is Borel-definable over $M$, then $\mu\ast\nu$ is Borel-definable over $M$.

        \item If $\mu$ and $\nu$ are definable over $M$, then $\mu\ast\nu$ is definable over $M$. 

        \item If $\mu,\nu\in\mathfrak{M}_{\bar{m}}^{\fs}(\FC,M)$, then $\mu\ast\nu\in\mathfrak{M}_{\bar{m}}^{\fs}(\FC,M)$.

        \item Let $\pi(\bar{x};\bar{y})$ be a partial type over $\emptyset$ containing ``$\bar{x}\equiv_{\emptyset}\bar{y}$'', such that
        $G_{\pi,\FC}$ is a subgroup of $\aut(\FC)$.
        If $\mu,\nu\in\mathfrak{M}_{\pi(\bar{m};\bar{y})}^{\inv}(\FC,M)$, then
        $\mu\ast\nu\in \mathfrak{M}_{\pi(\bar{m};\bar{y})}^{\inv}(\FC,M)$.
    \end{enumerate}
\end{lemma}

\begin{proof}
Proof of (1):
    Let $\bar{b},\bar{c}\in \FC^{\bar{x}}$ such that $\bar{b}\equiv_M\bar{c}$.
    For each $\mathcal{L}_{\bar{x}}(M)$-formula $\theta(\bar{x};\bar{m})$, we have
    $$\big((h_{\bar{b}})_{\ast}(\nu)\big)\big(\theta(\bar{x};\bar{m})\big)=\nu\big(\theta(\bar{b};\bar{y})\big)=\nu\big(\theta(\bar{c};\bar{y})\big)= \big((h_{\bar{c}})_{\ast}(\nu)\big)\big(\theta(\bar{x};\bar{m})\big),$$
    and so $(h_{\bar{b}})_{\ast}(\nu)=(h_{\bar{c}})_{\ast}(\nu)$. Hence for any $\mathcal{L}_{\bar{x};\bar{y}}$-formula $\varphi(\bar{x};\bar{y})$, 
    \begin{align*}
    (\mu\ast\nu)\big(\varphi(\bar{b};\bar{y})\big) &= \int_{S_{\bar{x}}(M)}F^{\varphi^{\textrm{opp}}(\bar{y};\bar{x})}_{\mu}d (h_{\bar{b}})_{\ast}(\nu) \\
    &= \int_{S_{\bar{x}}(M)}F^{\varphi^{\textrm{opp}}(\bar{y};\bar{x})}_{\mu}d (h_{\bar{c}})_{\ast}(\nu) = (\mu\ast\nu)\big(\varphi(\bar{c};\bar{y})\big). 
    \end{align*}

 Proof of (2): We show $\mu\ast\nu\in \mathfrak{M}_{\bar{m}}(\FC)$.
    Let $\varphi(\bar{y})\in\tp(\bar{m}/\emptyset)$, then
    $$(\mu\ast\nu)\big(\varphi(\bar{y})\big) = \int_{S_{\bar{m}}(\FC)}G^{\varphi(\bar{y})}_{\mu}d\nu=\int_{S_{\bar{m}}(\FC)}\mu\big(\varphi(\bar{y})\big)d\nu=\mu\big(\varphi(\bar{y})\big)=1.$$

 Proof of (3): Follows directly from (1) and (2).

Proof of (4) and (5): Fix an $\CL$-formula $\varphi(\bar{x};\bar{y})$ and $\epsilon > 0$. 
Since $\mu$ is definable over $M$, 
the function $F^{\varphi(\bar{y}:\bar{x})}_{\mu}:S_{\bar{x}}(M)\to[0,1]$
is continuous. Hence there exist formulas $\{\psi_{i}(\bar{x};\bar{m})\}_{i=1}^{n}$ and real numbers $r_1,\dots,r_n$ such that 
\begin{equation*}
    \sup_{q \in S_{\bar{x}}(M)} \left| F_{\mu}^{\varphi^{\textrm{opp}}(\bar{y};\bar{x})}(q) - \sum_{i=1}^{n} r_i\mathbf{1}_{[\psi_i(\bar{x};\bar{m})]}(q) \right| < \epsilon. 
\end{equation*}
Note that for any $q(\bar{x})\in S_{\bar{x}}(M)$ and $\bar{b} \models q$ for some $\bar{b}\in \FC^{\bar{x}}$, we have:
\begin{align*}\tag{$\diamondsuit$}
F_{\mu \ast \nu}^{\varphi^{\textrm{opp}}(\bar{y}:\bar{x})}(q) &= (\mu\ast\nu)(\varphi(\bar{b};\bar{y}))  = (\mu \otimes (h_{\bar{b}})_{\ast}(\nu))(\varphi(\bar{x};\bar{y})) \\ 
&= \int_{S_{\bar{x}}(M)} F_{\mu}^{\varphi^{\textrm{opp}}(\bar{y};\bar{x})} \,d(h_{\bar{b}})_{\ast}(\nu)\\
&\approx_{\epsilon}  \int_{S_{\bar{x}}(M)} \sum_{i=1}^{n}r_i \mathbf{1}_{[\psi_{i}(\bar{x};\bar{m})]} \,d(h_{\bar{b}})_{\ast}(\nu) \\ 
&= \sum_{i=1}^{n} r_i ((h_{\bar{b}})_{\ast}(\nu))(\psi_{i}(\bar{x};\bar{m})) = \sum_{i=1}^{n} r_i \nu(\psi_{i}(\bar{b};\bar{y})) \\ 
&= \sum_{i=1}^{n} r_i F_{\nu}^{\psi_{i}^{\textrm{opp}}(\bar{y};\bar{x})}(q). 
\end{align*}
Therefore if $\nu$ is Borel-definable over $M$, the final term is a linear combination of Borel functions, hence Borel. Additionally, if $\nu$ is definable over $M$, then the final term is a linear combination of continuous functions, thus continuous. 

Proof of (6): Fix an $\CL$-formula $\varphi(\bar{x};\bar{y})$ and an element $\bar{b}\in \FC^{\bar{x}}$ such that
    $$0<(\mu\ast\nu)\big(\varphi(\bar{b};\bar{y})\big)=\int_{S_{\bar{m}}(\FC)}\Big( F_{\mu}^{\varphi^{\textrm{opp}}(\bar{y};\bar{x})}\circ h_{\bar{b}}\Big) d\nu.$$
    Then there exists $q(\bar{y})\in\supp(\nu)\subseteq S_{\bar{m}}(\FC)$ such that $(F_{\mu}^{\varphi(\bar{y};\bar{x})}\circ h_{\bar{b}})(q) > 0 $.
    Consider $\sigma\in\aut(\FC')$ such that $q=\tp(\sigma(\bar{m})/\FC)$ and
    $\bar{c}\in \FC^{\bar{x}}$ with $\sigma^{-1}(\bar{b})\equiv_M\bar{c}$.
    Then we have
    $$0<\Big(F^{\varphi^{\textrm{opp}}(\bar{y};\bar{x})}_{\mu}\circ h_{\bar{b}}\Big)(q)=\mu\big(\varphi(\bar{c};\bar{y})\big),$$
    and, since $\mu$ is finitely satisfiable in $M$, there is $\bar{a}\in M^{\bar{y}}$ such that $\FC\models\varphi(\bar{c};\bar{a})$. 
    Since $\bar{c}$ and $\sigma^{-1}(\bar{b})$ have the same type over $M$, $\FC' \models \varphi(\sigma^{-1}(\bar{b});\bar{a})$. Thus $\FC' \models \varphi(\bar{b};\sigma(\bar{a}))$. 
Hence there are indices $i_1,\dots,i_n$ such that
    $$\varphi(\bar{b};y_{i_1},\dots,y_{i_n}) \in \tp(\sigma(\bar{m})/\FC)=q.$$
    Finally, since $q \in \supp(\nu)$ and $\nu$ is finitely satisfiable in $M$, we conclude that $q$ is also finitely satisfiable in $M$. Thus, there exists $\bar{d} \in M^{\bar{y}}$ such that $\FC \models \varphi(\bar{b};\bar{d})$, completing the proof.  
    
Proof of (7): By (3), we have that $\mu\ast\nu\in\mathfrak{M}_{\bar{m}}^{\inv}(\FC,M)$.
Consider $\bar{b}\in \FC^{\bar{x}}$ and $\varphi(\bar{b};\bar{y})$
such that $\pi(\bar{m};\bar{y})\vdash\varphi(\bar{b};\bar{y})$.
Note that for every $\sigma\in\aut(\FC')$ with $\models\pi(\bar{m};\sigma(\bar{m}))$ we have also
$\pi(\bar{m};\bar{y})\vdash\varphi(\sigma^{-1}(\bar{b});\bar{y})$.
Indeed, let $\tp(\tau(\bar{m})/\FC)\in[\pi(\bar{m};\bar{y})]\subseteq S_{\bar{m}}(\FC)$.
By applying $\tau^{-1}\sigma^{-1}$, we have 
$\pi(\tau^{-1}\sigma^{-1}(\bar{m});\bar{y})\vdash\varphi(\tau^{-1}\sigma^{-1}(\bar{b});\bar{y})$.
Because $\sigma,\tau\in G_{\pi,\FC'}$, which is a subgroup of $\aut(\FC')$ by assumption and Remark \ref{remark: group in every model}, we conclude that $\models\pi(\tau^{-1}\sigma^{-1}(\bar{m});\bar{m})$.
Hence, $\models\varphi(\tau^{-1}\sigma^{-1}(\bar{b});\bar{m})$,
which implies $\tp(\tau(\bar{m})/\FC)\in [\varphi(\sigma^{-1}(\bar{b});\bar{y})]$.

Recall that $\supp(\mu),\supp(\nu)\subseteq[\pi(\bar{m};\bar{y})]$, so we compute:
\begin{align*}
    (\mu\ast\nu)(\varphi(\bar{b};\bar{y})) &= \int_{S_{\bar{m}}(\FC)} G^{\varphi(\bar{b};\bar{y})}_{\mu}d\nu = \int\limits_{\tp(\sigma(\bar{m})/\FC)\in S_{\bar{m}}(\FC)}\mu\big(\varphi(\sigma^{-1}(\bar{b});\bar{y})\big)d\nu \\
    &= \int\limits_{\substack{\tp(\sigma(\bar{m})/\FC)\in S_{\bar{m}}(\FC) \\ \models\pi(\bar{m};\sigma(\bar{m}))}}\mu\big(\varphi(\sigma^{-1}(\bar{b});\bar{y})\big)d\nu \\
    &\geqslant  \int\limits_{\substack{\tp(\sigma(\bar{m})/\FC)\in S_{\bar{m}}(\FC) \\ \models\pi(\bar{m};\sigma(\bar{m}))}}\mu\big([\pi(\bar{m};\bar{y})]\big)d\nu=1. \qedhere 
\end{align*}
\end{proof}

We will now prove a few additional properties of the convolution product. In the context of NIP theories, this product is left continuous. This follows from the fact that under the NIP hypothesis, the Morley product is left continuous. The left continuity of the Morley product relies on the existence of smooth extension. It is open whether or not the product is left continuous in general. We presume there is probably a counterexample, but we do not know one.

\begin{proposition}\label{prop:conv_left_cont} 
Suppose that $T$ is NIP,
$\mu\in\mathfrak{M}^{\inv}_{\bar{y}}(\FC,M)$, and 
$\nu\in\mathfrak{M}_{\bar{m}}(\FC)$. If $(\mu)_{i \in I}$ is a net of elements from $ \mathfrak{M}^{\inv}_{\bar{y}}(\FC,M)$ converging to $\mu$,
then
$$\lim\limits_{i\in I}\mu_i\ast\nu=\mu\ast\nu.$$
In particular, the convolution product in $(\mathfrak{M}^{\inv}_{\bar{m}}(\FC,M),*)$ is left continuous, i.e., for any $\nu \in \mathfrak{M}^{\inv}_{\bar{m}}(\FC,M)$, the map $-*\nu:\mathfrak{M}^{\inv}_{\bar{m}}(\FC,M) \to \mathfrak{M}^{\inv}_{\bar{m}}(\FC,M)$ is continuous. 
\end{proposition}

\begin{proof}
For every $\CL$-formula $\varphi(\bar{x};\bar{y})$ and $\bar{b} \in \FC^{\bar{y}}$, we have 
    \begin{align*}
    (\mu\ast\nu)\big(\varphi(\bar{b};\bar{y})\big) &= \big(\mu\otimes (h_{\bar{b}})_{\ast}(\nu)\big)\big(\varphi(\bar{x};\bar{y})\big) \\
    &= \big((\lim\limits_{i\in I}\mu_i)\otimes (h_{\bar{b}})_{\ast}(\nu)\big)\big(\varphi(\bar{x};\bar{y})\big) \\
    &\overset{(*)}{=} \lim\limits_{i\in I}\big(\mu_i\otimes (h_{\bar{b}})_{\ast}(\nu)\big)\big(\varphi(\bar{x};\bar{y})\big) \\
    &= \lim\limits_{i\in I}(\mu_i\ast\nu)\big(\varphi(\bar{b};\bar{y})\big).
    \end{align*}
    Equation $(*)$ follows from left continuity of the Morley product (see Fact \ref{fact:MP}). 
    Thus, $\mu\ast\nu=\lim\limits_{i\in I}\mu_i\ast\nu$, and the proposition follows.
\end{proof}

\noindent The structure $(\mathfrak{M}^{\inv}_{\bar{m}}(\FC,M),*)$ also admits an identity element, namely $\delta_{\tp(\bar{m}/\FC)}$. 

\begin{proposition}\label{prop:neutral.element} Suppose that $\mu \in \mathfrak{M}_{\bar{m}}^{\inv}(\FC,M)$ and $\mu$ is Borel-definable over $M$. Then,
    $$\mu\ast\delta_{\tp(\bar{m}/\FC)}=\mu\quad\text{ and }\quad \delta_{\tp(\bar{m}/\FC)}\ast\mu=\mu. $$
    As convention, we write $\delta_{\tp(\bar{m}/\FC)}$ simply as $\delta_{\bar{m}}$. 
\end{proposition} 

\begin{proof}
Fix an $\CL$-formula $\varphi(\bar{x};\bar{y})$ and a tuple $\bar{b}\in \FC^{\bar{x}}$. Note, 
\begin{align*}
    (\mu * \delta_{\bar{m}}) (\varphi(\bar{b};\bar{y})) &= 
    ( \mu \otimes (h_{\bar{b}})_{\ast}(\delta_{\bar{m}}))(\varphi(\bar{x};\bar{y})) \\
    &= F_{\mu}^{\varphi^{\textrm{opp}}(\bar{y};\bar{x})}(h_{\bar{b}}(\tp(\bar{m}/N))) \\ 
    &= F_{\mu}^{\varphi^{\textrm{opp}}(\bar{y};\bar{x})}(\tp(\bar{b}/M))=\mu\big(\varphi(\bar{b};\bar{y})\big).
\end{align*}
Likewise, by Lemma \ref{prop:compute}, 
\begin{align*}
    (\delta_{\bar{m}} \ast \mu)(\varphi(\bar{b};\bar{y})) &= (\delta_{\bar{m}} \otimes (h_{\bar{b}})_{\ast}(\mu))(\varphi(\bar{x};\bar{y}))\\
    &= (h_{\bar{b}})_{*}(\mu)(\varphi(\bar{x};\bar{m})) \\
    &= \mu(\varphi(\bar{b};\bar{y})). \qedhere 
\end{align*}
\end{proof}

We now show that the convolution product is bi-affine. This essentially follows directly from bi-linearity of integration.

\begin{proposition}\label{prop:left.affine}
If $T$ is NIP, then
for any $\nu \in \mathfrak{M}^{\inv}_{\bar{m}}(\FC,M)$, 
the map $-\ast\nu: \mathfrak{M}^{\inv}_{\bar{m}}(\FC,M) \to \mathfrak{M}^{\inv}_{\bar{m}}(\FC,M)$ is affine. 
More generally, without NIP,
if $\mu,\lambda\in \mathfrak{M}_{\bar{y}}^{\inv}(\FC,M)$ are Borel-definable over $M$, $\nu\in\mathfrak{M}_{\bar{m}}(\FC)$, and $r\in[0,1]$,
then
$$\big(r\mu+(1-r)\lambda\big)\ast\nu=r(\mu\ast\nu)+(1-r)(\lambda\ast\nu).$$
\end{proposition}

\begin{proof} Fix an $\CL$-formula $\varphi(\bar{x};\bar{y})$ with $\bar{b}\in \FC^{\bar{x}}$ and set $s:=1-r$. We compute:
\begin{align*}
    ((r\mu +s\lambda)*\nu)(\varphi(\bar{b};\bar{y}))
    &= \int_{S_{\bar{x}}(M)} F_{r\mu + s\lambda}^{\varphi^{\textrm{opp}}(\bar{y};\bar{x})} d(h_{\bar{b}})_{\ast}(\nu) \\ 
    &= \int_{S_{\bar{x}}(M)} \Big(rF_{\mu}^{\varphi^{\textrm{opp}}(\bar{y};\bar{x})} +sF_{\lambda}^{\varphi^{\textrm{opp}}(\bar{y};\bar{x})}\Big) d(h_{\bar{b}})_{\ast}(\nu) \\ 
    &= r \int_{S_{\bar{x}}(M)} F_{\mu}^{\varphi^{\textrm{opp}}(\bar{y};\bar{x})}d(h_{\bar{b}})_{\ast}(\nu) +
    s\int_{S_{\bar{x}}(M)}F_{\lambda}^{\varphi^{\textrm{opp}}(\bar{y};\bar{x})} d(h_{\bar{b}})_{\ast}(\nu) \\ 
    &=(r(\mu * \nu) + s(\lambda * \nu))(\varphi(\bar{b};\bar{y})). \qedhere
\end{align*}
\end{proof}

\begin{proposition}\label{prop:right.affine}
Suppose that $\mu \in \mathfrak{M}^{\inv}_{\bar{y}}(\FC,M)$ and $\mu$ is Borel-definable over $M$. Then the map $\mu\ast -: \mathfrak{M}_{\bar{m}}(\FC) \to \mathfrak{M}_{\bar{y}}(\FC)$ is affine. 
\end{proposition}

\begin{proof}
    Fix $\lambda, \nu \in \mathfrak{M}_{\bar{m}}(\FC)$, $r \in [0,1]$, an $\CL$-formula $\varphi(\bar{x};\bar{y})$, and $\bar{b}\in \FC^{\bar{x}}$.
 We compute:
    \begin{align*}
        (\mu * (r\nu + (1-r)\lambda))(\varphi(\bar{b};\bar{y})) 
        &= \int_{S_{\bar{m}}(\FC)} G_{\mu}^{\varphi(\bar{b};\bar{y})} d\big(r\nu + (1-r)\lambda\big) \\ 
        &= r \int_{S_{\bar{m}}(\FC)} G_{\mu}^{\varphi(\bar{b};\bar{y})} d\nu 
        + (1-r)\int_{S_{\bar{m}}(\FC)} G_{\mu}^{\varphi(\bar{b};\bar{y})} d\lambda \\ 
        &= \big(r(\mu *\nu) + (1-r)(\mu*\lambda)\big)\big(\varphi(\bar{b};\bar{y})\big). \qedhere 
    \end{align*}
\end{proof}

\subsubsection{Product for types}\label{subsubsection: * on types}
In this subsection, we consider the convolution product restricted to types. In general, this product can be defined for arbitrary invariant types and not just ones which are Borel definable. Hence, we provide the following definition.

\begin{definition} Let $p(\bar{y}) \in S_{\bar{y}}^{\inv}(\FC,M)$ and $q(\bar{y}) \in S_{\bar{m}}(\FC)$. Then the product $p *q$ in $S_{\bar{y}}(\FC)$ is defined as follows: for any $\mathcal{L}$-formula $\varphi(\bar{x};\bar{y})$ and $\bar{b} \in \FC^{\bar{x}}$, we define
\begin{equation*} 
\varphi(\bar{b};\bar{y}) \in p * q \Longleftrightarrow \varphi(\bar{x};\bar{y}) \in p_{\bar{y}} \otimes (h_{\bar{b}}(q))_{\bar{x}}, 
\end{equation*}
where $\otimes$ is the Morley product for invariant types. As in the case of measures, the object on the right, $(h_{\bar{b}}(q))_{\bar{x}}$, is not a global type, 
but a type in $S_{\bar{x}}(M)$. 
However, since $p_{\bar{y}}$ is $M$-invariant, the product is well-defined, i.e., one can replace $(h_{\bar{b}}(q))_{\bar{x}}$ with any global extension. 
\end{definition} 

It is obvious by definition that if $p$ is Borel-definable over $M$, then $\delta_{p} * \delta_{q} = \delta_{p*q}$. As convention, if $p$ is Borel-definable over $M$ and $\mu \in \mathfrak{M}_{\bar{m}}(\FC)$ we often write $\delta_{p} * \mu$ simply as $p * \mu$. 

We will usually be concerned with the space $S^{\inv}_{\bar{m}}(\FC,M)$. We will see that this space with the operation defined above is 
a left topological semigroup. 
The next proposition yields an explicit formula for $\ast$ on $S^{\inv}_{\bar{m}}(\FC,M)$, which is then used to deduce that $\ast$ restricted to $S^{\inv}_{\bar{m}}(\FC,M)$ is an operation on $S^{\inv}_{\bar{m}}(\FC,M)$.

\begin{proposition}\label{prop:formula.for.star.on.types}
    Let $p(\bar{y}),q(\bar{y})\in S^{\inv}_{\bar{m}}(\FC,M)$.
    Then
    $$p\ast q=\tau(p|_{\FC'})|_\FC,$$
    where $q(\bar{y})=\tp(\tau(\bar{m})/\FC)$ for some $\tau\in\aut(\FC')$ and $p|_{\FC'}$ is the unique $M$-invariant extension of $p$ to $\FC'$.
Yet more explicitly, taking a monster model $\FC'' \succ \FC'$ in which $\FC'$ is small, we can write $p|_{\FC'}=\tp(\sigma(\bar{m})/\FC')$ for some $\sigma\in\aut(\FC'')$, and then
    $$p\ast q= \tp(\tau''\sigma(\bar{m})/\FC),$$
where $\tau'' \in \aut(\FC'')$ is an arbitrary extension of $\tau$.
\end{proposition}

\begin{proof}
    Take any $\CL$-formula $\varphi(\bar{x};\bar{y})$ and $\bar{b}\in \FC^{\bar{x}}$. 
    \begin{align*}
  \varphi(\bar{b};\bar{y}) \in p *q &\Longleftrightarrow \varphi(\bar{x};\bar{y}) \in p_{\bar{y}} \otimes (h_{\bar{b}}(q))_{\bar{x}} \\
    &\Longleftrightarrow \varphi(\tau^{-1}(\bar{b});\bar{y}) \in p|_{\FC'} \Longleftrightarrow \varphi(\bar{b};\bar{y}) \in \tau(p|_{\FC'})|_{\FC}. 
    \end{align*}

Thus,
    \begin{align*}
    \varphi(\bar{b};\bar{y})\in p\ast q & \iff  \varphi(\bar{b};\bar{y})\in \tau(p|_{\FC'}) \\
    &\iff \varphi(\tau^{-1}(\bar{b});\bar{y})\in p|_{\FC'}=\tp(\sigma(\bar{m})/\FC') \\
    &\iff \varphi(\bar{b};\bar{y})\in \tp(\tau''\sigma(\bar{m})/\FC). \qedhere
    \end{align*} 
\end{proof}

\begin{cor}\label{corollary: closedness under *}
 $S^{\inv}_{\bar{m}}(\FC,M)$ is closed under $\ast$.
\end{cor}

\begin{proof}
Consider $p(\bar{y}),q(\bar{y})\in S^{\inv}_{\bar{m}}(\FC,M)$. The fact that $p* q \in S_{\bar m}(\FC)$ follows immediately from the explicit formulas from Proposition \ref{prop:formula.for.star.on.types}. It remains to show that $p*q$ is $M$-invariant.

Suppose for a contradiction that there is an $\mathcal{L}$-formula $\varphi(\bar x;\bar y)$ and $\bar a \equiv_M \bar b$ in $\FC^{\bar x}$ such that $\varphi(\bar a;\bar y) \wedge \neg \varphi(\bar b;\bar y) \in p*q$. Choose $\sigma, \tau,\tau''$ as in Proposition \ref{prop:formula.for.star.on.types}. 

By Proposition \ref{prop:formula.for.star.on.types}, we get 
$$\models \varphi(\bar a; \tau''\sigma(\bar m)) \wedge \neg \varphi(\bar b;\tau''\sigma(\bar m)),$$ and so 
 $$\models \varphi(\tau^{-1}(\bar a);\sigma(\bar m)) \wedge \neg \varphi(\tau^{-1}(\bar b);\sigma(\bar m)).$$
Thus, since $\tp(\sigma(\bar m)/\FC')=p|_{\FC'}$ is $M$-invariant, we conclude that $\tau^{-1}(\bar a) \not\equiv_M \tau^{-1}(\bar b)$. Hence, $\bar a \not\equiv_{\tau(\bar m)} \bar b$, which implies that $q=\tp(\tau(\bar m)/\FC)$ is not $M$-invariant, a contradiction.
\end{proof}

\begin{proposition}\label{prop:Borel.preimage}
Suppose that $p(\bar{y})\in S^{\inv}_{\bar{y}}(\FC,M)$ and $p$ is Borel-definable over $M$. Let $\mu\in\mathfrak{M}_{\bar{m}}(\FC)$, $\varphi(\bar{x};\bar{y})$ be an $\CL$-formula and $\bar{b}\in \FC^{\bar{x}}$.
    Then
    $$(p\ast\mu)\big(\varphi(\bar{b};\bar{y})\big)=\mu\big(h_{\bar{b}}^{-1}[D_{p,M}^{\varphi}]\big),$$
where $D_{p,M}^{\varphi}$ is introduced in Definition \ref{definition: definition of p}.
\end{proposition}

\begin{proof}
Notice that 
\begin{align*}
(p * \mu)(\varphi(\bar{b};\bar{y})) &= (p \otimes (h_{\bar{b}})_{*}(\mu))(\varphi(\bar{x};\bar{y})) \\
&= (h_{\bar{b}})_{*}(\mu)(\{ q \in S_{\bar{x}}(M) : F_{p}^{\varphi^{\textrm{opp}}(\bar{y};\bar{x})}(q) = 1 \}) \\ 
&= \mu(h_{\bar{b}}^{-1}[D_{p,M}^{\varphi}]). \qedhere
\end{align*} 
\end{proof}

We remark that the following statement is true without the NIP assumption.

\begin{proposition}\label{prop:cont_types}
For any $q \in S_{\bar{m}}^{\inv}(\FC,M)$, the map $-\ast q: S_{\bar{m}}^{\inv}(\FC,M) \to S_{\bar{m}}^{\inv}(\FC,M)$ is continuous.
\end{proposition}

\begin{proof}
Follows by the left continuity of the Morley product for types. Similar to the proof of Proposition \ref{prop:conv_left_cont}.
\end{proof}

We now prove that the product is associative on triples of types from $S_{\bar{m}}^{\inv}(\FC,M)$. Again, no NIP assumption is necessary. 

\begin{proposition}\label{prop: associative for types}
    Let $p(\bar{y}),q(\bar{y}),r(\bar{y})\in S^{\inv}_{\bar{m}}(\FC,M)$. Then
    $$(p\ast q)\ast r= p\ast(q\ast r).$$
\end{proposition}

\begin{proof}
We can write $r=\tp(\eta(\bar{m})/\FC)$ for some $\eta\in\aut(\FC')$, and $q|_{\FC'}=\tp(\tau(\bar{m})/\FC')$ for some $\tau \in \aut(\FC'')$ (where $\FC'' \succ \FC$ is a monster model in which $\FC'$ is small). 

Take an $\CL$-formula $\varphi(\bar{x};\bar{y})$ and a tuple $\bar{b}\in \FC^{\bar{x}}$. 
Since $\bar m$ and $\bar b$ are short in $\FC$, we can choose the above $\eta$ more carefully so that $\eta^{-1}(\bar b) \in \FC^{\bar{x}}$.
    \ 
    \\ \textbf{Claim:} $h_{\eta^{-1}(\bar{b})}(q)=h_{\bar{b}}(q\ast r)$.
\begin{clmproof}
Note that 
    $h_{\bar{b}}(r)=\tp(\eta^{-1}(\bar{b})/M)$ and
    $h_{\eta^{-1}(\bar{b})}(q)=\tp(\tau^{-1}\eta^{-1}(\bar{b})/M)$,
and so for an $\mathcal{L}$-formula $\theta(\bar{x};\bar{y})$ we can write:
    \begin{align*}
    \theta(\bar{x};\bar{m})\in h_{\eta^{-1}(\bar{b})}(q) &\iff \models\theta(\tau^{-1}\eta^{-1}(\bar{b});\bar{m}) \\
    &\iff \models\theta(\eta^{-1}(\bar{b});\tau(\bar{m})) \\
    &\iff \theta(\eta^{-1}(\bar{b});\bar{y})\in q \\
&\iff \theta(\bar{x};\bar{y})\in  q\otimes h_{\bar{b}}(r) \\
    &\iff \theta(\bar{b};\bar{y})\in q\ast r.
    \end{align*} 
    We end using Lemma \ref{prop:compute}, to have that the last line is equivalent to $\theta(\bar{x};\bar{m})\in (h_{\bar{b}})_{\ast}\delta_{q\ast r}=h_{\bar{b}}(q\ast r)$, 
so the claim is proved.
\end{clmproof}

Now, using the claim for the fourth equivalence below, we have
    \begin{align*}
    \varphi(\bar{b};\bar{y})\in (p\ast q)\ast r &\iff \varphi(\bar{x};\bar{y})\in (p\ast q)\otimes h_{\bar{b}}(r) \\
    &\iff \varphi(\eta^{-1}(\bar{b});\bar{y})\in p\ast q \\
    &\iff \varphi(\bar{x};\bar{y})\in p\otimes h_{\eta^{-1}(\bar{b})}(q) \\
    &\iff \varphi(\bar{x};\bar{y})\in p\otimes h_{\bar{b}}(q\ast r) \\
    &\iff \varphi(\bar{b};\bar{y})\in p\ast(q\ast r). \qedhere
    \end{align*}
\end{proof}

\begin{remark}\label{rem: ast.preserves.pi}
    Let $\pi(\bar{x};\bar{y})$ be a partial type over $\emptyset$ containing ``$\bar{x}\equiv_{\emptyset}\bar{y}$'' and such that
    $G_{\pi,\FC}$ is a subgroup of $\aut(\FC)$.
    If $p,q\in S^{\inv}_{\pi(\bar{m};\bar{y})}(\FC,M)$,
    then $p\ast q\in S^{\inv}_{\pi(\bar{m};\bar{y})}(\FC,M)$. 
\end{remark}

\begin{proof}
We want to use the last formula from Proposition \ref{prop:formula.for.star.on.types}, so write
$q(\bar{y})=\tp(\tau(\bar{m})/\FC)$ for some $\tau\in\aut(\FC')$,
and $p|_{\FC'}=\tp(\sigma(\bar{m})/\FC')$ for some $\sigma\in\aut(\FC'')$. Let $\tau'' \in \aut(\FC'')$ be any extension of $\tau$.
By Proposition \ref{prop:formula.for.star.on.types} and Corollary \ref{corollary: closedness under *}, $p\ast q=\tp(\tau''\sigma(\bar{m})/\FC)\in S^{\inv}_{\bar{m}}(\FC,M)$.

As $p,q$ both contain the type $\pi(\bar m,\bar y)$, we get $\models\pi(\bar{m};\tau''(\bar{m}))$ and $\models\pi(\bar{m};\sigma(\bar{m}))$, hence $\models\pi(\tau''^{-1}(\bar{m});\bar{m})$ and $\models\pi(\sigma^{-1}(\bar{m});\bar{m})$, and so $\tau''^{-1},\sigma^{-1} \in G_{\pi,\FC''}$. Since $G_{\pi,\FC}$ is a subgroup of $\aut(\FC)$, $G_{\pi,\FC''}$ is subgroup of $\aut(\FC'')$ by Remark \ref{remark: group in every model}. Therefore, $\sigma^{-1}\tau''^{-1} \in  G_{\pi,\FC''}$, i.e. $\models\pi(\sigma^{-1}\tau''^{-1}(\bar{m});\bar{m})$, and so $\models\pi(\bar{m};\tau''\sigma(\bar{m}))$. Thus, since $p\ast q=\tp(\tau''\sigma(\bar{m})/\FC)$, we conclude that $p \ast q \in [\pi(\bar{m};\bar{y})]$.

The conclusions of the above two paragraphs imply that $p\ast q \in S^{\inv}_{\pi(\bar{m};\bar{y})}(\FC,M)$. 
\end{proof}

\subsubsection{Associativity of convolution product}
From Definition \ref{def:star.product.definition.0}, we know that $*$ is associative on $\mathfrak{M}_{\bar{m}}^{\sfs}(\FC,M)$ in the NIP context. The general question of whether $*$ is associative on $\mathfrak{M}_{\bar{m}}^{\inv}(\FC,M)$ in the NIP context (i.e. Question \ref{question: associativity of *} in the introduction) remains open. The goal of this section is to prove associativity of $*$ for some large practical families of Keisler measures. We prove that $*$ is associative on triples of finitely satisfiable Keisler measures in the NIP context and triples of invariant measures over countable models of an NIP theory in a countable language. We also prove that $*$ is associative on triples of measures with some definability assumptions without the NIP hypothesis. (Recall also that in Subsection \ref{subsubsection: * on types} we proved that $*$ is associative on invariant types, without any extra assumptions).

Before we start, let us recall that the convolution product for definable groups is associative on invariant measures in the NIP context. However, this proof does not directly generalize to the setting of theories because of a ``piecewise character'' of the definition of the convolution product (i.e., the measure $(h_{\bar{b}})_{*}(\nu)$ in Definition \ref{def:star.definition} changes with $\bar{b}$). More explicitly, all the proofs of associativity for both the Morley product and the convolution product for definable groups for invariant measure (and invariant types) involve replacing a measure with a smooth extension. One should think about this operation as something akin to \emph{realizing a type}. However, in the case of general NIP theories, these are some subtle complications which make this processes much less clear than usual.

\begin{theorem}\label{thm:star.associativity}
Let $\nu,\lambda \in \mathfrak{M}_{\bar{m}}^{\inv}(\FC,M)$, where $\nu$ is Borel-definable over $M$, and let $\mu\in\mathfrak{M}_{\bar{y}}^{\inv}(\FC,M)$ be $M$-definable.
Then 
$$(\mu \ast \nu) \ast \lambda = \mu \ast(\nu \ast \lambda).$$
\end{theorem}

\begin{proof} 
By Lemma \ref{lemma:def.Borel.def}, $\mu\ast\nu$ is Borel-definable over $M$, so we can compute $(\mu\ast\nu)\ast\lambda$.
Fix an $\CL$-formula $\varphi(\bar{x};\bar{y})$, a tuple $\bar{b}\in \FC^{\bar{x}}$, and any $\epsilon > 0$. 
Since $\mu$ is definable over $M$, the map $F^{\varphi^{\textrm{opp}}(\bar{y};\bar{x})}_{\mu}:S_{\bar{x}}(M)\to[0,1]$ is continuous, and so 
\begin{equation*}
    \sup_{q \in S_{\bar{x}}(M)} | F_{\mu}^{\varphi^{\textrm{opp}}(\bar{y};\bar{x})}(q) - \sum_{i=1}^{n} r_i\mathbf{1}_{[\psi_i(\bar{x};\bar{m})]} (q) | < \epsilon,
\end{equation*}
where $\{\psi_{i}(\bar{x};\bar{m})\}_{i=1}^{n}$ are $\CL_{\bar{x}}(M)$-formulas
and $r_1,\dots,r_n\in\Rr$.
As in ($\diamondsuit$) of the proof of Lemma \ref{lemma:def.Borel.def},
we have
$$F_{\mu \ast \nu}^{\varphi^{\textrm{opp}}(\bar{y};\bar{x})}(q) \approx_{\epsilon} \sum_{i=1}^{n} r_i F_{\nu}^{\psi_{i}^{\textrm{opp}}(\bar{y};\bar{x})}(q).$$
for every $q\in S_{\bar{x}}(M)$. Hence,
\begin{align*}
    (\mu \ast (\nu \ast \lambda))(\varphi(\bar{b};\bar{y})) &= \int_{S_{\bar{x}}(M)} F_{\mu}^{\varphi^{\textrm{opp}}(\bar{y};\bar{x})} \,d(h_{\bar{b}})_{\ast}(\nu \ast \lambda)\\
    &\approx_{\epsilon} \int_{S_{\bar{x}}(M)} \sum_{i=1}^{n} r_i \mathbf{1}_{[\psi_{i}(\bar{x};\bar{m})]} \,d(h_{\bar{b}})_{\ast}(\nu \ast\lambda) \\
    &= \sum_{i=1}^{n} r_i (h_{\bar{b}})_{\ast}(\nu \ast \lambda)(\psi_{i}(\bar{x};\bar{m}))= \sum_{i=1}^{n} r_i (\nu \ast \lambda)(\psi_{i}(\bar{b};\bar{y})) \\ 
    &= \sum_{i=1}^{n} r_i \int_{S_{\bar{x}}(M)} F_{\nu}^{\psi_{i}^{\textrm{opp}}(\bar{y};\bar{x})}\, d (h_{\bar{b}})_{\ast}(\lambda) =\int_{S_{\bar{x}}(M)}  \sum_{i=1}^{n} r_i  F_{\nu}^{\psi_{i}^{\textrm{opp}}(\bar{y};\bar{x})} \,d (h_{\bar{b}})_{\ast}(\lambda) \\
    &\approx_{\epsilon} \int_{S_{\bar{x}}(M)}  F_{\mu\ast\nu}^{\varphi^{\textrm{opp}}(\bar{y};\bar{x})} \,d (h_{\bar{b}})_{\ast}(\lambda) = \big((\mu \ast \nu)\ast\lambda\big) (\varphi(\bar{b};\bar{y})). \qedhere
\end{align*}
\end{proof}

\begin{cor}\label{corollary: * associative on fs types}
    Assume that $T$ is NIP, $\mu\in\mathfrak{M}_{\bar{m}}^{\fs}(\FC,M)$ and $\nu,\lambda\in\mathfrak{M}_{\bar{m}}^{\inv}(\FC,M)$. Then
    $$(\mu\ast\nu)\ast\lambda=\mu\ast(\nu\ast\lambda).$$
\end{cor}

\begin{proof}
    As $\mu$ is finitely satisfiable in $M$, there exists a net $\big(\Av(\bar{p}_i) \big)_{i\in I}$ converging to $\mu$ such that $\bar{p}_i=(p^i_1,\ldots,p^i_{n_i})$ and each $p^i_j(\bar{y})$ is equal to $\tp(\bar{a}_{i,j}/\FC)$ for some $\bar{a}_{i,j}\in M^{\bar{y}}$ (see \cite[Proposition 2.11]{Artem_Kyle} and \cite[Fact 2.2]{gannon2022sequential}). Then
    \begin{align*}
    (\mu\ast\nu)\ast\lambda &= \left( \left( \lim\limits_{i\in I}\Av(\bar{p}_i)\right)\ast\nu\right)\ast\lambda \\
    &\overset{2\times\text{Prop. \ref{prop:conv_left_cont}}}{=}  \lim\limits_{i\in I}\left(\big(\Av(\bar{p}_i)\ast\nu\big)\ast\lambda\right) \\
    &\overset{2\times\text{Prop. \ref{prop:left.affine}}}{=}  \lim\limits_{i\in I}\left(\frac{1}{n_i}\sum\limits_{j\leqslant n_i}(p^i_j\ast\nu)\ast\lambda\right) \\
    &\overset{\text{Thm \ref{thm:star.associativity}}}{=}  \lim\limits_{i\in I}\left(\frac{1}{n_i}\sum\limits_{j\leqslant n_i}p^i_j\ast(\nu\ast\lambda)\right) \\
    &\overset{\text{Prop \ref{prop:left.affine}}}{=}  \lim\limits_{i\in I}\left(\Av(\bar{p}_i)\ast(\nu\ast\lambda)\right) \\
    &\overset{\text{Prop. \ref{prop:conv_left_cont}}}{=}  \left(\lim\limits_{i\in I}\Av(\bar{p}_i)\right)\ast(\nu\ast\lambda)= \mu\ast(\nu\ast\lambda). \qedhere
    \end{align*}
\end{proof}

\begin{cor}\label{cor:fs.semigroup}
    Assume that $T$ is NIP.
    Then $(\mathfrak{M}_{\bar{m}}^{\fs}(\FC,M),\ast)$ is a compact (Hausdorff) left topological semigroup with neutral element $\delta_{\bar{m}}$.
\end{cor}

\begin{proof}
    The corollary follows by 
Corollary \ref{corollary: * associative on fs types}, Propositions \ref{prop:conv_left_cont} and \ref{prop:neutral.element}, and the fact that
    if $\mu,\nu\in\mathfrak{M}_{\bar{m}}^{\fs}(\FC,M)$ then $\mu\ast\nu\in\mathfrak{M}_{\bar{m}}^{\fs}(\FC,M)$, i.e., (6) of Lemma \ref{lemma:def.Borel.def}.
\end{proof}

Since $(\mathfrak{M}_{\bar{m}}^{\fs}(\FC,M),*)$ forms a compact left topological  semigroup,  one can study it through the lens of Ellis theory. The semigroup $(\mathfrak{M}_{\bar{m}}^{\fs}(\FC,M),*)$ has minimal left ideals which are disjoint unions of Ellis subgroups 
(e.g. see \cite[Fact A.8]{rzepecki2018}), and so one might be curious about the possible Ellis groups one may encounter. 
The next proposition essentially shows that studying Ellis groups in this context is not interesting -- they are all trivial. 
This phenomenon has been observed before in the definable group setting and is related to the proof that there are no non-trivial compact convex groups.

\begin{proposition}\label{proposition: Ellis groups are trivial}
    Assume that $T$ is NIP. Then the Ellis group is trivial for each of the following semigroups:
    \begin{enumerate}
        \item 
        $(\mathfrak{M}^{\sfs}_{\bar{m}}(\FC,M),\ast)$,
        (and thus the Ellis group of $\E(\mathfrak{M}_{\bar{x}}(M),\conv(\aut(M)))$ is trivial as well by Definition \ref{def:star.product.definition.0}),

        \item 
        $(\mathfrak{M}^{\fs}_{\bar{m}}(\FC,M),\ast)$,

        \item
        $(\mathfrak{M}^{\inv}_{\bar{m}}(\FC,M),\ast)$,
        provided $\ast$ is associative (e.g. in case of a countable $M$ and countable language, cf. Theorem \ref{thm: associativity for NIP and countable}).
    \end{enumerate}
\end{proposition}

\begin{proof}
The proofs of all the three points are very similar and follow a similar argument from Proposition 5.10 in \cite{Artem_Kyle2}).
\end{proof}

The following proposition reduces the problem of associativity of $*$ in the NIP setting to checking the associativity for the case of one type and two measures.

    \begin{proposition}\label{thm: associativity for NIP and countable}
    Assume that $T$ is NIP and $\mu,\nu \in \mathfrak{M}_{\bar{m}}^{\inv}(\FC,M)$. If for every $p \in S_{\bar{m}}^{\inv}(\FC,M)$ 
    we have that 
    $$(p * \mu) * \nu = p * (\mu * \nu),$$
    then the convolution product is associative on $\mathfrak{M}_{\bar{m}}^{\inv}(\FC,M)$. 
\end{proposition}

\begin{proof} 
Fix $\epsilon >0$. Consider an $\CL$-formula $\varphi(\bar{x};\bar{y})$ and a tuple $\bar{b}\in \FC^{\bar{x}}$. 
By NIP (Fact \ref{fact:MP}(2)), there exist $p_1(\bar{y}),\ldots,p_k(\bar{y}) \in \supp(\mu)$, $\bar{p}:=(p_1,\ldots,p_k)$, such that 
\begin{equation*}
    \sup_{q \in S_{\bar{x}}(M)} |F_{\mu}^{\varphi^{\textrm{opp}}(\bar{y};\bar{x})}(q) - F_{\Av(\bar{p})}^{\varphi^{\textrm{opp}}(\bar{y};\bar{x})}(q)| < \epsilon. 
\end{equation*}
We claim that 
\begin{equation*}
    \sup_{q \in S_{\bar{x}}(M)} |F_{\mu \ast \nu}^{\varphi^{\textrm{opp}}(\bar{y};\bar{x})}(q) - F_{\Av(\bar{p}) \ast \nu}^{\varphi^{\textrm{opp}}(\bar{y};\bar{x})}(q)| < \epsilon. 
\end{equation*}
Indeed, notice that for any $q \in S_{\bar{x}}(M)$ and $\bar{c} \models q$ with $\bar{c}\in \FC^{\bar{x}}$, we have that 
\begin{align*}
    F_{\mu * \nu}^{\varphi^{\textrm{opp}}(\bar{y};\bar{x})}(q) &= (\mu * \nu)(\varphi(\bar{c};\bar{y})) 
    = \int_{S_{\bar{x}}(M)} F_{\mu}^{\varphi^{\textrm{opp}}(\bar{y};\bar{x})} \,d(h_{\bar{c}})_*(\nu) \\ 
    &\approx_{\epsilon} \int_{S_{\bar{x}}(M)} F_{\Av(\bar{p})}^{\varphi^{\textrm{opp}}(\bar{y};\bar{x})}\, d(h_{\bar{c}})_{*}(\nu) = (\Av(\bar{p}) * \nu) (\varphi(\bar{c};\bar{y})) 
    = F_{\Av(\bar{p}) * \nu}^{\varphi^{\textrm{opp}}(\bar{y};\bar{x})}(q). 
\end{align*}
Hence, 
\begin{align*}
    (\mu * (\nu * \lambda))(\varphi(\bar{b};\bar{y})) &= \int_{S_{\bar{x}}(M)} F_{\mu}^{\varphi^{\textrm{opp}}(\bar{y};\bar{x})}\, d(h_{\bar{b}})_*(\nu * \lambda)  \approx_{\epsilon} \int_{S_{\bar{x}}(M)} F_{\Av(\bar{p})}^{\varphi^{\textrm{opp}}(\bar{y};\bar{x})}\, d(h_{\bar{b}})_*(\nu * \lambda)\\
    &= \sum_{i=1}^{k} \frac{1}{k} \int_{S_{\bar{x}}(M)} F_{p_i}^{\varphi^{\textrm{opp}}(\bar{y};\bar{x})}\, d(h_{\bar{b}})(\nu * \lambda) =  \sum_{i=1}^{k} \frac{1}{k} (p_i * (\nu * \lambda))(\varphi(\bar{b};\bar{y})) \\
    &\overset{(\ddagger)}{=} \sum_{i=1}^{k} \frac{1}{k} ((p_i * \nu) * \lambda))(\varphi(\bar{b};\bar{y})) = \sum_{i=1}^{k} \frac{1}{k} \int_{S_{\bar{x}}(M)} F_{p_i * \nu}^{\varphi^{\textrm{opp}}(\bar{y};\bar{x})}\, d(h_{\bar{b}})_{*}(\lambda) \\
     &=  \int_{S_{\bar{x}}(M)} \sum_{i=1}^{k} \frac{1}{k}F_{p_i * \nu}^{\varphi^{\textrm{opp}}(\bar{y};\bar{x})}\, d(h_{\bar{b}})_{*}(\lambda)  =  \int_{S_{\bar{x}}(M)} F_{\Av(\bar{p}) * \nu}^{\varphi^{\textrm{opp}}(\bar{y};\bar{x})} \,d(h_{\bar{b}})_{*}(\lambda) \\
     &\approx_{\epsilon} \int_{S_{\bar{x}}(M)} F_{\mu * \nu}^{\varphi^{\textrm{opp}}(\bar{y};\bar{x})}\, d(h_{\bar{b}})_{*}(\lambda) = ((\mu * \nu)*\lambda)(\varphi(\bar{b};\bar{y})), 
    \end{align*}
where equation $(\ddagger)$ follows by our hypothesis. 
As $\epsilon>0$ can be chosen to be arbitrarily small, the statement holds. 
\end{proof}

Thus, the question from the beginning of this subsection reduces to the following:
Let $p(\bar{y})\in S^{\inv}_{\bar{m}}(\FC,M)$ and let $\mu,\nu\in\mathfrak{M}_{\bar{m}}^{\inv}(\FC,M)$, does it follow that
$$(p\ast\mu)\ast\nu=p\ast(\mu\ast\nu)?$$

When $\mathcal{L}$ and $M$ are both countable and $T$ is NIP, then the question above has a positive answer. The proof is similar to the proof that the Morley product is associative over countable models of NIP theories (again, in a countable language). This essentially follows from the fact that Borel functions on Polish spaces are well-behaved.

\begin{theorem} Asumme that $\mathcal{L}$ is countable, $T$ is NIP, and $|M| = \aleph_0$. Suppose that $\mu,\nu \in \mathfrak{M}_{\bar{\bar{m}}}^{\inv}(\FC,M)$ and $p \in S_{\bar{m}}^{\inv}(\FC,M)$. Then $p * (\mu * \nu) = (p * \mu) * \nu$. 
\end{theorem}

\begin{proof} 
Fix an $\mathcal{L}$-formula $\varphi(\bar{x};\bar{y})$. Since $p$ is invariant over $M$, it is Borel-definable over $M$ (i.e., see Fact \ref{fact:MP}). In fact, $p$ is \emph{strongly Borel-definable over $M$}, cf. Proposition 2.6 in \cite{Anand_Udi2011}. 
This means that
there exist $M$-type definable sets $A_1,\dots,A_N$ and $B_1,\dots,B_N$ such that 
\begin{equation*}
    d_{p}^{\varphi} = \bigcup_{i=1}^{n} A_i \cap B_i^{c},
\end{equation*}
(see Definition \ref{definition: definition of p} for $d_{p}^{\varphi}$).
Now, for each $i \leq n$, there exist sequences of $\mathcal{L}_{\bar{x}}(M)$-formulas 
$(\theta_{j_i}(\bar{x};\bar{c}_{j_i}))_{j \in \omega}$ and $(\chi_{j_{i}}(\bar{x};\bar{d}_{j_i}))_{j \in \omega}$ such that:
\begin{enumerate}
    \item $[\theta_{j_i}(\bar x,\bar{c}_{j_i})] \supseteq [\theta_{(j+1)_{i}}(\bar x,\bar{c}_{(j+1)_{i}})]$ and $[\chi_{j_{i}}(\bar x,\bar{d}_{j_i})] \subseteq ([\chi_{(j+1)_{i}}(x,\bar{d}_{j_{i}})])$,
    \item $\bigcap_{j \in \omega} [\theta_{j_i}(\bar x,\bar{c}_{j_i})] = [A_i]$ and $\bigcup_{j \in \omega} [\chi_{j_{i}}(\bar x,\bar{c}_{j_i})] = [B_i]^{c}$.  
\end{enumerate}
For each $(k,l) \in \omega \times \omega$, we consider the $\mathcal{L}$-formula given by $\gamma_{k,l} = \bigvee_{i =1}^{n} \theta_{k_i}(\bar{x};\bar{y}_{\bar{c}_{k_i}}) \wedge \chi_{l_i}(\bar{x};\bar{y}_{\bar{d}_{l_i}})$, 
where $\bar{y}_{\bar{c}_{k_i}}$ and $\bar{y}_{\bar{d}_{l_i}}$ are the variables corresponding to $\bar{c}_{k_i}$ and $\bar{d}_{l_i}$ in the enumeration of $\bar{m}$.  We compute; 

\begin{align*}
&(p * ( \mu * \nu))(\varphi(\bar b;\bar y)) \overset{(a)}{=} ((h_{\bar{b}})_{*}(\mu * \nu))(D_{p,M}^{\varphi}) \\  
&= ((h_{\bar{b}})_{*}(\mu * \nu)) \left( \bigcup_{i=1}^{n} \left(\bigcap_{j < \omega} [\theta_{j_{i}}(\bar{x};\bar{c}_{j_i})] \cap \bigcup_{j < \omega} [\chi_{j_{i}}(\bar{x};\bar{d}_{j_i})]  \right) \right) \\ 
&\overset{(b)}{=} \lim_{l \to \infty} \lim_{k \to \infty}  ((h_{\bar{b}})_{*}(\mu * \nu)) \left( \bigcup_{i=1}^{n} \left(\bigcap_{j < k} [\theta_{j_{i}}(\bar{x};\bar{c}_{j_i})] \cap \bigcup_{j < l} [\chi_{j_{i}}(\bar{x};\bar{d}_{j_i})]  \right) \right) \\ 
&\overset{(c)}{=} \lim_{l \to \infty} \lim_{k \to \infty}   ((h_{\bar{b}})_{*}(\mu * \nu))  \left( \bigcup_{i=1}^{n} \left([\theta_{k_{i}}(\bar{x};\bar{c}_{k_i})] \cap  [\chi_{l_{i}}(\bar{x};\bar{d}_{l_i})]  \right) \right) \\
&\overset{(d)}{=} \lim_{l \to \infty} \lim_{k \to \infty}   (\mu * \nu)  \left( \bigcup_{i=1}^{n} \left([\theta_{k_{i}}(\bar{b};\bar{y}_{\bar{c}_{k_i}})] \cap  [\chi_{l_{i}}(\bar{b};\bar{y}_{\bar{d}_{l_i}})]  \right) \right) \\
&\overset{(e)}{=} \lim_{l \to \infty} \lim_{k \to \infty}   (\mu_{\bar{y}} \otimes ((h_{\bar{b}})_{*}(\nu))_{\bar{x}})  \left( \bigcup_{i=1}^{n} \left([\theta_{k_{i}}(\bar{x};\bar{y}_{k_i})] \cap  [\chi_{l_{i}}(\bar{x};\bar{y}_{l_i})]  \right) \right) \\
&= \lim_{l \to \infty} \lim_{k \to \infty} \int_{S_{\bar{m}}(\FC)} (F_{\mu}^{\gamma_{k,l}^{\textrm{opp}}} \circ h_{\bar{b}} )d\nu \\ 
&\overset{(f)}{=}  \lim_{l \to \infty} \int_{S_{\bar{m}}(\FC)} \lim_{k \to \infty} (F_{\mu}^{\gamma_{k,l}^{\textrm{opp}}} \circ h_{\bar{b}} ) d\nu \\ 
&\overset{(g)}{=} \int_{S_{\bar{m}}(\FC)} \lim_{l \to \infty} \lim_{k \to \infty} (F_{\mu}^{\gamma_{k,l}^{\textrm{opp}}} \circ h_{\bar{b}} ) d\nu \\ 
&\overset{(h)}{=} \int_{S_{\bar{m}}(\FC)} (F_{p* \mu}^{\varphi^{\textrm{opp}}} \circ h_{\bar{b}} ) d\nu  = ((p * \mu) * \nu) (\varphi(\bar b;\bar y)). 
\end{align*}
Now, we justify some of the above equations. 
\begin{enumerate}[($a$)]
\item Straightforward from the definition. 
\item Continuity from above and below. 
\item Choice of $\theta$ and $\chi$ as decreasing and increasing families respectively. 
\item  Lemma \ref{prop:compute}. 
\item Directly from the definition of the convolution product. 
\item This is an application of the dominated convergence theorem. We claim that for every fixed $l$, 
the sequence of functions $\lim_{k \to \infty} (F_{\mu}^{\gamma_{k,l}} \circ h_{\bar{b}})$ converges. Therefore, we can bring the limit inside the integral. It suffices to prove that for any $q \in S_{\bar{m}}(\FC)$, 
the limit $\lim_{k \to \infty} (F_{\mu}^{\gamma_{k,l}^{\textrm{opp}}} \circ h_{\bar{b}})(q)$ exists.
Now, let $\bar e \models h_{\bar b}(q)$ and notice that
\begin{align*}
    \lim_{k \to \infty} (F_{\mu}^{\gamma_{k,l}^{\textrm{opp}}} \circ h_{\bar{b}})(q) &= \lim_{k \to \infty} F_{\mu}^{\gamma_{k,l}^{\textrm{opp}}} (h_{\bar{b}}(q)) \\ &= \lim_{k \to \infty} \mu\left( \bigvee_{i=1}^{n} \theta_{k_i}( \bar{e};\bar{y}_{\bar{c}_{k_i}}) \wedge \chi_{l_i}( \bar{e};\bar{y}_{\bar{d}_{l_i}}) \right) \\ &= \lim_{k \to \infty} (h_{ \bar{e}})_*(\mu) \left( \bigvee_{i=1}^{n} \theta_{k_i}(\bar x ;\bar{c}_{k_i}) \wedge \chi_{l_i}(\bar x;\bar{d}_{l_i}) \right) \\ &=  (h_{ \bar{e}})_*(\mu) \left( \bigcup_{i=1}^{n} [A_{i}] \cap [\chi_{l_i}(\bar{x};\bar{d}_{l_i})] \right). 
\end{align*}
\item Similar to $(f)$, another application of the dominated convergence theorem. 
\item Fix $q \in S_{\bar{m}}(\FC)$ and let $\bar{e} \models h_{\bar{b}}(q)$. Notice that: 

\begin{align*}
    \lim_{l \to \infty} \lim_{k \to \infty} (F_{\mu}^{\gamma_{k,l}^{\textrm{opp}}} \circ h_{\bar{b}})(q) 
    &= \lim_{l \to \infty} \lim_{k \to \infty} F_{\mu}^{\gamma_{k,l}^{\textrm{opp}}}( h_{\bar{b}}(q)) \\ 
    &= \lim_{l \to \infty} \lim_{k \to \infty} \mu \left(\bigvee_{i=1}^{n} \theta_{k_{i}}(\bar{e};\bar{y}_{\bar{c}_{k_i}}) \wedge \chi_{l_{i}}(\bar{e};\bar{y}_{\bar{d}_{l_i}}) \right) \\
    &= \lim_{l \to \infty} \lim_{k \to \infty} (h_{\bar{e}})_{*}(\mu) \left(\bigvee_{i=1}^{n} \theta_{k_{i}}(\bar{x};\bar{c}_{k_i}) \wedge \chi_{l_{i}}(\bar{x};\bar{d}_{l_i}) \right) \\
    &=  (h_{\bar{e}})_{*}(\mu) \left(\bigcup_{i=1}^{n} \left(\bigcap_{j < k} [\theta_{j_{i}}(\bar{x};\bar{c}_{k_i})] \cap \bigcup_{t < l}  [\chi_{t_{i}}(\bar{x};\bar{d}_{l_i})] \right) \right) \\
    &= (h_{\bar{e}})_{*}(\mu)(D_{p,M}^{\varphi}). 
 \end{align*}
 After pausing for a breath of fresh air, we further compute:
 \begin{align*}
 (h_{\bar{e}})_{*}(\mu)(D_{p,M}^{\varphi}) &= (p \otimes (h_{\bar{e}})_{*}(\mu))(\varphi(\bar{x};\bar{y})) = (p * \mu)(\varphi(\bar{e};\bar{y})) \\ &= F_{p *\mu}^{\varphi}(h_{\bar{b}}(q)) = (F_{p*\mu}^{\varphi} \circ h_{\bar{b}})(q). \qedhere
    \end{align*}  
\end{enumerate} 
\end{proof}

\begin{cor} Assume that $\mathcal{L}$ is countable, $T$ is NIP, and $|M| = \aleph_0$. Then $(\mathfrak{M}_{\bar{m}}^{\inv}(\FC,M),*)$ is a compact left topological semigroup with unit $\delta_{\tp(\bar{m}/\emptyset)}$. 
\end{cor}

\section{Adding an affine sort}\label{sec: affine sort}
A well-known construction of adding an affine sort, first studied by Hrushovski, starts from a group $G$ definable in a structure $M$ and expands $M$ by a new sort $S$ and a strictly 1-transitive action of $G$ on $S$ (and no other new relations or functions).

The purpose of this section is to argue, using the aforementioned construction, that convolution for theories \emph{encodes} convolution for definable groups. We show that given a structure $M$ with a definable group $G$, the convolution semigroup for the expansion of $M$ by an affine sort (with the appropriate choice of a partial type) is isomorphic to the definable convolution semigroup over the definable group $G$. Moreover, basic properties of measures (e.g., definable, fim) transfer along this isomorphism.  In particular, this encoding provides a rich source of examples of generically stable and fim subgroups of the group of automorphisms of the monster model. We begin with an auxiliary subsection containing general observations on pushforwards.

As usual, we let $\FC$ be a monster model of $T$ and $M$ be a small elementary submodel.

\subsection{Definable function transfer}

Here we make some remarks regarding pushforwards by definable functions.

\begin{fact}\label{fact: pushforward}
Let $\bar x$ and $\bar y$ be finite tuples of variables. Assume that $f \colon \FC^{\bar x} \to \FC^{\bar y}$ is an $M$-definable function. Let $\mu \in \mathfrak{M}_{\bar x}(\FC)$ and $f_*(\mu) \in \mathfrak{M}_{\bar y}(\FC)$ be the pushforward of $\mu$ via $f$,
where for any formula $\varphi(\bar{y})\in \CL_{\bar{y}}(\FC)$, 
$f_{\ast}(\mu)(\varphi(\bar{y})) := \mu(\varphi(f(\bar{x})))$.
Then, if $\mu$ is invariant over $M$ [definable over $M$, Borel-definable over $M$, finitely satisfiable in $M$, or fim over $M$], then the measure $f_*(\mu)$ has the corresponding property.
\end{fact}

\begin{proof}
The cases ``invariant over $M$'', ``definable over $M$'', ``fim over $M$'' are contained in Proposition 3.26 of \cite{CGK}. The case ``Borel-definable over $M$'' follows be the same very short argument as in the ``definable over $M$'' case (see Proposition 3.26(2) of \cite{CGK}). The ``finitely satisfiable in $M$'' case is easy: Consider any formula $\varphi(\bar y, \bar c)$ of positive $f_*(\mu)$-measure. Then $0<f_*(\mu) (\varphi(\bar y; \bar c))= \mu(\varphi(f(\bar x);\bar c))$, so, as $\mu$ is finitely satisfiable in $M$, there is some $\bar a \in M^{\bar x}$ such that $\models \varphi(f(\bar a);\bar c))$ and clearly $f(\bar a) \in M^{\bar y}$.
\end{proof}

Take the situation from Fact \ref{fact: pushforward}. Then $f_* |_{S_{\bar x}(\FC)} \colon S_{\bar x}(\FC) \to S_{\bar y}(\FC)$ is a continuous map. As such it induces the pushforward map $(f_* |_{S_{\bar x}(\FC)})_* \colon \mathfrak{M}_{\bar x}(\FC) \to \mathfrak{M}_{\bar y}(\FC)$ (treating Keisler measures as Borel measures on type spaces).

\begin{remark}\label{remark **=*}
$(f_* |_{S_{\bar x}(\FC)})_* =f_*$.
\end{remark}

\begin{proof}
It follows easily by the definition of pushforwards.
\end{proof}

The next remark is Proposition 3.26(4) of \cite{CGK}.
\begin{remark}\label{remark: improved 3.26(4)}
$f_*[\supp(\mu)] = \supp(f_*(\mu))$. 
\end{remark}

\subsection{Affine sort construction}
We now begin the affine sort construction. We fix an enumeration $\bar{m}$ of a small model $M\models T$, say $\bar{m}\in M^{\bar{x}}$. Assume that $G(x)\in\CL$ is a $\emptyset$-definable group in $T$.
Then, let $\bar{g}=(g_\alpha)_{\alpha}$ enumerate the elements of $G(M)$ starting from $g_0=1$ (i.e., the neutral element of $G(M)$). 
We expand $M$ by a new sort $S$ together with a regular (or strictly 1-transitive) action  $\cdot$ of $G(M)$ on $S$ (and no other new structure), and denote 
$$\bar{M}:=(M,S,\cdot).$$
Then $S$ is called the {\em affine sort}. Let us fix some $s_0 \in S$. Then $S =G(M)\cdot s_0$, and
$\bar{s}:=\bar{g}\cdot s_0$ lists all the element of the sort $S$. 
Let $\bar{y}$ be a tuple of variable corresponding to $\bar{s}$ and set $\bar{n}:=(\bar{m};\bar{s})$. 
We use ``$\mathcal{L}^{\textrm{aff}}$'' to denote the expansion of the language $\mathcal{L}$, which corresponds to the structure $\bar{M}$.
Take a monster model $\bar{\FC}\succeq\bar{M}$; then $\bar{\FC}=(\FC,G(\FC)\cdot s_0, \cdot)$ for some monster model $\FC$ of $T$.
We consider the automorphism group of $\bar{\FC}$ and the map:
$$\bar{F}:\aut(\bar{\FC})\to G(\FC)\rtimes\aut(\FC)$$
$$\bar{\sigma}\mapsto (g,\bar{\sigma}|_\FC),$$
where $g\in G(\FC)$ is the unique element such that $g\cdot\bar{\sigma}(s_0)=s_0$. Recall that the group structure on $G(\FC)\rtimes\aut(\FC)$ is given by:
$$(g_1,\sigma_1)\cdot(g_2,\sigma_2)=(g_1\cdot\sigma_1(g_2),\sigma_1\sigma_2).$$

For the next remark, see \cite[Proposition 3.3]{GisNew08}.
\begin{remark}
\begin{enumerate}
    \item The map $\bar{F}$ is a group isomorphism.

    \item The group $G(\FC)\rtimes\aut(\FC)$ acts on $(\FC,\, G(\FC)\cdot s_0)$ via
    $$(g,\sigma)\cdot (c,\,g'\cdot s_0)=\big( \sigma(c),\,\sigma(g')g^{-1}\cdot s_0\big).$$
\end{enumerate}
\end{remark}

The next fact is well-known and follows from \cite[Section 3]{GisNew08}.
\begin{fact}\label{rem: affine Lascar and KP}
    After identifying $\aut(\bar{\FC})$ with $G(\FC)\rtimes\aut(\FC)$ thorough $\bar{F}$, we have:
    \begin{enumerate}
        \item $\autf_{\KP}(\bar{\FC})=G(\FC)^{00}\rtimes\autf_{\KP}(\FC) $, and $\aut(\bar{\FC})/\autf_{\KP}(\bar{\FC}) \cong G(\FC)/G(\FC)^{00} \rtimes \gal_{\KP}(T)$ as topological groups;
        \item $\autf_{\mathrm{L}}(\bar{\FC})=G(\FC)^{000}\rtimes\autf_L(\FC)$, and $\aut(\bar{\FC})/\autf_{\mathrm{L}}(\bar{\FC}) \cong G(\FC)/G(\FC)^{000} \rtimes \gal_{\mathrm{L}}(T)$ as topological groups.
    \end{enumerate}
\end{fact}

We now define a partial type $\pi$ and the associated relatively type-definable subgroup $H$ of the automorphism group of $\bar{\FC}$ (intimately connected to our fixed $\emptyset$-definable group $G$) which are {\bf central to the rest of this section}. Namely,
$$H:=  G_{\pi,\bar{\FC}} =\{\bar{\sigma}\in\aut(\bar{\FC})\;\colon\;\models\pi(\sigma(\bar{n});\bar{n})\},$$
where $\pi(\bar{x},\bar{y};\bar{x}',\bar{y}')$ is a partial type (over $\emptyset$) 
expressing that
``$\bar{x}\bar{y}\equiv_{\emptyset}\bar{x}'\bar{y}'$''
and $x_\alpha=x_\alpha'$ for 
$\alpha$ labeling the enumeration $\bar{m}\in M^{\bar{x}}$. 
In particular, $\bar x'$ and $\bar y'$ correspond to $\bar m$ and $\bar s$, respectively. Note that after the identification $\aut(\bar{\FC})=G(\FC)\rtimes\aut(\FC)$, we have $H=G(\FC)\rtimes\aut(\FC/M)$.
We have the following canonical embeddings:
$$G(\FC)\ni g\mapsto(g,\id_{\FC})\in\aut(\bar{\FC}),$$
$$\aut(\FC)\ni \sigma\mapsto(1,\sigma)\in\aut(\bar{\FC}),$$
under which we can present $H$ inside the semi-product as follows:
$$H=G(\FC)\cdot\aut(\FC/M).$$
With $H$ we associate the following collection of types in $S_{\bar{n}}(\bar{\FC})$:

$$\widetilde{H}_{\bar{\FC},\bar{n}}:=[\pi(\bar{x},\bar{y};\bar{n})] = \{p(\bar{x},\bar{y})\in S_{\bar{n}}(\bar{\FC})\;\colon\; \pi(\bar{x},\bar{y};\bar{n})\subseteq p(\bar{x},\bar{y}) \}.$$

Let $\bar{\FC}\preceq\bar{\FC}'=(\FC',\,G(\FC')\cdot s_0,\cdot)$ be a bigger monster model, $p(\bar{x},\bar{y})\in \widetilde{H}_{\bar{\FC},\bar{n}}$, 
and $(\bar{d},\bar{h}\cdot s_0)\models p$. 
Then $\bar{d}=\bar{m}$ and there exists $\bar{\tau}=(g,\tau)\in\aut(\bar{\FC}')$ such that $\tau\in\aut(\bar{\FC}'/M)$ and
$$(\bar{d},\bar{h}\cdot s_0)=(g,\tau)(\bar{m},\bar{g}\cdot s_{0})=(\bar{m},(g_{\alpha}g^{-1}\cdot s_0)_{\alpha}).$$
We see that $p(\bar{x},\bar{y})=\tp(\bar{m},(g_{\alpha}g^{-1}\cdot s_0)_{\alpha}\;/\bar{\FC})$.

We now define the map $f$ which will lead to our transfer results:
$$f:\widetilde{H}_{\bar{\FC},\bar{n}}\to S_G(\FC),$$
$$f\Big(\tp(\bar{m},(g_{\alpha}g^{-1}\cdot s_0)_{\alpha}\;/\bar{\FC})\Big):=\tp(g^{-1}/\FC).$$
It is well-defined. To see this, consider any $g,g' \in G(\FC')$ such that $\tp(\bar{m},(g_{\alpha}g^{-1}\cdot s_0)_{\alpha}\;/\bar{\FC})=\tp(\bar{m},(g_{\alpha}g'^{-1}\cdot s_0)_{\alpha}\;/\bar{\FC})$. Then we can find $(h,\sigma) \in \aut(\bar{\FC}'/\bar{\FC})$ mapping $(\bar{m},(g_{\alpha}g^{-1}\cdot s_0)_{\alpha})$ to $(\bar{m},(g_{\alpha}g'^{-1}\cdot s_0)_{\alpha})$. Thus, $\sigma \in \aut(\FC'/\FC)$ and $h=1$ (because $h^{-1}\cdot s_0=s_0$, as $s_0 \in \bar{\FC}$). Hence, $g'^{-1}=\sigma(g^{-1})$, and so $\tp(g'^{-1}/\FC)= \tp(g^{-1}/\FC)$.

\begin{remark}\label{rem:aff.f.homeo}
    The map $f$ is a homeomorphism.
\end{remark}

\begin{proof}
First, let us notice that $f$ is a composition of two maps
$r:\widetilde{H}_{\bar{\FC},\bar{n}} \to S_S(\bar{\FC})$ and $s:S_S(\bar{\FC}) \to S_G(\FC)$, 
where $r$ is just the restriction to $y_0$, i.e. $p(\bar{x},\bar{y})\mapsto p|_{y_0}$,
and $s$ is given by $\tp(g\cdot s_0/\bar{\FC})\mapsto\tp(g/\FC)$.
It is explained in Proposition 2.22 from \cite{KruRze16} that $s$ is a homeomorphism. Hence, it suffices to show that $r$ is a homeomorphism as well.

It is clear that $r$ is a continuous surjection, so it remains to show that it is 
injective (then compactness will imply that $r$ is a homeomorphism). So consider any $g,h \in \bar G(\FC')$ such that 

$$f\Big( \tp\big(\bar{m},(g_\alpha g\cdot s_0)_{\alpha}\,/\bar{\FC}\big) \Big)=\tp(g\cdot s_0/\bar{\FC}) = \tp(h\cdot s_0/\bar{\FC})=f\Big(\tp\big(\bar{m},(g_\alpha h\cdot s_0)_\alpha\,/\bar{\FC}\big)\Big).$$
There exists $\bar{\tau}=(1,\tau)\in \aut(\bar{\FC}'/\bar{\FC}) = \{1\}\times\aut(\FC'/\FC)$
such that $\bar{\tau}(g\cdot s_0)=h\cdot s_0$ and so $h=\tau(g)$.
Then
$$\bar{\tau}\big(\bar{m},(g_\alpha g\cdot s_0)_{\alpha}\big)=\big(\bar{m},(g_\alpha h\cdot s_0)_\alpha\big),$$
and we obtain $\tp\big(\bar{m},(g_\alpha g\cdot s_0)_{\alpha}\,/\bar{\FC}\big)= \tp\big(\bar{m},(g_\alpha h\cdot s_0)_\alpha\,/\bar{\FC}\big)$.
\end{proof}

In the next lemma, we change our convention concerning $\bar x$, $\bar y$, allowing them to be finite tuples of variables from the home and from the affine sort, respectively.

\begin{lemma}\label{lemma: canonical form of formulas}
Each $\mathcal{L}^{\textrm{aff}}$-formula $\psi(\bar x;\bar y)$ (where $\bar x$ are from the home sort with $n:=|\bar x|$, and $\bar y$ from the affine sort with $m:=|\bar y|$) is equivalent 
(in $\Th(\bar{M},s_0)$) to the formula 
$$(\,\exists \bar t\,)\left( \bigwedge_{i \leqslant m} G(t_i) \wedge \bigwedge_{i \leqslant m} y_i = t_i \cdot s_0 \wedge \varphi(\bar x;\bar t)\right),$$
for a unique (up to equivalence in $\Th(M)$) $\mathcal{L}$-formula $\varphi(\bar x;\bar t)$.
\end{lemma}

\begin{proof}
Let $\Phi \colon S_{\bar x \bar y}^{\bar M}(s_0) \to S_{\bar x G(\bar t)}^M(\emptyset)$ be given by 
$$\tp((a_i)_{i\leqslant n},(h_i\cdot s_0)_{i \leqslant m}/s_0) \mapsto \tp((a_i)_{i \leqslant n},(h_i)_{i \leqslant m}/\emptyset),$$
where $S_{\bar x \bar y}^{\bar M}(s_0)$ is the space of complete types over $s_0$ in the sense of $\bar M$ and $S_{\bar x G(\bar t)}^M(\emptyset)$ is the space of those complete types over $\emptyset$ in the sense of $M$ which contain the formula $\bigwedge_{i \leqslant m} G(t_i)$.

By an argument similar to the proof of Remark \ref{rem:aff.f.homeo}, it can be shown that 
$\Phi$ is a well-defined homeomorphism. 
Thus, each clopen $[\psi(\bar x; \bar y)]$ in $S_{\bar x \bar y}^{\bar M}(s_0)$ is the preimage under $\Phi$ of a unique clopen $[\varphi(\bar x;\bar t)]$ in $S_{\bar x G(\bar t)}^M(\emptyset)$, and this $\varphi(\bar x;\bar t)$ does the job. Uniqueness (up to equivalence) of $\varphi(\bar x; \bar t)$ is also clear.
\end{proof}

We now focus on measures which concentrate on $\widetilde{H}_{\bar{\FC},\bar{n}}$. We want to connect the following two spaces of regular Borel probability measures: $\mathcal{M}(\widetilde{H}_{\bar{\FC},\bar{n}})$
and $\CM(S_G(\FC))$.
$\CM(S_G(\FC))$ is identified with $\mathfrak{M}_G(\FC)$,
and $\mathcal{M}(\widetilde{H}_{\bar{\FC},\bar{n}})$ with
$$\mathfrak{M}_{\pi,\bar{n}}(\bar{\FC}):= \mathfrak{M}_{\pi(\bar{x},\bar{y};\bar{n})}(\bar{\FC})= \{\mu\in\mathfrak{M}_{\bar{x}\bar{y}}(\bar{\FC})\;\colon\;\mu([\pi(\bar{x},\bar{y};\bar{n})])=1\}.$$
The homeomorphism $f$ naturally induces a pushforward map 
$$f_\ast: \mathfrak{M}_{\pi,\bar{n}}(\bar{\FC}) \to \mathfrak{M}_G(\FC),$$
which is a homeomorphism as well. For example, we have
$$(f_\ast(\mu))(\varphi(t;\bar{b}))=\mu(f^{-1}[\varphi(t;\bar{b})])=
\mu([(\exists t)\big(\varphi(t;\bar{b})\,\wedge\,G(t)\,\wedge\,y_0=t\cdot s_0 \big)]).$$
for $\mu\in \mathfrak{M}_{\pi,\bar{n}}(\bar{\FC})$ and $\varphi(t;\bar{b})\in\CL(\FC)$.

Note that $H$ acts on the space $\widetilde{H}_{\bar{\FC},\bar{n}}$, i.e. $\bar{\sigma}\cdot p\in\widetilde{H}_{\bar{\FC},\bar{n}}$ for every $\bar{\sigma}\in H$ and $p\in\widetilde{H}_{\bar{\FC},\bar{n}}$. On the other hand, if $p(\bar{x},\bar{y})\in S^{\inv}_{\bar{n}}(\bar{\FC},\bar{M})$ and $\bar{\sigma}=(g,\sigma)\in H$, we do not know if $\bar{\sigma}\cdot p\in S^{\inv}_{\bar{n}}(\bar{\FC},\bar{M})$, but we can still compute:
$$\bar{\sigma}\cdot p=(g,\sigma)\cdot p=(g,\id_{\FC})\cdot p.$$
Indeed, $(g,\sigma)=(g,\id_{\FC})\cdot(1,\sigma)$ and $(1,\sigma)\in\aut(\bar{\FC}/\bar{M})$ 
does not move the $\bar{M}$-invariant type $p$. For a similar reason, for $\bar{\sigma}=(g,\sigma)\in H$, if $\mu\in\mathfrak{M}_{\pi,\bar{n}}(\bar{\FC})$ then $\bar{\sigma}\cdot\mu\in \mathfrak{M}_{\pi,\bar{n}}(\bar{\FC})$,
and if $\mu\in\mathfrak{M}^{\inv}_{\bar{n}}(\bar{\FC},\bar{M})$ then
$$\bar{\sigma}\cdot\mu=(g,\sigma)\cdot\mu=(g,\id_\FC)\cdot\mu.$$

The next lemma demonstrates how the action of $H$ on $\mathcal{M}(\widetilde{H}_{\bar{\FC},\bar{n}})$ interacts with  the pushforward of the map $f$.

\begin{lemma}\label{lemma:action.transfer}
    Let $\bar{\sigma}=(\sigma,g)\in H$ and $\mu\in\mathfrak{M}_{\pi,\bar{n}}(\bar{\FC})$. Then
    $$f_{\ast}\big( (g,\sigma)\cdot \mu\big)=\big(\sigma\cdot (f_\ast(\mu))\big)\cdot g^{-1}.$$
\end{lemma}

\begin{proof}
    As $(\sigma,g)=(g,\id_{\FC})\cdot (1,\sigma)$, for every $[ \varphi(t;\bar{b})]\subseteq S_G(\FC)$ we can compute:
    \begin{align*}
    f_\ast\big( (g,\sigma)\cdot\mu\big)\big(  \varphi(t;\bar{b})\big) &= (\big((g,\sigma)\cdot\mu\big)\big(f^{-1}[ \varphi(t;\bar{b})]\big) \\
    &= \big( (g,\id_{\FC})(1,\sigma)\cdot\mu\big)\big(f^{-1}[ \varphi(t;\bar{b})]\big) \\
    &= \big( (g,\id_{\FC})(1,\sigma)\cdot\mu \big)\big([(\exists t)(G(t)\wedge \varphi(t;\bar{b})\wedge y_0=t\cdot s_0)]\big) \\
    &= \mu\big((1,\sigma^{-1})(g^{-1},\id_{\FC})[(\exists t)(G(t)\wedge \varphi(t;\bar{b})\wedge y_0=t\cdot s_0)]\big) \\
    &= \mu\big((1,\sigma^{-1})[(\exists t)(G(t)\wedge \varphi(t;\bar{b})\wedge y_0=tg\cdot s_0)]\big) \\
    &= \mu\big((1,\sigma^{-1})[(\exists t)(G(t)\wedge \varphi(tg^{-1};\bar{b})\wedge y_0=t\cdot s_0)]\big) \\
    &= \mu\big([(\exists t)(G(t)\wedge \varphi(t\sigma^{-1}(g^{-1});\sigma^{-1}(\bar{b}))\wedge y_0=t\cdot s_0)]\big) \\
    &= \mu\big(f^{-1}[ \varphi(t\sigma^{-1}(g^{-1});\sigma^{-1}(\bar{b}))]\big) \\
    &= (f_\ast \mu)\big([ \varphi(t\sigma^{-1}(g^{-1});\sigma^{-1}(\bar{b}))]\big) \\
    &= \big(\sigma\cdot(f_\ast \mu)\big)\big([ \varphi(tg^{-1};\bar{b})]\big) \\
    &= \Big( \big(\sigma\cdot(f_\ast(\mu))\big)\cdot g^{-1}\Big)([ \varphi(t;\bar{b})]\big). \qedhere
    \end{align*}
\end{proof}

\begin{remark}\label{rem:aff01}
      For every $\bar{\sigma}=(g,\sigma)\in H$ and  $\mu\in\mathfrak{M}_{\pi,\bar{n}}(\bar{\FC})\cap\mathfrak{M}_{\bar{n}}^{\inv}(\bar{\FC},\bar{M})$ we have
        $$f_\ast(\bar{\sigma}\cdot\mu)=f_\ast(\mu)\cdot g^{-1}.$$
\end{remark}

\begin{proof}
It follows by  Lemma \ref{lemma:action.transfer}, because $\bar{\sigma}\cdot\mu=(g,\id_\FC)\cdot\mu$.
\end{proof}

\begin{cor}\label{cor:aff01}
     For every $\mu\in\mathfrak{M}_{\pi,\bar{n}}(\bar{\FC})\cap\mathfrak{M}_{\bar{n}}^{\inv}(\bar{\FC},\bar{M})$, we have that
        $\mu$ is left $H$-invariant if and only if $f_\ast(\mu)$ is right $G(\FC)$-invariant.
\end{cor}

\begin{proof}
 It  follows by Remark \ref{rem:aff01} and the fact that $f_*$ is bijective.
\end{proof}

Our goal is to prove that the various properties transfer from $\mathfrak{M}_{\pi,\bar{n}}(\FC)$ to $\mathfrak{M}_G(\FC)$. 
We note that it suffices to check a \emph{finitary reduction}. 
To prove that the properties transfer, we describe two maps $F$ and $\chi$ such that for any finite subtuples of variables, the pushforward of the restriction of $\mu$ to this subtuple along $f$ is precisely $F \circ \chi$. 
We show that both $F$ and $\chi$ separately transfer the properties in question, and thus our map $f$ will transfer them as well.

Starting from this point and ending with Remark \ref{remark: restrictions to finite tuples}, we completely {\bf change the meaning} of $\bar x',\bar y'$. 
Namely, 
let $\bar x' =(x_j)_{j \in J}$ and $\bar y' =(y_i)_{i \in I}$ be finite subtuples of $\bar x$ and $\bar y$, respectively, such that $\bar y'$ contains $y_0$ and for every $i \in I$ the there is $j \in J$ for which $m_j = g_i$ (recall that $G(M)$ is enumerated via $(g_\alpha)_{\alpha}$ with $g_0=1$).

\noindent Let $\rho \colon \bar{\FC}^{\bar x'\bar y'} \to G(\FC)$ be the $s_0$-definable map given by
$$\rho(\bar a', (h_i \cdot s_0)_{i \in I}):= h_0.$$
Let $\delta \colon G(\FC) \to \bar{\FC}^{\bar x'\bar y'}$ be the $\bar M$-definable map given by 
$$\delta(h_0):= (\bar m', (g_i h_0 \cdot s_0)_{i \in I}),$$
where $\bar m'$ is the subtuple of $\bar m$ corresponding to $\bar x'$. Note that $\delta$ is a section of $\rho$, i.e. $\rho \circ \delta = \id_{G(\FC)}$. Thus, $\rho_* \circ \delta_* = \id_{\mathfrak{M}_G(\bar{\FC})}$.

Let $\bar n'$ be the subtuple of $\bar n$ corresponding to $\bar x'\bar y'$, $\pi'(\bar{x}',\bar{y}';\bar{n}'):=\pi(\bar{x},\bar{y};\bar{n})|_{\bar{x}',\bar{y}';\bar{n}'}$, and let
$\widetilde{H}_{\bar{\FC},\bar{n}'}:=[\pi'(\bar{x}',\bar{y}';\bar{n}')]\subseteq S_{\bar{x}'\bar{y}'}(\bar{\FC})$.
Note that $\rho_*[\tilde{H}_{\bar{\FC},\bar n'}] = S_G(\bar{\FC})$ and $\delta_*[S_G(\bar{\FC})]= \tilde{H}_{\bar{\FC},\bar n'}$.
In fact, using the choice of $I,J$,
$$\tilde{H}_{\bar{\FC},\bar n'}= \{\tp(\bar m',(g_i g^{-1} \cdot s_0)_{i \in I}/\bar{\FC}) : g \in G(\FC')\}.$$
Let $F \colon \tilde{H}_{\bar{\FC},\bar n'} \to S_G(\bar{\FC})$ be given by 
$$F(\tp(\bar m',(g_i g^{-1} \cdot s_0)_{i \in I}/\bar{\FC})):=\tp(g^{-1}/\bar{\FC}).$$
This is a homeomorphism (by an argument as in Remark \ref{rem:aff.f.homeo} above). It is also clear that $F=\rho_* |_{\tilde{H}_{\bar{\FC},\bar n'}}$. Since $F$ is a homeomoprhism, we clearly have

\begin{remark}\label{remark: F_* homeo}
$F_* \colon \mathfrak{M}_{\pi',\bar n'}(\bar{\FC}) \to \mathfrak{M}_G(\bar{\FC})$ is a homeomorphism.
\end{remark}

\begin{lemma}\label{lemma: first lemma}
$F_*= \rho_*|_{\mathfrak{M}_{\pi',\bar n'}(\bar{\FC})}$.
\end{lemma}

\begin{proof}
 First, we show that   $F_*= (\rho_*|_{S_{\bar x' \bar y'}(\bar{\FC})})_*|_{\mathfrak{M}_{\pi',\bar n'}(\bar{\FC})}$. Take $\mu \in \mathfrak{M}_{\pi',\bar n'}(\bar{\FC})$ and clopen $U \subseteq S_{G}(\bar{\FC})$. We have $F_*(\mu)(U) = (\rho_* |_{\tilde{H}_{\bar{\FC},\bar n'}})_*(\mu)(U) = \mu((\rho_* |_{\tilde{H}_{\bar{\FC},\bar n'}})^{-1}[U]) = \mu((\rho_* |_{S_{\bar x' \bar y'}(\bar{\FC})})^{-1}[U]) = (\rho_* |_{S_{\bar x' \bar y'}(\bar{\FC})})_*(\mu)(U)$, where the third equality follows from the assumption that $\mu$ is concentrated on $\tilde{H}_{\bar{\FC},\bar n'}$.

On the other hand, by Remark \ref{remark **=*}, $(\rho_*|_{S_{\bar x' \bar y'}(\bar{\FC})})_*|_{\mathfrak{M}_{\pi',\bar n'}(\bar{\FC})}=\rho_*|_{\mathfrak{M}_{\pi',\bar n'}(\bar{\FC})}$, which completes the proof.
\end{proof}

\begin{lemma}\label{lemma: second lemma}
$F_* \circ \delta_*|_{\mathfrak{M}_G(\bar{\FC})}  = \id_{\mathfrak{M}_G(\bar{\FC})}$.
\end{lemma}

\begin{proof}
Since $\delta_*[\mathfrak{M}_G(\bar{\FC})]\subseteq \mathfrak{M}_{\pi',\bar n'}(\bar{\FC})$ (by Remark \ref{remark: improved 3.26(4)}), using Lemma \ref{lemma: first lemma}, $F_* \circ \delta_*|_{\mathfrak{M}_G(\bar{\FC})} = \rho_* \circ \delta_* |_{\mathfrak{M}_G(\bar{\FC})}$. We are done as $\rho_* \circ \delta_* =\id_{\mathfrak{M}_G(\bar{\FC})}$.
\end{proof}

From Remark \ref{remark: F_* homeo} and Lemma \ref{lemma: second lemma}, we get 

\begin{lemma}\label{lemma: the inverse of F*}
$F_*^{-1} = \delta_* |_{\mathfrak{M}_G(\bar{\FC})}$.
\end{lemma}

By Fact \ref{fact: pushforward} and Lemmas \ref{lemma: first lemma} and  \ref{lemma: the inverse of F*}, we conclude:

\begin{cor}\label{corollary: preservation under F_*}
Let $\mu \in \mathfrak{M}_{\pi',\bar n'}(\bar{\FC})$. Then $\mu$ is invariant over $\bar M$ [definable over $\bar M$, Borel-definable over $\bar M$, finitely satisfiable in $\bar M$, or fim over $\bar M$] if and only if the measure $F_*(\mu)$ has the corresponding property.
\end{cor}

Consider the map $\chi \colon S^{\bar{\FC}}_G(\bar{\FC}) \to S^{\FC}_G(\FC)$ given by $\tp^{\bar{\FC}}(g/\bar{\FC}) \mapsto \tp^{\FC}(g/\FC)$. It is induced by passing from $\bar{\FC}$ to the reduct $\FC$. It is clear 
that $\chi$ is a homeomophism 
(where injectivity follows from the fact that $\aut(\FC'/\FC)=\aut(\bar{\FC}'/\bar{\FC})|_{\FC'}$), 
and so is the pushforward $\chi_* \colon \mathfrak{M}_G(\bar{\FC}) \to \mathfrak{M}_G(\FC)$.

\begin{lemma}\label{lemma: chi_*}
Let $\mu \in \mathfrak{M}_{G}(\bar{\FC})$. Then $\mu$ is invariant over $\bar M$ [definable over $\bar M$, Borel-definable over $\bar M$, finitely satisfiable in $\bar M$, or fim over $\bar M$] if and only if the measure $\chi_*(\mu)$ has the corresponding property over $M$.
\end{lemma}

\begin{proof}
(1) Invariance. $(\Rightarrow)$ follows from the fact that $\aut(\FC/M)=\aut(\bar{\FC}/\bar M)|_{\FC}$. $(\Leftarrow)$ Consider any $\mathcal{L}^{\textrm{aff}}$-formula $\psi(x,\bar{w}, \bar{z})$, 
where $x,\bar w$ are from the home sort and $\bar z$ from the affine sort. By Lemma \ref{lemma: canonical form of formulas},  $\psi(x,\bar{w}, \bar{z})$ is equivalent to 
$$(\exists \bar t)\left( \bigwedge_{i < m} G(t_i) \wedge \bigwedge_{i < m} z_i = t_i \cdot s_0 \wedge \varphi(x,\bar w,\bar t)\right),$$
for a unique (up to equivalence) $\mathcal{L}$-formula $\varphi(x,\bar w,\bar t)$. Consider any $\mu \in \mathfrak{M}_G(\bar{\FC})$, an instance $\psi(x,\bar a, (h_i \cdot s_0)_{i<m})$ of $\psi$, and $(\bar a'_i,  (h'_i \cdot s_0)_{i<m}) \equiv_{\bar M} (\bar a, (h_i \cdot s_0)_{i<m})$. Then $(\bar a,\bar h') \equiv_M (\bar a,\bar h)$, and so $\mu(\psi(x, \bar a, (h'_i \cdot s_0)_{i<m})) =\chi_*(\mu)(\varphi(x,\bar a, \bar h')) = \chi_*(\mu)(\varphi(x,\bar a, \bar h)) = \mu(\psi(x, \bar a, (h_i \cdot s_0)_{i<m}))$.

(2) Definability. $(\Rightarrow)$ follows from the fact that each $\CL^{\textrm{aff}}$-formula over $\bar{M}$ in home variables is equivalent to an $\CL$-formula over $M$ (which holds by Lemma  \ref{lemma: canonical form of formulas}). For  $(\Leftarrow)$ consider any $\mathcal{L}^{\textrm{aff}}$-formula $\psi(x,\bar{w}, \bar{z})$ as in (1) and closed $K \subseteq [0,1]$, and choose a formula $\varphi(x,\bar{w},\bar{t})$ as in (1). Since $\chi_*(\mu)$ is $M$-definable, $\{ (\bar a, \bar h) : \chi_*(\mu)(\varphi(x,\bar a,\bar h)) \in K\}$ is $M$-type-definable. As for $\bar h$ contained in $G(\FC)$ we have $\chi_*(\mu)(\varphi(x,\bar a, \bar h)) = \mu(\psi(x, \bar a, (h_i \cdot s_0)_{i<m}))$, we conclude that $\{ (\bar a, \bar b) : \mu(\psi(x,\bar a,\bar b)) \in K\}$ is $\bar M$-type-definable.

(3) Borel-definability. A similar argument. 

(4) Finite satisfiability. Again follows easily using Lemma  \ref{lemma: canonical form of formulas}.

(5) 
Fim. For $(\Rightarrow)$ first note that $M$-invariance of $\chi_*(\mu)$ follows from (1). Now, consider any $\mathcal{L}$-formula $\varphi(x,\bar w)$. Let the formulas $\theta_n(x_0,\dots,x_{n-1})$, $n<\omega$, be $\mathcal{L}^{\textrm{aff}}_{\bar M}$-formulas witnessing fim for $\mu \in \mathfrak{M}_G(\bar{\FC})$ for the formula $\varphi(x,\bar w)$. By Lemma \ref{lemma: canonical form of formulas}, each formula $\theta_n(x_0,\dots,x_{n-1})$ is equivalent to an $\mathcal{L}_M$-formula $\theta_n'(x_0,\dots,x_{n-1})$.

For $(\Leftarrow)$, again $\bar M$-invariance of $\mu$ follows from (1). Consider any $\mathcal{L}^{\textrm{aff}}$-formula $\psi(x,\bar{w}, \bar{z})$ as in (1), and take $\varphi(x,\bar w,\bar t)$ as in (1). Take $\mathcal{L}_M$-formulas $\theta_n(\bar x_0,\dots,x_{n-1})$, $n<\omega$, witnessing that $\chi_*(\mu)$ is fim for the formula $\varphi(x,\bar w,\bar t)$. It is easy to check that the same $\theta_n$'s witness fim for $\mu$ for the formula $\psi(x,\bar{w}, \bar{z})$.
\end{proof}

Let $f' \colon \tilde{H}_{\bar{\FC},\bar n'} \to S_G(\FC)$ be the composition $\chi \circ F$, i.e. 
$$f'(\tp(\bar m',(g_i g^{-1} \cdot s_0)_{i \in I}/\bar{\FC}):=\tp(g^{-1}/\FC).$$
By Corollary \ref{corollary: preservation under F_*}  and Lemma \ref{lemma: chi_*}, we get

\begin{cor}\label{corollary: preservation for finite tuples}
Let $\mu \in \mathfrak{M}_{\pi',\bar n'}(\bar{\FC})$. Then $\mu$ is invariant over $\bar M$ [definable over $\bar M$, Borel-definable over $\bar M$, finitely satisfiable in $\bar M$, or fim over $\bar M$] if and only if the measure $f'_*(\mu)$ has the corresponding property over $M$.
\end{cor}

Finally, note that $f'$ is the restriction of $f$ to the variables $\bar x'\bar y'$ in the sense that $f= f' \circ r$, where $r \colon \tilde{H}_{\bar{\FC},\bar n} \to \tilde{H}_{\bar{\FC},\bar n'}$ is the restriction map to the variables $\bar x' \bar y'$. Therefore, for $\mu \in \mathfrak{M}_{\pi,\bar n}(\bar{\FC})$ we have $f'_*(\mu|_{\bar x'\bar y'}) = f_*(\mu)$. 
Using this together with the next remark and Corollary \ref{corollary: preservation for finite tuples}, we conclude with Corollary \ref{corollary: preservation for infinite tuples}.

\begin{remark}\label{remark: restrictions to finite tuples}
A measure $\mu \in \mathfrak{M}_{\bar x \bar y}(\FC)$ is invariant over $\bar M$ [definable over $\bar M$, Borel definable over $\bar M$, finitely satisfiable in $\bar M$, or fim over $\bar M$] if and only if the restrictions of $\mu|_{\bar x' \bar y'}$ to all finite subtuples $\bar x' \bar y'$ as above have the corresponding property.
\end{remark}

\begin{cor}\label{corollary: preservation for infinite tuples}
Let $\mu \in \mathfrak{M}_{\pi,\bar n}(\bar{\FC})$. Then $\mu$ is invariant over $\bar M$ [definable over $\bar M$, Borel-definable over $\bar M$, finitely satisfiable in $\bar M$, or fim over $\bar M$] if and only if the measure $f_*(\mu)$ has the corresponding property over $M$.
\end{cor}

We conclude with a corollary which can be treated as a source of examples of generically stable and fim relatively type-definable subgroups of the group of automorphisms: starting from a generically stable [resp. fim] definable group $G$ the corollary yields a generically stable [resp. fim] relatively type-definable subgroup of $\aut(\bar \FC)$.
(Recall that the notions of a generically stable and fim relatively type-definable subgroup of the group of automorphisms were introduced in Definition \ref{def: fim subgroups}; the corresponding notions for definable groups can be found e.g. in Definitions 2.4 and 3.29 of \cite{CGK}.) 
Before stating the corollary, let us list a few examples of generically stable or fim definable groups.

\begin{itemize}
\item The definable groups with fsg (finitely satisfiable generics) in NIP theories are fim by Remark 4.4 in \cite{HruPiSi13}. In particular, all definable groups in stable theories as well as all definably compact groups definable in expansions of real closed fields are fim.
\item Any pseudofinite group with NIP is fim (see Example 7.32 in \cite{Guide_NIP} and Section 3 in \cite{Macpherson-Tent:Pseudofinite_gropus_with_NIP} for a nice non-solvable example).
\item Stable connected groups are generically stable: the unique left-invariant global type is precisely the unique global generic type.
\item The group $(\mathbb{R},+, R_n)_{n \in \mathbb{N}}$, where the relations $R_n$ are local orders given by $R_n(x,y) \iff 0 \leq y-x \leq n$, is generically stable. This is witnessed by the unique global $1$-type $p \in S(\FC)$ whose any/some realization is not infinitesimally close to an element of $\FC$ (using quantifier elimination for $\textrm{Th}((\mathbb{R},+, R_n)_{n \in \mathbb{N}})$ with constant $1$ established in Proposition 4.8 of \cite{Krup_Portillo:On_stable_quotients}).
\end{itemize}

\begin{cor}
    $G(\FC)$ is (right) generically stable [resp. fim] over $M$ 
    if and only if $H$ is (left) generically stable [resp. fim] over $\bar{M}$.
\end{cor}

\begin{proof}
    Follows by  Corollary \ref{cor:aff01} and Corollary \ref{corollary: preservation for infinite tuples}.
\end{proof}

We are now ready to prove the isomorphism theorem advertised 
at the beginning of Section \ref{sec: affine sort}.

From now on, we {\bf come back to the meaning} of $\bar x',\bar y'$ fixed in the paragraph following Fact \ref{rem: affine Lascar and KP} (which was changed after Corollary \ref{cor:aff01}), i.e. $\bar x'$ and $\bar y'$ correspond to $\bar m$ and $\bar s$, respectively.

Let us first remark that the role of the type $\pi(\bar x;\bar y)$ from Section \ref{sec:star.product} is now played by the type $\pi^{\textrm{opp}}(\bar x',\bar y';\bar x,\bar y)$ (in particular, the role of the tuple $\bar x$ is now played by the tuple $\bar x' \bar y'$, and the role of $\bar y$ is played by $\bar x \bar y$). Note also that by the definition of $\pi(\bar x,\bar y;\bar x',\bar y')$, the types $\pi(\bar x',\bar y';\bar x,\bar y)$ and $\pi(\bar x,\bar y;\bar x',\bar y')$ are equivalent and $G_{\pi^{\textrm{opp}},\bar \FC}=G_{\pi,\bar \FC}
=H$ is a subgroup of $\aut(\bar{\FC})$.

Define the main two objects of interest:
\begin{align*}
\mathcal{S}_0 &:= \widetilde{H}_{\bar{\FC},\bar{n}}\cap S^{\inv}_{\bar{n}}(\bar{\FC},\bar{M}),     \\
\mathcal{S} &:= \mathfrak{M}_{\pi,\bar{n}}(\bar{\FC})\cap\mathfrak{M}^{\inv}_{\bar{n}}(\bar{\FC},\bar{M}). 
\end{align*}

We will demonstrate that each of the above spaces is isomorphic to its corresponding space over our fixed definable group $G$. We first prove the case for types. 

\begin{lemma}\label{lemma: affine star transfer for types}
    The space $\mathcal{S}_0$ equipped with operation $\ast$ is a compact left topological semigroup. Moreover, we have the following
    isomorphism between topological semigroups 
(the semigroup on the right is equipped with the Newelski product; see Section \ref{subsection: convolution for definable groups}):
    $$(\mathcal{S}_0,\ast)\xrightarrow[f]{\cong}(S_G^{\inv}(\FC,M),\ast), $$
where formally the map $f$ above is the restriction of $f$ to $\mathcal{S}_0$.
\end{lemma}

\begin{proof}
    By Remark \ref{rem: ast.preserves.pi} and Propositions  \ref{prop:cont_types} and \ref{prop: associative for types}, we know that $(\mathcal{S}_0,\ast)$ is a left topological semigroup.
    By Corollary \ref{corollary: preservation for infinite tuples}, we have that $f$ (restricted to $\mathcal{S}_0$) is a homeomorphism between spaces $\mathcal{S}_0$ and
    $S_G^{\inv}(\FC,M)$. So we only need to show that $f(p\ast q)=f(p)\ast f(q)$ for $p,q\in\mathcal{S}_0$.

Let $\bar{\FC}\preceq\bar{\FC}'\preceq\bar{\FC}''$ be a bigger and bigger monster model sequence, and
let $\hat{p} \in S(\bar{\FC}')$ be the unique $\bar{M}$-invariant extension of $p$. 
There is $h\in G(\FC'')$ such that $\hat{p}=\tp(\bar{m},(g_\alpha h^{-1}\cdot s_0)_\alpha\,/\bar{\FC}')$. There is also $\bar{\tau}=(g,\tau)\in\aut(\FC')$ such that $\bar{\tau}(\bar{n})\models q$ and so $q=\tp(\bar{m}, (g_\alpha g^{-1}\cdot s_0)_\alpha\,/\bar{\FC})$.
Since $\aut(\bar{\FC'})\ni (g,\id_{\FC}')\subseteq(g,\id_{\FC''}) \in \aut(\bar{\FC}'')$ and $(1,\tau)\in\aut(\bar{\FC}'/\bar{M})$, we get
$$\bar{\tau}(\hat{p})=(g,\id_{\FC'})(1,\tau)\hat{p}=(g,\id_{\FC'})\hat{p}=\tp(\bar{m},(g_\alpha h^{-1}g^{-1}\cdot s_0)_\alpha\,/\bar{\FC}').$$
Thus, using Proposition \ref{prop:formula.for.star.on.types}, we obtain
$$p\ast q=\bar{\tau}(\hat{p})|_{\bar{\FC}}=\tp(\bar{m},(g_\alpha h^{-1}g^{-1}\cdot s_0)_\alpha\,/\bar{\FC})\mapsto f(p\ast q)=\tp(h^{-1}g^{-1}/\FC).$$
On the other hand, as  $\hat{p}$ is $\bar{M}$-invariant,
by Corollary  \ref{corollary: preservation for infinite tuples},
we have that $f^{\bar{\FC}'}(\hat{p})=\tp(h^{-1}/\FC')$ is the unique $M$-invariant extension of $f(p)$, where $f^{\bar{\FC}'}$ is the counterpart of the map $f$ for the bigger monster model $\bar{\FC}'$.
Therefore, $h^{-1}\models f(p)|_{\FC g^{-1}}$. Since $q=\tp(\bar{m}, (g_\alpha g^{-1}\cdot s_0)_\alpha\,/\bar{\FC})$, we also have that  $g^{-1}\models f(q)$. Using the last two observations together with the above formula $f(p*q)=\tp(h^{-1}g^{-1}/\FC)$, by the definition of the Newelski product, we obtain
\begin{equation*}f(p)\ast f(q)=\tp(h^{-1}g^{-1}/\FC)=f(p\ast q). \qedhere
\end{equation*}
\end{proof}

We now show the isomorphism theorem for measures under the NIP hypothesis. We first need the following lemma. 

\begin{lemma}\label{lemma: affine isomorphism for measures}
Let $\mu,\nu\in\mathcal{S}$, and assume that $\mu$ be Borel-definable over $\bar{M}$ and there is $\mu'\in\mathfrak{M}^{\inv}_{\bar{n}}(\bar{\FC}',\bar{M})$,
    for a bigger monster model $\bar{\FC}'\succeq\bar{\FC}$,
    such that
    $\mu'|_{\bar{\FC}}=\mu$ and $\supp(\mu')\subseteq S_{\bar{n}}^{\inv}(\bar{\FC}',\bar{M})$ (i.e. $\mu'$ is invariantly supported over $\bar{M}$).
    Then
    $$f_\ast(\mu\ast\nu)=f_\ast(\mu)\ast f_\ast(\nu),$$
where the $*$-product on the right is the convolution product recalled in Definition \ref{definition: convolutionfor definable groups}.
\end{lemma}

\begin{proof}
Consider a clopen $[\varphi(t;\bar{b})]\subseteq S_G(\FC)$ and note that
$$f^{-1}[\varphi(t;\bar{b})]=[(\exists t)\big( G(t)\;\wedge\;\varphi(t;\bar{b})\;\wedge\;y_0=t\cdot s_0\big)]=[\psi(\bar{b};y_0)]\subseteq S_{\bar{n}}(\bar{\FC}),$$
for some $\psi(\bar{x}',\bar{y}';\bar{x},\bar{y})\in\CL_{\bar{x}',\bar{y}';\bar{x},\bar{y}}$ - without loss of generality we assume that $\bar{b}\in\FC^{\bar{x}'\bar{y}'}$ 
(recall that $\bar{x}'$ and $\bar{y}'$ are copies of variables $\bar{x}$ and $\bar{y}$, respectively, and in fact $\psi$ uses only $(\bar{x}',\bar{y}'; y_0)$ as variables (even only $(\bar{x}';y_0)$), but we need to track full tuples for the application of the definition of the $\ast$-product). 
Recall that in Section \ref{subsec:star.definitions} we defined the map $h_{\bar{b}}:S_{\bar{n}}(\bar{\FC})\to S_{\bar{x}',\bar{y}'}(\bar{M})$ as follows:
$$h_{\bar{b}}:\tp(\bar{\sigma}(\bar{n})/\bar{\FC})\mapsto\tp(\bar{\sigma}^{-1}(\bar{b})/\bar{M}),$$
where $\bar{\sigma}=(g,\sigma)\in\aut(\bar{\FC}')$.
Note that by Lemma \ref{lemma:def.Borel.def}(7), $\mathcal{S}$ is closed under $\ast$, so it makes sense to compute $f_\ast(\mu\ast\nu)$.
Because $\supp(\nu)\subseteq \widetilde{H}_{\bar{\FC},\bar{n}}$, we can compute
\begin{align*}
f_\ast(\mu\ast\nu)(\varphi(t;\bar{b})) &= (\mu\ast\nu)(\psi(\bar{b};y_0)) \\
&= \big(\mu\otimes (h_{\bar{b}})_{\ast}(\nu)\big)\big( \psi(\bar{x}',\bar{y}';\bar{x},\bar{y})\big) \\
&=  \int_{S_{\bar{n}}(\bar{\FC})}\Big( F^{\psi^{\textrm{opp}}(y_0;\bar{x}',\bar{y}')}_{\mu}\circ h_{\bar{b}} \Big)d\nu \\
&= \int_{\substack{\tp(\bar{\sigma}(\bar{n})/\bar{\FC})\\ \in \widetilde{H}_{\bar{\FC},\bar{n}}}} \mu'\big( \psi(\bar{\sigma}^{-1}(\bar{b});y_0) \big)d\nu =:(\spadesuit),
\end{align*}
where $\mu'\in\mathfrak{M}_{\bar{n}}(\bar{\FC}')$ is the $\bar{M}$-invariant extension of $\mu$ given in the assumptions,
and where (on the last line) $\bar{\sigma}=(g,\sigma)\in\aut(\bar{\FC}')$ is used in listing all the types from the space $\widetilde{H}_{\bar{\FC},\bar{n}}$. 
Note that $\bar{\sigma}(\bar{n})=(\bar{m},(g_\alpha g^{-1}\cdot s_0)_\alpha)$,
so $\sigma\in\aut(\FC'/M)$.
Consider any $p=\tp(\bar{m},(g_\alpha h^{-1}\cdot s_0)_\alpha\,/\bar{\FC}')\in \widetilde{H}_{\bar{\FC}',\bar{n}}\cap S^{\inv}_{\bar{n}}(\bar{\FC}',\bar{M})$, where $h\in G(\FC'')$. 
Then we have the following sequence of equivalences, the fourth of which follows by the regularity of the action on the affine sort:
\begin{align*}
p\in [\psi(\bar{\sigma}^{-1}(\bar{b});y_0)] &\iff \psi(\bar{b};y_0)\in \bar{\sigma}p=(g,\id_{\FC'})p \\
&\iff(\exists t)(G(t)\wedge\varphi(t;\bar{b})\wedge y_0=t\cdot s_0) \in \tp(\bar{m}, (g_\alpha h^{-1}g^{-1}\cdot s_0)_{\alpha}\,/\bar{\FC}) \\
&\iff \models (\exists t)(G(t)\wedge\varphi(t;\bar{b})\wedge h^{-1}g^{-1}\cdot s_0=t\cdot s_0) \\
&\iff \models (\exists t)(G(t)\wedge\varphi(tg^{-1};\bar{b})\wedge h^{-1}\cdot s_0=t\cdot s_0) \\
&\iff p\in\big[(\exists t)(G(t)\wedge\varphi(tg^{-1};\bar{b})\wedge y_0=t\cdot s_0)\big].
\end{align*}
We see that 
\begin{align*}
\mu'\Big( [\psi(\bar{\sigma}^{-1}(\bar{b});y_0)] \cap \widetilde{H}_{\bar{\FC}',\bar{n}}\cap S^{\inv}_{\bar{n}}(\bar{\FC}',\bar{M}) \Big)  &=  \\
 \mu'\Big( \big[(\exists t)(G(t)\wedge\varphi(tg^{-1};\bar{b})\wedge y_0=t\cdot s_0)\big] \cap \widetilde{H}_{\bar{\FC}',\bar{n}}\cap S^{\inv}_{\bar{n}}(\bar{\FC}',\bar{M})  \Big). & &
\end{align*}
Recall that $\mu'$ is invariantly supported over $\bar{M}$, i.e. $\supp(\mu')\subseteq \widetilde{H}_{\bar{\FC}',\bar{n}}\cap S^{\inv}_{\bar{n}}(\bar{\FC}',\bar{M})$.
Thus,
$$(\spadesuit)= \int_{\substack{\tp(\bar{m},(g_\alpha g^{-1}\cdot s_0)_\alpha\,/\bar{\FC}) \\
\in \widetilde{H}_{\bar{\FC};\bar{n}}  }}\; \mu'\big( (\exists t)(G(t)\wedge\varphi(tg^{-1};\bar{b})\wedge y_0=t\cdot s_0) \big)d\nu.$$

Now, we will compute the other side of the desired equality. By Corollary \ref{corollary: preservation for infinite tuples} and the assumption that $\mu$ is Borel-definable over $\bar{M}$, we have that  $f_\ast(\mu)$ is Bore-definable over $M$. Thus, $f_\ast(\mu)\ast f_\ast(\nu)$ is well-defined. Using Corollary \ref{corollary: preservation for infinite tuples} again (but this time for $\FC'$ in place of $\FC$), we get that  $(f^{\bar{\FC}'})_{\ast}(\mu')$ is the unique $M$-invariant extension of $f_\ast(\mu)$,  where $f^{\bar{\FC}'}$ is the counterpart of the map $f$ for the bigger monster model $\bar{\FC}'$.
Using the fact that $\supp(\nu)\subseteq \widetilde{H}_{\bar{\FC},\bar{n}}$, we get:

\begin{align*}
\big((f_\ast(\mu))\ast(f_\ast(\nu))\big)(\varphi(t;\bar{b})) &= \big( (f_\ast(\mu))_{t_2}\otimes (f_\ast(\nu))_{t_1}\big)(\varphi(t_2\cdot t_1;\bar{b})) \\
&= \int_{\tp(d/\FC)\in S_G(\FC)} ((f^{\bar{\FC}'})_{\ast}(\mu'))(\varphi(t\cdot d;\bar{b}))d f_\ast(\nu) \\
&= \int_{\substack{\tp(\bar{m},(g_\alpha g^{-1}\cdot s_0)_\alpha\,/\bar{\FC}) \\
\in \widetilde{H}_{\bar{\FC};\bar{n}}  }} ((f^{\bar{\FC}'})_{\ast}(\mu'))\big(\varphi(t\cdot g^{-1};\bar{b}) \big)d\nu \\
&= \int_{\substack{\tp(\bar{m},(g_\alpha g^{-1}\cdot s_0)_\alpha\,/\bar{\FC}) \\
\in \widetilde{H}_{\bar{\FC};\bar{n}}  }} \;\mu'\big((f^{\FC'})^{-1}[\varphi(t\cdot g^{-1};\bar{b})] \big)d\nu \\
&= (\spadesuit)
\end{align*}

We have proved that $f_\ast(\mu\ast\nu)(\varphi(t;\bar{b}))= \big(f_\ast(\mu)\ast f_\ast(\nu)\big)(\varphi(t;\bar{b}))$.
\end{proof}

\begin{theorem}\label{thm: affine isomorhpism for measures}
(NIP)
    The space $\mathcal{S}$ equipped with operation $\ast$ is a compact left topological semigroup,
    and $\mathcal{S}_0$ is its closed sub-semigroup.
    Moreover, the horizontal arrows in the following diagram are 
isomorphisms between topological semigroups 
(the semigroups on the right are equipped the Newelski product and definable convolution recalled in Defintion \ref{definition: convolutionfor definable groups}):
    $$\xymatrix{(\mathcal{S},\ast) \ar[r]^-{f_\ast} & (\mathfrak{M}_G^{\inv}(\FC,M),\ast) \\
    (\mathcal{S}_0,\ast) \ar[r]_-{f} \ar[u]_{\subseteq}& (S_G^{\inv}(\FC,M),\ast) \ar[u]_{\subseteq}\, , }$$
where formally the above $f$ and $f_*$ are the restrictions $f|_{\mathcal{S}_0}$ and $f_*|_\mathcal{S}$, respectively.
\end{theorem}

\begin{proof}
By NIP and Lemma \ref{lemma:def.Borel.def}(7), we know that 
$\mathcal{S}$ is closed under $\ast$.
By NIP, the unique $\bar{M}$-invariant extension of $\mu$ to each bigger monster model $\bar{\FC}'\succeq\bar{\FC}$ is invariantly supported over $\bar{M}$ (by Proposition \ref{prop:NIP}), so we can apply Lemma \ref{lemma: affine isomorphism for measures} to get that $f_\ast \colon (\mathcal{S},*) \to (\mathfrak{M}_G^{\inv}(\FC,M),*)$
is a homomorphism.
By Remark \ref{rem:aff.f.homeo} and Corollary \ref{corollary: preservation for infinite tuples}, we know that $f_\ast$ is a homeomorphism.
Hence, as $(\mathfrak{M}_G^{\inv}(\FC,M),*)$ is known to be a compact left topological semigroup (see Fact \ref{fact: definable convolution is a semigroup}), so is $(\mathcal{S},*)$.
\end{proof}

\begin{cor} (NIP)
    Assume that $M$ is strongly $\aleph_0$-homogeneous.
    For every $\mu\in\mathfrak{M}_{\pi(\bar{x},\bar{y};\bar{n})}(\bar{\FC})$, $\mu$ is $\pi$-strongly finitely satisfiable in $\bar{M}$ if and only if $f_\ast(\mu)$ is finitely satisfiable in $M$ 
(which in turn is equivalent to  $\mu$ being finitely satisfiable in $\bar{M}$ by Corollary \ref{corollary: preservation for infinite tuples}). 
    Moreover, we have
    $$(\mathfrak{M}^{\fs}_G(\FC,M),\,\ast)\xrightarrow[\mu\mapsto f^{-1}_\ast(\mu)]{\cong}(\mathfrak{M}^{\sfs}_{\pi(\bar{x},\bar{y};\bar{n})}(\bar{\FC},\bar{M}),\,\ast)\leqslant(\mathfrak{M}^{\sfs}_{\bar{n}}(\bar{\FC},\bar{M}),\,\ast),$$
    where on the left we have definable convolution, and \emph{convolution in theories} in the middle and on the right.
\end{cor}

\begin{proof}
   By Theorem \ref{thm: mother star space}, we have that $(\mathfrak{M}^{\sfs}_{\pi(\bar{x},\bar{y};\bar{n})}(\bar{\FC},\bar{M}),\,\ast)\leqslant(\mathfrak{M}^{\sfs}_{\bar{n}}(\bar{\FC},\bar{M}),\,\ast)$.
    By Corollary \ref{corollary: preservation for infinite tuples}
    and Theorem \ref{thm: affine isomorhpism for measures},
    the map ``$ \mu\mapsto f^{-1}_\ast(\mu)$'' is an isomorphism between
    $(\mathfrak{M}^{\fs}_G(\FC,M),\,\ast)$ and 
    $(\mathfrak{M}_{\pi,\bar{n}}(\bar{\FC})\cap\mathfrak{M}^{\fs}(\bar{\FC},\bar{M}),\ast)$.
    We show that
    $$\mathfrak{M}_{\pi,\bar{n}}(\bar{\FC})\cap\mathfrak{M}^{\fs}(\bar{\FC},\bar{M})=\mathfrak{M}^{\sfs}_{\pi,\bar{n}}(\bar{\FC},\bar{M}).$$
    The inclusion $\supseteq$ follows by definition.
    Now, take $\mu\in \mathfrak{M}_{\pi,\bar{n}}(\bar{\FC})\cap\mathfrak{M}^{\fs}(\bar{\FC},\bar{M})$ and
    $[\psi(\bar{x},\bar{y};\bar{b})]\subseteq S_{\pi,\bar{n}}(\bar{\FC})$ such that $0<\mu(\psi(\bar{x},\bar{y};\bar{b}))$.
    There is a clopen $[\varphi(t;\bar{c})]\subseteq S_G(\FC)$ such that $f^{-1}[\varphi(t;\bar{c})]=[\psi(\bar{x},\bar{y};\bar{b})]\cap[\pi(\bar{x},\bar{y};\bar{n})]$. Hence,
    $$0<\mu(\psi(\bar{x},\bar{y};\bar{b}))=f_\ast(\mu)(\varphi(t;\bar{c})),$$
    and because $f_\ast(\mu)$ is finitely satisfiable in $M$, there exists $g\in G(M)$ such that $\models \varphi(g;\bar{c})$.
    We have $\models (\exists t)(G(t)\wedge\varphi(t;\bar{c})\wedge g\cdot s_0=t\cdot s_0)$ and so
    $$\tp(\bar{m},(g_\alpha g\cdot s_0)_\alpha\,/\bar{\FC})\in f^{-1}[\varphi(t;\bar{c})]=[\psi(\bar{x},\bar{y};\bar{b})]\cap[\pi(\bar{x},\bar{y};\bar{n})],$$
    and finally $\mu\in \mathfrak{M}^{\sfs}_{\pi,\bar{n}}(\bar{\FC},\bar{M})$.
\end{proof}

\subsection{Stabilizers}
Finally, we make some connections regarding the stabilizer of a measure and the stabilizer of its pushforward along the map $f$.
In Section \ref{sec:fim idempotent}, Proposition \ref{proposition:unique.invariant.transfer} (proved below) will allow us to deduce Conjecture \ref{conjecture: main conjecture for definable groups} from our main Conjecture \ref{conjecture: main conjecture} (in particular, we have the same deduction in the situations in which we prove Conjecture  \ref{conjecture: main conjecture}).
Throughout this section, injectivity of $f_*$ is used many times without mention.

Take $\mu\in\mathfrak{M}_{\pi,\bar{n}}^{\inv}(\bar{\FC},\bar{M}):=\mathfrak{M}_{\pi,\bar{n}}(\bar{\FC})\cap\mathfrak{M}^{\inv}_{\bar{n}}(\bar{\FC},\bar{M})$ and
consider $K:=\stab(\mu)\leqslant\aut(\bar{\FC})$.
By Lemma \ref{lemma: rel.inv}, $K$ is relatively $\bar n$-invariant over $\bar{M}$,
i.e. $K=\{\bar{\tau} \in \aut(\bar{\FC}): \models \rho(\bar{\tau}(\bar n),\bar n)\}$,
where $\rho(\bar{x}',\bar{y}';\bar{x},\bar{y})$ is a disjunction of (possibly infinitely many)
complete types over $\emptyset$. Put $\Theta(\bar{x}',\bar{y}';\bar{x},\bar{y})=[\rho(\bar{x}',\bar{y}';\bar{x},\bar{y})]\subseteq S(\emptyset)$.
We set
$$\widetilde{K}_{\bar{\FC},\bar n}:=\{q(\bar{x},\bar{y})\in S_{\bar{n}}(\bar{\FC})\;\colon\; q(\bar{x},\bar{y}) \in\Theta(\bar{x},\bar{y};\bar{n}) \}.$$
Note that also $K \leqslant H$ (by Lemma \ref{lemma:stab.mu.subgroup.pi}).

Moreover, $R_0:=\{g \in G(\FC)\;\colon\; f_\ast(\mu)\cdot g=f_\ast(\mu)\}$ is an $M$-invariant subgroup of $G(\FC)$. 
This follows by Remark \ref{rem:aff01}: for every $\sigma\in\aut(\FC/M)$ and $g\in R_0$ we have
$$(f_\ast(\mu))\cdot \sigma(g)^{-1}=f_\ast\big((\sigma(g),\id_{\FC})\mu \big)=f\big((1,\sigma)(g,\id_{\FC})(1,\sigma^{-1}) \mu \big)= f_\ast(\mu),$$
because $\mu$ is $\bar{M}$-invariant, $(1,\sigma),(1,\sigma^{-1})\in\aut(\bar{\FC}/\bar{M})$,
and the equality $(f_\ast(\mu))\cdot g^{-1}=f_\ast(\mu)$  translates into $(g,\id_{\FC})\mu=\mu$.
Therefore, there exists a disjunction of (possibly infinitely many) complete types from $S_G(M)$, denoted $R(t)$, such that $R_0=R(\FC)$.

\begin{proposition}\label{proposition:unique.invariant.transfer}
    Let $\mu\in\mathfrak{M}_{\pi,\bar{n}}(\bar{\FC})\cap \mathfrak{M}^{\inv}_{\bar{n}}(\bar{\FC},\bar{M})$ and let $K$ and $R(t)$ be as above.
    Then
    \begin{enumerate}
        \item $\widetilde{K}_{\bar{\FC},\bar n}= \big\{\tp(\bar{m},(g_\alpha g\cdot s_0)_{\alpha}\,/\bar{\FC})\;\colon\; g \in R(\FC')\big\}.$

        \item $f[\widetilde{K}_{\bar{\FC},\bar n}]= S_R(\FC)$.

        \item 
For every $\nu\in \mathfrak{M}_{\pi,\bar{n}}^{\inv}(\bar{\FC},\bar{M})$, we have that $\nu(\widetilde{K}_{\bar{\FC},\bar n})=1$ if and only if $f_\ast(\nu)(S_R(\FC))=1$.

        \item 
For every $\nu\in \mathfrak{M}_{\pi,\bar{n}}^{\inv}(\bar{\FC},\bar{M})$, we have that $\nu$ is left $K$-invariant if and only if $f_\ast(\nu)$ is right $R(\FC)$-invariant.   
 
        \item 
        The measures $\mu$ is the unique left $K$-invariant measure in
        $\mathfrak{M}^{\inv}_{\bar{n}}(\bar{\FC},\bar{M})$ concentrated on $K$ (i.e. $\mu( \widetilde{K}_{\bar{\FC},\bar n} )=1$) if and only if
        $f_{\ast}(\mu)$ is the unique right $R(\FC)$-invariant measure in $\mathfrak{M}_G^{\inv}(\FC,M)$ concentrated on $R$ (i.e. $f_\ast(\mu)(S_R(\FC))=1$).
    \end{enumerate}
\end{proposition}

\begin{proof}
(1) First, we prove $\subseteq$.    Let $q(\bar{x},\bar{y})\in\widetilde{K}_{\bar{\FC},\bar{n}}$ and let $\bar{\sigma}\in\aut(\bar{\FC}')$ be such that $\bar{\sigma}(\bar{n})\models q$.
    Since $K$ is relatively $\bar{n}$-invariant over $\bar{M}$, there exists $\bar{\tau}=(g,\tau)\in K$ such that $q(\bar{x},\bar{y})\in[\tp(\bar{\tau}(\bar{n})/\bar{n})]$.
    By Remark \ref{rem:aff01}, we have 
$\bar{\tau} \cdot \mu=\mu$ if and only if $(f_\ast(\mu))\cdot g^{-1}=f_\ast(\mu)$. Therefore, $\bar{\tau}\in R(\FC)\rtimes\aut(\FC/M)$.
    We see that
    $$\bar{\sigma}(\bar{n})\equiv_{\bar{n}}\bar{\tau}(\bar{n})=\bar{m}(g_\alpha g^{-1}\cdot s_0)_\alpha.$$
    There exists some $\bar{\zeta}=(h,\zeta)\in\aut(\bar{\FC}'/\bar{n})$ (in particular $\zeta\in\aut(\FC'/M)$ and $h=1$) such that
    $$\bar{\sigma}(\bar{n})=\bar{\zeta}\big( \bar{m}(g_\alpha g^{-1}\cdot s_0)_\alpha \big)=\bar{m}(g_\alpha \zeta(g^{-1})\cdot s_0)_\alpha,$$
    and we obtain that $q(\bar{x},\bar{y})=\tp\big(\bar{m}(g_\alpha \zeta(g^{-1})\cdot s_0)_\alpha/\bar{\FC}\big)$.
    Note that since $g^{-1}\in R(\FC)$, $R(t)$ is $M$-invariant and $\zeta\in\aut(\FC'/M)$, 
we get that $\zeta(g^{-1})\in R(\FC')$, as required.

    To show $\supseteq$ in (1), we start from some $g\in R(\FC')$.
    Our goal is to prove that $\models\rho(\bar{m},(g_\alpha g\cdot s_0)_{\alpha};\bar{n})$. By the definition of $R$, there exists $h\in R(\FC)=R_0$
    and $\zeta\in\aut(\FC'/M)$ such that $g=\zeta(h)$. We have that $\bar{\zeta}:=(1,\zeta)\in\aut(\bar{\FC}'/\bar{M})$ and that $\bar{\tau}:=(h,\id_{\FC})\in K$ 
(by Remark \ref{rem:aff01}).
    Hence $\models\rho(\bar{\tau}(\bar{n});\bar{n})$, in other words
    $\models\rho(\bar{m}, (g_\alpha h\cdot s_0)_\alpha;\bar{n})$.
    After applying $\bar{\zeta}$ to the last term, we obtain our goal.

Point (2) follows immediately from (1).
Point (3) follows from (2):
     $$\nu(\widetilde{K}_{\bar{\FC},\bar n})=\nu(f^{-1}f[\widetilde{K}_{\bar{\FC},\bar n}])=(f_\ast(\nu))(S_R(\FC)).$$
Point (4) follows by $K=R(\FC)\rtimes\aut(\FC/M)$ and Remark \ref{rem:aff01}. 

(5)
    Let $\nu\in\mathfrak{M}_{\pi,\bar{n}}(\bar{\FC})$. Then $\nu$ is 
        a left $K$-invariant measure 
        in $\mathfrak{M}^{\inv}_{\bar{n}}(\bar{\FC},\bar{M})$ concentrated on $K$
        if and only if
        $f_{\ast}(\mu)$ is a right $R(\FC)$-invariant measure 
        in $\mathfrak{M}_G^{\inv}(\FC,M)$ concentrated on $R$.
    Indeed, we have $\bar{M}$-invariant/$M$-invariant transfer by Corollary \ref{corollary: preservation for infinite tuples}.
    Then we can use points (3) and (4) of this proposition.
    Therefore $f_\ast$ is a bijection between the set of left $K$-invariant measures
    in $\mathfrak{M}^{\inv}_{\bar{n}}(\bar{\FC},\bar{M})$ concentrated on $K$
    and right $R(\FC)$-invariant measures in $\mathfrak{M}_G^{\inv}(\FC,M)$ concentrated on $R$, and the conclusion follows.
\end{proof}

\begin{cor}
        Let $p\in\widetilde{H}_{\bar{\FC},\bar{n}}\cap S^{\inv}_{\bar{n}}(\bar{\FC},\bar{M})$.
    Then the type $p$ is the unique left $K$-invariant type in $S^{\inv}_{\bar n}(\bar{\FC},\bar{M})$ concentrated on $K$ 
        if and only if $f(p)$ is the unique right $R(\FC)$-invariant type in $S^{\inv}_G(\FC)$ concentrated on $R$. 
\end{cor}

\begin{proof}
    Follows by the proof Proposition \ref{proposition:unique.invariant.transfer}(5) and the fact that $f$ is a homeomorphism between $\widetilde{H}_{\bar{\FC},\bar{n}}$ and $S_G(\FC)$.
\end{proof}

\section{On classification of idempotent fim measures and generically stable types I: results for types and KP-invariant measures}\label{sec:fim idempotent}

Classical results from harmonic analysis demonstrate deep connections between measure theoretic objects living on a group and algebraic properties of that group. As explained in the introduction, one of these connections is the correspondence between idempotent probability measures and compact subgroups. In the context of definable groups, a series of papers  \cite{Artem_Kyle,Artem_Kyle2,CGK} studied the connections between idempotent Keisler measures and type-definable subgroups and established a family of connections. One of the main open questions from that line of research is
 the following: Are idempotent fim Keisler measures in one-to-one correspondence to fim type-definable subgroups via the map $\mu \to \stab(\mu)$
(see Conjecture \ref{conjecture: main conjecture for definable groups})? Under many different hypotheses, this question has a positive solution. Here, we extend the conjecture to the setting of arbitrary theories and prove the conjecture under several different hypotheses. We recall the conjecture below:

\begin{customConj}{(A)}\label{conjecture: main conjecture}
Let $\mu \in \mathfrak{M}^{\inv}_{\bar m}(\FC,M)$ be fim over $M$. 
We know that $\stab(\mu)=G_{\pi,\FC}$ for some partial type $\pi(\bar x;\bar y)\vdash\bar{x}\equiv_{\emptyset}\bar{y}$.
Then the following are equivalent:
\begin{enumerate}
\item $\mu$ is an idempotent, i.e. $\mu * \mu = \mu$. 

\item $\mu$ is the unique (left) $G_{\pi,\FC}$-invariant measure in $\mathfrak{M}^{\inv}_{\pi(\bar{m};\bar{y})}(\FC,M)$.

\end{enumerate}
In particular, there is a correspondence between idempotent fim measures in $\mathfrak{M}^{\inv}_{\bar m}(\FC,M)$ and relatively $\bar m$-type-definable over $M$ fim subgroups of $\aut(\FC)$.
\end{customConj}

We remark that the conjecture above implies the conjecture in the definable group setting. In other words, 
    Conjecture \ref{conjecture: main conjecture} implies Conjecture \ref{conjecture: main conjecture for definable groups}. Indeed,
    let $G$ be a definable group.
Notice that if $\mu\in\mathfrak{M}^{\inv}_G(\FC,M)$ is fim over $M$, then
    Corollary \ref{corollary: preservation for infinite tuples} implies that
    $f_\ast^{-1} (\mu)\in\mathfrak{M}^{\inv}_{\pi(\bar{x},\bar{y};\bar{n})}(\bar{\FC},\bar{M})$ is fim over $\bar{M}$ (in the notation from Section \ref{sec: affine sort}). By Lemma \ref{lemma: affine isomorphism for measures}, we have that $f_\ast^{-1} (\mu)$ is idempotent if and only if $\mu$ is idempotent. Statement (5) of Proposition \ref{proposition:unique.invariant.transfer} finishes the claim 
working with the ``right'' version in item (2) of Conjecture \ref{conjecture: main conjecture for definable groups}. (The equivalence between ``right'' and ``left'' version in item (2) of Conjecture \ref{conjecture: main conjecture for definable groups} is a separate result which follows from Proposition 3.33 and Corollary 3.34 in \cite{CGK}.)

\subsection{Examples of idempotents} 

The affine sort construction gives us many examples of idempotent measures in the context of first-order theories. 
Another source for idempotent measures is Ellis theory combined with our results from Section \ref{sec:star.product}. Namely, a classical fact (e.g. see \cite[Fact A.8]{rzepecki2018}) tells us that every compact left topological  semigroup  has a minimal left ideal which in turn is a union of groups whose neutral elements are idempotents. On the other hand, in Section \ref{sec:star.product}, we proved that under NIP, $\mathfrak{M}_{\bar{m}}^{\fs}(\mathfrak{C},M)$ and $\mathfrak{M}_{\bar{m}}^{\inv}(\mathfrak{C},M)$ (here assuming additionally that both the language and the model $M$ are countable) are compact left topological  semigroups. These two observations together yield idempotent measures. In fact, Proposition \ref{proposition: Ellis groups are trivial} implies that every minimal left ideal in any of the above two semigroups consists only of idempotents.

Below we provide several explicit examples of idempotent measures.

\begin{example}[Random graph]Let $T$ be the theory of the random graph. Let $M \prec \mathfrak{C}$ be models of $T$. For simplicity, assume that $|M| = \aleph_0$. Let $\bar{m}$ be an enumeration of $M$. Let $\Phi(\bar{y})$ be a formula without parameters. Then there is a unique measure $\mu$ in $\mathfrak{M}_{\bar{m}}^{\inv}(\mathfrak{C},M)$ which satisfies the following: For any finite sets of parameters $B_1,\dots,B_n$, possibly pairwise indistinct, and for any $\epsilon \colon \mathbb{N} \times \bigcup_{i=1}^{n} B_i \to \{0,1\}$ we have that

\begin{equation*}
    \mu\left(\Phi(\bar{y}) \wedge \bigwedge_{i=1}^{n} \bigwedge_{b \in B_i} R^{\epsilon(i,b)}(y_i,b) \right) =\begin{cases}
\begin{array}{cc}
\frac{1}{2^{|B_1|+ \dots +|B_n|}} & \models\Phi(\bar{m}),\\
0 & \text{otherwise},
\end{array}\end{cases}
\end{equation*}
where $R^{1}(y_i,b) = R(y_i,b)$ and $R^{0}(y_i,b) = \neg R(y_i,b)$. 

We claim that $\mu$ constructed above is a $\emptyset$-definable idempotent, however it is not finitely satisfiable in $M$ and so not generically stable. 
Let us only prove idempotency (the remaining properties are easy):
\begin{align*}
    &(\mu * \mu) \left(\Phi(\bar{y}) \wedge \bigwedge_{i=1}^{n} \bigwedge_{b \in B_i} R^{\epsilon(i,b)}(y_i,b) \right) \\ 
    &= (\mu_{y} \otimes (h_{\bar{b}})_{*}(\mu)_{x}) \left(\Phi(\bar{y}) \wedge \bigwedge_{i=1}^{n} \bigwedge_{b \in B_i} R^{\epsilon(i,b)}(y_i,x_{b}) \right) \\ 
    &= \int F_{\mu_{\bar{y}}}^{\theta(\bar{y};\bar{x})} d(h_{\bar{b}})_{*}(\mu)_{\bar{x}} \\
    &\overset{(*)}{=} \int \mu_{\bar{y}}\left(\Phi(\bar{y}) \wedge \bigwedge_{i=1}^{n} \bigwedge_{b \in B_i} R^{\epsilon(i,b)}(y_i,b) \right) d(h_{\bar{b}})_{*}(\mu)_{\bar{x}}\\ 
    &= \mu\left(\Phi(\bar{y}) \wedge \bigwedge_{i=1}^{n} \bigwedge_{b \in B_i} R^{\epsilon(i,b)}(y_i,b) \right), 
\end{align*}
where
\begin{equation*}
    \theta(\bar{y};\bar{x}) := \Phi(\bar{y}) \wedge \bigwedge_{i=1}^{n} \bigwedge_{b \in B_i} R^{\epsilon(i,b)}(y_i,x_{b}). 
\end{equation*}
We now briefly justify equation $(*)$. The map $F_{\mu}^{\theta(\bar{y};\bar{x})}: S_{\bar{y}}(M) \to [0,1]$ is constant on a set of measure $1$. 
Indeed, we first observe that $(h_{\bar{b}})_{*}(\mu)_{\bar{x}}(\bigwedge_{a \neq b \in B}x_{b} \neq x_{a}) = 1$ where $B = \bigcup_{i=1}^{n} B_i$. Now, for any $q \in [\bigwedge_{a \neq b \in B}x_{b} \neq x_{a}]$, we claim that $F_{\mu}^{\theta}(q) = \mu\left(\Phi(\bar{y}) \wedge \bigwedge_{i=1}^{n} \bigwedge_{b \in B_i} R^{\epsilon(i,b)}(y_i,b) \right)$, since the value of the measure of the formula depends only on the structure of the formula, not on any individual tuple of parameters (as long as they are all distinct!). 
\end{example}

\emph{Blow-ups} give another example of interesting idempotent generically stable types. We provide the following concrete example, but remark that this example clearly generalizes to similar structures. 

\begin{example}[Blow-up]
Let $\mathcal{L} = \{<,E\}$. Let $T$ be the theory of the structure $M = (\mathbb{Q} \times \mathbb{N},<,E)$ where 
\begin{enumerate}
    \item $M \models (q,m) < (p,n)$ if and only if $q < p$. 
    \item $M \models (q,m)E(p,n)$ if and only if $q = p$. 
\end{enumerate} 
The structure $M$ looks like a dense linear ordering without endpoints, but each point is replaced by infinitely many points. We remark that $T$ has quantifier elimination and is NIP. Fix $\mathfrak{C}$ a monster model of $T$. Let $\bar{m} = ((q_1,m_1),(q_2,m_2),\ldots)$ be an enumeration of $M$. For each point $q \in \mathbb{Q}$, there is a unique global type $r_{q}$ such that $(q,m)Ex \in r_{q}$ for some/any $m \in \mathbb{N}$, for any $a \in M$, $r_{q} \vdash x \neq a$, and $r_{q}$ is generically stable over $M$. Consider the type 
\begin{equation*}
    p(\bar y) = \bigotimes_{i=1}^{\omega} r_{q_{i}}(y_{i}). 
\end{equation*}
Then clearly $p \in S_{\bar m}(\FC)$, and $p$ is generically stable over $M$ by Remark \ref{remark: fim implies super-fim?}. 
Moreover, $p$ is an idempotent, which follows from the second explicit formula for $*$ in Proposition \ref{prop:formula.for.star.on.types} and the observation that a tuple $\bar{a}= (a_i)_{i<\omega}$  realizes 
$p|_{\FC'}$ if and only if all $a_i$'s are pairwise distinct, $\bar a$ is disjoint from $\bar \FC'$, and $a_i E(q_i,0)$ for every $i<\omega$ (where $\FC' \succeq \FC$ is a bigger monster model).
\end{example}

Finally, we give an example of an idempotent generically stable measure in a stable theory. 

\begin{example}[Stable] Let $\mathcal{L} = \{E\}$ be a binary relation symbol. Consider the structure $M_{n}$ which has $n$-many countable equivalence classes. For each class $E_i$, let $a_i$ be a representative from $E_i$. Let $\bar{m}$ be an enumeration of $M_n$ such that for every $k \geq 0$ and $0 \leq j \leq n -1$, $M_{n} \models m_{k \cdot n + j} E m_{(k+1)\cdot n + j} $. 
Let $\FC \succ M_n$ be a monster model.
Then for every $\sigma \in \sym(n)$, there exists a unique global type $p_{\sigma}$ which concentrates on $\tp(\bar m/\emptyset)$ such that for any 
$a \in \FC$ and $i \in \mathbb{N}$, $p\vdash x_i \neq a$, and for each $k \geq 0$ and $0 \leq j \leq n-1$, 
\begin{equation*}
    p_{\sigma} \vdash x_{k\cdot n + j}Ea_{\sigma(j)}. 
\end{equation*}
Then we claim that the measure $\mu := \frac{1}{|\sym(n)|} \sum_{\sigma \in \sym(n)} \delta_{p_{\sigma}}$ is a generically stable idempotent. 
Indeed, $\mu$ is clearly invariant under all automorphisms of $\FC$, so generically stable (by stability of $M_n$) and idempotent by Lemma \ref{lemma: equivalence of invariance} and Proposition \ref{proposition: adaptation of 3.19 from [ChGa]} below.
\end{example}

\subsection{FIM subgroups, measures and types}\label{subsec:6.1}
We refer the reader to Definition \ref{def: fim subgroups} for the notion of a relatively $\bar{m}$-type definable over $M$ fim/generically stable subgroup of $\aut(\FC)$. 

In this subsection, we prove some fundamental results about relatively type-definable generically stable and fim subgroups of $\aut(\FC)$. In particular, we prove several results regarding uniqueness of measures which are invariant under the action by left translations by the elements of the relatively type-definable subgroups in question.

\begin{remark}
  By Corollaries \ref{cor:aff01} and \ref{corollary: preservation for infinite tuples}, we can produce examples of relatively $\bar{m}$-type definable fim/generically stable subgroups of $\aut(\FC)$. Simply, a fim [generically stable] over $M$, $\emptyset$-definable group $G$ yields a fim [generically stable] over $\bar{M}$ relatively $\bar n$-type-definable over $\bar{M}$ (even over $M$) group $H\leqslant\aut(\bar{\FC})$ (in the notation of Section \ref{sec: affine sort}).
    
    Another point is that using the properties obtained in Section \ref{sec: affine sort}, starting from any example of a $\emptyset$-definable group $G$ with a definable global idempotent type [measure] which is not concentrated on its right stabilizer, we obtain a corresponding example of a relatively $\bar n$-type-definable group over $\bar{M}$ with an analogous property; this shows that generic stability [or fim] assumption in Conjecture \ref{conjecture: main conjecture} is necessary. Relevant examples in the context of definable groups are given e.g. in Remarks 2.28 and 3.2 from \cite{CGK}.
\end{remark}

\subsubsection{Uniqueness}\label{subsubsection: uniqueness}
As before, let $M\preceq\FC$ be enumerated by $\bar{m}$ and let $\bar{x}$ be a tuple of variables corresponding to $\bar{m}$, and let $\bar{y}$ be a copy of $\bar{x}$.
Consider a $\emptyset$-definable partial type $\pi(\bar{x};\bar{y})$ which contains ``$\bar{x}\equiv\bar{y}$'' and assume that $G_{\pi,\FC}$ forms a subgroup of 
$\aut(\FC)$. 

By the last assumption and Remark \ref{remark: G_pi=G_pi^opp}, we know that $G_{\pi,\FC}=G_{\pi^{\textrm{opp}},\FC}$, which together with the assumption that  $\pi(\bar{x};\bar{y}) \vdash \bar{x}\equiv\bar{y}$ implies that $[\pi(\bar m,\bar y)]=\{\tp(\tau(\bar{m})/M)\;\colon\;\tau\in G_{\pi,\FC}\} = [\pi(\bar{y};\bar{m})]$. These basic properties  will be in use often without comment. 

We are interested in measures $\mu\in\mathfrak{M}_{\pi(\bar{m};\bar{y})}^{\inv}(\FC,M)$ which are left $G_{\pi,\FC}$-invariant, i.e. invariant under the action of $G_{\pi,\FC}$ induced from the standard action of $\aut(\FC)$ via pushforwards. 
Such an invariance implies the invariance under the right action by $\ast$-product in the following sense:

\begin{lemma}\label{lemma: equivalence of invariance}
Let $\mu \in \mathfrak{M}^{\inv}_{\pi(\bar{m};\bar{y})}(\FC,M)$
be Borel-definable over $M$.
We have
$$G_{\pi,\FC}\cdot \mu =\{\mu\} \quad\Rightarrow\quad
\mu \ast S_{\pi(\bar{m};\bar{y})}(\FC)=\{\mu\}.$$
\end{lemma}

\begin{proof}
Consider $q(\bar{y})\in S_{\pi(\bar{m};\bar{y})}(\FC)$, an $\CL$-formula $\varphi(\bar{x};\bar{y})$ and a tuple $\bar{b}\in\FC^{\bar{x}}$.
There exists $\sigma\in\aut(\FC')$ for some bigger monster model $\FC'\succeq\FC$ with $q=\tp(\sigma(\bar{m})/\FC)$.
Let $\bar{d}\in\FC^{\bar{x}}$ realize $\tp(\sigma^{-1}(\bar{b})/M)$,
then there exists $h\in\aut(\FC'/M)$ such that $\bar{d}=h\sigma^{-1}(\bar{b})$.
Note that $h\sigma^{-1}(\bar{m})\models\big(\bar{b}\bar{m}\equiv_{\emptyset}\bar{d}\bar{y}\;\wedge\;\pi(\bar{m};\bar{y})\big)$, so there is $\bar{e}\in\FC^{\bar{y}}$ such that $\bar{e}\models \big(\bar{b}\bar{m}\equiv_{\emptyset}\bar{d}\bar{y}\;\wedge\;\pi(\bar{m};\bar{y})\big)$. 
Then we find $g\in\aut(\FC)$ with $g(\bar{b}\bar{m})=\bar{d}\bar{e}$ and note that $g\in G_{\pi,\FC}$.

We compute: 
\begin{align*}
    (\mu\ast q)\big(\varphi(\bar{b};\bar{y})\big) &= \big(\mu_{\bar{y}}  \otimes h_{\bar{b}}(q)_{\bar{x}}\big)\big(\varphi(\bar{x};\bar{y})\big)\\
&= (F^{\varphi^{\textrm{opp}}(\bar{y};\bar{x})}_{\mu_{\bar{y}}}\circ h_{\bar{b}})(q)\\
    &= \mu\big(\varphi(\bar{d};\bar{y})\big)=\mu\big(\varphi(g(\bar{b});\bar{y})\big) \\
    &= (g^{-1} \cdot \mu)\big(\varphi(\bar{b};\bar{y})\big)=\mu\big(\varphi(\bar{b};\bar{y})\big). \qedhere 
\end{align*}
\end{proof}

\begin{proposition}\label{proposition: adaptation of 3.19 from [ChGa]}
Let $\mu \in \mathfrak{M}^{\inv}_{\pi(\bar{m};\bar{y})}(\FC,M)$
be Borel-definable over $M$ and right $S_{\pi(\bar{m};\bar{y})}(\FC)$-invariant with respect to $\ast$. Then $\mu$ is an idempotent.
\end{proposition}

\begin{proof}
Take an $\CL$-formula $\varphi(\bar{x};\bar{y})$ and a tuple $\bar{b}\in\FC^{\bar{x}}$. For every $q\in S_{\pi(\bar{m};\bar{y})}(\FC)$
we have
\begin{align*}
     (F^{\varphi^{\textrm{opp}}(\bar{y};\bar{x})}_{\mu_{\bar{y}}}\circ h_{\bar{b}})(q) &=
     \big( \mu_{\bar{y}}\otimes h_{\bar{b}}(q)_{\bar{x}}\big)\big(\varphi(\bar{x};\bar{y})\big) \\
     &= (\mu\ast q)\big(\varphi(\bar{b};\bar{y})\big) = \mu\big(\varphi(\bar{b};\bar{y})\big).
\end{align*}
Therefore, the fiber function is constant in the following integral, and we obtain:
\begin{equation*}
(\mu\ast\mu)\big(\varphi(\bar{b};\bar{y})\big)=\int_{S_{\pi(\bar{m};\bar{y})}(\FC)} (F^{\varphi^{\textrm{opp}}(\bar{y};\bar{x})}_{\mu_{\bar{y}}}\circ h_{\bar{b}})\,d\mu_{\bar{x}}=\mu\big(\varphi(\bar{b};\bar{y})\big). \qedhere
\end{equation*}
\end{proof}

By Lemma \ref{lemma: equivalence of invariance} and Proposition \ref{proposition: adaptation of 3.19 from [ChGa]}, we conclude:

\begin{cor}\label{cor: two to one}
$(2)\Rightarrow(1)$ in Conjecture \ref{conjecture: main conjecture} holds.
\end{cor}

\noindent
The implication $(1)\Rightarrow(2)$ in Conjecture \ref{conjecture: main conjecture} is the more difficult one.
Now, we will examine when a (left) $G_{\pi,\FC}$-invariant
fim measure $\mu\in\mathfrak{M}^{\inv}_{\pi(\bar{m};\bar{y})}(\FC,M)$
is unique among (left) $G_{\pi,\FC}$-invariant measures in $\mathfrak{M}^{\inv}_{\pi(\bar{m};\bar{y})}(\FC,M)$.
This does not require assumption $(1)$ from Conjecture \ref{conjecture: main conjecture}. 
Assumption $(1)$ will be important to show that 
a fim measure $\mu\in\mathfrak{M}^{\inv}_{\bar{m}}(\FC,M)$ is concentrated on $[\pi(\bar{m};\bar{y})]$ (cf. Conjecture \ref{conj: gen transitivity} and Corollary \ref{corollary: generic transitivity for types in rosy theories} for the case of types).

\begin{question}\label{question: basic question on fim} 
\begin{enumerate}
    \item Do we have uniqueness of a (left) $G_{\pi,\FC}$-invariant type in $S^{\inv}_{\pi(\bar m;\bar y)}(\FC,M)$ in a generically stable subgroup $G_{\pi,\FC}$?
    \item Do we have uniqueness of a (left) $G_{\pi,\FC}$-invariant
    measure in $\mathfrak{M}^{\inv}_{\pi(\bar m;\bar y)}(\FC,M)$
    in a fim subgroup $G_{\pi,\FC}$?
\end{enumerate}
\end{question}

\begin{remark}\label{remark: correspondence from positive answer}
A positive answer to Question \ref{question: basic question on fim}(2) together with the equivalence between $(1)$ and $(2)$ in Conjecture \ref{conjecture: main conjecture} implies the \emph{in particular} part of said conjecture. 
\end{remark}

\begin{proof}
Assume that the answer to Question \ref{question: basic question on fim}(2) is positive. Then for any relatively $\bar m$-type-definable over $M$ fim subgroup $G$ of $\aut(\FC)$ (defined by a partial type $\pi(\bar x;\bar y)$ as above) there exists a unique measure $\mu_G \in \mathfrak{M}^{\inv}_{\pi(\bar m;\bar y)}(\FC,M)$ which is (left) $G$-invariant. Let $J$ be the set of fim idempotent measures in $\mathfrak{M}^{\inv}_{\bar m}(\FC,M)$ and $\mathcal{G}$ the set of relatively $\bar m$-type-definable over $M$ fim subgroups of $\aut(\FC)$. Then we have well-defined maps $\Phi \colon J \to \mathcal{G}$ and $\Psi \colon \mathcal{G} \to J$ given by $\Phi(\mu):=\stab(\mu)$ and $\Psi(G):=\mu_G$. We need to show that $\Psi \circ \Phi = \id_J$ and $\Phi \circ \Psi =\id_{\mathcal{G}}$. Assume that we have the equivalence between (1) and (2) in Conjecture \ref{conjecture: main conjecture}.

The equality $\Psi \circ \Phi = \id_J$ follows since $\Psi(\Phi(\mu))=\mu_{\stab(\mu)}=\mu$, which holds because  $\mu$ is a fim, (left) $\stab(\mu)$-invariant measure in $\mathfrak{M}^{\inv}_{\pi(\bar m;\bar y)}(\FC,M)$ by $(1) \rightarrow (2)$ (where $\pi(\bar x;\bar y) \vdash \bar x \equiv \bar y$ is such that  $\stab(\mu)=G_{\pi,\FC}$). To show that $\Phi \circ \Psi =\id_{\mathcal{G}}$, first note that $\Phi(\Psi(G))=\Phi(\mu_G)=\stab(\mu_G)$ and we want to prove that it is equal to $G$. The inclusion $\supseteq$ is immediate from the (left) $G$-invariance of $\mu_G$. For the opposite inclusion, consider any $g \in \stab(\mu_G)$ and $p \in \supp(\mu_G)$. Since $\mu_G \in \mathfrak{M}^{\inv}_{\pi(\bar m;\bar y)}(\FC,M)$ (where $\pi(\bar x;\bar y) \vdash \bar x \equiv \bar y$ is such that  $G=G_{\pi,\FC}$), we see that $p=\tp(h'(\bar m)/\FC)$ for some $h' \in G_{\FC'}:=\{\sigma' \in \aut(\FC'): \models \pi(\sigma'(\bar m);\bar m)\}$. Take any extension $g' \in \aut(\FC')$ of $g$. Then $\tp(g'h'(\bar m)/\FC) =gp \in \supp(g\mu_G)=\supp(\mu_G) \subseteq S_{\pi(\bar m;\bar y)}(\FC)$, hence $g'h' \in G_{\pi,\FC'}$, and so $g' \in G_{\pi,\FC'}$ (as $G_{\pi,\FC'}$ is a subgroup). Thus, $g \in G$, as required.
\end{proof}

We resolve the first question above in full generality with a positive answer. We resolve the second question above with a positive answer under the hypothesis of NIP. 
In general, we do not need the NIP assumption, however we need the assumption that the witnessing fim measure is \emph{super-fim}, i.e., the Morley powers of our measure are additionally fim (see Definition \ref{def:fim}). 
We also give a positive answer when the model $M$ and the language are both countable.
To be precise, in the last two cases, we show the uniqueness of a (left) $G_{\pi,\FC}$-invariant measure in $\mathfrak{M}^{\inv}_{\pi(\bar m;\bar y)}(\FC,M)$ which is additionally Borel-definable over $M$.

We first give a positive answer to Question \ref{question: basic question on fim}(1).

\begin{proposition}\label{cor: uniqueness for types}
    Let $G_{\pi,\FC}\leqslant\aut(\FC)$ be a relatively $\bar{m}$-type-definable over $M$ subgroup which is generically stable (where without loss of generality $\pi(\bar x;\bar y) \vdash \bar x \equiv \bar y$).
    Let $p\in S^{\inv}_{\pi(\bar{m};\bar{y})}(\FC,M)$ be (left) $G_{\pi,\FC}$-invariant and generically stable over $M$. 
Then $p$ is the unique (left) $G_{\pi,\FC}$-invariant type in $S^{\inv}_{\pi(\bar{m};\bar{y})}(\FC,M)$.
\end{proposition}

\begin{proof}
 Let $q\in S^{\inv}_{\pi(\bar{m};\bar{y})}(\FC,M)$ be (left) $G_{\pi,\FC}$-invariant. 
\begin{clm}
$p|_M=q|_M$.
\end{clm}

\begin{clmproof}
Consider any formula $\varphi(\bar x;\bar y)$. We have the following equivalences. (In the justifications of these equivalences, one should bear in mind the properties mentioned at the beginning of Section \ref{subsubsection: uniqueness}.) 
\begin{enumerate}
\item  $\varphi(\bar x;\bar{m}) \in q_{\bar x} \iff \varphi(\bar{x};\bar{y}) \in q_{\bar x} \otimes p_{\bar y}$, which holds by the assumptions that $q_{\bar x}$ is $G_{\pi,\FC}$-invariant and $p_{\bar y} \in [\pi(\bar m;\bar y)]$. 

\item $\varphi(\bar{x};\bar{y}) \in q_{\bar x} \otimes p_{\bar y} \iff \varphi(\bar{x};\bar{y}) \in p_{\bar y} \otimes q_{\bar x}$, which holds by generic stability of $p$ (see \cite[Proposition 2.1(iii)]{PiTa}). 

\item $\varphi(\bar{x};\bar{y}) \in p_{\bar y} \otimes q_{\bar x}  \iff \varphi(\bar{m};\bar{y}) \in p_{\bar y}$, which holds by the assumptions that $p_{\bar y}$ is $G_{\pi,\FC}$-invariant and $q_{\bar x} \in [\pi(\bar m;\bar x)]$.
\end{enumerate}
Therefore, 
$$\varphi(\bar x;\bar{m}) \in q_{\bar x} \iff  \varphi(\bar{m};\bar{y}) \in p_{\bar y}.$$ 
Applying this in the special case of $q=p$, we get 
$$\varphi(\bar x;\bar{m}) \in p_{\bar x} \iff  \varphi(\bar{m};\bar{y}) \in p_{\bar y}.$$
The last two equivalences imply 
$$\varphi(\bar x;\bar{m}) \in q_{\bar x} \iff \varphi(\bar x;\bar{m}) \in p_{\bar x}.$$
Since $\varphi(\bar x;\bar y)$ was arbitrary, we conclude that $p|_M=q|_M$.
\end{clmproof}

    As $p$ is generically stable over $M$ and $q$ is $M$-invariant, using the above claim, we may conclude that $p = q$ 
(i.e., see \cite[Proposition 2.1(iv)]{PiTa}).
\end{proof}

To settle Question \ref{question: basic question on fim}(2) in the aforementioned situations requires more work. We start from the following lemma whose proof is a measure-theoretic variant on the proof of Claim 1 in the proof of Proposition \ref{cor: uniqueness for types}.

\begin{lemma}\label{lemma: invariance to uniqueness over M}
Let $\mu,\nu \in \mathfrak{M}^{\inv}_{\pi(\bar{m};\bar{y})}(\mathfrak{C})$ be Borel-definable over $M$ and $G_{\pi,\FC}$-invariant.
If $\mu$ is fim over $M$, then $\mu|_M=\nu|_M$. 
\end{lemma}

\begin{proof} 
Note that in this proof we use Borel $M$-definability only to be able to compute Morley products.
Consider a formula $\varphi(\bar{x};\bar{y})$.
    The map $F^{\varphi(\bar{x};\bar{y})}_{\nu_{\bar{x}}}$
    is constant over $[\pi(\bar{m};\bar{y})]\supseteq\supp(\mu_{\bar{y}}|_M)$ and equal to $\nu(\varphi(\bar{x};\bar{m}))$.
To see it, 
first recall that $[\pi(\bar{m};\bar{y})]=\{\tp(\tau(\bar{m})/M)\;\colon\;\tau\in G_{\pi,\FC}\}$.
    Then, by $G_{\pi,\FC}$-invariance, if $\tau\in G_{\pi,\FC}$ we have
    $$F^{\varphi(\bar{x};\bar{y})}_{\nu_{\bar{x}}}\big(\tp(\tau(\bar{m})/\FC)\big)=\nu\big(\varphi(\bar{x};\tau(\bar{m})\big)=(\tau^{-1})_\ast (\nu)\big(\varphi(\bar{x};\bar{m})\big)=\nu\big(\varphi(\bar{x};\bar{m})\big).$$

    A similar argument shows that
    $F^{\varphi^{\textrm{opp}}(\bar{y};\bar{x})}_{\mu_{\bar{y}}}$
    is constant over $[\pi(\bar{x};\bar{m})]\supseteq\supp(\nu_{\bar{x}}|_M)$ and equal to $\mu(\varphi(\bar{m};\bar{y}))$. 
Only the inclusion $[\pi(\bar{x};\bar{m})]\supseteq\supp(\nu_{\bar{x}}|_M)$ requires a short justification. It follows, because by assumption,  $[\pi(\bar{m};\bar{x})] \supseteq \supp(\nu_{\bar{x}}|_M)$,  and we know that $[\pi(\bar{m};\bar{x})] = [\pi(\bar{x};\bar{m})]$. 

\begin{clm}
    $\nu_{\bar x}(\varphi(\bar{x};\bar{m})) = (\nu_{\bar{x}} \otimes \mu_{\bar{y}})(\varphi(\bar{x};\bar{y}))$.
\end{clm}

\begin{clmproof}
    Notice that 
\begin{equation*}
    (\nu_{\bar{x}} \otimes \mu_{\bar{y}})(\varphi(\bar{x};\bar{y})) = \int_{\supp(\mu)} F_{\nu}^{\varphi(\bar{x};\bar{y})} d\mu=\int_{\supp(\mu)}\nu(\varphi(\bar{x};\bar{m}))d\mu=\nu(\varphi(\bar{x};\bar{m})). 
\end{equation*}
\end{clmproof}

\begin{clm}
    $(\nu_{\bar{x}} \otimes \mu_{\bar{y}})(\varphi(\bar{x};\bar{y})) = \mu_{\bar y}(\varphi(\bar{m};\bar{y}))$.
\end{clm}

\begin{clmproof}
    By Theorem 5.16(a) from \cite{Wild},
    $$(\nu_{\bar{x}} \otimes \mu_{\bar{y}})(\varphi(\bar{x};\bar{y}))=
    (\mu_{\bar{y}} \otimes \nu_{\bar{x}})(\varphi(\bar{x};\bar{y}))= \int_{\supp(\nu)} F^{\varphi^{\textrm{opp}}(\bar{y};\bar{x})}_\mu d\nu=\mu(\varphi(\bar{m};\bar{y})).$$
\end{clmproof}
\noindent
By the above claims, we conclude that
$$\nu_{\bar x}(\varphi(\bar{x};\bar{m}))=\mu_{\bar y}(\varphi(\bar{m};\bar{y})).$$
Applying this in the special case of $\nu = \mu$, we obtain 
$$\mu_{\bar{x}} (\varphi(\bar{x};\bar{m}))=\mu_{\bar{y}} (\varphi(\bar{m};\bar{y})).$$
The last two exposed lines imply
$$\nu_{\bar{x}} (\varphi(\bar{x};\bar{m}))=\mu_{\bar{y}} (\varphi(\bar{m};\bar{y}))=\mu_{\bar{x}} (\varphi(\bar{x};\bar{m})).$$
Hence, we conclude that $\nu|_{M} = \mu|_{M}$, completing the proof. 
\end{proof}

\begin{lemma}\label{lemma: superfim and Borel-def}
In the following, we let $y$ denote a tuple of variables. Let $\mu,\nu \in \mathfrak{M}_{y}(\mathfrak{C})$ be Borel-definable over $M$ with $\mu|_{M} = \nu|_{M}$, and let $\mu$ be fim over $M$. 
If 
\begin{enumerate}
    \item $\mu$ is super-fim over $M$, or 
    \item $|\CL|\leqslant\aleph_0$ and
    $\nu$ is Borel-definable over a countable model,
\end{enumerate}
then $\mu = \nu$. 
\end{lemma}

\begin{proof} 
The proof follows the structure of the proof of Lemma 2.14 from \cite{HruPiSi13}. 
We first construct a measure $\lambda\in\mathfrak{M}_{\mathbf{y}}(\FC)$ where $\mathbf{y} = (y_i)_{i = 1}^{\omega}$ and for each $i < \omega$, $|y_i| = |y|$. 
\begin{enumerate}
    \item $({\lambda_1})_{y_1} = \mu_{y_1}$. 
    \item If $n$ is even, then $({\lambda_{n}})_{y_1,\dots,y_n}= \nu_{y_n} \otimes ({\lambda_{n-1}})_{y_1,\dots,y_{n-1}}$. 
    \item If $n$ is odd, then $({\lambda_n})_{y_1,\ldots, y_n} = \mu_{y_n} \otimes ({\lambda_{n-1}})_{y_1,\dots,y_{n-1}}$.
    \item $\lambda = \bigcup_{n=1}^{\omega} ({\lambda_{n}})_{y_1,\dots,y_n}$.
\end{enumerate}

\begin{clm}
We claim that $\lambda|_{M} = \mu^{(\omega)}|_{M}$.
\end{clm}

\begin{clmproof}
The proof proceeds by induction on the power of the Morley product. The base case is obvious. Now, assume that $\lambda|_{M,y_1,\dots,y_n}=\mu^{(n)}_{y_1,\dots,y_n}|_M$ and consider some
$\theta(y_1,\dots,y_{n+1}) \in \mathcal{L}(M)$.
For odd $n+1$, let $\bar{y} := (y_1,\dots,y_n)$ and notice 
that
\begin{align*}
    \lambda\big( \theta(y_1, \dots,y_{n+1})  \big) &=
    \big(\mu_{y_{n+1}} \otimes \lambda_{n})\big(\theta(y_1,\dots,y_{n+1})\big) \\
    &= \int_{S_{\bar{y}}(M)}  F^{\theta^{\textrm{opp}}(y_{n+1};\bar{y})}_{\mu_{y_{n+1}}}\,d\left( \lambda_{n}|_{M} \right)
    = \int_{S_{\bar{y}}(M)} F^{\theta^{\textrm{opp}}(y_{n+1};\bar{y})}_{\mu_{\bar{y}_{n+1}}}\,d \left( \mu^{(n)}_{\bar{y}}|_{M} \right) \\
    &= 
    \mu^{(n+1)} \big(\theta(y_{1},\dots,y_{n+1})\big),
\end{align*}
where the third equality follows from the induction hypothesis. When $n+1$ is even, we compute the following
\begin{align*}
    \lambda\big( \theta(y_1, \dots,y_{n+1}) \big) &=
\big(\nu_{y_{n+1}}\otimes\lambda_{n}\big)\big(\theta(y_1, \dots,y_{n+1})\big) \\
    &= \int_{S_{\bar{y}}(M)}  F^{\theta^{\textrm{opp}}(y_{n+1};\bar{y})}_{\nu_{{y}_{n+1}}}\,d\left(\lambda_{n}|_M \right)\\
    &\overset{(a)}{=} \int_{S_{\bar{y}}(M)}  F^{\theta^{\textrm{opp}}(y_{n+1};\bar{y})}_{\nu_{{y}_{n+1}}}\,d\left( \mu^{(n)}_{\bar{y}} |_M \right)\\
    &= \big(\nu_{y_{n+1}}\otimes\mu^{(n)}_{\bar{y}}\big) \big(\theta(y_1, \dots,y_{n+1}) \\
    & \overset{(\diamondsuit)}{=}
    \big(\mu^{(n)}_{\bar{y}} \otimes \nu_{\bar{y}_{n+1}}\big)
    \big(\theta(y_{1},\ldots,y_{n+1})\big) \\
    &= \int_{S_{y_{n+1}}(M)} F^{\theta(\bar{y};y_{n+1})}_{\mu^{(n)}_{\bar{y}}}\,d\left( \nu_{y_{n+1}}|_{M} \right) \\
    &\overset{(b)}{=}\int_{S_{y_{n+1}}(M)} F^{\theta(\bar{y};y_{n+1})}_{\mu^{(n)}_{\bar{y}}}\,d\left( \mu_{y_{n+1}}|_{M} \right) \\
    &= \big(\mu^{(n)}_{\bar{y}} \otimes \mu_{y_{n+1}}\big)
    \big(\theta(y_{1},\ldots,y_{n+1})\big) \\
    &\overset{(c)}{=}  
    \big(\mu_{y_{n+1}}\otimes \mu^{\otimes (n)}_{\bar{y}}\big)
    \big(\theta(y_{1},\ldots,y_{n+1})\big) \\
    &= \mu^{(n+1)}
    \big(\theta(y_{1},\ldots,y_{n+1})\big). 
\end{align*}
Equation $(a)$ follows from the induction hypothesis. Equation $(b)$ is the hypothesis of the lemma. Equation $(c)$ follows from the fact that fim measures commute with Borel definable measures in arbitrary theories (i.e., Threorem 5.16.(a) of \cite{Wild}).

We now justify equation $(\diamondsuit)$ under the two separate hypotheses of the lemma. First, we suppose that $\mu$ is \emph{super-fim} over $M$. By definition, this implies that $\mu^{(n)}$ is fim over $M$. Again, by Threorem 5.16.(a) from \cite{Wild}, this implies that $\mu^{(n)}$ commutes with any Borel-definable measure, in this case $\nu_{y_{n+1}}$, and thus the equation is justified.

Now we assume that $|\CL|\leqslant\aleph_0$ and that $\nu$ is Borel-definable over a countable model.
Then $\nu$ and $\mu$ are Borel-definable over a countable model $M_0\preceq\FC$. 
Theorem 2.13 from \cite{Wild} implies that the Morley product of $\nu$ with any Morley power of $\mu$ is Borel-definable over $M_0$ and any product of these measures is associative. Hence, by associativity and Theorem 5.16.(a) from \cite{Wild},
        \begin{align*}
            \nu \otimes \mu^{(n)} &= (\nu \otimes \mu) \otimes \mu^{(n-1)} = (\mu \otimes \nu) \otimes \mu^{(n-1)} \\ 
             &= \mu \otimes (\nu \otimes \mu^{(n-1)}) = \ldots = \mu^{(n)} \otimes \nu,
        \end{align*}
and we derive $(\diamondsuit)$. 
\end{clmproof}

Suppose that $\mu\neq\nu$. Then there exists an $\mathcal{L}_{y}(\FC)$-formula $\varphi(y;\bar{b})$ such that  $\mu(\varphi(y;\bar{b}))=r\neq s=\nu(\varphi(y;\bar{b}))$. Then 
$\lambda(\varphi(y_n;\bar{b}))$
is equal to $r$ for odd $n$ and to $s$ for even $n$. However, since $\mu$ is fim over $M$, this implies that $\mu$ is \emph{self-averaging} (cf. Definition 2.2 and Theorem 2.7 in \cite{ConGanHan23}). In other words, because $\lambda|_{M} = \mu^{(\omega)}|_{M}$, we may conclude that
\begin{equation*}
\lim_{i \to \infty} \lambda(\varphi(y_i;\bar{b})) = \mu(\varphi(y;\bar{b})). 
\end{equation*} 
Thus, we have a contradiction. 
\end{proof}

\begin{cor}\label{cor: uniqueness for measures}
    Let $G_{\pi,\FC}\leqslant\aut(\FC)$ be a relatively $\bar{m}$-type-definable over $M$ subgroup which is fim 
(where without loss of generality $\pi(\bar x;\bar y) \vdash \bar x \equiv \bar y$). 
    Let $ \mu \in \mathfrak{M}^{\inv}_{\pi(\bar{m};\bar{y})}(\FC,M)$ be (left) $G_{\pi,\FC}$-invariant and fim over $M$. Then $\mu$ is the unique (left) $G_{\pi,\FC}$-invariant measure in $\mathfrak{M}^{\inv}_{\pi(\bar{m};\bar{y})}(\FC,M)$ which is Borel-definable over $M$, 
assuming one of the following: 
    \begin{enumerate} 
    \item $T$ is NIP. 
(In this case, $\mu$ is the unique measure in $\mathfrak{M}^{\inv}_{\pi(\bar{m};\bar{y})}(\FC,M)$ which is (left) $G_{\pi,\FC}$-invariant measure, since $M$-invariance is equivalent to Borel-definability over $M$.)
    \item $\mathcal{L}$ is countable and $M$ is countable. 
    \item The measure $\mu$ is super-fim over $M$. 
    \end{enumerate} 
\end{cor}

\begin{proof} 
Items (2) and (3) follow directly from Lemmas \ref{lemma: invariance to uniqueness over M} and \ref{lemma: superfim and Borel-def}. Item (1) is a particular case of (3), because NIP implies that fim measures are super-fim (see Remark \ref{remark: fim implies super-fim?}).
\end{proof}

\subsection{Main conjecture for types}\label{subsec:6.2}
Here, we prove a variant of the main conjecture, Conjecture \ref{conjecture: main conjecture}, in the context of types under the various model-theoretic assumptions.
These results are similar to ones proved by the first and third author (along with Chernikov) in the context of definable groups. These proofs are quite similar and we refer the reader to Section 2 of \cite{CGK} to see for themselves. 

As usual, we let $\FC$ be a monster model of a fixed theory and $\FC' \succeq \FC$ be a bigger monster model in which $\FC$ is small.
Let $M\preceq\FC$ be enumerated by $\bar{m}$ and let $\bar{y}$ be a tuple of variables corresponding to $\bar{m}$, and let $\bar{x}$ be a copy of $\bar{y}$.
Consider a $\emptyset$-definable partial type $\pi(\bar{x};\bar{y})$ which contains ``$\bar{x}\equiv\bar{y}$'' and assume that $G_{\pi,\FC}$ forms a subgroup of 
$\aut(\FC)$.

Below we first prove a counterpart of a fact on definable generically stable groups (namely, see \cite[Fact 2.5(2)]{CGK} and the proof of \cite[Lemma 2.1]{PiTa}). Finding the correct statement in the context of theories is not so straightforward (see the example below). This proposition and the example illustrate both the similarities and the differences between the definable group setting and the setting of arbitrary theories.   

\begin{proposition}\label{proposition: generic stability and the inverse}
Let $p \in S^{\inv}_{\pi(\bar{m};\bar{y})}(\FC,M)$ be generically stable and (left) $G_{\pi,\FC}$-invariant. 
Then there exists $\sigma \in \aut(\FC')$ such that $\sigma^{-1}(\bar m) = \sigma(\bar m) \models p$.
\end{proposition}

\begin{proof}
Take $g \in \aut(\FC')$ with $g(\bar m) \models p$. It is easy to find a small $N \preceq \FC'$ containing $\FC$ and satisfying $g[N]=N$. Choose $f \in  \aut(\FC')$ such that $f(\bar m) \models p|_N$. Then $f,g \in G_{\pi,\FC'}$. By the $G_{\pi,\FC}$-invariance of $p$, we get that $p|_N$ is $G_{\pi,N}$-invariant, and so $g^{-1}f(\bar m) \models p|_N$; hence, $g^{-1}f(\bar m) \models p$. 

Since $p$ is generically stable and $(f(\bar m),g(\bar m)) \models p^{(2)}$, we get that $(f(\bar m),g(\bar m)) \equiv_{\FC} (g(\bar m),f(\bar m))$. Hence, there is $h \in \aut(\FC'/\FC)$ such that $h(f(\bar m)) =g(\bar m)$ and $h(g(\bar m))=f(\bar m)$. So there exist $\chi_1,\chi_2 \in \aut(\FC'/M)$ such that $f=h^{-1}g\chi_1$ and $g=h^{-1}f\chi_2$. This implies that $f^{-1}g= \chi_1^{-1}g^{-1}f\chi_2$. 

Put $\sigma:= g^{-1}f \chi_1^{-1}$. Then $\sigma(\bar m)=g^{-1}f(\bar m) \models p$. On the other hand, $\sigma^{-1}(\bar m)= \chi_1 f^{-1}g(\bar m) = \chi_1 \chi_1^{-1}g^{-1}f\chi_2(\bar m) =g^{-1}f(\bar m)$ which equals $\sigma(\bar m)$.
\end{proof}

The following example shows that one cannot strengthen the conclusion of Proposition \ref{proposition: generic stability and the inverse} by saying that for every $\sigma \in \aut(\FC')$ such that $\sigma(\bar m) \models p$, we have that $\sigma^{-1}(\bar m) \models p$. 

\begin{example}
We use the affine sort notation from Section \ref{sec: affine sort}.
Let $G=M:=(\mathbb{Q},+)$.
Let $q = \tp(g/\FC)\in S_G(\FC)$ be the unique nonalgebraic complete type over $\FC$, where $g \in \FC'$.  Choose $\sigma \in \aut(\FC'/M)$ so that $\sigma^{-1}(g) \in \FC$. Then $p:=\tp((g,\sigma)(\bar n)/\bar{\FC}) \in S^{\inv}_{\pi(\bar{x}\bar{y};\bar{n})}(\bar{\FC},\bar M)=\widetilde{H}_{\bar{\FC},\bar{n}}\cap S^{\inv}_{\bar{n}}(\bar{\FC},\bar{M})$ is (left) $H$-invariant and idempotent in the stable theory $\Th(\bar M)$ (where recall that $H \leqslant \aut(\FC)$ is relatively $\bar n$-type-definable over $\bar n$), but 
$\tp((g,\sigma)^{-1}(\bar n)/\bar \FC) \ne p$.
\end{example}

\begin{proof}
It is clear that $p \in  \widetilde{H}_{\bar{\FC},\bar n}$.
It is also clear that $f(p)=\tp(-g/\FC) =q$ and $f(\tp((g,\sigma)^{-1}(\bar n)/\bar \FC))=f(\tp((\sigma^{-1}(-g),\sigma^{-1})(\bar n)/\bar \FC)) = \tp(\sigma^{-1}(g)/\FC) \ne q$, as $\tp(\sigma^{-1}(g)/\FC)$ is algebraic (where $f \colon \widetilde{H}_{\bar{\FC},\bar n} \to S_G(\FC)$ is the homeomorphism from Section \ref{sec: affine sort}). So, by Corollary \ref{corollary: preservation for infinite tuples},
Remark \ref{cor:aff01} and Lemma \ref{lemma: affine star transfer for types}, it remains to show the following statements:

\begin{enumerate}
\item $q$ is invariant over $M$;
\item $q$ is right $G(\FC)$-invariant;
\item $q$ is idempotent.
\end{enumerate}

All three items ((1) even over $\emptyset$) follow easily from the fact that $q$ is the unique nonalgebriac type in $S_G(\FC)$.
\end{proof}

We begin by listing several properties of a generically stable type each of which  turns out to be equivalent to being concentrated on the stabilizer, and we call the types with these equivalent properties {\em generically transitive}. We deduce from our earlier observations that if a generically stable type is generically transitive, then the main conjecture holds for this type. We then use it to prove the main conjecture  for types under the various assumptions.

Let $p \in S^{\inv}_{\bar m}(\FC,M)$ be generically stable. Let $p' \in  S^{\inv}_{\bar m}(\FC',M)$ be the unique extension of $p$ to an $M$-invariant complete type over $\FC'$. Let $G =\stab(p)$, a subgroup of $\aut(\FC)$ which is relatively $\bar m$-type-definable over $M$ by Lemma \ref{lemma: rel. type-definability}, so can be written as $G_{\pi, \FC}=\{\sigma \in \aut(\FC) \;\colon\; \models \pi(\sigma(\bar m);\bar m)\}$ for some type $\pi(\bar x; \bar y)$ without parameters such that
$\pi(\bar{x};\bar{y})$ contains ``$\bar{x}\equiv_{\emptyset}\bar{y}$''. Then $p'$ is definable over $M$ by the same defining scheme as $p$, and  $\stab(p')=G_{\pi,\FC'}$. 
Recall that since $G_{\pi, \FC}$ is a group and $\pi(\bar{x};\bar{y}) \vdash \bar{x}\equiv_{\emptyset}\bar{y}$, we we have $G_{\pi^{\textrm{opp}},\FC}=G_{\pi^{\textrm{opp}},\FC}$ and $\pi(\bar y;\bar m)$ is equivalent to $\pi(\bar m;\bar y)$.

\begin{remark}\label{remark: generic transitivity}
The following conditions are equivalent.
\begin{enumerate}
\item 
$p \in S^{\inv}_{\pi(\bar{m};\bar{y})}(\FC,M)$, i.e. $p$ extends the partial type $\pi(\bar m;\bar y)$.

\item 
For every $\sigma \in \aut(\FC')$ such that $\sigma(\bar m) \models p$, we have $\sigma(p')=p'$.

\item 
There exists  $\sigma \in \aut(\FC')$ such that $\sigma(\bar m) \models p$ and $\sigma(p')=p'$.

\item 
For every $A \subseteq \FC'$ and for every $\sigma \in \aut(\FC')$ such that $\sigma(\bar m) \models p$, if $\bar b \models p|_{A\sigma^{-1}(\bar m)}$, then $\sigma(\bar b) \models p|_{\sigma[A]\bar m}$.

\item 
There exists a small  $A \subseteq \FC'$ containing $M$ such that for every $\sigma \in \aut(\FC')$ with $\sigma(\bar m) \models p$, if $\bar b \models p|_{A\sigma^{-1}(\bar m)}$, then $\sigma(\bar b) \models p|_{\sigma[A]\bar m}$.

\item
There exists a small  $A \subseteq \FC'$ containing $M$ and $\sigma \in \aut(\FC')$ with $\sigma(\bar m) \models p$, such that if  $\bar b \models p|_{A\sigma^{-1}(\bar m)}$, then $\sigma(\bar b) \models p|_{\sigma[A]\bar m}$.
\end{enumerate}
\end{remark}

\begin{proof}
$(1) \Leftrightarrow (2) \Leftrightarrow (3)$. Consider any $\sigma \in \aut(\FC')$ such that $\sigma(\bar m) \models p$. Then we have the following equivalences:
$$\sigma(p')=p' \iff \sigma \in G_{\pi, \FC'} \iff \models \pi(\sigma(\bar m);\bar m) \iff p=\tp(\sigma(\bar m)/\FC) \in S_{\pi(\bar{m};\bar{y})}(\FC).$$

$(2) \Rightarrow (4)$. We have $\tp(\sigma(\bar b)/\sigma[A] \bar m) = \sigma(\tp(\bar b/ A\sigma^{-1}(\bar m)) =\sigma(p'|_{A\sigma^{-1}(\bar m)}) = \sigma(p')|_{\sigma[A] \bar m} = p' |_{\sigma[A] \bar m} = p |_{\sigma[A] \bar m}$.

$(4) \Rightarrow (5)$ and $(5) \Rightarrow (6)$ are trivial.

$(6) \Rightarrow (3)$. 
Pick $A$ and $\sigma$ witnessing that (6) holds. We will show that $\sigma$ witnesses that (3) holds. Suppose not. Then there is a small $B \subseteq \FC'$ containing $A\sigma^{-1}[M]$ such that $\sigma(p'|_B) = \sigma(p')|_{\sigma[B]} \ne p' |_{\sigma[B]}$. So there is $\bar b \in p'|_B(\FC')$ such that $\sigma(\bar b) \not\models  p' |_{\sigma[B]}$. On the other hand, by generic stability of $p$ over $M \subseteq A$, we have $\bar b \ind_{A\sigma^{-1}(\bar m)} B$. This clearly implies $\sigma(\bar b) \ind_{\sigma[A] \bar m} \sigma[B]$, but also, since $A$ and $\sigma$ witness (6), we get $\sigma(\bar b) \models p|_{\sigma[A] \bar m}$. These two observations together with generic stability of $p$ over $M$ imply that $\sigma(\bar b) \models p'|_{\sigma[B]}$, a contradiction.
\end{proof}

\begin{remark}\label{remark: extension of equivalences for gen. trans.}
In Remark \ref{remark: generic transitivity}, each assumption ``$\sigma(\bar m) \models p$'' can be replaced by ``$\sigma^{-1}(\bar m) \models p$''.
\end{remark}

\begin{proof}
In items (2) and (3), it follows immediately from the fact that $\sigma(p')=p'$ if and only if $\sigma^{-1}(p')=p'$. The proofs of $(2) \Rightarrow (4) \Rightarrow (5) \Rightarrow (6) \Rightarrow (3)$ do not actually use $\sigma(\bar m) \models p$, i.e. if we replace this condition by $\sigma^{-1}(\bar m) \models p$ in one of these items, we automatically get this condition in the remaining ones.
\end{proof}

\begin{definition}
    Following the terminology from Definition 2.13 of \cite{CGK}, we will call a generically stable type $p \in S^{\inv}_{\bar m}(\FC,M)$ {\em generically transitive over $M$} if the equivalent conditions from Remark \ref{remark: generic transitivity} hold.
\end{definition}

\begin{remark}\label{remark: generic transitivity and affine sort}
Working in the set-up of Section \ref{sec: affine sort}:
A type $p \in \widetilde{H}_{\bar{\FC}, \bar n}$ is generically transitive over $\bar M$ if and only if $f(p) \in S_G(\FC)$ is generically transitive over $M$.
\end{remark}

\begin{proof}
(Warning: in this remark, there are two different relatively $\bar{n}$-type-definable over $\bar{M}$ subgroups of $\aut(\bar{\FC})$ - the one from Section \ref{sec: affine sort} fixing pointwise the home sort, and the one given by $\stab(p)$.)

By Corollary \ref{corollary: preservation for infinite tuples}, $p$ is generically stable over $\bar M$ if and only if $f(p)$ is generically stable over $M$. So in the proof of both implications below, these equivalent conditions hold. 
 Let $f' \colon \widetilde{H}_{\bar{\FC'}, \bar n} \to S_G(\bar{\FC}')$ be the counterpart of $f$ from Section \ref{sec: affine sort} for $\FC'$ in place of $\FC$. 

($\Rightarrow$).
Assume that $p$ is generically transitive, i.e. condition (3) from Remark \ref{remark: generic transitivity} holds. This is witnessed by some $\bar{\sigma}=(g,\sigma) \in G(\FC') \rtimes \aut(\FC') = \aut(\bar{\FC}')$. As $\bar{\sigma}(\bar n) \models p \in \widetilde{H}_{\bar{\FC}, \bar n}$, we have that $\bar{\sigma} \in H_{\bar{\FC}'}$ (i.e., the counterpart of $H$ computed in $\bar{\FC}'$). By assumption and Remark \ref{rem:aff01}, $f'(p') = f'(\bar{\sigma}(p')) = f'(p')g^{-1}$. We also have $f(p) = \tp(g^{-1}/\FC)$, and 
$f'(p') \in S_G(\FC')$ is the unique $M$-invariant extension of $f(p)$. So $f(p)$ is generically transitive by item (5) of \cite[Remark 2.12]{CGK}.

$(\Leftarrow)$. Assume that $f(p)$ is generically transitive, i.e. condition (5) of  \cite[Remark 2.12]{CGK} holds. This is witnessed by some $g \in G(\FC')$. Put $\bar \sigma: =(g^{-1},\id_{\FC'}) \in \aut(\bar{\FC}'/\bar{M})\leqslant H_{\bar{\FC}'}$. Then $f(\tp(\bar{\sigma}(\bar n)/\bar{\FC}))=\tp(g/\FC) =f(p)$ (by the choice of $g$), so $\tp(\bar{\sigma}(\bar n)/\bar{\FC})=p$ by injectivity of $f$. By the choice of $g$ and Remark \ref{rem:aff01} (and the fact that $f'(p') \in S_G(\FC')$ is the unique $M$-invariant extension of $f(p)$), $f'(\bar{\sigma}(p')) = f'(p')g =f'(p')$, so $\bar{\sigma}(p') = p'$ since $f'$ is injective. Hence, $p$ is generically transitive by item (3) of Remark \ref{remark: generic transitivity}.
\end{proof}

By Propositions \ref{proposition: adaptation of 3.19 from [ChGa]} and Corollary \ref{cor: uniqueness for types}, using item (1) of  Remark \ref{remark: generic transitivity}, we see that Conjecture \ref{conjecture: main conjecture} for types is equivalent to the following:

\begin{conj}\label{conj: gen transitivity}
An idempotent generically stable type is generically transitive.
(In other words, if a generically stable type is idempotent, then it is concentrated on the type defining its stabilizer.)
\end{conj}

The material from Sections 2.7-2.10 from \cite{CGK} goes through in a slightly simplified form. Namely, we do not need any stratified local ranks; we just use standard local ranks (or Shelah degrees or thorn-ranks) and the trivial fact that they are invariant under automorphisms in place of invariance under the action of the definable group in question. In conclusion, we get the following:

\begin{cor}\label{corollary: generic transitivity for types in rosy theories}
In every rosy (in particular, in every stable or even simple) theory, each idempotent generically stable type is generically transitive. In particular, Conjecture\ref{conjecture: main conjecture} holds for types in rosy theories. Moreover, in any theory, each idempotent type which is stable over some model $M$ is generically transitive.
\end{cor}

Corollary \ref{corollary: generic transitivity for types in rosy theories} together with Remark \ref{remark: generic transitivity and affine sort} and Lemma \ref{lemma: affine star transfer for types}
yields the main conclusions of Sections 2.7-2.10 from \cite{CGK} without using stratified local ranks.

\subsection{Main conjecture for special measures in NIP}\label{subsec:6.3}
In this section, we prove 
Conjecture \ref{conjecture: main conjecture} under the NIP hypothesis, assuming the measure in consideration is also $\autf_{\KP}(\FC)$-invariant. These measures are 
essentially controlled by their pushforwards to $\gal_{\KP}(T)$.
Many of the results in this section are variants of the ones proved in the definable group setting under the hypothesis of 
$\mathcal{G}^{00}$-invariance (i.e., see \cite{Artem_Kyle2}).

For a short overview of strong types and associated Galois groups see Section \ref{subsection: Galois groups} and the general references in there. We will be using the notation from the last paragraph of that section. Recall that $M \preceq \FC\preceq \FC' \preceq \FC''$ and $\bar m$ is an enumeration of $M$. The map $\rho_{\FC}$ will be denoted by $\rho$.

We will need the following classical fact (for a proof of surjectivity, e.g. see the argument in \cite[Proposition 3.4]{Artem_Kyle2}).

\begin{fact}\label{fact: pushforward onto}
Let $f \colon X \to Y$ be a continuous map between compact spaces. Then the pushforward $f_* \colon \mathcal{M}(X) \to \mathcal{M}(Y)$ is continuous. If $f$ is also surjective, then so is $f_*$.
\end{fact}

Applying it to our continuous surjection  $\rho \colon S_{\bar m}(\FC) \to \gal_{\KP}(T)$, we get that  
$$\rho_{\ast}\colon \mathfrak{M}_{\bar{m}}(\FC)\to \mathcal{M}(\gal_{\KP}(T)),$$
is continuous, onto, and clearly affine.

In the case of a definable group $G=G(M)$, the counterpart of our map $\rho$ is the natural continuous surjection from the space of global types concentrated on $G$ to $G(\FC)/G(\FC)^{00}$, and it is easy see (using coheirs) that the restriction of this map to the types finitely satisfiable in $G$ is still surjective so that one can still apply Fact \ref{fact: pushforward onto} to this restriction. We first show an analogous result for theories.

Extending the context of Definition \ref{def:sfs}, a partial type $r(\bar y)$ which extends $\tp(\bar{m}/\emptyset)$ will be called {\em strongly finitely satisfiable} in $M$ if for every formula $\varphi(\bar y;\bar b) \in r(\bar y)$ there exists $\bar a \in M^{\bar y}$ such that $\bar a \equiv \bar m$ and $\models \varphi(\bar a;\bar b)$. By $S^{\sfs}_{\bar{m}}(N,M)$ we denote the set of complete types over $N$ concentrated on $\tp(\bar m/\emptyset)$ and strongly finitely satisfiable in $M$. 

\begin{remark}\label{remark:sfs=fs}
    If $M$ is $\aleph_0$-saturated and strongly $\aleph_0$-homogeneous, then
    $$S^{\sfs}_{\bar{m}}(M,M)=S^{\fs}_{\bar{m}}(M,M)=S_{\bar{m}}(M).$$
\end{remark}

\begin{proof}
    It is enough to show that $S_{\bar{m}}(M)\subseteq S^{\sfs}_{\bar{m}}(M,M)$.
Let $\varphi( \bar{y};\bar{m})\in p( \bar{y})\in S_{\bar{m}}(M)$. There exist finite $ \bar{y}_0\subseteq \bar{y}$ and finite $\bar{m}_0\subseteq\bar{m}$ such that $\varphi( \bar{y},\bar{m})=\varphi_0( \bar{y}_0;\bar{m}_0)$ for some formula $\varphi_0$. Because $\tp(\bar{m}/\emptyset)\cup\{\varphi( \bar{y};\bar{m})\}\subseteq p$, we have
    $$\tp(\bar{m}/\emptyset)|_{ \bar{y}_0}\cup\{\varphi_0( \bar{y}_0;\bar{m}_0)\}\subseteq (p|_{\bar{m}_0})|_{ \bar{y}_0}.$$
    Due to the $\aleph_0$-saturation of $M$, this last type is realized in $M$ by some $\bar{a}_0$.
    As $M$ is strongly $\aleph_0$-homogeneous, $\bar{m}|_{ \bar{y}_0}\equiv\bar{a}_0$ implies existence of $\sigma\in\aut(M)$ such that $\sigma(\bar{m}|_{ \bar{y}_0})=\bar{a}_0$.
    Then $\bar{a}:=\sigma(\bar{m})$ realizes $\tp(\bar{m})\cup\{\varphi( \bar{y};\bar{m})\}$ in $M$.
\end{proof}

\begin{lemma}\label{lemma:sfs.extensions}
    Assume that $M\preceq N\preceq \FC$ and $p( \bar{y})\in S^{\sfs}_{\bar{m}}(N,M)$.
    Then there exists $q( \bar{y})\in S^{\sfs}_{\bar{m}}(\FC,M)$ such that $p\subseteq q$.
\end{lemma}

\begin{proof}
    Let $r(\bar{y})$ be a maximal set of $\CL_{ \bar{y}}(\FC)$-formulas which contains $p( \bar{y})$ and is strongly finitely satisfiable in $M$ (such a maximal set exists by Zorn's lemma). The set $r( \bar{y})$ is a complete type. 
Indeed, let $\varphi( \bar{y};\bar{a})\in \CL_{ \bar{y}}(\FC)$ be such that $\varphi\not\in r$ and $\neg\varphi\not\in r$.
    Then both $r\cup\{\varphi\}$ and $r\cup\{\neg\varphi\}$ are not strongly finitely satisfiable in $M$. 
    So there exists a finite set $\Delta_1( \bar{y})\subseteq r( \bar{y})$ such that there is no $\bar{a}\in M^{ \bar{y}}$ with $\bar{a}\equiv\bar{m}$ and $\bar{a}\models \Delta_1\cup\{\varphi\}$. 
    Similarly, there is a finite set $\Delta_2( \bar{y})\subseteq r( \bar{y})$ such that there is no $\bar{a}\in M^{ \bar{y}}$ with $\bar{a}\equiv\bar{m}$ and $\bar{a}\models\Delta_2\cup\{\neg\varphi\}$. 
    Because $\Delta_1( \bar{y})\cup\Delta_2( \bar{y})\subseteq r( \bar{y})$ and $r$ is strongly finitely satisfiable in $M$, there is $\bar{a}\in M^{\bar{y}}$ such that $\bar{a}\equiv\bar{m}$ and $\bar{a}\models\Delta_1\cup\Delta_2$, but then $\bar{a}$ must satisfy $\varphi$ or $\neg\varphi$ and we get a contradiction.
\end{proof}

\begin{lemma}\label{lemma: on sfs also onto}
Assume that $M$ is $\aleph_0$-saturated and strongly $\aleph_0$-homogeneous. Then the map 
$$\rho |_{S_{\bar m}^{\mathrm{sfs}}(\FC,M)} \colon S_{\bar m}^{\mathrm{sfs}}(\FC,M)  \to \gal_{\KP}(T)$$ 
is surjective. Thus, the map
$$\rho |_{S_{\bar m}^{\mathrm{inv}}(\FC,M)} \colon S_{\bar m}^{\mathrm{inv}}(\FC,M)  \to \gal_{\KP}(T)$$ 
is also surjective.
\end{lemma}

\begin{proof}
We need to show that for every $\sigma \in \aut(\FC)$ there exists $p \in S_{\bar m}^{\mathrm{sfs}}(\FC,M)$ and $\tau \in \aut(\FC)$ such that $\tau(\bar m) \models p|_M$ and $\tau/\autf_{\KP}(\FC) = \sigma/\autf_{KP}(\FC)$.  We check that $\tau=\sigma$ works.

By Remark \ref{remark:sfs=fs}, $\tp(\sigma(\bar m)/M) \in  S^{\sfs}_{\bar{m}}(M,M)$. Hence, using Lemma \ref{lemma:sfs.extensions}, there exists $p \in S_{\bar m}^{\mathrm{sfs}}(\FC,M)$ extending $\tp(\sigma(\bar m)/M)$. Then $\sigma(\bar m) \models p|_M$, and we are done.
\end{proof}

The next corollary follows from Fact \ref{fact: pushforward onto} and Lemma \ref{lemma: on sfs also onto}, bearing in mind the fact that for $\mu \in \mathfrak{M}_{\bar{m}}(\FC)$:
\begin{enumerate}
\item if the support of $\mu$ is contained in $S_{\bar m}^{\mathrm{sfs}}(\FC,M)$, then $\mu \in  \mathfrak{M}_{\bar{m}}^{\sfs}(\FC,M)$;
\item if the support of $\mu$ is contained in $S_{\bar m}^{\mathrm{inv}}(\FC,M)$, then $\mu \in  \mathfrak{M}_{\bar{m}}^{\inv}(\FC,M)$.
\end{enumerate}

\begin{cor}
Assume that $M$ is $\aleph_0$-saturated and strongly $\aleph_0$-homogeneous. Then the map 
    $$\rho_\ast|_{\mathfrak{M}_{\bar{m}}^{\sfs}(\FC,M)}:\mathfrak{M}_{\bar{m}}^{\sfs}(\FC,M)\to \mathcal{M}(\gal_{\KP}(T)),$$
    is surjective. In particular, the map
        $$\rho_\ast|_{\mathfrak{M}_{\bar{m}}^{\inv}(\FC,M)}:\mathfrak{M}_{\bar{m}}^{\inv}(\FC,M)\to \mathcal{M}(\gal_{\KP}(T)),$$
    is also surjective.
\end{cor}

We use the symbol $\circledast$ to denote the standard convolution product on $\mathcal{M}(\gal_{\KP}(T))$, which was recalled in the introduction for arbitrary locally compact groups. We will now argue that both $\rho$ and $\rho_{*}$ are homomorphisms of semigroups.

\begin{lemma}\label{lemma:pi.homomorphism.types}
    If $p(\bar{y}),q(\bar{y})\in S^{\inv}_{\bar{m}}(\FC,M)$ then,
    $$\rho(p\ast q)=\rho(p)\cdot\rho(q),$$
    and if $T$ is NIP, 
    $$\rho_{\ast}(\delta_p\ast \delta_q)=\rho_{\ast}(\delta_p)\circledast \rho_{\ast}(\delta_q).$$
\end{lemma}

\begin{proof}
Choose $\tau\in\aut(\FC')$ and $\sigma \in \aut(\FC'')$ so that  $q(\bar{y})=\tp(\tau(\bar{m})/\FC)$ and $p|_{\FC'}(\bar y)=\tp(\sigma(\bar m)/\FC')$, where $p|_{\FC'}$ is the unique $M$-invariant extension of $p$ to $\FC'$. Let $\tau'' \in \aut(\FC'')$ be any extension of $\tau$. By Proposition \ref{prop:formula.for.star.on.types},
    $$p\ast q=\tp(\tau''\sigma(\bar{m})/\FC).$$
Thus,
    \begin{align*}
        \rho(p\ast q) &= \rho\Big(\tp\big(\tau''\sigma(\bar{m})/\FC\big)\Big) = 
        \mathfrak{r}_{\FC''}\Big(\rho_{\FC''}\big(\tp\big(\tau''\sigma(\bar{m})/\FC''\big)\big)\Big) \\
        &= \mathfrak{r}_{\FC''}\Big( (\tau''\sigma)^{-1}/\autf_{\KP}(\FC'')\Big)=
        \mathfrak{r}_N\Big( \sigma^{-1}/\autf_{\KP}(\FC'') \cdot \tau''^{-1}/\autf_{\KP}(\FC'')\Big) \\
        &= \mathfrak{r}_{\FC''}\rho_{\FC''}\Big(\tp\big(\sigma(\bar{m})/\FC''\big)\Big)\,\cdot\,\mathfrak{r}_N\rho_{\FC''}\Big(\tp\big(\tau''(\bar{m})/\FC''\big)\Big) \\
        &= \rho\Big(\tp\big(\sigma(\bar{m})/\FC\big)\Big)\,\cdot\,
        \rho\Big(\tp\big(\tau(\bar{m})/\FC\big)\Big)=\rho(p)\cdot\rho(q),
    \end{align*}
where the second and sixth equation follows from the diagram at the end of Section \ref{subsection: Galois groups} applied to $\FC''$ in place of $\FC'$.

    Since $\delta_{p\ast q}=\delta_p\ast\delta_q$, we have that
    $$\rho_{\ast}(\delta_p\ast\delta_q)=\rho_{\ast}(\delta_{p\ast q})=\delta_{\rho(p\ast q)}.$$
    On the other hand, using the first part of the proposition, we have
    \begin{equation*} \rho_{\ast}(\delta_p)\circledast\rho_{\ast}(\delta_q)
    =\delta_{\rho(p)}\circledast\delta_{\rho(q)}
    =\delta_{\rho(p)\cdot\rho(q)}=\delta_{\rho(p\ast q)}.
    \end{equation*}
By the last two exposed lines, we conclude that $\rho_{\ast}(\delta_p\ast\delta_q)=\rho_{\ast}(\delta_p)\circledast\rho_{\ast}(\delta_q)$.
\end{proof}

\begin{theorem}\label{thm:pi_star.homomorphism}
    Assume that $T$ is NIP and $\mu,\nu\in\mathfrak{M}_{\bar{m}}^{\inv}(\FC,M)$.
    Then
    $$\rho_{\ast}(\mu\ast\nu)=\rho_{\ast}(\mu)\,\circledast\,\rho_{\ast}(\nu).$$
\end{theorem}

\begin{proof}
    The idea of this proof is similar to the one of Theorem 3.10 from \cite{Artem_Kyle2}. 
        It is enough to show that for every $f\in C(\gal_{\KP}(T))$ we have
    $$\int_{\gal_{\KP}(T)}  f\, d \rho_{\ast}(\mu\ast\nu) = \int_{\gal_{\KP}(T)} f\, d \big( \rho_{\ast}(\mu)\circledast\rho_{\ast}(\nu)\big).$$
    Fix $\epsilon>0$. Note that $f\circ\rho:S_{\bar{m}}(\FC)\to\Rr$ is a continuous function so there exist formulas $\{\psi_i(\bar{b};\bar{y})\}_{i\leqslant n}$, where $\bar{b}\in \FC^{\bar{x}}$ and real numbers $r_1,\ldots,r_n$ such that
        $$\sup\limits_{q(\bar{y})\in S_{\bar{m}}(\FC)} |(f\circ\rho)(q)\;-\;\sum\limits_{i\leqslant n} r_i\mathbf{1}_{[\psi_i(\bar{b};\bar{y})]}(q)|<\epsilon.$$
    We observe that for every $\sigma\in\aut(\FC)$ the following holds:
        $$ \sup\limits_{q(\bar{y})\in S_{\bar{m}}(\FC)} |(f\circ\rho)(\sigma(q))\;-\;\sum\limits_{i\leqslant n} r_i\mathbf{1}_{[\psi_i(\bar{b};\bar{y})]}(\sigma(q))|<\epsilon.$$
    We introduce an auxiliary function $H\colon\gal_{\KP}(T)\to\Rr$, given by
    $$H(a):=\int\limits_{b\in\gal_{\KP}(T)}f(b\cdot a)d\rho_{\ast}(\mu).$$
    Note that $H$ is continuous by Fact 417.B in \cite{Fremlin}.  By the definition of convolution, we see that
    \begin{align*}
    \int\limits_{\gal_{\KP}(T)}  f\,d\big(\rho_{\ast}(\mu)\circledast\rho_{\ast}(\nu)\big)
    &= \int\limits_{a\in\gal_{\KP}(T)} \int\limits_{b\in\gal_{\KP}(T)} f(b\cdot a)\,d\rho_{\ast}(\mu)_b \, d\rho_{\ast}(\nu)_a \\
    &= \int\limits_{a\in\gal_{\KP}(T)} H(a)\,d\rho_{\ast}(\nu).
    \end{align*}
    Before going into the main computation, we prove an approximation lemma. 
    
    \begin{clm}
    For any $p(\bar{y})\in S_{\bar{m}}^{\inv}(\FC,M)$, we have 
    $$(H\circ\rho)(p)\approx_{\epsilon}\sum\limits_{i\leqslant n}r_i\big(F^{\psi_i^{\textrm{opp}}(\bar{y};\bar{x})}_\mu \circ h_{\bar{b}} \big)(p).$$
    \end{clm}

    \begin{clmproof}
Choose $\sigma\in\aut(\FC)$ such that $\sigma(\bar{m})\models p|_{M\bar{b}}$. Then $\rho(p)=\sigma^{-1}/\autf_{\KP}(\FC) = \rho(\tp(\sigma(\bar m)/\FC))$. 
Let $\sigma' \in \aut(\FC')$ be any extension of $\sigma$.
Notice

    \begin{align*}
    (H\circ\rho)(p) &= \int\limits_{b\in\gal_{KP}(T)} f\big(b\cdot\rho(p)\big)\,d\rho_{\ast}(\mu) & & \\
    &= \int\limits_{q\in S_{\bar{m}(\FC)}}f\big(\rho(q)\cdot\rho(p)\big)\,d\mu & & \\
    &= \int\limits_{q\in S^{\inv}_{\bar{m}(\FC,M)}}f\big(\rho(q)\cdot\rho(p)\big)\,d\mu &  & \supp(\mu)\subseteq S_{\bar{m}}^{\inv}(\FC,M) \\
    &= \int\limits_{q\in S^{\inv}_{\bar{m}(\FC,M)}}f\big(\rho(q)\cdot\rho(\tp(\sigma(\bar{m})/\FC)\big)\,d\mu &  & \\
    &= \int\limits_{q\in S^{\inv}_{\bar{m}(\FC,M)}}f\big(\rho(q\ast\tp(\sigma(\bar{m})/\FC))\big)\,d\mu &  & \text{by Lemma \ref{lemma:pi.homomorphism.types}} \\
    &= \int\limits_{q\in S^{\inv}_{\bar{m}(\FC,M)}}(f\circ\rho)\Big(\big(\sigma'(q|_{\FC'})\big)|_{\FC} \Big)\,d\mu &  & \text{by Proposition \ref{prop:formula.for.star.on.types}} \\
    &= \int\limits_{q\in S^{\inv}_{\bar{m}(\FC,M)}}(f\circ\rho)\big(\sigma(q)\big)\,d\mu &  & \text{since }\sigma\in\aut(\FC) \\
    &\approx_{\epsilon}  \int\limits_{q\in S^{\inv}_{\bar{m}}(\FC,M)}\Big(\sum\limits_{i\leqslant n}r_i\mathbf{1}_{[\psi_i(\bar{b};\bar{y})]}\Big)(\sigma(q))\,d\mu & & \\
    &= \sum\limits_{i\leqslant n}r_i\int\limits_{q\in S^{\inv}_{\bar{m}}(\FC,M)}\mathbf{1}_{[\psi_i(\sigma^{-1}(\bar{b});\bar{y})]}(q)\,d\mu & & \\
    &= \sum\limits_{i\leqslant n}r_i \mu\big( \psi_i(\sigma^{-1}(\bar{b});\bar{y}) \big).
    \end{align*}
    Now, let $p=\tp(\tau(\bar{m})/\FC)$ for some $\tau\in\aut(\FC')$.
    Then $\tau(\bar{m})\equiv_{M\bar{b}}\sigma(\bar{m})$, so there exists $h\in\aut(\FC/M\bar{b})$ such that $h\sigma(\bar{m})=\tau(\bar{m})$.
    Thus, $\tau^{-1}h\sigma(\bar{m})=\bar{m}$ and we have that $\tau^{-1}h\sigma\in\aut(\FC/M)$. Note that
    $$\sigma^{-1}(\bar{b})\equiv_M (\tau^{-1}h\sigma)\big(\sigma^{-1}(\bar{b})\big)=\tau^{-1}(\bar{b}).$$
    Since $\mu$ is $M$-invariant,
there is a unique $M$-invariant extension $\hat{\mu} \in \mathfrak{M}^{\inv}_{\bar m}(\FC',M)$ of $\mu$, and we have
    \begin{align*}
        \sum\limits_{i\leqslant n}r_i \mu\big( \psi_i(\sigma^{-1}(\bar{b});\bar{y}) \big) &= \sum\limits_{i\leqslant n}r_i \hat{\mu}\big( \psi_i(\tau^{-1}(\bar{b});\bar{y}) \big)
        = \sum\limits_{i\leqslant n}r_i \,F^{\psi_i^{\textrm{opp}}(\bar{y};\bar{x})}_\mu\big( \tp(\tau^{-1}(\bar{b})/M)\big) \\
        &=  \sum\limits_{i\leqslant n}r_i \,\big(F^{\psi_i^{\textrm{opp}}(\bar{y};\bar{x})}_\mu\circ h_{\bar{b}}\big)\big( \tp(\tau(\bar{m})/\FC)\big) \\
        &= \sum\limits_{i\leqslant n}r_i \,\big(F^{\psi_i^{\textrm{opp}}(\bar{y};\bar{x})}_\mu\circ h_{\bar{b}}\big)(p). 
    \end{align*}
    \end{clmproof}
    
Using the above claim together with the earlier observations, we complete the proof of the theorem via the following computation: 
    \begin{align*}
        \int\limits_{\gal_{\KP}(T)}f\,d\rho_{\ast}(\mu\ast\nu) &= \int\limits_{S_{\bar{m}}(\FC)} ( f\circ\rho )\,d(\mu\ast\nu) \approx_{\epsilon}  \int\limits_{S_{\bar{m}}(\FC)}\sum\limits_{i\leqslant n}r_i\mathbf{1}_{[\psi_i(\bar{b};\bar{y})]}\,d(\mu\ast\nu) \\
        &= \sum\limits_{i\leqslant n}r_i\,(\mu\ast\nu)\big(\psi_i(\bar{b};\bar{y})\big) = 
\sum\limits_{i\leqslant n}r_i\,\big(\mu\otimes(h_{\bar{b}})_{\ast}(\nu)\big)\big(\psi_i(\bar{x};\bar{y})\big) \\
        &= \sum\limits_{i\leqslant n}r_i\int\limits_{S_{\bar{m}}(\FC)}\big(F^{\psi_i^{\textrm{opp}}(\bar{y};\bar{x})}_{\mu}\circ h_{\bar{b}}\big)\,d\nu 
=
        \int\limits_{S_{\bar{m}}(\FC)}\Big( \sum\limits_{i\leqslant n}r_i\,F^{\psi_i^{\textrm{opp}}(\bar{y};\bar{x})}_\mu\circ h_{\bar{b}} \Big)\,d\nu \\
        &= \int\limits_{S_{\bar{m}}^{\inv}(\FC,M)}\Big( \sum\limits_{i\leqslant n}r_i\,F^{\psi_i^{\textrm{opp}}(\bar{y};\bar{x})}_\mu\circ h_{\bar{b}} \Big)\,d\nu
        \approx_{\epsilon} \int\limits_{S_{\bar{m}}^{\inv}(\FC,M)}\big(H\circ\rho\big)\,d\nu \\
       &= \int\limits_{S_{\bar{m}}(\FC)}\big(H\circ\rho \big)\,d\nu = \int\limits_{\gal_{\KP}(T)}H\,d\rho_{\ast}(\nu) \\
        &= \int\limits_{\gal_{\KP}(T)} f\,d\big(\rho_{\ast}(\mu)\circledast\rho_{\ast}(\nu)\big)   
    \end{align*}
    Because $\epsilon>0$ was arbitrary, the desired statement holds. 
\end{proof}

\begin{remark}\label{rem: affine homomorhism 2}
Theorem 3.10 of \cite{Artem_Kyle2} follows from Theorems \ref{thm:pi_star.homomorphism} and \ref{thm: affine isomorhpism for measures}.
\end{remark}

\begin{proof}
Let us work in the set-up of Section \ref{sec: affine sort}. Notice that the map $\rho_{\bar{\FC}}$ restricted to $S_{\pi(\bar{x}\bar{y};\bar{n})}(\bar{\FC})$ looks as follows
$$S_{\pi(\bar{x}\bar{y};\bar{n})}(\bar{\FC})\ni p\mapsto p|_{\bar{M}}=\tp\Big((g,\id_{\FC})(\bar{n})/\bar{M}\Big)\mapsto(g^{-1},\id_{\FC})/\autf_{\KP}(\bar{\FC}),$$
for some $g\in G(\FC)$.
Under the isomorphism $\aut(\bar{\FC})/\autf_{\KP}(\bar{\FC}) \cong G(\FC)/G(\FC)^{00} \rtimes \aut(\FC)/\autf_{\KP}(\FC)$ from Fact \ref{rem: affine Lascar and KP}, $(g^{-1},\id_{\FC})/\autf_{\KP}(\bar{\FC})$ is identified with the pair $(g^{-1}/G(\FC)^{00}, \id_{\FC}/\autf_{\KP}(\FC))$. 
Thus, we can compose $\rho_{\bar{\FC}}|_{S_{\pi(\bar{x}\bar{y};\bar{n})}(\bar{\FC})}$
with the map $(g /G(\FC)^{00}, \id_{\FC}/\autf_{\KP}(\FC))\mapsto g /G(\FC)^{00}$ being an isomorphism of topological groups between the image of 
$\rho_{\bar{\FC}}|_{S_{\pi(\bar{x}\bar{y};\bar{n})}(\bar{\FC})}$ and $G(\FC)/G(\FC)^{00}$
(which is the compact group used in \cite{Artem_Kyle2} in place of $\aut(\FC)/\autf_{\KP}(\FC)$ in Theorem \ref{thm:pi_star.homomorphism} above).

Let $\tilde{\rho}_{\bar{\FC}}:p\mapsto g/G(\FC)^{00}$ be the aforementioned composition, and let $\hat{\pi} \colon S_G^{\inv}(\FC,M) \to G(\FC)/G(\FC)^{00}$ be the map described in Fact 3.1 of \cite{Artem_Kyle2}. Then the following diagram commutes (where $f$ is the main homeomorphism from Section \ref{sec: affine sort}):

\begin{center}
			\begin{tikzcd}
			S_G^{\inv}(\FC,M) \ar[dr,"\hat{\pi}"]\ar[rr,"f^{-1}"] & & S_{\pi(\bar{x}\bar{y};\bar{n})}^{\inv}(\bar{\FC},\bar{M})\ar[dl,"\tilde{\rho}_{\bar{\FC}}"]\\
			&G(\FC)/G(\FC)^{00} & 
			\end{tikzcd}
		\end{center}
This induces the commutative diagram of pushforwards:

\begin{center}
			\begin{tikzcd}
			\mathfrak{M}^{\inv}_G(\FC,M) \ar[dr,"\hat{\pi}_*"]\ar[rr,"(f^{-1})_*"] & &\mathfrak{M}^{\inv}_{\pi(\bar{x}\bar{y};\bar{n})}(\bar{\FC},\bar{M})\ar[dl,"(\tilde{\rho}_{\bar{\FC}})_*"]\\
			&\mathcal{M}(G(\FC)/G(\FC)^{00}) & 
			\end{tikzcd}
		\end{center}
Having this, Theorem 3.10 of \cite{Artem_Kyle2} follows from Theorems \ref{thm:pi_star.homomorphism} and \ref{thm: affine isomorhpism for measures}.
\end{proof}

Recall that the main goal of this section is to prove Conjecture \ref{conjecture: main conjecture} under the NIP hypothesis, assuming the measure in consideration is additionally $\autf_{\KP}(\FC)$-invariant. The next theorem is crucial for that. It is a counterpart of Theorem 4.11 from \cite{Artem_Kyle}.

Recall from Section \ref{subsection: Galois groups} that $\mathfrak{p}_{\KP}^{\FC}$ was the quotient map $\aut(\FC) \to \aut(\FC)/\autf_{\KP}(\FC)$. From now on, we will denote it by $\mathfrak{p}$. The corresponding map for $\FC'$ in place of $\FC$ will be denoted by $\mathfrak{p}_{\FC'}$.

\begin{theorem}\label{thm: old 8.8}
    If $T$ is NIP and $\mu\in\mathfrak{M}_{\bar{m}}^{\inv}(\FC,M)$ is an idempotent measure, then:
    \begin{enumerate}

\item $\supp(\rho_{\ast}(\mu))$ is a compact group and $(\rho_{\ast}(\mu))|_{\supp(\rho_{\ast}(\mu))}$ is precisely the normalized Haar measure on $\supp(\rho_{\ast}(\mu))$,

 \item $\mathfrak{p}^{-1}[\supp(\rho_{\ast}(\mu))]$ is a relatively $\bar{m}$-type definable subgroup of $\aut(\FC)$ which contains $\Stab(\mu)$,

        \item if additionally $\mu$ is $\autf_{\KP}(\FC)$-invariant, then
        $$\stab(\mu)=\mathfrak{p}^{-1}[\supp(\rho_{\ast}(\mu))].$$
    \end{enumerate}
\end{theorem}

\begin{proof}
Proof of (1): Since $T$ is NIP and $\mu\in\mathfrak{M}_{\bar{m}}^{\inv}(\FC,M)$ is an idempotent, by Theorem \ref{thm:pi_star.homomorphism}, we obtain that $\rho_{\ast}(\mu)\in\CM(\gal_{\KP}(T))$ is also idempotent. Thus, using the classical Fact \ref{fact: classical Pym}, we conclude that $\supp(\rho_{\ast}(\mu))$ is a compact group and $\rho_{*}(\mu)$ is precisely the normalized Haar measure on $\supp(\rho_{\ast}(\mu))$.

Proof of (2): By (1),  $\supp(\rho_{\ast}(\mu))$ is a subgroup of $\gal_{\KP}(T)$, so $\mathfrak{p}^{-1}[\supp(\rho_{\ast}(\mu))]$ is  a subgroup of $\aut(\FC)$, because $\mathfrak{p}$ is a group homomorphism. By (1), we also know that $\supp(\rho_{\ast}(\mu))$ is closed, which together with
the third diagram in Section \ref{subsection: Galois groups} applied for $N=M$ and the definition of the logic topology on $\gal_{\KP}(T)$ implies that $\mathfrak{p}^{-1}[\supp(\rho_{\ast}(\mu))]$ is relatively $\bar{m}$-type definable. It remains to prove that $\Stab(\mu) \subseteq \mathfrak{p}^{-1}[\supp(\rho_{\ast}(\mu))]$.

Let $\sigma\in\stab(\mu)$. Then also $\sigma^{-1}\in\stab(\mu)$.
    We want to show that  $\mathfrak{p}(\sigma)\in\supp(\rho_{\ast}(\mu))$.
    Consider an open neighborhood $U\subseteq\gal_{\KP}(T)$ of $\mathfrak{p}(\sigma)$. 

\begin{clm}
$\sigma\rho^{-1}[U]=\rho^{-1}[U \cdot \mathfrak{p}(\sigma)^{-1}]$.
\end{clm}

\begin{clmproof}
Let us prove $\subseteq$ (the opposite inclusion uses a similar computation). Consider any $q \in \sigma\rho^{-1}[U]$, i.e. $q=\sigma p$ for some $p \in \rho^{-1}[U]$. Write $p = \tp(\tau'(\bar m)/\FC)$ for some $\tau' \in \aut(\FC')$, and pick $\tau \in \aut(\FC)$ with $\tau(\bar m) \equiv_M \tau'(\bar m)$. Then $\rho(p)=\tau^{-1}/\autf_{\KP}(\FC) \in U$. On the other hand, choosing any extension $\sigma' \in \aut(\FC')$ of $\sigma$, we have $\sigma p =\tp(\sigma'\tau'(\bar m)/\FC)$ and $\sigma(\tau(\bar m)) \equiv_{\mathrm{Ls}} \sigma'(\tau'(\bar m))$, and therefore  $\rho (\sigma p) = (\sigma \tau)^{-1}/\autf_{\KP}(\FC)= \tau^{-1}/\autf_{\KP}(\FC) \cdot \sigma^{-1}/\autf_{\KP}(\FC)$. Thus,  $\rho (\sigma p) \in U \cdot \mathfrak{p}(\sigma)^{-1}$, as required.
\end{clmproof}
Using the fact that $\sigma^{-1}\in\stab(\mu)$ and the above claim, we have
   $$(\rho_{\ast}(\mu))(U)=\mu(\rho^{-1}[U])=\big((\sigma^{-1})_{\ast}(\mu)\big)(\rho^{-1}[U])=\mu(\sigma\rho^{-1}[U])=(\rho_{\ast}(\mu))(U\cdot \mathfrak{p}(\sigma)^{-1})>0, $$ 
 where the inequality holds as $\id/\autf_{\KP}(\FC)\in (U\cdot \mathfrak{p}(\sigma)^{-1}) \cap \supp(\rho_{\ast}(\mu))$.
    
    Proof of (3): By (2), it remains to prove $\supseteq$, i.e.
that for every $\tau \in \aut(\FC)$ such that $\mathfrak{p}(\tau)\in\supp(\rho_{\ast}(\mu))$ we have $\tau_{\ast}(\mu)=\mu$. 
Since $\Stab(\mu)$ is a subgroup, it is enough to show that $(\tau^{-1})_{\ast}(\mu)=\mu$.
    Consider $\varphi(\bar{x};\bar{y})\in\CL$, $\bar{b}\in N^{\bar{x}}$,
    and functions
    $$f:=(F^{\varphi^{\textrm{opp}}(\bar{y};\bar{x})}_{\mu}\circ h_{\tau(\bar{b})}):S_{\bar{m}}(\FC)\to \Rr,$$
    $$h:=(F^{\varphi^{\textrm{opp}}(\bar{y};\bar{x})}_{\mu}\circ h_{\bar{b}}):S_{\bar{m}}(\FC)\to \Rr.$$
    As $\mu$ is  $\autf_{\KP}(\FC)$-invariant,
    both $f$ and $h$ factor thorough $\rho\colon S_{\bar{m}}(\FC)\to \gal_{\KP}(T)$,
    via the functions $\hat{f},\hat{h}\colon \gal_{\KP}(T)\to\Rr$ given by
    $$\hat{f}(\theta/\autf_{\KP}(\FC)):=f\big(\tp(\theta^{-1}(\bar{m})/\FC)\big),$$
    $$\hat{h}(\theta/\autf_{\KP}(\FC)):=h\big(\tp(\theta^{-1}(\bar{m})/\FC)\big).$$
    Moreover, for every $\theta/\autf_{\KP}(\FC)\in\gal_{\KP}(T)$ we have
    $$\hat{f}(\theta/\autf_{\KP}(\FC))=\mu(\varphi(\theta\tau(\bar{b});\bar{y}))=
    \hat{h}(\theta/\autf_{\KP}(\FC)\cdot\tau/\autf_{\KP}(\FC)).$$
    Therefore, we can compute
    \begin{align*}
   (\tau^{-1})_{\ast}(\mu)(\varphi(\bar b;\bar y)) &=  \mu\big(\varphi(\tau(\bar{b});\bar{y})\big) = (\mu\ast\mu)\big( \varphi(\tau(\bar{b});\bar{y}) \big) = \int\limits_{S_{\bar{m}}(\FC)}f\,d\mu \\     
	 &= \int\limits_{S_{\bar{m}}(\FC)}\hat{f}\circ\rho\,d\mu = \int\limits_{\gal_{\KP}(T)} \hat{f}(g)\,d(\rho_{\ast}(\mu))(g) \\
        &= \int\limits_{\gal_{\KP}(T)} \hat{h}\big(g\cdot\mathfrak{p}(\tau)\big)\,d(\rho_{\ast}(\mu))(g)  \stackrel{(\spadesuit)}{=} 
        \int\limits_{\gal_{\KP}(T)} \hat{h}(g)\,d(\rho_{\ast}(\mu))(g)  \\
        &= \int\limits_{S_{\bar{m}}(\FC)}\hat{h}\circ\rho\,d\mu = 
        \int\limits_{S_{\bar{m}}(\FC)}h\,d\mu \\
        &= (\mu\ast\mu)\big( \varphi(\bar{b};\bar{y}) \big)
        = \mu\big(\varphi(\bar{b};\bar{y})\big),
    \end{align*}
where $(\spadesuit)$ follows by the assumption that $\mathfrak{p}(\tau)\in\supp(\rho_{\ast}(\mu))$ and by (1).
Since $\varphi(\bar{b};\bar{y})$ was arbitrary, we conclude that 
$(\tau^{-1})_{\ast}(\mu)=\mu$
\end{proof}

We now prove Conjecture \ref{conjecture: main conjecture} for $\autf_{\KP}(\FC)$-invariant measures under NIP.

\begin{theorem}\label{thm: 0.7 for NIP} Suppose $T$ is NIP. Let $\mu \in \mathfrak{M}^{\inv}_{\bar m}(\FC,M)$ be
$\autf_{\KP}(\FC)$-invariant.
Then $\stab(\mu)$ is relatively $\bar{m}$-definable over $M$,
say $\stab(\mu)=G_{\pi,\FC}$ for $\pi(\bar{x};\bar{y})$ being a partial type over $\emptyset$ which contains ``$\bar{x}\equiv_{\emptyset}\bar{y}$''.
Then the following are equivalent:
\begin{enumerate}
\item $\mu$ is an idempotent,
\item $\mu$ is a (left) $G_{\pi,\FC}$-invariant measure in $\mathfrak{M}^{\inv}_{\pi(\bar{m};\bar{y})}(\FC,M)$.
\end{enumerate}
If $\mu$ is additionally fim over $M$, then (2) is equivalent
to
\begin{enumerate}
\item[(2')] $\mu$ is the unique (left) $G_{\pi,\FC}$-invariant measure in $\mathfrak{M}^{\inv}_{\pi(\bar{m};\bar{y})}(\FC,M)$.
\end{enumerate}
Thus, Conjecture \ref{conjecture: main conjecture} holds under NIP and the additional assumption that $\mu$ is $\autf_{\KP}(\FC)$-invariant.
\end{theorem}

\begin{proof}
By Theorem \ref{thm: old 8.8}, $\stab(\mu)=\mathfrak{p}_{\FC}^{-1}[\supp((\rho_{\FC})_\ast(\mu))]$ is a relatively $\bar{m}$-type definable subgroup of $\aut(\FC)$.
Hence, we can indeed write $\stab(\mu)=G_{\pi,\FC}$  for $\pi$ as in the statement.

The implication $(2)\Rightarrow(1)$ holds by Corollary \ref{cor: two to one}.
We prove the implication $(1)\Rightarrow(2)$.

    Note that $\mu$ is (left) $G_{\pi,\FC}$-invariant. 
    Thus, it suffices to show that $\mu([\pi(\bar{m};\bar{y})])=1$.
By assumption, $\autf_{\KP}(\FC)\leqslant G_{\pi,\FC}$. Using compactness (or rather saturation of $\FC$) and the fact that an automorphism $\sigma \in \aut(\FC)$ belongs to $\autf_{\KP}(\FC)$ if and only if $\sigma(\bar m) \equiv_{\KP} \bar m$, this implies that $\autf_{\KP}(\FC')\leqslant G_{\pi,\FC'}$.

To avoid confusion, in the computations below we will emphasize the monster models over which we are working by writing them as subscripts of $\rho$ and $\mathfrak{p}$.

\begin{clm}
    $\rho_{\FC}^{-1}\rho_{\FC}[\pi(\bar{m};\bar{y})]=[\pi(\bar{m};\bar{y})]$
    (as subsets of $S_{\bar{m}}(\FC)$).
\end{clm}

\begin{clmproof}
    Let $p\in S_{\bar{m}}(\FC)$, and let $\sigma\in\aut(\FC')$ be such that $p=\tp(\sigma(\bar{m})/\FC)$. 
Then:
    \begin{align*}
        &\tp(\sigma(\bar{m})/\FC)\in \rho_{\FC}^{-1}\rho_{\FC}\big([\pi(\bar{m};\bar{y})]\big) \\
        &\iff \rho_{\FC}\big(\tp(\sigma(\bar{m})/\FC)\big)\in\rho_{\FC}
        \big([\pi(\bar{m};\bar{y})]\big) \\
        & \stackrel{(1)}{\iff} \mathfrak{r}_{\FC}\rho_{\FC'}\big(\tp(\sigma(\bar{m})/\FC')\big)
        \in\rho_{\FC}\big([\pi(\bar{m};\bar{y})]\big) \\
        &\iff (\exists\tau\in G_{\pi,\FC'})\Big(\mathfrak{r}_{\FC}\rho_{\FC'}\big(\tp\big(\sigma(\bar{m})/\FC'\big)\big)=\rho_{\FC}\big(\tp \big(\tau(\bar{m})/\FC \big) \big) \Big) \\
        &\stackrel{(2)}{\iff} (\exists\tau\in G_{\pi,\FC'})\Big(\mathfrak{r}_{\FC}\rho_{\FC'}\big(\tp\big(\sigma(\bar{m})/\FC'\big)\big)=\mathfrak{r}_{\FC}\rho_{\FC'}\big(\tp \big(\tau(\bar{m})/\FC' \big) \big) \Big) \\
        &\stackrel{(3)}{\iff} (\exists\tau\in G_{\pi,\FC'})\Big(\rho_{\FC'}\big(\tp\big(\sigma(\bar{m})/\FC'\big)\big)=\rho_{\FC'}\big(\tp \big(\tau(\bar{m})/\FC' \big) \big) \Big) \\
        &\iff (\exists\tau\in G_{\pi,\FC'})\big( \sigma\in\tau\autf_{\KP}(\FC') \big) \\
        &\stackrel{(4)}{\iff} \sigma\in G_{\pi,\FC'}\cdot\autf_{\KP}(\FC')=G_{\pi,\FC'}
        \;\iff\; \tp(\sigma(\bar{m})/\FC)\in [\pi(\bar{m};\bar{y})], 
    \end{align*}
where (1) and (2) follow from the last diagram in Section \ref{subsection: Galois groups}, (3) from injectivity of $\mathfrak{r}_{\FC}$, and (4) from the above observation that  $\autf_{\KP}(\FC')\leqslant G_{\pi,\FC'}$.
\end{clmproof}

\begin{clm}
    $\rho_{\FC}\big([\pi(\bar{m};\bar{y})]\big)=\supp((\rho_\FC)_\ast(\mu)).$
\end{clm}

\begin{clmproof}
    Let $\sigma\in\aut(\FC)$. 
By the explicit formula formula for $\rho_{\FC}$ (given before the last diagram in Section \ref{subsection: Galois groups}), we have that 
    $$\sigma/\autf_{\KP}(\FC)\in \rho_{\FC}\big([\pi(\bar{m};\bar{y})]\big)\;\iff\;
    (\exists\tau\in G_{\pi,\FC})\big(\sigma/\autf_{\KP}(\FC)=\tau/\autf_{\KP}(\FC) \big).$$
 On the other hand,   
by the first paragraph of the proof of Theorem \ref{thm: 0.7 for NIP}, we have $G_{\pi,\FC}=\mathfrak{p}^{-1}_{\FC}[\supp((\rho_{\FC})_\ast(\mu))]$.
Thus,  
     $$\sigma/\autf_{\KP}(\FC)\in \rho_{\FC}\big([\pi(\bar{m};\bar{y})]\big)\;\iff\;
    \sigma/\autf_{\KP}(\FC)\in  \supp((\rho_{\FC})_\ast(\mu)).$$
\end{clmproof}
\noindent
Using both claims, we can compute
\begin{align*}
    \mu\big([\pi(\bar{m};\bar{y})]\big) &= \mu\big(\rho_{\FC}^{-1}\rho_{\FC}[\pi(\bar{m};\bar{y})]\big)
    =((\rho_{\FC})_\ast(\mu))\big(\rho_{\FC}[\pi(\bar{m};\bar{y})]\big)\\
    &= ((\rho_{\FC})_\ast(\mu))\big(\supp((\rho_{\FC})_\ast(\mu))\big) =1.
\end{align*}
\noindent
Finally, $(2)\iff(2')$ follows by Corollary \ref{cor: uniqueness for measures}(1).
\end{proof}

\begin{cor}\label{cor: NIP conj for definable}
(NIP)
    Let $G$ be a $\emptyset$-definable group, let $\mu\in\mathfrak{M}^{\inv}_G(\FC,M)$ be $G(\FC)^{00}$-invariant.  
Then the right stabilizer $\stab(\mu)$ is a type definable subgroup of $G(\FC)$, say 
    $\stab(\mu)=H(\FC)$.
    Then the following are equivalent:
    \begin{enumerate}
        \item $\mu$ is idempotent,
        \item $\mu$ is a (right) $\stab(\mu)$-invariant measure in $\mathfrak{M}^{\inv}_H(\FC,M)$.
    \end{enumerate}
If $\mu$ is additionally fim over $M$, then (2) is equivalent
to
\begin{enumerate}
\item[(2')] $\mu$ is the unique (right) $\stab(\mu)$-invariant measure in $\mathfrak{M}^{\inv}_H(\FC,M)$.
\end{enumerate}
Thus, Conjecture \ref{conjecture: main conjecture for definable groups} holds under NIP and the additional assumption that $\mu$ is $G(\FC)^{00}$-invariant.
\end{cor}

\begin{proof} 
We use the notation of Section \ref{sec: affine sort}.
By Fact \ref{rem: affine Lascar and KP} and Remark \ref{rem:aff01}, the assumption that $\mu$ is right $G(\FC)^{00}$-invariant implies that $f_\ast^{-1}(\mu)$ is (left) $\autf_{\KP}(\bar{\FC})$-invariant. Since $\mu\in\mathfrak{M}^{\inv}_G(\FC,M)$ [is fim, resp.], Corollary \ref{corollary: preservation for infinite tuples} implies that $f_\ast^{-1}(\mu) \in \mathfrak{M}^{\inv}_{\bar n}(\bar{\FC},\bar{M})$ [is fim]. Therefore, the assumptions of Theorem \ref{thm: 0.7 for NIP} are satisfied for the measure $f_\ast^{-1}(\mu)$, and so the conclusion also holds. Then the conclusion of Theorem \ref{cor: NIP conj for definable} follows as explained in the paragraph after Conjecture \ref{conjecture: main conjecture} at the beginning of Section \ref{sec:fim idempotent}. Only type-definability of $\stab(\mu)$ was not explained there, as it was assumed to be known (by \cite[Proposition 5.3]{Artem_Kyle}). But it also follows from relative $\bar n$-type-definability over $\bar{M}$ of the left stabilizer of $f_\ast^{-1}(\mu)$ (using Proposition \ref{proposition:unique.invariant.transfer}(2) and Remark \ref{rem:aff.f.homeo}).
\end{proof}

\section{Newelski's group chunk theorem for automorphisms}\label{sec:group chunk}
In this section, we generalize portions of stable group theory to the context of automorphism groups. This section is vital in proving the main conjecture in the context of stable theories, which we will prove 
in Section \ref{sec: 0.7 for stable}. The main theorem in this section is a counterpart of Newelski's Group Chunk Theorem for groups of automorphisms.

Let $T$ be a complete first-order theory and $\FC \models T$ its monster model. Let $\bar c$ be an enumeration of $\FC$. Let $\bar x$ be a tuple of variables corresponding to $\bar c$. 
Usually we use ``$\bar{x}$'' to denote a tuple of variables corresponding to 
an enumeration of small model $M$. However, in Subsections \ref{section: stable groups for Aut(C)} and \ref{subsec: group chunk 2} we change the convention and use ``$\bar{x}$'' to denote variables corresponding to
the enumeration $\bar{c}$ of the monster model $\FC$ and ``$\bar{x}'$'' for variables corresponding to 
the enumeration $\bar{m}$ of the small model $M$ (with $\bar{x}'\subseteq\bar{x}$ and $\bar{m}\subseteq\bar{c}$). In Subsection \ref{subsec: group chunk 3}, we come back to the usual meaning of ``$\bar{x}$''.
By $\FC' \succeq \FC$ we will denote a bigger monster model in which $\FC$ is small.

Let $\pi(\bar{x}';\bar{y}')$ be a partial type over $\emptyset$ containing ``$\bar{x}'\equiv_{\emptyset}\bar{y}'$'' and such that 
$$G_{\FC}:=G_{\pi,\FC}=\{\sigma\in\aut(\FC)\;\colon\;\models \pi(\sigma(\bar{m});\bar{m})\}$$
is a subgroup of $\aut(\FC)$. Define also
$$G_{\varphi,\FC}:=\{\sigma\in G_{\FC}\;\colon\;\models \varphi (\sigma(\bar{c});\bar{a})\},$$ where $\varphi(\bar{x};\bar{a})$ is any formula with parameters $\bar a$ from $\FC$.

We also set:
$$\widetilde{G}_{\FC}:=\widetilde{G}_{\pi,\FC}=[\pi(\bar{x}';\bar{m})]\cap S_{\bar{c}}(\FC)=\{p(\bar x) \in S_{\bar c}(\FC): \pi(\bar x';\bar m) \subseteq p(\bar x)\},$$
$$\widetilde{G}:=S_{\pi(\bar{x}';\bar{m})}(M)=\{p(\bar x') \in S_{\bar m}(M): \pi(\bar x';\bar m) \subseteq p(\bar x')\}.$$

\subsection{Generics in $G_{\pi,\FC}$}\label{section: stable groups for Aut(C)}
In the following material, the terminology and notation from the paragraph preceding Remark \ref{remark: trivial remark} is used.

\begin{definition}
We say that a relatively definable subset
$D$ of $G_\FC$ is {\em left [right] generic} if $G_{\FC}$ is covered by finitely many left [right] translates of $D$ by elements of $G_\FC$. A filter on the Boolean algebra $\Df(G_{\FC})$ of relatively definable subsets of $G_{\FC}$ is {\em left [right] generic} if every set in this algebra is left [right] generic.
\end{definition}

\begin{remark}\label{remark: generic in various models}
For a fixed formula $\varphi(\bar x; \bar a)$, where $\bar a$ is a finite tuple from $\FC$, for any monster model $\mathfrak{D} \succeq \FC$, we have that $G_{\varphi,\FC}$ is left [right] generic subset of $G_{\FC}$ if and only if $G_{\varphi,\mathfrak{D}}$ is a left [right] generic subset of $G_{\mathfrak{D}}$. 
\end{remark}

\begin{proof}
In fact, $\varphi(\bar x;\bar a)$ uses only a finite tuple of variables $\bar x''$. Let $\bar c''$ be the corresponding subtuple of $\bar c$.

 The conclusion for the left version follows from the observation that  the condition that $G_{\varphi,\FC}$ is left generic means precisely that there are $\bar a_1, \dots, \bar a_n \in G_{\FC} \cdot \bar a$ (equivalently, for every $i$, $\models \exists \bar x'( \pi(\bar x';\bar m) \wedge \bar m \bar a \equiv \bar x' \bar a_i$))  such that the type $\exists \bar x'(\pi(\bar x';\bar m) \wedge \bar m \bar c'' \equiv \bar x' \bar x'')$ implies the formula $\varphi(\bar x'';\bar a_1) \lor \dots \lor \varphi(\bar x'';\bar a_n)$.

 The conclusion for the right version follows from the observation that  the condition that $G_{\varphi,\FC}$ is right generic means precisely that  there are $\bar c_1'', \dots, \bar c_n'' \in \bar  G_{\FC} \cdot \bar c''$ such that the type $\exists \bar x'(\pi(\bar x';\bar m)  \wedge \bar m \bar c_1''\dots \bar c_n'' \equiv \bar x' \bar x_1'' \dots \bar x_n'')$ implies the formula  $\varphi(\bar x_1'';\bar a) \lor \dots \lor \varphi(\bar x_n'';\bar a)$.
\end{proof}

\begin{proposition}\label{proposition: existence of generics}
Assume $T$ is stable. Then for every formula $\varphi(\bar x;\bar a)$ either $G_{\varphi,\FC}$ or its complement in $G_{\FC}$ (which equals $G_{\neg \varphi,\FC}$) is left [right] generic. Thus, non left [right] generic sets in $\Df(G_{\FC})$ form an ideal, and so each left [right] generic filter on $\Df(G_{\FC})$ extends to a left [right] generic ultrafilter; in particular, a left [right] generic ultrafilter on $\Df(G_{\FC})$ exists.
\end{proposition}

\begin{proof}
We argue similarly to the case of type-definable groups from \cite[Chapter 1, Lemma 6.4]{anandgeometric}.
Let us focus on the left version (the right version follows by a symmetric argument). Let $C:=(G_1,G_2,R)$ (a 2-sorted structure with sorts $G_1$ and $G_2$ being copies of $G_{\FC}$ equipped with no structure), where $C \models R(\sigma,\tau)$ if $\sigma \in \tau \cdot G_{\varphi,\FC}$ (for $\sigma \in G_1$ and $\tau \in G_2$). Let $T_C:=\Th(C)$.

\begin{clm}
$R(x,y)$ is stable in $T_C$. 
\end{clm}

\begin{clmproof}
If not, then for every $n<\omega$ there exist $(\sigma_i,\tau_i)_{i \leq n}$ from $G_{\FC}$ such that $\sigma_j \in \tau_i \cdot G_{\varphi,\FC} \iff i \leqslant j$. As the condition $\sigma_j \in \tau_i \cdot G_{\varphi,\FC}$ is equivalent to $\models \varphi(\sigma_j(\bar c);\tau_i(\bar a))$, we conclude that $\varphi(\bar x; \bar y)$ is unstable, a contradiction.
\end{clmproof}

The rest of the proof is the same as in \cite[Chapter 1, Lemma 6.4]{anandgeometric}, but we give some details for the readers convenience. Note that for any $\sigma \in G_{\FC}$ (which we identify with the corresponding elements in $G_1$ and $G_2$) we have an automorphism $F_\sigma$ of $C$ given by $F_\sigma(g):=\sigma \cdot g$ and $F_\sigma(h):=\sigma \cdot h$ for any $g \in G_1$ and $h \in G_2$. Hence, $\aut(C)$ acts transitively on $G_1$ and $G_2$, and so $|S^{T_C}_{G_1}(\emptyset)|=1$ and  $|S^{T_C}_{G_2}(\emptyset)|=1$. Note also that $R(C,\sigma)= \sigma \cdot G_{\varphi,\FC}$, in particular $R(C,e)= G_{\varphi,\FC}$. 

\begin{clm}
$G_{\varphi,\FC}$ is left generic if and only if $R(x,e)$ does not fork over $\emptyset$ in $T_C$.
\end{clm}

\begin{clmproof}
By Claim 1, we get the following sequence of equivalences:
$R(x,e)$ does not fork over the $\emptyset$ in $T_C$ if and only if some positive Boolean combination of conjugates of $R(x,e)$ in the sense of $T_C$ is consistent and definable over $\emptyset$, 
if and only if there exist $\tau_1,\dots,\tau_n \in G_2$ such that $R(C,\tau_1) \cup \dots \cup R(C,\tau_n) = G_1$, if and only if there exist $\tau_1,\dots,\tau_n \in  G_{\FC}$ such that $ \tau_1 \cdot G_{\varphi,\FC} \cup \dots \cup  \tau_n \cdot G_{\varphi,\FC}=  G_{\FC}$. 
\end{clmproof}

Note that $G_{\neg \varphi,\FC} =  G_{\FC} \setminus G_{\varphi,\FC}$. By the proof of Claim 2 applied to $\neg\varphi(\bar x; \bar a)$ and $\neg R(x,y)$, we get that $G_{\neg \varphi,\FC}$ is left generic if and only if $\neg R(x,e)$ does not fork $\emptyset$ in $T_C$. Since by Claim 1 either $R(x,e)$ or $\neg R(x,e)$ does not fork over $\emptyset$, we conclude that either $G_{\varphi,\FC}$ or its complement in $G_\FC$ is left generic.

The remaining part of the proposition follows from the first part in a standard way.
\end{proof}

We will say that $G_{\varphi,\FC}$ is {\em two-sided generic} if there are $\sigma_0,\dots,\sigma_{n-1}, \tau_0,\dots,\tau_{m-1} \in G_\FC$ such that $G_\FC = \bigcup_{i<n,j<m}\sigma_i \cdot G_{\varphi,\FC} \cdot_r \tau_j$.

\begin{cor}\label{corollary: left=right}
Assume $T$ is stable. Then the following conditions are equivalent for a given formula $\varphi(\bar x;\bar a)$ with parameters $\bar a$ from $\FC$.
\begin{enumerate}
\item $G_{\varphi,\FC}$ is two-sided generic.
\item $G_{\varphi,\FC}$ is left generic.
\item $G_{\varphi,\FC}$ is right generic.
\end{enumerate}
\end{cor}

\begin{proof}
(2) $\Rightarrow$ (1) is trivial. 

(1) $\Rightarrow$ (2). By (1), $G_\FC = \bigcup_{i<n,j<m}\sigma_i \cdot G_{\varphi,\FC} \cdot_r \tau_j$ for some $\sigma_i,\tau_j \in G_\FC$. By Proposition \ref{proposition: existence of generics}, there exists a left generic ultrafilter $D$ on the Boolean algebra $\Df(G_\FC)$. Then there are $i<n$ and $j<m$ such that $\sigma_i G_{\varphi,\FC} \cdot_r \tau_j \in D$, and so $\rho_0 \cdot(\sigma_i \cdot G_{\varphi,\FC} \cdot_r \tau_j) \cup \dots \cup \rho_{k-1} \cdot(\sigma_i \cdot G_{\varphi,\FC} \cdot_r \tau_j) = G_\FC$ for some $\rho_0,\dots,\rho_{k-1} \in G_\FC$. Hence, $(\rho_0\sigma_i)\cdot G_{\varphi,\FC}\cup \dots \cup (\rho_{k-1}\sigma_i)\cdot G_{\varphi,\FC} =G_{\FC}$, i.e. $G_{\varphi,\FC}$ is left generic.

The proof of (1) $\iff$ (3) is similar.
\end{proof}

We should remark here that Proposition \ref{proposition: existence of generics} and Corollary \ref{corollary: left=right} alternatively follow from Theorem 4.7 of \cite{Conant}. In order to see that, observe that 
each member of $\Def(G_{\FC})$ is stable in the sense of Definition 4.1 of \cite{Conant} which follows from the proof of Claim 1 in the proof of Proposition \ref{proposition: existence of generics} above. For the reader's convenience, we decided to include complete proofs following a classical stability theory approach. Further results in this subsection (in particular those involving $\Delta$-ranks) do not follow from \cite{Conant}).

Because of the above corollary, one can speak about \emph{genericity} of relatively definable subsets of $G_\FC$ in the stable context without specifying \emph{left}, \emph{right}, or \emph{two-sided}. Hence we may forgo using these adjectives in this context. 

Let $S(\Df(G_{\FC}))$ be the Stone space of $\Df(G_{\FC})$. Then $\cdot$ and $\cdot_r$ naturally induce left and right actions of $G_\FC$ by homeomorphisms on $S(\Df(G_{\FC}))$, which we denote by $\odot$ and $\odot_r$, respectively. With these actions, $S(\Df(G_{\FC}))$ becomes a left and right $G_\FC$-flow.

Recall that we also have two natural continuous actions of $\aut(\FC)$ on $S_{\bar c}(\FC)$, which (abusing notation) we also denote by $\cdot$ and $\cdot_r$, namely:
$$\sigma \cdot p:= \sigma(p)=\{\varphi(\bar x;\sigma(\bar a))\;\colon\;\varphi(\bar x;\bar a) \in p\} = \tp(\sigma'(\tau'(\bar c))/\FC),$$
$$p\cdot_r \sigma:= \tp(\tau'(\sigma(\bar c))/\FC),$$ 
where $\sigma \subseteq \sigma' \in \aut(\FC')$ and $\tau' \in \aut(\FC')$ is such that $\tp(\tau'(\bar c)/\FC)=p$.
Recall that 
$$\widetilde{G}_{\FC}:=\widetilde{G}_{\pi,\FC}=[\pi(\bar{x}';\bar{m})]\cap S_{\bar{c}}(\FC)=\{p(\bar x) \in S_{\bar c}(\FC): \pi(\bar x';\bar m) \subseteq p(\bar x)\}.$$

\begin{remark}\label{remark: restricted actions}
$\cdot$ [resp. $\cdot_r$] restricts to a left [resp. right] action of $G_{\FC}$ on $\widetilde{G}_\FC$. This restricted action will be still denoted by $\cdot$ [resp. $\cdot_r$].
\end{remark}

\begin{proof}
This follows from Remark \ref{remark: group in every model}.
\end{proof}
\noindent
Thus, these actions turn $\widetilde{G}_{\FC}$ into a left and right  $G_{\FC}$-flow.

\begin{remark}\label{remark: identification of flows}
The map $f \colon \widetilde{G}_{\FC} \to S(\Df(G_{\FC}))$ given by $f(p):= \{G_{\varphi,\FC}: \varphi(\bar x;\bar a) \in p\}$ is a well-defined isomorphism of left and right $G_\FC$-flows.
\end{remark}

\begin{proof}
This follows easily using Remark \ref{remark: two descriptions of the type space}.
\end{proof}

An element $p$ of $\widetilde{G}_{\FC}$ is said to be {\em left [right, or two-sided] generic} if any neighborhood $U$ of $p$ is {\em left [right, or two-sided] generic} which means that finitely many left [right, or two-sided] translates of $U$ under $G_{\FC}$ cover $\widetilde{G}_{\FC}$. By Proposition \ref{proposition: existence of generics}, Corollary \ref{corollary: left=right}, and Remark \ref{remark: identification of flows}, we get

\begin{cor}\label{corollary: basic properties of generics} Assume that $T$ is stable.
\begin{enumerate}
\item There exists a left generic type $p \in \widetilde{G}_{\FC}$. 
Moreover, every collection of left generic clopens in $\widetilde{G}_\FC$ which is closed under finite intersections extends to a left generic type in $p \in \widetilde{G}_{\FC}$.
\item For every clopen subset $X$ of $\widetilde{G}_{\FC}$ the following are equivalent:
\begin{enumerate}
\item $X$ is two-sided generic;
\item $X$ is left generic;
\item $X$ is right generic.
\end{enumerate}
\item For every $p \in \widetilde{G}_{\FC}$ the following are equivalent:
\begin{enumerate}
\item $p$ is two-sided generic;
\item $p$ is left generic;
\item $p$ is right generic.
\end{enumerate}
\end{enumerate}
\end{cor}

Thus, speaking about generic types in $ \widetilde{G}_{\FC}$ in the stable context, we will be skipping the adjective left, right, or two-sided. 

For the rest of this subsection, let us assume that {\bf $T$ is stable}.
Denote by $\Gen(\widetilde{G}_{\FC})$ the set of all generic types in  $\widetilde{G}_{\FC}$. It is clearly closed, and in fact a left and right $G_\FC$-subflow of $ \widetilde{G}_{\FC}$.
We are going to present several characterizations of when $p \in \widetilde{G}_{\FC}$, analogous to the definable group case. 

For a finite family $\Delta=\Delta(\bar x;\bar y)$ of formulas in variables $\bar x,\bar y$, by a {\em $\Delta$-formula} we mean a formula equivalent to a Boolean combination of instances of some formulas from $\Delta$. Recall that the {\em rank $R_\Delta$} is a unique function from the collection of all consistent formulas $\varphi(\bar x)$ with parameters from $\FC$ to $\Ord \cup \{\infty\}$ satisfying: $R_\Delta(\varphi(\bar x)) \geq \alpha+1$ if and only if there exist pairwise inconsistent $\Delta$-formulas $(\psi_i(\bar x))_{i<\omega}$ (with parameters from $\FC$) with $R_\Delta(\varphi(\bar x) \wedge \psi_i(\bar x)) \geq \alpha$ for all $i<\omega$.
For a partial type $\Phi(\bar x)$ with parameters from $\FC$, $R_{\Delta}(\pi(\bar x))$ is defined as the minimum of the $R_\Delta(\varphi(\bar x))$ where $\varphi(\bar x)$ ranges over all formulas implies by $\Phi(\bar x)$. Stability of the theory is equivalent to saying that all the $R_\Delta$-ranks (for all possible finite sets of formulas $\Delta)$ are less than $\infty$ (equivalently, less than $\omega$). By $\Mlt_\Delta(\varphi(\bar x))$ we denote the {\em $\Delta$-multiplicity} of the formula $\varphi(\bar x)$, i.e. the maximal number $n<\omega$ of pairwise inconsistent $\Delta$-formulas $(\psi_i(\bar x))_{i<n}$ such that $R_\Delta(\varphi(\bar x) \wedge \psi_i(\bar x)) = R_\Delta(\varphi(\bar x))$ for all $i<n$. Finally, $\Mlt_{\Delta}(\Phi(\bar x))$ is defined as the minimum of the multiplicities $\Mlt_\Delta(\varphi(\bar x))$ for $\varphi(\bar x)$ implied by $\Phi(\bar x)$ and satisfying $R_\Delta(\varphi(\bar x))=R_\Delta(\Phi(\bar x))$. It is easy to see that  $\Mlt_\Delta(\Phi(\bar x))$ is the number of global $\Delta$-types $p(\bar x)$ such that $R_\Delta(\Phi(\bar x) \cup p(\bar x)) =R_\Delta(\Phi(\bar x))$. For more details on $R_\Delta$-ranks  (including equivalent definitions and fundamental properties which we are using below without mention) the reader is referred to \cite{anandgeometric,ShelahModels}.
By the $R_\Delta$-rank and $\Delta$-multiplicity of a closed subset of the space of complete global types we mean the  $R_\Delta$-rank and $\Delta$-multiplicity of the corresponding partial type.

By $\vec{R}_\Delta(p)$ we mean the sequence of all $R_\Delta(p)$ (in some fixed order), where $\Delta$ ranges over all finite collections formulas $\varphi(\bar x;\bar y)$. We say that $\vec{R}_\Delta(p)$ is {\em maximal} if it is maximal coordinatewise among all $\vec{R}_\Delta(q)$ when $q$ ranges over  $\widetilde{G}_{\FC}$.

\begin{remark}\label{remark: generic dnf over M}
If $p \in \Gen(\widetilde{G}_{\FC})$, then $p$ does not fork over $M$.
\end{remark}

\begin{proof}
Take $\varphi(\bar x;\bar a) \in p$. Then finitely many left  $G_\FC$-translates of the clopen $[\varphi(\bar x;\bar a)] \subseteq \widetilde{G}_{\FC}$ cover  $\widetilde{G}_{\FC}$. Since $R_\Delta$ is invariant under the left action of $\aut(\FC)$, we get 
$$R_\Delta(\varphi(\bar x;\bar a)) \geq  R_\Delta([\varphi(\bar x;\bar a)]) = R_\Delta(\widetilde{G}_{\FC}) = R_\Delta(\pi(\bar x';\bar m) \cup \tp(\bar c/\emptyset)) \geq R_\Delta(p|_M).$$
As this is true for every finite $\Delta$, we conclude that $p$ does not fork over $M$.
\end{proof}

\begin{proposition}\label{proposition: characterziations of genericity}
Let $p\in \widetilde{G}_{\FC}$. Then the following conditions are equivalent:
\begin{enumerate}
\item $p$ is generic;
\item $\vec{R}_\Delta(p)=\vec{R}_\Delta(\widetilde{G}_\FC)$;
\item $\vec{R}_\Delta(p)$ is maximal;
\item for every $\sigma \in G_\FC$, the type $\sigma \cdot p$ does not fork over $M$.
\end{enumerate}
\end{proposition}

\begin{proof}
(1) $\Rightarrow$ (2) follows from the proof of Remark \ref{remark: generic dnf over M}.

(2) $\Rightarrow$ (3) is trivial.

(3) $\Rightarrow$ (4). Consider any $\sigma \in G_\FC$. Since $R_\Delta$ is invariant under the left action of $\aut(\FC)$, by (2), we get that  $\vec{R}_\Delta(\sigma \cdot p)$ is maximal. Suppose for a contradiction that  $\sigma \cdot p$ forks over $M$. Then  $\vec{R}_\Delta(\sigma \cdot p) <  \vec{R}_\Delta((\sigma \cdot p)|_M) =  \vec{R}_\Delta(q)$, where $q\in \widetilde{G}_\FC$ is the unique non-forking extension of $(\sigma \cdot p)|_M$. This is a contradiction with maximality of $\vec{R}_\Delta(\sigma \cdot p)$.

(4) $\Rightarrow$ (1). Take $\varphi(\bar x;\bar a) \in p$. We need to show that $G_{\varphi,\FC}$ is generic. 
The formula $\varphi(\bar x;\bar a)$ uses only a finite subtuple $\bar x''$ of $\bar x$. Let $\bar c''$ be the subtuple of $\bar c$ corresponding to $\bar x''$.
Let $M \prec N \prec \FC$ be such that $N$ contains $\bar c''$ and is $|M|^+$-saturated, strongly $|M|^+$-homogeneous, but small in $\FC$. 

Consider any $\sigma \in G_{\FC}$. By (4), $\sigma(p)$ does not fork over $M$, so also over $N$, and so $\sigma(p) |_{\bar x',\bar x''}$ does not fork over $N$. Hence, by stability and smallness of $|\bar x'\bar x''|$ in $N$, we get that  $\sigma(p) |_{\bar x'\bar x''}$ is strongly finitely satisfiable in $N$ in the language $\CL_M$
(in the sense that for any formula $\psi(\bar x'\bar x'') \in \sigma(p) |_{\bar x'\bar x''}$ the type $\sigma(p) |_{\bar x'\bar x'',M} \cup \psi(\bar x'\bar x'')$ has a realization in $N$). In particular, there are $\bar m', \bar d''$ in $N$ satisfying the type $\varphi(\bar x;\sigma(\bar a)) \wedge \pi(\bar x';\bar m) \wedge \sigma(p)|_{\bar x'\bar x'', \emptyset}$. Pick $\tau_\sigma \in \aut(N)$ with $\tau_\sigma(\bar m \bar c'')=\bar m' \bar d''$, and any extension $\tilde{\tau}_\sigma \in \aut(\FC)$ of $\tau_\sigma$. Then $\models \pi(\tilde{\tau}_\sigma(\bar m);\bar m)$, so $\tilde{\tau}_\sigma \in G_\FC$. Also, $\models \varphi(\tilde{\tau}_\sigma(\bar c'');\sigma(\bar a))$, so $\models \varphi(\sigma^{-1}\tilde{\tau}_\sigma(\bar c'');\bar a)$. Thus, $\sigma^{-1}\tilde{\tau}_\sigma \in G_{\varphi,\FC}$, so $\sigma^{-1} \in G_{\varphi,\FC} \cdot_r \tilde{\tau}_\sigma^{-1}$. We have proved that $G_\FC = \bigcup_{i \in I}  G_{\varphi,\FC} \cdot_r \tau_i$ for an index set $I$ small with respect to $\FC$, and some $\tau_i \in G_\FC$. Using smallness of $I$, an easy compactness argument yields a finite subset $I_0$ of $I$ such that  $G_\FC = \bigcup_{i \in I_0}  G_{\varphi,\FC} \cdot_r \tau_i$. Thus, $G_{\varphi,\FC}$ is generic.
\end{proof}

In the case of stable groups, one has left [right, or two-sided] invariant stratified ranks. In our context, the usual $R_\Delta$-ranks are left invariant. However, they are not right invariant and we do not see a candidate for right invariant local ranks which would take finite values, witness forking, etc. Starting from a finite $\Delta$, to make the $R_\Delta$-rank right invariant, a natural trick would be to take the closure of $\Delta$ under all permutations of variables $\bar x$ induced by $\aut(\FC)$ or just by $G_\FC$. But this transition often makes $R_\Delta(\widetilde{G}_\FC)$ infinite. For example, consider $T$ to be the theory of an infinite set in the empty language, $M$ a countable model, $\pi(\bar x';\bar m)$ equal to ``$\bar{x}'\equiv_\emptyset\bar{m}$''. 
Then $\widetilde{G}_\FC = S_{\bar c}(\FC)$. Starting from $\Delta(\bar x;\bar y):=\{x_0=y_0\}$ and taking the closure $\cl(\Delta)$ under all permutations of $\bar x$, we easily see that $R_{\cl(\Delta)}(S_{\bar c}(\FC))=\infty$.
Even more: there is no possibly infinite set of formulas $\Delta(\bar x;\bar y)$ containing the formula $x_0=y_0$ for which $R_\Delta$ is right invariant and 
$R_\Delta (S_{\bar c}(\FC))< \infty$.
In particular, it is not clear whether we could add the right version of (4) in Proposition \ref{proposition: characterziations of genericity}.\\

Let $\widetilde{G}_{\FC, \Delta}:=\{ p|_\Delta: p \in \widetilde{G}_\FC\}$ and $\Gen (\widetilde{G}_{\FC, \Delta}):= \{p|_\Delta: p \in \Gen(\widetilde{G}_\FC)\}$ for any finite $\Delta$. By the equivalence of (1) and (2) in Proposition \ref{proposition: characterziations of genericity}, we get that $\Gen (\widetilde{G}_{\FC, \Delta})$ is always finite. Thus, the left action of $G_\FC$ on $\Gen(\widetilde{G}_\FC)$ induces a left action on $\Gen (\widetilde{G}_{\FC, \Delta})$ which is transitive, and clearly $\Gen(\widetilde{G}_{\FC}) \cong \varprojlim_\Delta (\widetilde{G}_{\FC, \Delta})$ as $G_{\FC}$-flows.
Indeed, finiteness of $\Gen(\widetilde{G}_{\FC,\Delta})$ is clear - the size of this set is precisely $\Mlt_{\Delta}(\widetilde{G}_\FC)$ which is finite.
To see transitivity of the left action of $\widetilde{G}_\FC$ on $\Gen(\widetilde{G}_{\FC,\Delta})$,
enumerate $\Gen(\widetilde{G}_{\FC,\Delta})$ as $p_0,\dots,p_{n-1}$ and
choose generic clopens $X_0, \dots, X_{n-1}$ in $\widetilde{G}_{\FC}$ which are relatively defined by $\Delta$-formulas $\varphi_0(\bar x),\dots,\varphi_{n-1}(\bar x)$ which separate the types in $\Gen(\widetilde{G}_{\FC,\Delta})$ in the sense that for every $i,j <n$ we have that $\varphi_i(\bar x) \in p_j \iff i=j$. Since finitely many left $G_\FC$-translates of each $X_i$ cover $\widetilde{G}_\FC$, we get that for every $i,j < n$ there exists $g \in G_\FC$ such that $gX_i \cap X_j$ is generic. 
This implies that $gp_i=p_j$ (as $gp_i$ is the only type in $\Gen(\widetilde{G}_{\FC,\Delta})$ containing $g\varphi_i(\bar x)$, and $p_j$ the only type in $\Gen(\widetilde{G}_{\FC,\Delta})$ containing $\varphi_j(\bar x)$), so we have transitivity.

\begin{remark}\label{remark: max rank implies generic}
Let $\Delta$ be a finite collection of formulas, and let $X$ be a clopen subset of $\widetilde{G}_{\FC}$ relatively defined by a $\Delta$-formula.
If $R_\Delta(X)=R_\Delta(\widetilde{G}_{\FC})$, then $X$ is generic.
\end{remark}

\begin{proof}
Without loss of generality $\Mlt_\Delta(X)=1$.
Present $\widetilde{G}_\FC$ as a disjoint union $X^\Delta_0 \cupdot \ldots \cupdot X^\Delta_{n_\Delta}$ of clopens relatively defined by $\Delta$-formulas such that $R_\Delta(X^\Delta_i)=R_\Delta(\widetilde{G}_\FC)=:N$ and $\Mlt_\Delta(X^\Delta_i)=1$. 
Then $X^\Delta_i$ is generic for some $i$. Hence, $R_\Delta(X\cap gX^\Delta_i)=N$ for some $g \in G_{\FC}$. 
Since $X$ and $gX^\Delta_i$ are relatively $\Delta$-defined and of $\Delta$-multiplicity 1, we get that $R_\Delta(X \vartriangle gX^\Delta_i)<N$. So $X \vartriangle gX^\Delta_i$ is not generic by Proposition \ref{proposition: characterziations of genericity}. 
Since $gX^\Delta_i$ is generic, we conclude that $X\cap gX^\Delta_i$ is generic, and hence so is $X$.
\end{proof}

Note that this remark yields an alternative proof of (3) $\Rightarrow$ (1) in Proposition \ref{proposition: characterziations of genericity}. Namely, assuming that $\vec{R}_\Delta(p)$ is maximal, by (1) $\Rightarrow$ (2) and existence of generics, we get that  $\vec{R}_\Delta(p) =\vec{R}_\Delta(\widetilde{G}_\FC)$. Hence, $p$ is generic by Remark \ref{remark: max rank implies generic}.

By the same argument, we get the following variant of the characterization of generics via local ranks. Below $\Delta'$ ranges over all finite sets of formulas $\varphi(\bar x';\bar y)$, and 
$$\widetilde{G}_{\FC,\bar m} := S_{\pi(\bar{x}';\bar{m})}(\FC)= \{ p(\bar x') \in S_{\bar m}(\FC): \pi(\bar x';\bar m) \subseteq p(\bar x')\}.$$ Then we have a left action $\cdot$ of $G_\FC$ on $\widetilde{G}_{\FC,\bar m}$ defined as before (but we do not have a natural right action). Left generics are defined in terms of this action as before.

\begin{proposition}\label{proposition: restricted to x'}
\begin{enumerate}
\item For each finite $\Delta'$, every clopen subset $X$ of $\widetilde{G}_{\FC,\bar{m}}$ relatively defined by a $\Delta'$-formula and with $R_{\Delta'}(X)=R_{\Delta'}(\widetilde{G}_{\FC,\bar m})$ is generic.

\item Let $p\in \widetilde{G}_{\FC,\bar m}$. Then the following conditions are equivalent:
\begin{enumerate}
\item $p$ is generic;
\item $\vec{R}_{\Delta'}(p)=\vec{R}_{\Delta'}(\widetilde{G}_{\FC,\bar m})$;
\item $\vec{R}_{\Delta'}(p)$ is maximal.
\end{enumerate}
\end{enumerate}
\end{proposition}

\begin{proposition}\phantomsection\label{proposition: transitivity}
$\Gen(\widetilde{G}_{\FC})$ is bounded and the left action of $G_\FC$ on $\Gen(\widetilde{G}_{\FC})$ is transitive.
\end{proposition}

\begin{proof}
First, we prove transitivity. Consider any $p,q \in \Gen (\widetilde{G}_{\FC})$, and we need to show that there exists $\sigma \in G_\FC$ with $\sigma \cdot p=\sigma(p)=q$.

For each finite $\Delta$ present $\widetilde{G}_\FC$ as a disjoint union $X^\Delta_0 \cupdot \ldots \cupdot X^\Delta_{n_\Delta}$ of clopens relatively defined by $\Delta$-formulas such that $R_\Delta(X^\Delta_i)=R_\Delta(\widetilde{G}_\FC)=:N$ and $\Mlt_\Delta(X^\Delta_i)=1$ for all $i\leqslant n_\Delta$. Then, by 
Proposition \ref{proposition: characterziations of genericity} and Remark \ref{remark: max rank implies generic}, $\Gen (\widetilde{G}_{\FC,\Delta})=\{s^\Delta_0,\dots,s^\Delta_{n_\Delta}\}$ with $s^\Delta_i \in X^\Delta_i$.  So, for any $r \in \Gen (\widetilde{G}_{\FC})$ and finite $\Delta$ there exists a unique $i_r\leqslant n_\Delta$ such that $r|_\Delta =s^\Delta_{i_r}$. 

By transitivity of the action of $G_\FC$ on $\Gen (\widetilde{G}_{\FC, \Delta})$, we get that for every finite $\Delta$, $G^\Delta:=\{\sigma \in G_\FC: \sigma \cdot s^\Delta_{i_p}=s^\Delta_{i_q}\}$ is non-empty. On the other hand, we will see that it is relatively type-definable. Using compactness, this implies the existence of $\sigma \in G_\FC$ with $\sigma \cdot p =q$.

Note that $G^\Delta =\{\sigma \in G_\FC: R_\Delta(\sigma[X^\Delta_{i_p}] \cap X^\Delta_{i_q}) \geq N\}$, and a fundamental property of $R_\Delta$-ranks implies that this set is relatively type-definable, so we are done. (In fact, one can see that this set is even relatively definable in  $\widetilde{G}_\FC$, which we leave as an easy exercise.)

Now, we show that $\Gen(\widetilde{G}_{\FC})$ is bounded. 

\begin{clm}
For any $p \in \Gen(\widetilde{G}_{\FC})$ and $\sigma,\tau \in G_{\FC}$, if $\sigma(\bar m) \equiv_M \tau(\bar m)$, then $\sigma \cdot p=\tau \cdot p$.
\end{clm}

\begin{clmproof}
By assumption, there is $\rho \in \aut(\FC/M) \subseteq G_{\FC}$ such that $\rho(\sigma(\bar m))=\tau(\bar m)$. Then $\tau^{-1}\rho\sigma \in \aut(\FC/M) \subseteq G_{\FC}$.

By Proposition \ref{proposition: characterziations of genericity}, both $p$ and $\sigma \cdot p$ do not fork over $M$, and so $p$ and $\sigma \cdot p$ are $M$-invariant (by stability). Therefore, by the previous paragraph, $\rho \cdot (\sigma \cdot p)=\sigma \cdot p$ and $(\tau^{-1}\rho\sigma) \cdot p =p$. Hence, $\sigma \cdot p = \rho \cdot (\sigma \cdot p) = \tau \cdot p$.
\end{clmproof}

We can find a set $I$ of cardinality at most $|S_{\bar m}(M)|$ (which is clearly bounded) and a family $\{\sigma_i: i \in I\}$ of elements of $G_{\FC}$ such that for every $\sigma \in G_{\FC}$ there exists $i \in I$ such that $\sigma(\bar m) \equiv_M \sigma_i(\bar m)$. Then, by transitivity of the action of $G_{\FC}$ on $\Gen(\widetilde{G}_{\FC})$ and the above claim, $\Gen(\widetilde{G}_{\FC}) =G_{\FC} \cdot p = \{\sigma_i: i \in I\} \cdot p$ which is bounded.
\end{proof}

\begin{proposition}\label{proposition: stabilizer contains Aut(C/M)}
For every $p \in \Gen(\widetilde{G}_{\FC})$, $\Stab(p):=\{\sigma \in G_\FC\;\colon\; \sigma \cdot p = p\}$ is relatively type-definable of bounded index, and as such it contains the normal closure in $G_\FC$ of $\aut(\FC/M)$.
\end{proposition}

\begin{proof}
Bounded index follows from boundedness of $\Gen(\widetilde{G}_{\FC})$. 
Relative type-definablity is clear: $\Stab(p)$ is precisely the intersection of $G_{\FC}$ with the pointwise stabilizer of the canonical base of $p$ which is contained in $M^{\mathrm{eq}}$ (as $p$ is definable over $M$ by Proposition \ref{proposition: characterziations of genericity}),
and so $\Stab(p)$ is relatively  $\bar m$-type-definable over $M$.
For the additional conclusion, we adapt the argument from \cite[Proposition 4.5]{HruKruPi}.

Assume that $H$ is a relatively type-definable subgroup of $G_\FC$ of bounded index. Then the intersection of all conjugates of $H$ is an intersection of boundedly many of them, so it is also relatively type-definable and of bounded index (formally one should be more precise here, but we skip the details). Thus, without loss of generality, $H$ is a normal subgroup of  $G_\FC$.

Choose a type $\pi'(\bar z;\bar b)$, where $\bar z$ is a short subtuple of $\bar x$ which corresponds to a subtuple $\bar a$ of $\bar c$ and $\bar b$ is a short tuple from $\FC$, so that 
$$(*)\;\;\;\;\;\;\;\; H=\{\sigma \in \aut(\FC): \models \pi'(\sigma(\bar a);\bar b)\}.$$

The orbit equivalence relation $E$ of the action of $H$ on $G_\FC (\bar a):=\{ \sigma(\bar a): \sigma \in G_\FC\}$ is bounded and invariant under the action of $G_\FC$.

On the other hand, whenever $\bar a' \equiv_M \bar a''$, where $\bar a'$ and $\bar a''$ both belong to $G_\FC (\bar a)$, then there is a sequence $(\bar a_1,\bar a_2,\dots)$ such that both $(\bar a',\bar a_1,\dots)$ and $(\bar a'',\bar a_1,\dots)$ are $M$-indiscernible. Using the fact that $\aut(\FC/M) \leqslant G_\FC$ and taking this sequence of length greater than the number of classes of $E$, we get that there are $i \ne j$ such that $\E(\bar a_i, \bar a_j)$. By the fact that $\aut(\FC/M) \leqslant G_\FC$, we get that any two distinct elements of each of the above two $M$-indiscernible sequences are $E$-related. We conclude that $\E(\bar a',\bar a'')$. 

Now, consider any $\sigma \in \aut(\FC/M)$. By the last paragraph, $\E(\bar a,\sigma(\bar a))$. Hence, there is $\tau \in H$ such that $\sigma(\bar a)=\tau(\bar a)$. Then $\tau^{-1}\sigma (\bar a) = \bar a$ and $\sigma = \tau(\tau^{-1}\sigma)$. 
Since $(*)$ shows that $H\Stab(\bar a) = H$ (where $\Stab(\bar a):=\{ \rho \in \aut(\FC): \rho(\bar a)=\bar a\}$), we conclude that $\sigma \in H$.

So we have proved that $\aut(\FC/M) \leqslant H$.
\end{proof}

Let $\res \colon \widetilde{G}_\FC \to \widetilde{G}$ be the restriction map to the variables $\bar x'$ and to the set of parameters $M$, 
where $\widetilde{G}:=S_{\pi(\bar{x}';\bar{m})}(M)=\{p(\bar x') \in S_{\bar m}(M): \pi(\bar x';\bar m) \subseteq p(\bar x')\}$. 
Let $\Gen(\widetilde{G})$ denote the set of all types $p \in \widetilde{G}$ such that for every $\varphi(\bar x';\bar m) \in p$ the set $G_{\varphi,\FC}$ is generic. Here and below we treat $\varphi(\bar x';\bar m)$ as $\varphi(\bar x;\bar m)$ whenever needed.

\begin{proposition}\label{proposition: res on generics}
The map $\res|_{\Gen(\widetilde{G}_\FC)}$ is a homeomorphism from $\Gen(\widetilde{G}_\FC)$ onto $\Gen(\widetilde{G})$.
\end{proposition}

\begin{proof}
It is clear that $\res[\Gen(\widetilde{G}_\FC)]\subseteq \Gen(\widetilde{G})$. The opposite inclusion follows from Remark \ref{remark: identification of flows} and Corollary \ref{corollary: basic properties of generics}.

For injectivity, take any $p,q \in \Gen(\widetilde{G}_\FC)$ with $\res(p)=\res(q)$. By Proposition \ref{proposition: transitivity}, we can find $\sigma \in G_\FC$ such that $\sigma(p)=q$. Take $\tau \in \aut(\FC')$ with $p=\tp(\tau(\bar c)/\FC)$. Then $\sigma'(\tau(\bar m)) \equiv_M \tau(\bar m)$, where $\sigma'$ is any extension of $\sigma$ to an automorphism of $\FC'$. So $\sigma'(\tau(\bar m)) =\eta (\tau(\bar m))$ for some $\eta \in \aut(\FC'/M)$. Then $(\tau^{-1} \eta^{-1} \tau)(\tau^{-1} \sigma' \tau) (\bar m)=\bar m$. Therefore, $\sigma'$ belongs to the normal closure of $\aut(\FC'/M)$. By Remark \ref{remark: generic in various models} and the argument as in the first paragraph of this proof, we get that $p$ extends to $p' \in \Gen(\widetilde{G}_{\FC'})$. Using Proposition \ref{proposition: stabilizer contains Aut(C/M)} for $\FC$ replaced by $\FC'$, we conclude that $\sigma(p')=p'$. Hence, $\sigma'(p) =p$, so $q=p$.

Continuity of $\res$ is trivial, hence $\res|_{\Gen(\widetilde{G}_\FC)} \colon \Gen(\widetilde{G}_\FC) \to \Gen(\widetilde{G})$ is a homeomorphism by compactness of the relevant spaces.
\end{proof}

Recall that $\Delta'$ ranges over all finite sets of formulas $\varphi(\bar x';\bar y)$.
\begin{proposition}\label{proposition: generic via ranks for x'}
 Let $p\in \widetilde{G}$. Then the following conditions are equivalent:
\begin{enumerate}
\item $p \in \Gen(\widetilde{G})$;
\item $\vec{R}_{\Delta'}(p)=\vec{R}_{\Delta'}(\widetilde{G}_{\FC,\bar m})$;
\item $\vec{R}_{\Delta'}(p)$ is maximal (among $\vec{R}_{\Delta'}(q)$ for $q \in \widetilde{G}$).
\end{enumerate}
\end{proposition}

\begin {proof}
(1) $\Rightarrow$ (2).
By Proposition \ref{proposition: res on generics}, $p=\res(q)$ for some $q \in \Gen(\widetilde{G}_\FC)$. Then $q|_{\bar x'} \in \Gen(\widetilde{G}_{\FC,\bar m})$. So $\vec{R}_{\Delta'}(q|_{\bar x'})=\vec{R}_{\Delta'}(\widetilde{G}_{\FC,\bar m})$ by Proposition \ref{proposition: restricted to x'}. As $(q|_{\bar x'})|_M=p$, we get $\vec{R}_{\Delta'}(p)=\vec{R}_{\Delta'}(\widetilde{G}_{\FC,\bar m})$.

(2) $\Rightarrow$ (3) is trivial.

(3) $\Rightarrow$ (2). It follows from (1) $\Rightarrow$ (2), because $\Gen(\widetilde{G}) \ne \emptyset$.

(2) $\Rightarrow$ (1). Let $\hat{p} \in \widetilde{G}_{\FC,\bar m}$ be the unique nonforking extension of $p$. Then $\vec{R}_{\Delta'}(\hat{p})=\vec{R}_{\Delta'}(p)=\vec{R}_{\Delta'}(\widetilde{G}_{\FC,\bar m})$, so $\hat{p}$ is a generic element in $\widetilde{G}_{\FC,\bar m}$ by Proposition \ref{proposition: restricted to x'}. We finish using the following variant of Remark \ref{remark: identification of flows}:
The map $f' \colon \widetilde{G}_{\FC,\bar m} \to S(\Df_{\bar m}(G_{\FC}))$ given by $f(p):= \{G_{\varphi,\FC}: \varphi(\bar x',\bar a) \in p\}$ is a well-defined isomorphism of left $G_\FC$-flows, where $\Df_{\bar m}(G_{\FC})$ is the Boolean algebra of relatively $\bar m$-definable subsets of $\widetilde{G}_\FC$.
\end{proof}

\subsection{Convolution product of types in stable theories}\label{subsec: group chunk 2}
Recall that we defined the convolution product on $S^{\inv}_{\bar m}(\FC,M)$ which was explicitly described in Proposition \ref{prop:formula.for.star.on.types}:
 $$p\ast q=\sigma(p|_{\FC'})|_\FC,$$
    where $q(\bar{y})=\tp(\sigma(\bar{m})/\FC)$ for some $\sigma\in\aut(\FC')$ and $p|_{\FC'}$ is the unique $M$-invariant extension of $p$ to $\FC'$.

For the rest of this subsection, assume that {\bf $T$ is stable}. Then a type $p \in S(\FC)$ is invariant over $M$ if and only if it is the unique nonforking extension of $p|_M$. Hence, the restriction map $S^{\inv}_{\bar m}(\FC,M) \to S_{\bar m}(M)$ is a homeomorphism which induces a semigroup operation $*$ on $S_{\bar m}(M)$ given by
$$p*q:=\sigma(\hat{p})|_M,$$
where $\sigma \in \aut(\FC)$ satisfies $\sigma(\bar m) \models q$ and $\hat{p} \in S_{\bar m}(\FC)$ is the unique global nonforking extension of $p$. We leave as an easy exercise to check that this is well-defied (i.e. does not depend on the choice of $\sigma$), and that it is indeed induced by the above restriction map. In particular, $(S_{\bar m}(M),*)$ is a left topological monoid.

\begin{remark}\label{remark: $*$ sep. cont.}
The map $*$ on $S_{\bar m}(M)$ is separately continuous.
\end{remark}

\begin{proof}
Right continuity follows from definability of types.
Alternatively, note that $(p\ast q)(\theta(\bar{x}';\bar{b}))=(p\otimes h_{\bar{b}}(q))(\theta(\bar{x}';\bar{y}))$ is composition of continuous functions.
\end{proof}

The above discussion applies to $(\FC,\FC')$ in place of $(M,\FC)$. Thus, we have a separately continuous semigroup operation $*$ on $S_{\bar c}(\FC)$ given by:
 $$p*q:=\sigma(\hat{p})|_\FC,$$
where $\sigma \in \aut(\FC')$ satisfies $\sigma(\bar c) \models q$ and $\hat{p} \in S_{\bar c}(\FC')$ is the unique nonforking extension of $p$. Then $S^{\inv}_{\bar c}(\FC,M)$ is closed under $*$. Indeed, assume that $p,q \in S^{\inv}_{\bar c}(\FC,M)$. If $\bar a \equiv_M \bar b$ are tuples from $\FC$, then $\bar a \equiv_{\sigma(\bar c)} \bar b$ by $M$-invariance of $q$. Hence, $\sigma^{-1}(\bar a) \equiv_{\bar c} \sigma^{-1}(\bar b)$. Hence, we get the equivalence $\varphi(\bar x;\sigma^{-1}(\bar a)) \in \hat{p} \iff  \varphi(\bar x;\sigma^{-1}(\bar b)) \in \hat{p}$, because $\hat{p}$ is $\FC$-invariant. So we get $M$-invariance of $p*q$, as the left hand side is equivalent to $\varphi(\bar x;\bar a) \in p*q$ and the right side to $\varphi(\bar x;\bar b) \in p*q$.

\begin{remark}\label{remark: restriction preserves *}
Let $\res \colon S^{\inv}_{\bar c}(\FC,M) \to S_{\bar m}(M)$ be the restriction map to the variables $\bar x'$ and to the set of parameters $M$ (i.e. the map introduced before Proposition \ref{proposition: res on generics} but on a modified domain). Then $\res$ is a homomorphism of semigroups, i.e. $\res(p*q)=\res(p) * \res(q)$ for any $p,q \in S_{\bar c}(\FC)$.
\end{remark}

\begin{proof}
Follows easily from the definitions.
\end{proof}

In order to apply some topological dynamics to our $\ast$-product, we need to treat $S_{\bar c}(\FC)$ as an $\aut(\FC)$-flow with respect to the following left action:
$$\sigma \bullet p:= p \cdot_r \sigma^{-1} = \tp(\tau(\sigma^{-1}(\bar c))/\FC),$$ 
where $\tau \in \aut(\FC')$ is such that $\tp(\tau(\bar c)/\FC)=p$. Whenever (in this subsection) we do not explicitly mention an action of $\aut(\FC)$, we always consider $\bullet$.

\begin{lemma}\phantomsection\label{lemma: Phi on S_c(C)}
\begin{enumerate}
\item For every $\sigma \in \aut(\FC)$ we have $\sigma \bullet q = \tp(\sigma^{-1}(\bar c)/\FC) *q$.
\item The assignment $\Phi \colon p \mapsto l_p$, where $l_p\colon S_{\bar c}(\FC) \to S_{\bar c}(\FC)$ is given by $l_p(q):=p*q$, is a topological  isomorphism from $(S_{\bar c}(\FC),*)$ to $\E(S_{\bar c}(\FC),\aut(\FC))$. It is also an isomorphism of $\aut(\FC)$-flows.
\end{enumerate}
\end{lemma}

\begin{proof}
(1) $\tp(\sigma^{-1}(\bar c)/\FC) * \tp(\rho(\bar c)/\FC) = \tp(\rho(\sigma^{-1}(\bar c))/\FC) = \sigma \bullet \tp(\rho(\bar c)/\FC)$ for any $\rho \in \aut(\FC')$.

(2) Since $\{ \tp(\sigma(\bar c)/\FC) : \sigma \in \aut(\FC)\}$ is dense in $S_{\bar c}(\FC)$ and $*$ is left continuous on $S_{\bar c}(\FC)$, using item (1), we get that $\Ima(\Phi) \subseteq \E(\aut(\FC),S_{\bar c}(\FC))$.  The opposite inclusion follows from item (1), left continuity of $*$, and compactness of $S_{\bar c}(\FC)$.

Associativity of $*$ implies that $\Phi$ is a homomorphism. Since $l_p(\tp(\bar c/\FC))=p$, we see that $\Phi$ is injective. Continuity of $\Phi$ follows from left continuity of $*$. 

To see that $\Phi$ is a flow homomorphism, consider any $p,q \in S_{\bar c}(\FC)$. Then, by item (1), $\Phi(\sigma \bullet p)(q) = (\tp(\sigma^{-1}(\bar c)/\FC) * p) * q = \tp(\sigma^{-1}(\bar c)/\FC) * (p * q) = \sigma \bullet (p*q)=\sigma \bullet \Phi(p)(q)$.
\end{proof}

Note that by Remark \ref{remark: restricted actions}, $\widetilde{G}_\FC$ is a $G_\FC$-flow with respect to the action $\bullet$ restricted to $G_\FC \times \widetilde{G}_\FC$. Moreover, by Remark \ref{remark: two descriptions of the type space}, $(G_\FC,\widetilde{G}_\FC, \tp(\bar c/\FC))$ is an ambit, i.e. the orbit $G_\FC \bullet  \tp(\bar c/\FC)$ is dense.

\begin{lemma}\phantomsection\label{lemma: Phi is an isom.}
\begin{enumerate}
\item $\widetilde{G}$ and $\widetilde{G}_\FC$ are closed under $*$, and so they are compact separately continuous monoids.
\item The restriction $\Phi|_{\widetilde{G}_{\FC}}$ is a semigroup and $G_\FC$-flow isomorphism from $\widetilde{G}_{\FC}$ to $\E(\widetilde{G}_{\FC},G_\FC)$.
\end{enumerate}
\end{lemma}

\begin{proof}
(1) The fact that $\widetilde{G}_\FC$ is closed under $*$ follows from Lemma \ref{lemma: Phi on S_c(C)}(1), left continuity of $*$, and (topological) closedness of $\widetilde{G}_\FC$.
Using this together with Remark \ref{remark: restriction preserves *} and an easy observation that $\res[S^{\inv}_{\bar c}(\FC,M)\cap \widetilde{G}_\FC] = \widetilde{G}$, we get that $\widetilde{G}$ is closed under $*$. Then we use Remar \ref{remark: $*$ sep. cont.} to obtain that $\ast$ is separately continuous.

(2) The proof is the same as in Lemma \ref{lemma: Phi on S_c(C)}(2), using Remark \ref{remark: two descriptions of the type space}.
\end{proof}

The next remark follows from the definition of $\bullet$.

\begin{remark}
$\Gen(\widetilde{G}_\FC)$ is precisely the set of all generic elements of the flow $(\widetilde{G}_\FC,G_\FC)$ (with respect $\bullet$).
\end{remark}

Since $\Gen(\widetilde{G}_\FC) \ne \emptyset$, by \cite[Corollary 1.9]{New2009}, we conclude: 

\begin{cor}\label{corollary: minimal subflow}
$\Gen(\widetilde{G}_\FC)$ is a unique minimal $G_\FC$-subflow and a unique minimal left ideal (with respect to $*$) of $\widetilde{G}_\FC$.
\end{cor}

{\em Weakly almost periodic} (or WAP, for short) flows, introduced in \cite{WAP}, play an important role in topological dynamics. Recall that a flow is WAP if each member of its Ellis semigroup is continuous. A strong connections (in a sense, equivalence) between WAP and stability was discovered by Ben-Yaacov \cite{Ben_Yaacov:Model_theoretic_stability_of_types_after_Grothendieck}.

\begin{proposition}\label{proposition: WAP flow}
The flows $(S_{\bar c}(\FC),\aut(\FC))$ and $(\widetilde{G}_\FC,G_\FC)$ are WAP.
\end{proposition}

\begin{proof}
Since WAP flows are closed under both decreasing of the acting group and taking subflows, it is enough to show that the first flow is WAP. 
But this is a standard application of Grothendieck's double limit theorem (cf. Corollary 2.15 and Proposition 2.17 in \cite{CodHoff23}).
\end{proof}

\begin{cor}\phantomsection\label{corollary: Gen is an ideal and group}
\begin{enumerate}
\item $(\Gen(\widetilde{G}_\FC),*)$ is a profinite group.
\item $\Gen(\widetilde{G})$ is closed under $*$, and $(\Gen(\widetilde{G}),*) \cong (\Gen(\widetilde{G}_\FC),*)$; thus $(\Gen(\widetilde{G}),*)$ is a profinite group.
\item $\Gen(\widetilde{G}_\FC)$ is a unique minimal left ideal and a unique minimal right ideal in $\widetilde{G}_\FC$.
\item $\Gen(\widetilde{G})$ is a unique minimal left ideal and a unique minimal right ideal in $\widetilde{G}$.
\end{enumerate}
\end{cor}

\begin{proof}
(1) By Proposition \ref{proposition: WAP flow} and Corollary \ref{corollary: minimal subflow}, $(\Gen(\widetilde{G}_\FC),G_\FC)$ is a WAP minimal flow. 
So $\E(\Gen(\widetilde{G}_\FC),G_\FC)$ is a group 
(using \cite[Proposition II.8]{WAP} and \cite[Theorem I.3.3(4)]{Glasner:Proximal_flows}). 
In particular, $\E(\Gen(\widetilde{G}_\FC),G_\FC)$ is a minimal left ideal in itself, and so, by \cite[Lemma 5.16]{CGK} and Proposition \ref{lemma: Phi is an isom.}(2), we get that $\E(\Gen(\widetilde{G}_\FC),G_\FC) \cong \Gen(\widetilde{G}_\FC)$. So  $\Gen(\widetilde{G}_\FC)$ is a group. The fact that it is a profinite group follows from the observations that $\Gen(\widetilde{G}_\FC)$ is a closed subspace of the profinite space $\widetilde{G}_\FC$, Remark \ref{remark: $*$ sep. cont.} (applied to $\FC$ in place of $M$), and the Ellis joint continuity theorem.

(2) follows from Proposition \ref{proposition: res on generics}, Remarks \ref{remark: generic dnf over M} and \ref{remark: restriction preserves *}, and item (1).

(3) Corollary \ref{corollary: minimal subflow} tells us that it is a unique minimal left ideal. The fact that it is a right ideal follows from Proposition \ref{proposition: characterziations of genericity} and an easy observation that $R_\Delta(p*q) \geqslant R_\Delta(p)$. Then minimality of this right ideal is immediate by (1). To see uniqueness, consider any minimal right ideal $I$ and an element $p \in I$. Take any $q \in \Gen(\widetilde{G}_\FC)$. Then, as $\Gen(\widetilde{G}_\FC)$ is a left ideal, we get $p * q \in I \cap \Gen(\widetilde{G}_\FC)$, so $I=\Gen(\widetilde{G}_\FC)$ by minimality of these right ideals.

(4) By Proposition \ref{proposition: res on generics}, Remarks \ref{remark: generic dnf over M}, \ref{remark: restriction preserves *}, item (3), and the fact that $\res[S^{\inv}_{\bar c}(\FC,M)\cap \widetilde{G}_\FC] = \widetilde{G}$, we get that $\Gen(\widetilde{G})$ is a two-sided ideal. Then the fact that it is a minimal left and minimal right ideal follows from (2). Uniqueness follows as at the end of the proof of (3).
\end{proof}
\noindent

\subsection{A counterpart of Newelski's theorem for $\aut(\FC)$}\label{subsec: group chunk 3}
Throughout this subsection, we assume that \textbf{$T$ is stable}.
In this section, let $\bar x$ be a tuple of variables corresponding to $\bar m$. 
(In the two previous subsections, $\bar x$ corresponded to $\bar c$, and $\bar x'$ to $\bar m$, but in this section we come back to 
the more standard notation, because we will not use types in $S_{\bar c}(\FC)$.) For a type $p \in S_{\bar x}(M)$, $\hat{p}\in S_{\bar x}(\FC)$ will denote its unique nonforking extension.

Let $P \subseteq S_{\bar m}(M)$ and $Q:= \cl(*P)$, i.e. the topological closure of the closure of $P$ under $*$. Since $*$ is separately continuous (see Remark \ref{remark: $*$ sep. cont.}), we get

\begin{remark}\label{remark: Q closed under *}
$Q$ is closed under $*$.
\end{remark}

Let 
$$\gen(P):= \{ q \in Q: \vec{R}_\Delta(q) \textrm{ is maximal}\},$$ 
where ``maximal'' means maximal among all $\vec{R}_\Delta(r)$ for $r$ ranging over $Q$ in the sense of the product order, and $\vec{R}_\Delta(p)$ is the sequence of all $R_\Delta(p)$ (in some fixed order), where $\Delta$ ranges over all finite collections formulas $\varphi(\bar x;\bar y)$.

For any $S \subseteq S_{\bar m}(M)$ put 
$$A_S:=\{ \sigma \in \aut(\FC) \;\colon\; \tp(\sigma(\bar m)/M) \in S\}.$$

The goal of this section is to prove the following theorem.

\begin{theorem}\label{theorem: counterpart of Newelski's theorem}
The set $H:=\{\sigma \in \aut(\FC): \gen(P) * \tp(\sigma(\bar m)/M) = \gen(P)\}$ is the smallest relatively $\bar m$-type-definable over $M$ subgroup of $\aut(\FC)$ containing $A_P$ and  we have $\gen(P)=\Gen(\widetilde{H})$, where $\widetilde{H}:=\{\tp(\sigma(\bar m)/M)\;\colon\; \sigma \in H\}$. Thus, $\gen(P)$ is a profinite group and a two-sided ideal of $Q$.
\end{theorem}

This theorem is a counterpart of \cite[Theorem 2.3]{New89}. Our proof is an adaptation of the proof of that theorem, but with various new ingredients. One of the main obstacles in comparison with definable groups is that in our context $R_\Delta(p*q)$ may be smaller than $R_\Delta(q)$. 

First of all, it turns out that the proof in \cite{New89} is not completely correct. The problem is that the formulas $\varphi_{\alpha,i}(x)$ at the top of page 176 in \cite{New89} should be $\Delta_\alpha$-formulas in order to proceed with the argument after the claim on the same page. However, in general they cannot be chosen to be $\Delta_\alpha$-formulas. In a private communication with the third author, Ludomir Newelski proposed an alternative initial part of the argument, using a modified version of $R_\Delta$ denoted by $R_\Delta'$. But in this new part he used $R'_\Delta(p*q) \geq R'_\Delta(q)$, which we do not have in our context. We will adapt Newelski's corrected argument to our context, but still working with $R_\Delta$'s and with the lexicographic order in place of product order. Secondly, the final part of the argument from \cite{New89} does not work in our context due to several reasons, one of which being the fact that $R_\Delta(p*q)$ may be smaller than $R_\Delta(q)$. So we give a different argument.

The closed set of types $Q$ corresponds to a partial type $Q(\bar x)$ (we will use the variables $\bar x$ to indicate the places when $Q$ is treated as a partial type).

Fix an enumeration $(\Delta_\alpha)_{\alpha<|T|+|M|}$ of
the collection of all finite sets of formulas (without parameters) in variables $(\bar x,\bar y)$, where $\bar y$ ranges over finite tuples. Let $\max(P)$ be the collection of all types $p \in Q$ for which the sequence $\vec{R}(p):=\langle R_{\Delta_\alpha}(p) \rangle_{\alpha<|T| +|M|}$ is the greatest element of the set $\{\vec{R}(q): q \in Q\}$ with respect to the {\bf lexicographic order}. Let 
$$\widehat{\max(P)}:=\{\hat{p}\;\colon\; p \in \max(P)\} \quad \textrm{ and } \quad \widehat{\max(P)}_\Delta:=\{p|_\Delta\;\colon\; p \in \widehat{\max(P)}\}.$$

\begin{lemma}\label{lemma: basic on max(P)}
\begin{enumerate}
\item $\max(P)$ is nonempty and closed.
\item $\widehat{\max(P)}_\Delta$ is finite for every finite $\Delta$.
\end{enumerate}
\end{lemma}

\begin{proof}
(1) Closedness is trivial, as the $R_\Delta$-rank of a partial type closed under conjunction equals the $R_\Delta$-rank of a formula in this type.

Let us show that  $\max(P) \ne \emptyset$. For that we will prove by induction on $\beta$ that for every $1\leqslant\beta\leqslant |T| +|M|$ there exists a type $p \in Q$ such that  $\vec{R}_\beta(p):=\langle R_{\Delta_\alpha}(\hat{p}) \rangle_{\alpha<\beta}$ is a greatest element of the set $\{\vec{R}_\beta(q)\;\colon\; q \in Q\}$ with respect to the lexicographic order. 

In the base step, notice that any $p\in Q$ with $R_{\Delta_0}(p)=R_{\Delta_0}(Q(\bar x))$ does the job.
Now, suppose the conclusion holds for all $\alpha<\beta$ and we want to prove it for $\beta$. So for any $\alpha<\beta$ the set $Q_\alpha$ of all $p \in Q$ such that $\vec{R}_{\alpha}(p)$ is greatest in the set $\{\vec{R}_{\alpha}(q)\;\colon\; q \in Q\}$ is nonempty and closed. Taking the intersection and using compactness of $Q$, we get an element $p \in \bigcap_{\alpha<\beta}Q_\alpha$. If $\beta$ is a limit ordinal, then $p \in Q_\beta$, so we are done. If $\beta = \gamma+1$ for some $\gamma$, then any $q\in Q_\gamma$ with maximal possible value $R_{\Delta_\gamma}(q)$ is as required.

(2) We need to show that  $\widehat{\max(P)}_{\Delta_\alpha}$ is finite for every $\alpha\leqslant |T| +|M|$. Denote by $N_\alpha$ the common value $R_{\Delta_\alpha}(p)$ for $p \in \max(P)$. Define the closed subset $Q_\alpha$ of $Q$ as in the proof of (1). 
Put $\widehat{Q_\alpha}:=\{\hat{q}\;\colon\; q \in Q_\alpha\}$. This set is closed in $\hat{Q}:=\{\hat{q}\;\colon\; q \in Q\}$. 

Observe that every $\hat{p} \in \widehat{\max(P)}$ belongs to $\widehat{Q_\alpha}$. So if  $\widehat{\max(P)}_{\Delta_\alpha}$ was infinite, we would get that $R_{\Delta_\alpha}(\widehat{Q_\alpha}) > N_\alpha$. Then any type $\hat{q} \in \widehat{Q_\alpha}$ with $R_{\Delta_\alpha}(\hat{q}) =R_{\Delta_\alpha}(\widehat{Q_\alpha})$ would contradict the definition of $N_\alpha$. 
\end{proof}

Let 
$$G:=\{ \sigma \in \aut(\FC)\;\colon\; \sigma(\widehat{\max(P)}) =\widehat{\max(P)}\},$$
$$G_\Delta:=\{ \sigma \in \aut(\FC)\;\colon\; \sigma(\widehat{\max(P)}_\Delta) =\widehat{\max(P)}_\Delta\}.$$
These are clearly subgroups of $\aut(\FC)$.

\begin{lemma}\phantomsection\label{lemma: relative definability of G}
\begin{enumerate}
\item $G:=\bigcap_{\Delta} G_\Delta$, where $\Delta$ ranges over all finite sets of formulas (in the object variables $\bar x$ and any parameter variables).
\item Each $G_\Delta$ is a relatively $\bar m$-definable over $M$ subgroup of $\aut(\FC)$, and $G$ is relatively $\bar m$-type-definable over $M$ subgroup of $\aut(\FC)$.
\item $G_\Delta:=\{ \sigma \in \aut(\FC)\;\colon\; \sigma(\widehat{\max(P)}_\Delta) \subseteq \widehat{\max(P)}_\Delta\}$.
\item $G=\{ \sigma \in \aut(\FC)\;\colon\; \sigma(\widehat{\max(P)}) \subseteq \widehat{\max(P)}\}$.
\end{enumerate}
\end{lemma}

\begin{proof}
(1) It follows easily from the fact that $\widehat{\max(P)} \cong \varprojlim_\Delta \widehat{\max(P)}_\Delta$ (which we have by Lemma \ref{lemma: basic on max(P)})
and automorphisms of $\FC$ commute with taking restrictions to $\Delta$.

(2) Let $\Delta:=\{\varphi_0(\bar x;\bar y),\dots, \varphi_{k-1}(\bar x;\bar y)\}$. 
By Lemma \ref{lemma: basic on max(P)}(2), 
$\widehat{\max(P)}_\Delta =\{\widehat{p_0}|_\Delta,\ldots,\widehat{p_{n-1}}|_\Delta\}$ for some $p_0,\dots,p_{n-1} \in \max(P)$. Then 
$$G_\Delta= \bigcup_{\sigma \in \Sym(n)} \bigcap_{i<n}\bigcap_{j<k} \{ \tau \in \aut(\FC)\;\colon\; \models \tau(d_{p_i} \varphi_j(\bar y)) \leftrightarrow d_{p_{\sigma(i)}} \varphi_j(\bar y)\},$$
which is clearly a finite union of a finite intersection of relatively $\bar m$-definable over $M$ subsets of $\aut(\FC)$, and so it is relatively $\bar m$-definable over $M$. Hence, $G$ is relatively $\bar m$-type-definable over $M$ by item (1).

(3) follows from Lemma \ref{lemma: basic on max(P)}(2).

(4) Assume that $\sigma(\widehat{\max(P)}) \subseteq \widehat{\max(P)}$. Then $\sigma(\widehat{\max(P)}_\Delta) \subseteq \widehat{\max(P)}_\Delta$ for every finite $\Delta$. Hence, by (3), $\sigma(\widehat{\max(P)}_\Delta) = \widehat{\max(P)}_\Delta$, i.e. $\sigma \in G_\Delta$, for every finite $\Delta$.  So, by virtue of (1), we conclude that $\sigma \in G$.
\end{proof}

By Lemma \ref{lemma: relative definability of G}, 
$G$ is relatively $\bar{m}$-type definable over $M$.
Choose $\pi(\bar{x};\bar{y})$ a partial type over $\emptyset$
such that $\pi(\bar{x};\bar{y})$ implies $\bar{x}\equiv_{\emptyset}\bar{y}$ and
$G=G_{\pi,\FC}=\{\sigma\in\aut(\FC)\;\colon\;\models\pi(\sigma(\bar{m});\bar{m})\}$.
Recall that
$$\widetilde{G}:=\{p(\bar x) \in S_{\bar m}(M): \pi(\bar x;\bar m) \subseteq p(\bar x)\},$$
$$A_P:=\{ \sigma \in \aut(\FC): \tp(\sigma(\bar m)/M) \in P\}.$$

\begin{lemma}\label{lemma: A_P contained in D}
$A_P \subseteq G$. Equivalently, $P \subseteq \widetilde{G}$.
\end{lemma}

\begin{proof}
The equivalence of $A_P \subseteq G$ with $P \subseteq \widetilde{G}$ is obvious. So it is enough to prove the first inclusion.
Take any $\sigma \in A_P$. Then $q:=\tp(\sigma(\bar m)/M) \in P$.  

Consider any $p \in \max(P)$.  Since all $R_\Delta$'s are invariant under $\aut(\FC)$, we have $R_\Delta(\sigma(\hat{p}))=R_\Delta(\hat{p})$ for all finite $\Delta$. On the other hand, $\sigma(\hat{p})|_M=p*q \in Q$ (as $p,q \in Q$ and using Remark \ref{remark: Q closed under *}). Therefore, since $p \in \max(P)$, if $\sigma(\hat{p}) \notin \widehat{\max(P)}$, we would get that $\sigma (\hat{p})$ forks over $M$. Then $\vec{R}(\sigma(\hat{p}))< \vec{R}(\sigma(\hat{p})|_M)$, so 
$\vec{R}(\hat{p})=\vec{R}(\sigma(\hat{p}))< \vec{R}(\widehat{\sigma(\hat{p})|_M})$. Since $\sigma(\hat{p})|_M \in Q$, this would contradict the fact that $p \in \max(P)$. 

We have proved that $\sigma(\widehat{\max(P)}) \subseteq \widehat{\max(P)}$, so $\sigma \in G$ by Lemma \ref{lemma: relative definability of G}(4).
\end{proof}

\begin{cor}\label{corollary: Q contained in Gen}
$Q \subseteq \widetilde{G}$.
\end{cor}

\begin{proof}
It follows from  Lemmas \ref{lemma: Phi is an isom.}(1) and \ref{lemma: A_P contained in D}.
\end{proof}

\begin{remark}\label{remark: very basic}
For every $\sigma \in G$ and $p \in \max(P)$, $\sigma(\hat{p}) = \widehat{p * q}$, where $q:=\tp(\sigma(\bar m)/M)$.
\end{remark}

\begin{proof}
This follows from the fact that $\sigma(\hat{p}) \in \widehat{\max(P)}$ and so $\sigma(\hat{p})$ does not fork over $M$, and $p*q =\sigma(\hat{p})|_M$.
\end{proof}

\begin{cor}\label{corollary: max(P)*r=max(P)}
For every $r \in \widetilde{G}$, $\max(P) * r = \max(P)$.
\end{cor}

\begin{proof}
$r=\tp(\sigma(\bar m)/M)$ for some $\sigma \in G$. Then $\sigma(\widehat{\max(P)}) =\widehat{\max(P)}$, so, by Remark \ref{remark: very basic}, we get $\widehat{\max(P) * r} = \widehat{\max(P)}$. Hence, $\max(P) * r = \max(P)$.
\end{proof}

\begin{proposition}\label{proposition: max(P) = Gen}
$\max(P) = \Gen(\widetilde{G})$.
\end{proposition}

\begin{proof}
By Corollary \ref{corollary: max(P)*r=max(P)} and the fact that $\Gen(\widetilde{G})$ is a left ideal in $\widetilde{G}$ (see Corollary \ref{corollary: Gen is an ideal and group}(4)), we get $\max(P) \subseteq \Gen(\widetilde{G})$. Hence, by Corollary \ref{corollary: max(P)*r=max(P)}, we get that $\max(P)$ is a right ideal in $\Gen(\widetilde{G})$. But $\Gen(\widetilde{G})$ is a group (by Corollary \ref{corollary: Gen is an ideal and group}(2)), so $\max(P)=\Gen(\widetilde{G})$.
\end{proof}

\begin{proposition}\label{proposition: G is the smallest relatively def}
$G$ is the smallest relatively $\bar m$-type-definable over $M$ subgroup of $\aut(\FC)$ containing $A_P$.
\end{proposition}

\begin{proof}
By Lemma \ref{lemma: A_P contained in D}, $A_P \subseteq G$. 

\begin{clm}
$A_{P^n} \subseteq A_P^n$, where $P^n = P* \dots * P$ and $A_P^n =A_P \circ \dots \circ A_P$ (both $n$-times). 
\end{clm}

\begin{clmproof}
The proof is by induction on $n$. The base step is trivial. Suppose the conclusion holds for $n$.
Consider any $\rho \in A_{P^{n+1}}$, i.e. $r:=\tp(\rho(\bar m)/M) \in P^n*P$. Then $r=p*q$ for some $p\in P^n$ and $q \in P$. So $q=\tp(\sigma(\bar m)/M)$ for some $\sigma \in A_P$. It is easy to construct a small $N$ so that $M\preceq N \preceq \FC$ and $\sigma[N]=N$. Let $p_N = \tp(\tau(\bar m)/N)$ be the unique nonforking extension of $p$ (for some $\tau \in \aut(\FC)$). Then $p*q = \sigma(p_N)|_M= \tp(\sigma(\tau(\bar m))/M)$. Since $p \in P^n$, we have that $\tau \in A_{P^n}$, so $\tau \in A_P^n$ by induction hypothesis. Hence, $\sigma \tau \in A_P^{n+1}$. As $\rho(\bar m) \equiv_M (\sigma\tau)(\bar m)$, we can find $\eta \in \aut(\FC/M)$ such that $\eta(\rho(\bar m))=(\sigma\tau)(\bar m)$. Then $f :=\tau^{-1}\sigma^{-1}\eta\rho \in \aut(\FC/M)$ and $\rho = \eta^{-1}(\sigma \tau) f \in A_P^{n+1}$ (as $\aut(\FC/M)  A_P  \aut(\FC/M) \subseteq A_P$).
\end{clmproof}

By the claim, we get $A_{*P} \subseteq \langle A_P \rangle$, where $*P$ is the closure of $P$ under $*$ and $\langle A_p \rangle$ is the subgroup of $\aut(\FC)$ generated by $A_P$. Hence, since $Q=\cl(*P)$, we conclude that $A_Q$ is contained in every relatively $\bar m$-type-definable over $M$ subgroup of $\aut(\FC)$ containing $A_P$.
On the other hand, by Proposition \ref{proposition: max(P) = Gen}, $\Gen(\widetilde{G})=\max(P) \subseteq Q$, and hence, $A_{\Gen(\widetilde{G})} \subseteq A_Q$. 
So in order to finish the proof, it is enough to show that every $\sigma \in G$ belongs to $A_{\Gen(\widetilde{G})} A_{\Gen(\widetilde{G})}^{-1}$.

Again, let $q:=\tp(\sigma(\bar m)/M)$ and pick a small $N$ so that $M\preceq N \preceq \FC$ and $\sigma[N]=N$. Take $p \in \Gen(\widetilde{G})$, and let $p_N= \tp(\tau(\bar m)/N)$ be the unique nonforking extension of $p$. Then $p*q = \sigma(p_N)|_M= \tp(\sigma(\tau(\bar m))/M) \in \Gen(\widetilde{G})$, as $\Gen(\widetilde{G})$ is a right ideal in $\widetilde{G}$. Hence, $\sigma \tau \in A_{\Gen(\widetilde{G})}$. But since $p \in \Gen(\widetilde{G})$, we also have $\tau \in A_{\Gen(\widetilde{G})}$. Therefore, $\sigma=(\sigma \tau) \tau^{-1} \in A_{\Gen(\widetilde{G})}A_{\Gen(\widetilde{G})}^{-1}$
\end{proof}

\begin{lemma}\label{lemma: description of G}
$G = \{\sigma \in \aut(\FC)\;\colon\; \max(P) * \tp(\sigma(\bar m)/M) = \max(P)\}$.
\end{lemma}

\begin{proof}
This is equivalent to the statement that for every $r \in S_{\bar m}(M)$
$$\max(P) * r = \max(P) \iff r \in \widetilde{G}.$$
The implication ($\Leftarrow$) is precisely Corollary \ref{corollary: max(P)*r=max(P)}.

($\Rightarrow$) Assume that $\max(P) * r = \max(P)$, where $r=\tp(\sigma(\bar m)/M)$ for some $\sigma \in \aut(\FC)$. By the definition of $*$,  $\max(P) * r =\sigma(\widehat{\max(P)})|_M$. Therefore, $\sigma(\widehat{\max(P)})|_M = \max(P)$. 

By Lemma \ref{lemma: relative definability of G}(4), it is enough to show that $\sigma(\widehat{\max(P)}) \subseteq \widehat{\max(P)}$. So take any $p \in \max(P)$ and suppose for a contradiction that $\sigma(\hat{p}) \notin \widehat{\max(P)}$. Then, since $\sigma(\hat{p})|_M \in \max(P)$, we get that $\sigma(\hat{p})$ forks over $M$. But then $\vec{R}(\hat{p})=\vec{R}(\sigma(\hat{p})) < \vec{R}(\sigma(\hat{p})|_M)= \vec{R}(\widehat{\sigma(\hat{p})|_M})$. 
As  $\sigma(\hat{p})|_M \in \max(P)\subseteq Q$, this contradicts the fact that $p \in \max(P)$.
\end{proof}

\begin{proof}[Proof of Theorem \ref{theorem: counterpart of Newelski's theorem}]
It follows directly from Lemma \ref{lemma: description of G}, Propositions \ref{proposition: max(P) = Gen}, \ref{proposition: G is the smallest relatively def}, and Corollaries \ref{corollary: Gen is an ideal and group}, 
\ref{corollary: Q contained in Gen}, modulo one detail. Namely, the definitions of $\gen(P)$ in Theorem \ref{theorem: counterpart of Newelski's theorem} and $\max(P)$ are different, because the first one is with respect to the product order on $\{\vec{R}(q)\;\colon\; q\in Q\}$ whereas the second one with respect to the lexicographic order. 
This can be resolved as follows. 
By Corollary  \ref{corollary: Q contained in Gen}, $Q \subseteq \widetilde{G}$.
By Proposition \ref{proposition: max(P) = Gen}, $\Gen(\widetilde{G})=\max(P)\subseteq Q$. Therefore, using Proposition \ref{proposition: generic via ranks for x'}, we obtain $\max(P) = \gen(P)$.
\end{proof}

\section{On classification of idempotent fim measures and generically stable types II - stable case}\label{sec: 0.7 for stable}

In this section, we apply the generalized stable group theory from Section \ref{sec:group chunk} to 
prove Conjecture \ref{conjecture: main conjecture} in the context of stable theories. 

As usual, we let $\FC$ be a monster model of $T$ and let $M\preceq\FC$ be small and enumerated by $\bar{m}$. Let $\bar{y}$ be a tuple of variables corresponding to the enumeration $\bar{m}$, and let $\bar{x}$ be its copy.
Measures in $\mathfrak{M}_{\bar m}(\FC)$ will be in variables $\bar y$.

Recall that a function from $M^n$ to $[0,1]$ is said to be {\em definable} if the premiages of any two closed disjoint subsets of $[0,1]$ can be separated by a definable set. Each such function extends uniquely to an $M$-definable function from $\FC^n$ to $[0,1]$ in the sense that the preimage of any closed subset of $[0,1]$ is $M$-type-definable. Conversely, the restriction to $M^n$ of any $M$-definable map from $\FC^n$ to $[0,1]$ is definable. See \cite[Lemma 3.2]{Gis-Pen-Pil} for these basic facts.

In Definition \ref{cheat}(4), we recalled what it means that a global measure $\mu$ is definable over $M$. This is equivalent to saying that for every formula $\varphi(\bar y;\bar z)$, the map $\bar b \mapsto \mu(\varphi(\bar y;\bar b))$ is $M$-definable in the above sense. Now, we say that a measure $\mu \in \mathfrak{M}_{\bar y}(M)$ is {\em definable} if for every formula $\varphi(\bar y;\bar z)$, the assignment $\bar b \mapsto \mu(\varphi(\bar y;\bar b))$ is a definable map from $M^{\bar z}$ to $[0;1]$. 

From the above discussion, it is easy to see that any definable measure $\mu \in \mathfrak{M}_{\bar y}(M)$ has a unique extension to a global Keisler measure definable over $M$.

The following fundamental fact on Keisler measures in stable theories essentially follows from \cite{Keisler1}, and is in the background of the main results of this section.

\begin{fact}\label{fact: basic fact on Keisler measure sin stable theories} The following statements are true. 
\begin{enumerate}
\item Each measure in $\mathfrak{M}_{\bar{m}}(M)$ is definable.
\item  Each measure in $\mathfrak{M}^{\inv}_{\bar{m}}(\FC,M)$ is definable over $M$.
\item Each measure in  $\mathfrak{M}_{\bar{m}}(M)$ has a unique extension to an $M$-definable measure in  $\mathfrak{M}^{\inv}_{\bar{m}} (\FC,M)$ which coincides with a unique extension to an ($M$-invariant) measure in  $\mathfrak{M}^{\inv}_{\bar{m}} (\FC,M)$.
\end{enumerate}
\end{fact}

\begin{proof}
Item (1) follows from approximation of measures by types (e.g., see \cite[Theorem 2.8]{Con-Ter}) and definability of types. By the same reason, each global measure is definable over some small model $N \preceq \FC$. If this measure is additionally invariant over $M$, then it must be definable over $M$, so we have item (2). The existence of a unique global $M$-definable extension follows from (1) and the above discussion; thus, the uniqueness of a global $M$-invariant extension follows from item (2).
\end{proof}

\subsection{Supports of idempotent Keisler measures} 
We prove several general results connecting measures and their support. So, in this subsection, $T$ is an arbitrary theory. The main result of this subsection says that if $\mu \in \mathfrak{M}^{\inv}_{\bar{m}}(\mathfrak{C},M)$ is $M$-definable, idempotent, and $M$-invariantly supported, then $(\supp(\mu),*)$ is a compact, Hausdorff, left-continuous semigroup without any closed two-sided ideals. 

\begin{proposition}\label{prop:support} Suppose that $\mu,\nu \in \mathfrak{M}^{\inv}_{\bar{m}}(\mathfrak{C},M)$. If $\mu$ is $M$-definable and $M$-invariantly supported, then $\supp(\mu) * \supp(\nu) \subseteq \supp(\mu *\nu)$. 
\end{proposition}

\begin{proof} Suppose that $p \in \supp(\mu)$ and $q \in \supp(\nu)$. Take a formula $\theta(\bar{b};\bar{y}) \in p * q$. It suffices to prove that $(\mu * \nu)(\theta(\bar{b};\bar{y})) > 0$. 

First, we claim that $\{h_{\bar{b}}(q)\;\colon\; q \in \supp(\nu)\} \subseteq \supp((h_{\bar{b}})_* (\nu))$. 
In order to see it, consider any $\theta(\bar{x};\bar{m})\in h_{\bar{b}}(q)\in S_{\bar{x}}(M)$,
where $q\in\supp(\nu)\subseteq S_{\bar{m}}(\FC)$.
As $h_{\bar{b}}(q)\in[\theta(\bar{x};\bar{m})]$, 
we have $q\in h_{\bar{b}}^{-1}[\theta(\bar{x};\bar{m})]=[\theta(\bar{b};\bar{y})]$. Then $0<\nu(\theta(\bar{b};\bar{y}))=((h_{\bar{b}})_\ast(\nu))(\theta(\bar{x};\bar{m}))$.

Now assume that $\theta(\bar{b};\bar{y}) \in p * q$. Then $\theta(\bar{x};\bar{y}) \in p_{\bar{y}} \otimes (h_{\bar{b}}(q))_{\bar{x}}$, and if $\bar{d} \models h_{\bar{b}}(q)$, we have that $\theta(\bar{d};\bar{y}) \in p$. 
Since $p \in \supp(\mu)$, this implies that $\mu(\theta(\bar{d};\bar{y})) > 0$ and so $F_{\mu}^{\theta^{\textrm{opp}}(\bar y;\bar x)}(h_{\bar{b}}(q)) > 0$. 
The integral $(\mu * \nu)(\theta(\bar{b};\bar{y}))=\int_{S_{\bar x}(M)} F_{\mu}^{\theta^{\textrm{opp}}(\bar y;\bar x)} d((h_{\bar{b}})_*(\nu))_{\bar x}$ is greater than $0$, since our map $F_{\mu}^{\theta^{\textrm{opp}}(\bar {y};\bar{x})}$ is continuous (as $\mu$ is $M$-definable) and greater than $0$ at some value in the support (as $h_{\bar{b}}(q) \in \supp((h_{\bar{b}})_{*}(\nu))$).
\end{proof}

\begin{proposition}\label{prop:product} Suppose that $\mu,\nu \in \mathfrak{M}^{\inv}_{\bar{m}}(\mathfrak{C},M)$ and $\mu$ is Borel-definable and $M$-invariantly supported. 
Then $\supp(\mu * \nu)$ is contained in the closure of $\{ p * q \;\colon\; p \in \supp(\mu), \,q \in \supp(\nu)\}$. Moreover, if $\mu$ is $M$-definable, then $\{ p * q \;\colon\; p \in \supp(\mu), \,q \in \supp(\nu)\}$ is a dense subset of $\supp(\mu * \nu)$.
\end{proposition}

\begin{proof}
Consider $\theta(\bar{b};\bar{y})\in r(\bar{y})\in\supp(\mu\ast\nu)$.
We have $(\mu * \nu )(\theta(\bar{b};\bar{y})) > 0$ and 
we want to show that there exist some $p \in \supp(\mu)$ and $q \in \supp(\nu)$ such that 
$\theta(\bar{b};\bar{y})\in p * q$. Note that 
\begin{equation*}
    0<(\mu * \nu )(\theta(\bar{b};\bar{y}))= \int_{S_{\bar{m}}(\mathfrak{C})} F_{\mu}^{\theta^{\textrm{opp}}(\bar y;\bar x)} \circ h_{\bar{b}}\, d \nu  
\end{equation*}
So, there exists some $q \in \supp(\nu)$  such that $(F_{\mu}^{\theta^{\textrm{opp}}(\bar y;\bar x)} \circ h_{\bar{b}})(q) > 0$. Hence, we have that $\mu(\theta(\bar{c};\bar{y})) > 0$ for some $\bar{c} \models h_{\bar{b}}(q)$. Now, there exists some $p \in \supp(\mu)$ such that $\theta(\bar{c};\bar{y}) \in p$, and that is equivalent to
$\theta(\bar{x};\bar{y})\in p\otimes h_{\bar{b}}(q)$.
By definition, $\theta(\bar{b};\bar{y}) \in p * q$. 

The moreover part follows from the first part and Proposition \ref{prop:support}.
\end{proof}

\begin{proposition}\label{prop:product2} Assume that $\mu \in \mathfrak{M}_{\bar{m}}^{\inv}(\mathfrak{C},M)$,
$\mu$ is idempotent, $M$-definable, and $M$-invariantly supported. Then $(\supp(\mu),*)$ is a compact Hausdorff left-continuous semigroup. 
\end{proposition}

\begin{proof} By Proposition \ref{prop:support}, for any $p, q\in \supp(\mu), p * q \in \supp(\mu)$. 
The product is associative on $M$-invariant types by Proposition \ref{prop: associative for types}. Moreover, 
the map $-*q$ is continuous by Proposition \ref{prop:cont_types}.
\end{proof}

\begin{cor}\label{cor:support} Assume that $\mu \in \mathfrak{M}_{\bar{m}}^{\inv}(\mathfrak{C},M)$. If $\mu$ is $M$-definable and $M$-invariantly supported, 
then for any $q \in S_{\bar{m}}^{\inv}(\mathfrak{C},M)$, we have that $\supp(\mu * q) = \supp(\mu) * q$. Moreover, if $q$ is definable, then $\supp(q * \mu) = q*\supp(\mu)$. 
\end{cor}

\begin{proof} We know that $\supp(\mu) * q = \supp(\mu) * \supp(q) \subseteq \supp(\mu * q)$ by Proposition \ref{prop:support}. Since $-*q$ is continuous and $\supp(\mu)$ is compact, we have that $\supp(\mu) * q$ is a closed subset of $\supp(\mu * q)$. 
On the other hand, by Proposition \ref{prop:product},
$\supp(\mu) * q$ is a dense subset of $\supp(\mu * q)$. 
Thus, $\supp(\mu * q) = \supp(\mu) * q$.

The moreover part follows by a symmetric argument, using the fact that the map $q*-$ is continuous under the assumption that $q$ is $M$-definable.
\end{proof}

\begin{lemma}\label{lemma:comp} Fix $\bar{b} \in \mathfrak{C}^{\bar{y}}$ and $q \in S_{\bar{m}}^{\inv}(\mathfrak{C},M)$. Let $\tp(\bar{c}/M) = h_{\bar{b}}(q)$. Then for any $t \in S_{\bar{m}}^{\inv}(\mathfrak{C},M)$, we have that $h_{\bar{c}}(t) = h_{\bar{b}}(t * q)$. 
\end{lemma}

\begin{proof} Fix $\theta(\bar{x};\bar{y}) \in \mathcal{L}$.
Then $\theta(\bar{x};\bar{m}) \in h_{\bar{b}}(t * q)$ 
if and only if $\theta(\bar{b};\bar{y}) \in t * q$. 
This is true if and only if $\theta(\bar{x};\bar{y}) \in t_{\bar{y}} \otimes (h_{\bar{b}}(q))_{\bar{x}}$ which is true if and only if $\theta(\bar{c};\bar{y}) \in t$. But this is true if and only if $\theta(\bar{x};\bar{m}) \in h_{\bar{c}}(t)$. 
\end{proof}

\begin{proposition}\label{prop:D-map-max} 
Let $\mu\in\mathfrak{M}^{\inv}_{\bar{m}}(\FC,M)$ be Borel-definable over $M$.
For any $ \mathcal{L}$-formula $\varphi(\bar{x};\bar{y})$ and $\bar{b}\in\FC^{\bar{x}}$, we define the map $D_{\mu}^{\varphi_{\bar{b}}}:S_{\bar{m}}^{\inv}(\mathfrak{C},M) \to [0,1]$ via 
\begin{equation*}
    D_{\mu}^{\varphi_{\bar{b}}}(q) = (\mu_{\bar{y}} \otimes (h_{\bar{b}}(q))_{\bar{x}}) (\varphi(\bar{x};\bar{y}))=(\mu\ast q)(\varphi(\bar{b};\bar{y})). 
\end{equation*}
If $\mu$ is $M$-definable, then the map $D_{\mu}^{\varphi_{\bar{b}}}$ is continuous. Suppose that $D_{\mu}^{\varphi_{{\bar{b}}}}|_{\supp(\mu)}$ achieves a maximum at $q_{*}$, say $\delta$. If $\mu$ is $M$-definable, idempotent, and $M$-invariantly supported, then for any $t \in \supp(\mu)$, $D_{\mu}^{\varphi_{{\bar{b}}}}(t * q_{*}) = D_{\mu}^{\varphi_{{\bar{b}}}}(q_{*}) = \delta$. 
\end{proposition}

\begin{proof} 
If $\mu$ is $M$-definable, then $F_{\mu}^{\varphi^{\textrm{opp}}(\bar y;\bar x)}$ is continuous, so continuity of  $D_{\mu}^{\varphi_{{\bar{b}}}}$ follows from the fact that $D_{\mu}^{\varphi_{{\bar{b}}}} = F_{\mu}^{\varphi^{\textrm{opp}}(\bar y;\bar x)} \circ h_{\bar b}$.

Now, consider $\varphi(\bar{b};\bar{y})$, $\mu$ and $q_{*}$ as in the statement. Assume that $\tp(\bar{c}/M) = h_{\bar{b}}(q_*)$. 
Then 
\begin{align*}
    \delta &= D_{\mu}^{\varphi_{{\bar{b}}}}(q_*) = (\mu_{\bar {y}} \otimes (h_{{\bar{b}}}(q_*))_{\bar{x}}) (\varphi(\bar{x};\bar{y})) = \mu(\varphi(\bar{c};\bar{y})) \\ &= (\mu * \mu)(\varphi(\bar{c};\bar{y})) = \int_{\supp(\mu)} (F_{\mu}^{\varphi^{\textrm{opp}}(\bar y;\bar x)} \circ h_{\bar{c}}) \,d\mu. 
\end{align*}
Now, the map $F_{\mu}^{\varphi^{\textrm{opp}}(\bar y;\bar x)} \circ h_{\bar{c}} : \supp(\mu) \to [0,1]$ is continuous and bounded by $\delta$, since 
\begin{equation*}
    (F_{\mu}^{\varphi^{\textrm{opp}}(\bar y;\bar x)} \circ h_{\bar{c}})(t) = F_{\mu}^{\varphi^{\textrm{opp}}(\bar y;\bar x)}(h_{\bar{b}}(t * q_{*})) = D_{\mu}^{\varphi_{{\bar{b}}}}(t * q_*) \leqslant D_{\mu}^{\varphi_{{\bar{b}}}}(q_*) = \delta,
\end{equation*}
where the first equality holds by Lemma \ref{lemma:comp} and the inequality holds by Proposition \ref{prop:product2} and the choice of $q_*$. Thus, it follows that $(F_{\mu}^{\varphi^{\textrm{opp}}(\bar y;\bar x)} \circ h_{\bar{c}})(t) = \delta$ for every $t \in \supp(\mu)$ (otherwise the above integral will be strictly less than $\delta$). And so, for any $t \in \supp(\mu)$, 
the above inequality becomes equality, i.e. $D_{\mu}^{\varphi_{{\bar{b}}}}(t * q_*) = D_{\mu}^{\varphi_{{\bar{b}}}}(q_*)$.
\end{proof}

\begin{theorem}\label{thm:two} 
Let $\mu \in \mathfrak{M}_{\bar{m}}^{\inv}(\mathfrak{C},M)$. 
Assume that $\mu$ is $M$-invariantly supported, idempotent, and $M$-definable. 
Then $(\supp(\mu),*)$ has no proper closed two-sided ideals. 
\end{theorem}

\begin{proof} Suppose that $I \subseteq \supp(\mu)$ such that $I$ is a closed two-sided ideal. Notice that if $I$ is dense in $\supp(\mu)$, then $I = \supp(\mu)$. Hence it suffices to assume that $I$ is not dense. So there exists a formula $\varphi(\bar{b};\bar{y}) \in \mathcal{L}_{\bar{y}}(\mathfrak{C})$, with $\bar{b}\in\FC^{\bar{x}}$,
such that $[\varphi(\bar{b};\bar{y})] \cap I = \emptyset$ and $[\varphi(\bar{b};\bar{y})] \cap \supp(\mu) \neq \emptyset$. 

\begin{clm}
    There exists some $q \in \supp(\mu)$ such that $D_{\mu}^{\varphi_{\bar{b}}}(q) > 0$.
\end{clm}

\begin{clmproof}
 Suppose not. Then for every $q \in \supp(\mu)$, we have that $D_{\mu}^{\varphi_{\bar{b}}}(q) = 0$. We now argue that the indicator function $\chi_{[\varphi(\bar{b};\bar{y})]}$ restricted to the set $\supp(\mu) * \supp(\mu) \subseteq \supp(\mu)$ is always $0$.
 Indeed, if $p,q\in\supp(\mu)$ and $\varphi(\bar{b};\bar{y})\in p\ast q$,
 then $D^{\varphi_{\bar{b}}}_\mu(q)=(\mu\ast q)(\varphi(\bar{b};\bar{y}))>0$.
 Then, by Proposition \ref{prop:product} we conclude that $\chi_{[\varphi(\bar{b};\bar{y})]}$ restricted to $\supp(\mu)$ is always $0$, but this contradicts $[\varphi(\bar{b};\bar{y})] \cap \supp(\mu) \neq \emptyset$. 
\end{clmproof}

\begin{clm}
    For every $t \in I$, we have $D_{\mu}^{\varphi_{\bar{b}}}(t) = 0$. 
\end{clm}

\begin{clmproof}
Let $\tp(\bar{c}/M) = h_{\bar{b}}(t)$, then we notice
\begin{equation*}
    D_{\mu}^{\varphi_{\bar{b}}}(t) = (\mu \otimes h_{\bar{b}}(t))(\varphi(\bar{x};\bar{y})) = \mu(\varphi(\bar{c};\bar{y})) 
    = \mu(\{ r \in \supp(\mu)\;\colon\; \varphi(\bar{b};\bar{y}) \in r * t\}). 
\end{equation*}
Indeed, $[\varphi(\bar{c};\bar{y})] = \{r \in \supp(\mu) \;\colon\; \varphi(\bar{b};\bar{y}) \in r * t\}$, since $\varphi(\bar{b};\bar{y}) \in r * t$ if and only if $\varphi(\bar{x};\bar{y}) \in r \otimes h_{\bar{b}}(t)$,
if and only if $\varphi(\bar{c};\bar{y}) \in r$. 
Now, $\supp(\mu) * t \subseteq I$ and since $[\varphi(\bar{b};\bar{y})] \cap I = \emptyset$, we have that $D_{\mu}^{\varphi_{\bar{b}}}(t) = 0$. 
\end{clmproof}

Finally, choose $r \in \supp(\mu)$ such that $D_{\mu}^{\varphi_{\bar{b}}}(r)$ is a maximum of $D_{\mu}^{\varphi_{\bar{b}}}$ on $\supp(\mu)$ 
(which exists by continuity of $D_{\mu}^{\varphi_{\bar{b}}}$). By Claim 1, we have that $D_{\mu}^{\varphi_{b}}(r) > 0$.
Since $I$ is a two sided ideal, we have that $I$ is a right ideal, and so for any $t \in I$, 
we have $t* r \in I$. Hence, by Proposition \ref{prop:D-map-max} and Claim 2 we get that 
\begin{equation*}
    0 < D_{\mu}^{\varphi_{\bar{b}}}(r) = D_{\mu}^{\varphi_{\bar{b}}}(t * r) = 0,
\end{equation*}
which is absurd. 
\end{proof}

\begin{proposition}\label{prop:min} 
Let $\mu \in \mathfrak{M}_{\bar{m}}^{\inv}(\mathfrak{C},M)$. 
Assume that $\mu$ is $M$-invariantly supported, idempotent, $M$-definable, and minimal (i.e. $\supp(\mu)$ is the unique minimal left ideal of $(\supp(\mu),*)$). Then for every $\varphi(\bar{b};\bar{y}) \in \mathcal{L}_{\bar{x}}(\mathfrak{C})$, where $\bar{b}\in\FC^{\bar{x}}$, we have that $D_{\mu}^{\varphi_{\bar{b}}}(p) = D_{\mu}^{\varphi_{\bar{b}}}(q)$ for all $p,q \in \supp(\mu)$. 
\end{proposition} 

\begin{proof} Suppose that $D_{\mu}^{\varphi_{\bar{b}}}|_{\supp(\mu)}$ attains a maximum at $q_*$. 
Consider any $p\in \supp(\mu)$.
By minimality, there exists some $r \in \supp(\mu)$ such that $r * q_* = p$. So, $D_{\mu}^{\varphi_{\bar{b}}}(p) = D_{\mu}^{\varphi_{\bar{b}}}(r * q_*) = D_{\mu}^{\varphi_{\bar{b}}}(q_*)$ by Proposition \ref{prop:D-map-max}. 
\end{proof}

\begin{proposition}\label{prop:inv} 
Let $\mu \in \mathfrak{M}_{\bar{m}}^{\inv}(\mathfrak{C},M)$. 
Assume that $\mu$ is $M$-invariantly supported, idempotent, $M$-definable, and minimal. 
Then for any $\varphi(\bar{b};\bar{y}) \in \mathcal{L}_{\bar{y}}(\mathfrak{C})$, where $\bar{b}\in\FC^{\bar{x}}$,
and $p \in \supp(\mu)$, we have that $\mu(\varphi(\bar{b};\bar{y})) = (\mu*p)(\varphi(\bar{b};\bar{y}))$. 
\end{proposition}

\begin{proof} 

Suppose for a contradiction that $\mu(\varphi(\bar{b};\bar{y})) \neq (\mu * p)(\varphi(\bar{b};\bar{y}))$. Since the right hand side equals $D_{\mu}^{\varphi_{\bar{b}}}(p)$, by Proposition \ref{prop:min}, we get $\mu(\varphi(\bar{b};\bar{y})) \neq D_{\mu}^{\varphi_{\bar{b}}}(u)$ for any/some idempotent $u \in \supp(\mu)$ (note that the existence of an idempotent follows from Ellis theorem, as $*$ is left continuous and associative on $\supp(\mu)$). 
Consider the case $\mu(\varphi(\bar{b};\bar{y})) > D_{\mu}^{\varphi_{\bar{b}}}(u)$ (the case of the opposite inequality is analogous). 
Let $\tp(\bar{c}/M) = h_{\bar{b}}(u)$; then $D_{\mu}^{\varphi_{\bar{b}}}(u) = \mu(\varphi(\bar{c};\bar{y}))$. 
Thus, $\mu(\varphi(\bar{b};\bar{y}) \wedge \neg \varphi(\bar{c};\bar{y})) > 0$. 

So there exists some $t \in \supp(\mu)$ such that $\varphi(\bar{b};\bar{y}) \wedge \neg \varphi(\bar{c};\bar{y}) \in t$.
Notice that since $\neg \varphi(\bar{c},\bar{y}) \in t$, this implies that $\neg \varphi(\bar{x};\bar{y}) \in t \otimes h_{\bar{b}}(u)$, which implies that $\neg \varphi(\bar{b};\bar{y}) \in t * u$. By minimality of $\supp(\mu)$, $t*u = t$ and so $\neg \varphi(\bar{b};\bar{y}) \in t$,
a contradiction with $\varphi(\bar{b};\bar{y}) \in t$.
\end{proof}

\begin{remark}\label{remark: definability and invariance}
Suppose that $T$ is stable and $\mu \in \mathfrak{M}_{\bar{m}}^{\inv}(\mathfrak{C},M)$.
Then $\mu$ is $M$-definable and $M$-invariantly supported.
\end{remark}

\begin{proof} Definability follows from stability (see Fact \ref{fact: basic fact on Keisler measure sin stable theories}). Invariantly supported follows from NIP and Proposition \ref{prop:NIP}.
\end{proof}

In Proposition \ref{proposition: profiniteness of supp}, we will see that under stability, each $\mu \in \mathfrak{M}_{\bar{m}}^{\inv}(\mathfrak{C},M)$ is also minimal.

\subsection{Uniqueness of measures in stable context}
In this subsection, we assume that \textbf{$T$ is stable}
and we show uniqueness of ``$\ast$-invariant'' Keisler measures.

Since $T$ is stable, by Fact \ref{fact: basic fact on Keisler measure sin stable theories}, all measures in $\mathfrak{M}_{\bar{m}}(M)$ are definable, and so we can consider the semigroup $(\mathfrak{M}_{\bar{m}}(M),*)$. Formally, if $\mu \in \mathfrak{M}_{\bar{m}}(M)$, we let $\hat{\mu}$ be the unique 
$M$-definable extension in $\mathfrak{M}^{\inv}_{\bar{m}}(\mathfrak{C},M)$ and if $\mu, \nu \in \mathfrak{M}_{\bar{m}}(M)$, we define $\mu \otimes \nu = (\hat{\mu} \otimes \hat{\nu})|_{M}$ and similarly $\mu * \nu = (\hat{\mu} * \hat{\nu})|_{M}$. 

Thus, for every measure $\mu \in \mathfrak{M}_{\bar{m}}(M)$ and formula $\varphi(\bar{x};\bar{y}) \in \mathcal{L}(M)$, there exists a unique continuous function $F_{\mu}^{\varphi^{\textrm{opp}}(\bar y;\bar x)}: S_{\bar{x}}(M) \to [0,1]$ such that for every $\bar{b} \in M^{\bar{x}}$, 
 $F_{\mu}^{\varphi^{\textrm{opp}}(\bar y;\bar x)}(\tp(\bar{b}/M)) = 
\mu(\varphi(\bar{b};\bar{y}))$, 
which clearly coincides with $F_{\hat{\mu}}^{\varphi^{\textrm{opp}}(\bar y;\bar x)}$. 
We have
\begin{equation*}
(\mu \otimes \nu)(\varphi(\bar{x};\bar{y}))= (\hat{\mu} \otimes \hat{\nu})(\varphi(\bar{x};\bar{y})) = \int_{S_{\bar{x}}(M)} F_{\mu}^{\varphi^{\textrm{opp}}(\bar y;\bar x)} \,d\nu.
\end{equation*}
Likewise, for any formula $\varphi(\bar{x};\bar{y}) \in \mathcal{L}_{\bar{x},\bar{y}}$ 
we have that 
\begin{equation*}
    (\mu * \nu)(\varphi(\bar{b};\bar{y})) = \int_{S_{\bar{m}}(M)} F_{\mu}^{\varphi^{\textrm{opp}}(\bar y;\bar x)} \circ h_{\bar{m}} \,d\nu. 
\end{equation*}

Regarding the function $h_{\bar m}$ used above, recall that at the beginning of Subsection \ref{subsec: star new definition}, we defined the map  $h_{\bar{m}}\colon S_{\bar{m}}(\FC)\to S_{\bar{x}}(M)$ ($S_{\bar{m}}(\FC)$ considered in variables $\bar{y}$). This definition can be extended to any $N \succeq M$ in place of $\FC$. In particular, we can apply it to $N=M$. Namely, we obtain the map
$h_{\bar{m}}\colon S_{\bar{m}}(M)\to S_{\bar{m}}(M)$ taking types in $\bar{y}$ into types in $\bar{x}$, given by
 $h_{\bar{m}}(p(\bar{y})) = q(\bar{x})$,
where 
$p(\bar{y})=\tp(\sigma(\bar{m})/M)$, 
$q(\bar{x})=\tp(\sigma^{-1}(\bar{m})/M)$, and $\sigma\in\aut(\FC)$. 
We remark that this map is both a homeomorphism and an involution (after identifying variables $\bar{x}$ with $\bar{y}$ and their corresponding spaces of types over $M$) with 
$h_{\bar{m}}([\theta(\bar{m};\bar{y})]) = [\theta(\bar{x};\bar{m})]$.
Hence, $h_{\bar{m}}$ induces a map 
$(h_{\bar{m}})_{*}: \mathfrak{M}_{\bar{m}}(M) \to \mathfrak{M}_{\bar{m}}(M)$ via the standard pushforward. We write $\mu^{-1}:=(h_{\bar{m}})_{*}(\mu)$. In other words, for any formula $\theta(\bar{x};\bar{y}) \in \mathcal{L}$, we have that 
\begin{equation*}
    \mu^{-1}(\theta(\bar{x};\bar{m})) = \mu(\theta(\bar{m};\bar{y})). 
\end{equation*}

Let $G$ be a relatively $\bar{m}$-type-definable over $M$ 
subgroup of $\aut(\mathfrak{C})$. 
There is a partial type $\pi(\bar{x};\bar{y})$ over $\emptyset$
such that $\pi(\bar{x};\bar{y})\vdash\bar{x}\equiv_{\emptyset}\bar{y}$
and $G=G_{\pi,\FC}$. We consider the space $S_{\pi(\bar{m};\bar{y})}(M)$
and shortly denote it by $\widetilde{G} = \{\tp(\sigma(\bar{m})/M)\;\colon\; \sigma \in G\}$, avoiding the reference to $\pi$.

Fix $\mu \in \mathfrak{M}_{\bar{m}}(M)$.
We say that $\mu$ is
{\em $G$-$\ast$-right invariant} if for every $\sigma \in G$,  $\mu * \tp(\sigma(\bar{m})/M) = \mu$. 
Similarly, we say that $\mu$ is $G$-$\ast$-left invariant if for every $\sigma \in G$, $\tp(\sigma(\bar{m})/M) *\mu = \mu$. 

\begin{remark}\phantomsection\label{rem:inverse}
\begin{enumerate}
    \item $\supp(\mu^{-1}) = \{p^{-1}\;\colon\; p \in \supp(\mu)\}$. 
    \item 
$p \in \widetilde{G}$ if and only if $p^{-1} \in \widetilde{G}$. 
\end{enumerate}
\end{remark}

\begin{proof} Clear by definitions.
\end{proof}

\begin{lemma}\label{lemma:implies-inverse} 
Assume that
\begin{enumerate}
    \item $\mu \in \mathfrak{M}_{\bar{m}}(M)$ is $G$-$*$-right invariant,
    \item $\mu(\widetilde{G}) = 1$.
\end{enumerate}
Then $\mu = \mu^{-1}$ 
(after the identification of $\bar x$ with $\bar y$).
\end{lemma}

\begin{proof}
Note that, by Remark \ref{rem:inverse} and the second point in assumptions, we have $\supp(\mu^{-1}) \subseteq \widetilde{G}$.

\begin{clm}
    For every $\nu\in\mathfrak{M}_{\bar{m}}(M)$ with $\nu(\widetilde{G})=1$ we have
    $\mu * \nu = \mu$.
\end{clm}

\begin{clmproof}
      Fix an $\CL$-formula $\theta(\bar{x};\bar{y})$, our goal is $(\mu\ast\nu)(\theta(\bar{m};\bar{y}))=\mu(\theta(\bar{m};\bar{y}))$.
       For every $q \in \supp(\nu)\subseteq \widetilde{G}\subseteq S_{\bar{y}}(M)$, we have that
\begin{equation*}
     F_{\mu_{\bar{x}}}^{\theta^{\textrm{opp}}(\bar y;\bar x)}(h_{\bar{m}}(q)) = \big(\mu_{\bar{y}} \otimes (h_{\bar{m}}(q))_{\bar x}\big)\big(\theta(\bar{x};\bar{y})\big) 
     = (\mu * q)(\theta(\bar{m};\bar{y})) = \mu(\theta(\bar{m},\bar{y})). 
\end{equation*} 
\noindent Therefore, 
\begin{align*}
    (\mu * \nu)(\theta(\bar{m};\bar{y})) &= \int_{S_{\bar{m}}(M)} \big(F_{\mu}^{\theta^{\textrm{opp}}(\bar y;\bar x)} \circ h_{\bar{m}}\big)\, d\nu = \int_{S_{\bar{m}}(M)} \mu(\theta(\bar{m};\bar{y}))\, d\nu = \mu(\theta(\bar{m};\bar{y})). 
\end{align*}
\end{clmproof}
\noindent
In particular, for every formula $\theta(\bar{x};\bar{y})\in\CL$, we have: 
$$(\mu_{\bar{y}}\ast(\mu^{-1})_{\bar y})(\theta(\bar{m};\bar{y}))=\mu(\theta(\bar{m};\bar{y})),$$
$$(\mu_{\bar{x}}\ast(\mu^{-1})_{\bar x})(\theta(\bar{x};\bar{m}))=\mu(\theta(\bar{x};\bar{m})).$$
We use commutativity of the Morley product in stable theories to compute:
\begin{align*}
    \mu(\theta(\bar{m};\bar{y})) &= (\mu_{\bar{y}}\ast(\mu^{-1})_{\bar y})(\theta(\bar{m};\bar{y})) = \Big(\mu_{\bar{y}}\otimes \big( (h_{\bar{m}})_{\ast}((\mu^{-1})_{\bar y}\big)_{\bar{x}}\Big)\big(\theta(\bar{x};\bar{y})\big) \\
    &= (\mu_{\bar{y}}\otimes\mu_{\bar{x}})\big(\theta(\bar{x};\bar{y})\big)
    = (\mu_{\bar{x}}\otimes\mu_{\bar{y}})\big(\theta(\bar{x};\bar{y})\big) \\
    &= \Big(\mu_{\bar{x}}\otimes \big( (h_{\bar{m}})_{\ast}((\mu^{-1})_{\bar x}\big)_{\bar{y}}\Big)\big(\theta(\bar{x};\bar{y})\big) = (\mu_{\bar{x}}\ast(\mu^{-1})_{\bar x})(\theta(\bar{x};\bar{m})) \\
    &= \mu_{\bar{x}}(\theta(\bar{x};\bar{m})) = \mu^{-1}(\theta(\bar{m};\bar{y})).
\end{align*}
Thus $\mu=\mu^{-1}$.
\end{proof}

\begin{lemma}\label{lemma:G inv right to left}
Assume that 
\begin{enumerate}
    \item $\mu \in \mathfrak{M}_{\bar{m}}(M)$ is $G$-$*$-right invariant,
    \item $\mu(\widetilde{G}) = 1$.
\end{enumerate}
Then $\mu$ is $G$-$\ast$-left-invariant.
\end{lemma}

\begin{proof} 
By Lemma \ref{lemma:implies-inverse} we know that $\mu^{-1}=\mu$
and so $\mu^{-1}$ is right $G$-invariant.
Fix an $\CL$-formula $\theta(\bar{x};\bar{y})$ and $\sigma\in G$.
We use commutativity of the Morley product in stable theories to compute:
\begin{align*}
\Big(\tp\big(\sigma(\bar{m})/M\big)\ast\mu\Big)
\big(\theta(\bar{m};\bar{y})\big) &= \Big(\tp\big(\sigma(\bar{m})/M\big)_{\bar{y}}\otimes((h_{\bar{m}})_\ast(\mu))_{\bar x}\Big)
\big(\theta(\bar{x};\bar{y})\big) \\
&= \Big(\tp\big(\sigma(\bar{m})/M\big)_{\bar{y}}\otimes(\mu^{-1})_{\bar x}\Big)
\big(\theta(\bar{x};\bar{y})\big) \\
&= \Big((\mu^{-1})_{\bar x}\otimes\tp\big(\sigma(\bar{m})/M\big)_{\bar{y}}\Big)
\big(\theta(\bar{x};\bar{y})\big) \\
&= \Big((\mu^{-1})_{\bar x}\otimes h_{\bar{m}}\big(\tp\big(\sigma^{-1}(\bar{m})/M\big)\big)_{\bar{y}}\Big)
\big(\theta(\bar{x};\bar{y})\big) \\
&= \Big(\mu^{-1}\ast \tp\big(\sigma^{-1}(\bar{m})/M\big)\Big)
\big(\theta(\bar{x};\bar{m})\big) \\
&= \mu^{-1}\big(\theta(\bar{x};\bar{m})\big)=\mu\big(\theta(\bar{m};\bar{y})\big). \qedhere 
\end{align*}
\end{proof}

\begin{lemma}\label{lemma:left vs right}
Let $\nu \in \mathfrak{M}_{\bar{m}}(M)$ and $p \in S_{\bar{m}}(M)$. For every $\CL$-formula $\varphi(\bar{x};\bar{y})$, we have the following
\begin{equation*}
    (p* \nu^{-1} )(\varphi(\bar{m};\bar{y})) = (\nu *p^{-1})(h_{\bar{m}}^{-1}([\varphi(\bar{m};\bar{y})])). 
\end{equation*}
In particular, $\nu$ is $G$-$*$-right invariant if and only if $\nu^{-1}$ is $G$-$*$-left invariant.
\end{lemma}

\begin{proof} 
Take $\varphi(\bar{x};\bar{y})\in\CL$ and compute:
\begin{align*}
    (p\ast\nu^{-1})\big(\varphi(\bar{m};\bar{y})\big) &= \Big(p_{\bar y}\otimes(h_{\bar{m}})_\ast\big((\nu^{-1})_{\bar y}\big)_{\bar x}\Big)\big(\varphi(\bar{x};\bar{y})\big) \\
    &= \big(p_{\bar y}\otimes\nu_{\bar{x}}\big)\big(\varphi(\bar{x};\bar{y})\big)= \big(\nu_{\bar{x}}\otimes p_{\bar y}\big)\big(\varphi(\bar{x};\bar{y})\big) \\
    &= \big(\nu_{\bar{x}}\otimes h_{\bar{m}}((p^{-1})_{\bar x})_{\bar y}\big)\big(\varphi(\bar{x};\bar{y})\big) \\
    &= \big(\nu_{\bar{x}}\ast (p^{-1})_{\bar x}\big)\big(\varphi(\bar{x};\bar{m})\big)=(\nu\ast p^{-1})\big(h_{\bar{m}}^{-1}[\varphi(\bar{m};\bar{y})]\big). \qedhere 
\end{align*}
\end{proof}

\begin{proposition}\label{prop:unique} 
Assume that 
\begin{enumerate}
    \item $\mu \in \mathfrak{M}_{\bar{m}}(M)$ is $G$-$*$-right invariant,
    \item $\mu(\widetilde{G}) = 1$.
\end{enumerate}
If $\nu \in \mathfrak{M}_{\bar{m}}(M)$ is such that
\begin{enumerate}
    \item $\nu$ is $G$-$*$-left invariant or $G$-$*$-right invariant, and
    \item $\nu(\widetilde{G}) = 1$,
\end{enumerate}
then $\mu = \nu$. 
\end{proposition}

\begin{proof}
By Lemma \ref{lemma:implies-inverse}, $\mu=\mu^{-1}$.
By Claim 1 from the proof of Lemma \ref{lemma:implies-inverse}, we have that $\mu\ast\nu=\mu$. 

If $\nu$ is $G$-$*$-right invariant, then we obtain that $\nu = \nu^{-1}$ and that $\nu$ is $G$-$*$-left invariant (by Lemma \ref{lemma:implies-inverse} and Lemma \ref{lemma:G inv right to left}).
Otherwise, if $\nu$ is $G$-$*$-left invariant, then, by Lemma \ref{lemma:left vs right} 
and Remark \ref{rem:inverse}, $\nu^{-1}$ is $G$-$*$-right invariant 
and $\nu^{-1}(\widetilde{G})=1$, 
so $\nu^{-1} = (\nu^{-1})^{-1} = \nu$ and $\nu^{-1}$ is $G$-$*$-left invariant (by Lemma \ref{lemma:implies-inverse} and Lemma \ref{lemma:G inv right to left}). 
In either case, $\nu = \nu^{-1}$ and $\nu$ is bi-$G$-$*$-invariant. 

Now, we use Claim 1 from the proof of Lemma \ref{lemma:implies-inverse} once again, but this time with switched roles of $\mu$ and $\nu$ to conclude that $\nu\ast\mu=\nu$.

We have that $\mu\ast\nu=(\nu\ast\mu)^{-1}$.
Indeed, let $\theta(\bar{x};\bar{y})$ be an $\CL$-formula. Then, since Morley product commutes,
\begin{align*}
    (\mu\ast\nu)\big(\theta(\bar{m};\bar{y})\big) &= \big(\mu_{\bar{y}}\otimes\big((h_{\bar{m}})_\ast(\nu)\big)_{\bar{x}}\big)\big(\theta(\bar{x};\bar{y})\big)= \big(\mu_{\bar{y}}\otimes(\nu^{-1})_{\bar{x}}\big)\big(\theta(\bar{x};\bar{y})\big) \\
    &= \big(\mu_{\bar{y}}\otimes\nu_{\bar{x}}\big)\big(\theta(\bar{x};\bar{y})\big) = \big(\nu_{\bar{x}}\otimes \mu_{\bar{y}}\big)\big(\theta(\bar{x};\bar{y})\big) \\
    &= 
    \big(\nu_{\bar{x}}\otimes 
    \big( (h_{\bar{m}})_\ast \mu\big)_{\bar{y}}\big)
    \big(\theta(\bar{x};\bar{y})\big)=(\nu\ast\mu)\big(\theta(\bar{x};\bar{m})\big) \\
&= (\nu \ast \mu)^{-1}(\theta(\bar m;\bar{y})).
\end{align*}
Finally, $\mu=\mu\ast\nu=(\nu\ast\mu)^{-1}=\nu^{-1}=\nu$.
\end{proof}

We remark that the version of Proposition \ref{prop:unique} with the assumption that  $\nu$ is $G$-$*$-right invariant alternatively follows from  Corollary \ref{cor: uniqueness for measures}(1) and Lemma \ref{lemma: relatively-type-definable} (proved in the next subsection). In order to see it, recall that a partial type $\pi(\bar x;\bar y) \vdash \bar x \equiv \bar y$ was chosen so that that $G=G_{\pi,\FC}$. Since $\mu,\nu$ are concentrated on  $\widetilde{G}$, we clearly have that $\hat{\mu},\hat{\nu} \in \mathfrak{M}^{\inv}_{\pi(\bar m;\bar y)}(\FC)$. On the other, since $\mu$ and $\nu$ are $G$-$*$-right invariant, by Lemma  \ref{lemma: relatively-type-definable}, $\hat{\mu},\hat{\nu}$ are (left) $G$-invariant. Thus, using Corollary \ref{cor: uniqueness for measures}(1), we conclude that $\hat{\mu}=\hat{\nu}$, which implies that $\mu=\nu$.

\subsection{Main conjecture in stable theories}\label{subsec:0.7 for stable}
Again, in this subsection, we assume that \textbf{$T$ is stable}. 
We will use the notation from Subsection \ref{subsec: group chunk 3}
and employ Theorem \ref{theorem: counterpart of Newelski's theorem}.

\begin{remark}\label{fact:extend}
Let $\mu \in \mathfrak{M}_{\bar{y}}(M)$. Then $\supp(\hat{\mu})=\{\hat{p}: p \in \supp(\mu)\}$, 
where the operation $\hat{}$ denotes taking the unique $M$-invariant extension to $\FC$.
\end{remark}

\begin{proof} 
The inclusion $\subseteq$ follows from the fact $\hat{\mu}$ is $M$-invariantly supported. For the opposite inclusion, consider any $p \in \supp(\mu)$ and observe that an easy compactness argument allows us to extend $p$ to some $q \in \supp(\hat{\mu})$. Since every element of the support of $\hat{\mu}$ is $M$-invariant, $q$ is $M$-invariant, and thus $q=\hat{p}$ (by uniqueness of an $M$-invariant ($=$ non-forking over $M$) extension).
\end{proof}

\begin{proposition}\label{proposition: profiniteness of supp}
Suppose that $T$ is stable and $\mu \in \mathfrak{M}_{\bar{m}}^{\inv}(\mathfrak{C},M)$ is idempotent. 
Then, $(\supp(\mu),*)$ is a profinite group. In particular, $\mu$ is minimal in the sense explained in Proposition \ref{prop:min}.
\end{proposition}

\begin{proof} 
Note that $\mu = \widehat{\mu|_M}$. On the other hand, by  Proposition \ref{prop:product2}, $(\supp(\mu),*)$ is a semigroup. Hence, by Remark \ref{fact:extend}, $(\supp(\mu),*) \cong (\supp(\mu|_M),*)$ as topological semigroups (with the restriction to $M$ as a witnessing isomorphism); in particular, $\supp(\mu|_M)$ is closed under $*$. 

Set $P := \supp(\mu|_M)$. Then, using the notation from Subsection \ref{subsec: group chunk 3}, $Q:= \cl(*P) = \cl(\supp(\mu|_M)) = \supp(\mu|_M)$.
So, by Theorem \ref{theorem: counterpart of Newelski's theorem}, $I:=\gen(\supp(\mu|_M))$ is a closed two-sided ideal in $\supp(\mu|_M)$. Hence, $\hat{I}:=\{\hat{p}: p \in I\}$ is a closed two-sided ideal in $\supp(\mu)$. By Theorem \ref{thm:two}, we conclude that $\hat{I}=\supp(\mu)$. This implies that $I=\supp(\mu|_M)$. Using Theorem \ref{theorem: counterpart of Newelski's theorem}, we know that $I$ is a profinite group, so we conclude that  $\supp(\mu|_M)$ is a profinite group, and so is its isomorphic copy $\supp(\mu)$.
\end{proof}

\begin{proposition}\label{prop:relative} 
Let $\mu \in \mathfrak{M}_{\bar{m}}(M)$ be idempotent and
let 
$$H_{\mu} := \{\sigma \in \aut(\mathfrak{C}) \;\colon\; \supp(\mu) * \tp(\sigma(\bar{m})/M) = \supp(\mu)\}.$$ 
Then $H_{\mu}$ is a relatively $\bar{m}$-type-definable over $M$ subgroup of $\aut(\mathfrak{C})$. 
Moreover, 
$\supp(\mu)=\gen(\supp(\mu)) = \Gen(\widetilde{H_{\mu}})$ and $(\supp(\mu),*)$ is a profinite group.
\end{proposition}

\begin{proof} 
The fact that $(\supp(\mu),*)$ is a profinite group and $\gen(\supp(\mu))=\supp(\mu)$ was obtained in the proof of Proposition \ref{proposition: profiniteness of supp} (applied to $\hat{\mu}$ in place of $\mu$) as a consequence of Proposition \ref{prop:product2}, Theorem \ref{thm:two}, and Theorem \ref{theorem: counterpart of Newelski's theorem}. Then the moreover part also follows from Theorem \ref{theorem: counterpart of Newelski's theorem}.
\end{proof}

It is sometimes convenient to work with the following version of the stabilizer. 

\begin{definition} Let $\mu \in \mathfrak{M}_{\bar{m}}(M)$. Then we define
\begin{equation*}
    \stab_{r}(\mu) := \{\sigma \in \aut(\mathfrak{C}) \;\colon\; \mu * \tp(\sigma(\bar{m})/M) = \mu\}. 
\end{equation*}
\end{definition}

\begin{lemma}\label{lemma: relatively-type-definable} 
Let $\mu \in \mathfrak{M}_{\bar{m}}(M)$. Then $\stab_{r}(\mu)$ is a group. In fact, $\stab_{r}(\mu) = \stab(\hat{\mu})$, 
where $\hat{\mu}\in\mathfrak{M}^{\inv}_{\bar{m}}(\FC,M)$ is the unique $M$-invariant extension of $\mu$. 
As consequence, $\stab_{r}(\mu)$ is 
a relatively $\bar{m}$-type-definable over $M$ subgroup of $\aut(\FC)$.  
\end{lemma}

\begin{proof}
We first prove that $\stab(\hat{\mu}) \subseteq \stab_{r}(\mu)$. Choose $\sigma \in \stab(\hat{\mu})$ and a formula $\varphi(\bar{x};\bar{y}) \in \CL$. Then

\begin{align*}
    \Big(\mu\ast\tp\big(\sigma(\bar{m})/M\big)\Big)\big(\varphi(\bar{m};\bar{y})\big) &=
    \Big(\mu\otimes h_{\bar{m}}\big(\tp\big(\sigma(\bar{m})/M\big)\big)\Big)\big(\varphi(\bar{x};\bar{y})\big) \\
    &= \Big(\mu\otimes\tp\big(\sigma^{-1}(\bar{m})/M\big)\Big)\big(\varphi(\bar{x};\bar{y})\big) \\
    &= \hat{\mu}\big(\varphi(\sigma^{-1}(\bar{m});\bar{y})\big)= (\sigma\cdot\hat{\mu})\big(\varphi(\bar{m};\bar{y})\big)\\
    &= \mu\big(\varphi(\bar{m};\bar{y})\big).
\end{align*}

We now prove that $\stab_{r}(\mu) \subseteq \stab(\hat{\mu})$. Fix some $\sigma \in \stab_{r}(\mu)$ and set $p:=\tp(\sigma(\bar{m})/M)$.
Notice that for any $\varphi(\bar{x};\bar{y}) \in \mathcal{L}$, we have
\begin{align*}
    \mu(\varphi(\bar{m};\bar{y})) &= (\mu * p)(\varphi(\bar{m};\bar{y}))  = (\mu \otimes h_{\bar{m}}(p)) (\varphi(\bar{x};\bar{y})) \\ 
    &= \hat{\mu}(\varphi(\sigma^{-1}(\bar{m}),\bar{y}))=(\sigma\cdot\hat{\mu})(\varphi(\bar{m};\bar{y})),
\end{align*}
and so we conclude that $(\sigma \cdot \hat{\mu})|_{M} = \mu$. 
Hence, it is enough to show that $\sigma \cdot \hat{\mu}$ does not fork over $M$, since if $\sigma \cdot \hat{\mu}$ does not fork over $M$, then it is the unique $M$-invariant extension of $\mu$, i.e. $\sigma\cdot\hat{\mu}=\hat{\mu}$, and so $\sigma\in\stab(\hat{\mu})$.

By Proposition \ref{prop:relative}, we have that 
$$H_{\mu}:= \{\sigma \in \aut(\mathfrak{C}): \supp(\mu) * \tp(\sigma(\bar{m})/M) = \supp(\mu)\}$$
is a relatively $\bar{m}$-type-definable over $M$ subgroup of $\aut(\mathfrak{C})$. Let $\pi(\bar{x};\bar{y})\vdash \bar{x}\equiv_{\emptyset}\bar{y}$
be a partial type over $\emptyset$, such that $H_{\mu}=G_{\pi,\FC}$. 
Then consider
\begin{equation*}
\widetilde{(H_\mu)}_{\mathfrak{C},\bar m} := \{p(\bar{x}) \in S_{\bar{m}}(\mathfrak{C}) \;\colon\; \pi(\bar{m};\bar{y}) \subseteq p(\bar{x})\}. 
\end{equation*}
In order to see that $\sigma \cdot \hat{\mu}$ does not fork over $M$, it is enough to show that $\supp(\hat{\mu}) \subseteq \Gen(\widetilde{(H_{\mu})}_{\mathfrak{C},\bar m})$ (by Proposition \ref{proposition: characterziations of genericity} and Corollary \ref{corollary: basic properties of generics}.(1)).

 Since $\hat{\mu}$ is $M$-invariant (and $T$ is stable), each type $p \in \supp(\hat{\mu})$ is $M$-invariant, so does not fork over $M$. In particular, for any $p \in \supp(\hat{\mu})$ and any finite collection of formulas $\Delta$, $\overset{\to}{R}_{\Delta}(p) = \overset{\to}{R}_{\Delta}(p|_{M})$. From Proposition \ref{prop:relative}, we have that if $p \in \supp(\hat{\mu})$, then $p|_{M} \in \supp(\mu) = \gen(\supp(\mu)) = \Gen(\widetilde{H_{\mu}})$, where $\widetilde{H_{\mu}}=S_{\pi(\bar{m};\bar{y})}(M)$.
 In particular, $p\in \widetilde{(H_{\mu})}_{\mathfrak{C},\bar m}=S_{\pi(\bar{m};\bar{y})}(\FC)$. 

Thus, by Proposition \ref{proposition: generic via ranks for x'}, we conclude that $\overset{\to}{R}_{\Delta}(p) = \overset{\to}{R}_{\Delta}(p|_M) = \overset{\to}{R}_{\Delta}(\widetilde{(H_\mu)}_{\mathfrak{C},\bar m})$. 
 Hence, $p$ is generic by Proposition \ref{proposition: restricted to x'}, 
which completes the proof. 
\end{proof}

\begin{proposition}\label{prop:existence/uniqueness} 
Let $H$ be a relatively $m$-type-definable over $M$ subgroup of $\aut(\FC)$, 
i.e. $H = G_{\pi,\FC}$ for some partial type $\pi(\bar{x};\bar{y})\vdash\bar{x}\equiv_{\emptyset}\bar{y}$.
We set $\widetilde{H}:=S_{\pi(\bar{m};\bar{y})}(M)$.
Then there exists a measures $\mu_{H} \in \mathfrak{M}_{\bar{m}}(M)$ such that 
\begin{enumerate}
    \item $\mu_{H}(\widetilde{H}) = 1$.
    \item $\mu_{H}$ is $H$-$*$-right invariant; i.e. for every $\sigma \in H$, $\mu_{H} * \tp(\sigma(\bar{m})/M) = \mu_{H}$. 
\end{enumerate}
Moreover, $\mu_{H}$ is unique. 
\end{proposition}

\begin{proof}
By Proposition \ref{prop:unique}, it suffices to show the existence. 
Let $\widetilde{H}_{\mathfrak{C}} : = S_{\pi(\bar{m};\bar{y})}(\FC)$. 
For each finite collection of formulas $\Delta$, there are only finitely many $p_1^{\Delta},\ldots,p_{n}^{\Delta} \in S_{\Delta}(\mathfrak{C})$ with maximal possible $R_{\Delta}([p_i^{\Delta}] \cap \widetilde{H}_{\mathfrak{C}})$. 
In particular, $p_{1}^{\Delta},\ldots,p_{n}^{\Delta}$ are the ``$\Delta$-generics'' of $H$ over $\mathfrak{C}$.
 For each finite $\Delta$, define the $\Delta$-measure $\mu_{\Delta} = \frac{1}{n}\sum_{i=1}^n \delta_{p_i^{\Delta}}$, i.e. $\mu_{\Delta}$ is 
 a measure on $S_{\Delta}(\mathfrak{C})$ concentrated on the subset $\{p(\bar{y}) \in S_{\Delta}(\mathfrak{C})\;\colon\; p(\bar{y}) \cup  \pi(\bar{m};\bar{y}) \text{ is consistent}\}$. 

It is clear that $\mu_{\Delta}$ is invariant under the natural left action of $H$ 
given by $(h \cdot \mu_\Delta)(\varphi(\bar{a};\bar{y})) := \mu_{\Delta}(\varphi(h^{-1}(\bar{a});\bar{y}))$, where $h \in H$.
Then, since $H$ is relatively $\bar{m}$-type-definable over $M$, we get that $\aut(\mathfrak{C}/M) \leqslant H$, and so $\mu_{\Delta}$
is $M$-invariant.

For each $\Delta$, choose $\mu'_{\Delta} \in \mathfrak{M}_{\bar{m}}(\mathfrak{C})$ such that $\mu'_{\Delta}|_{\Delta} = \mu_{\Delta}$, and let $\mu$ be an accumulation point of the net $(\mu'_{\Delta})_{\Delta}$. 
It follows that 
$\mu$ is (left) $H$-invariant, and so $\mu \in \mathfrak{M}^{\inv}_{\pi(\bar{m};\bar{y})}(\mathfrak{C},M)$.
We set $\mu_{H} := \mu|_{M}$ and note that $\mu_H(\widetilde{H})=1$.
To see that $\mu_H$ is $H$-$*$-right invariant, we simply
use that for $\sigma\in H$ and $\varphi(\bar{x};\bar{y})\in\CL$ we have
$$\Big(\mu_H\ast\tp\big(\sigma(\bar{m})/M\big)\Big)\big(\varphi(\bar{m};\bar{y})\big)=(\sigma\cdot\mu)\big(\varphi(\bar{m};\bar{y})\big).$$
This formula was already computed at the beginning of the proof of
Lemma \ref{lemma: relatively-type-definable}.
\end{proof}

\begin{proposition}\label{prop:idempotent} 
Let $\mu \in \mathfrak{M}_{\bar{m}}(M)$ be idempotent and let $H_{\mu} := \{\sigma \in \aut(\mathfrak{C})\;\colon\; \supp(\mu) * \tp(\sigma(\bar{m})/M) = \supp(\mu)\}$. 
Then:
\begin{enumerate}
 \item $\stab_{r}(\mu)  = H_{\mu}$ and $\mu = \mu_{\stab_{r}(\mu)}$ (in the notation from Proposition \ref{prop:existence/uniqueness}),
    \item $(\supp(\mu),*)$ is a profinite group and $\mu|_{\supp(\mu)}$ is the normalized Haar measure on $(\supp(\mu),*)$.
\end{enumerate}
\end{proposition}

\begin{proof} 
Proof of (1). We start with the proof of $\stab_{r}(\mu) = H_{\mu}$.

We first show that if $\sigma \in \stab_{r}(\mu)$, then $\sigma \in H_{\mu}$. Notice that $\mu * \tp(\sigma(\bar{m})/M) = \mu$ implies that $\supp(\mu * \tp(\sigma(\bar{m})/M)) = \supp(\mu)$. 
Using this together with Corollary \ref{cor:support}, we conclude that  $\supp(\mu) * \tp(\sigma(\bar{m})/M) = \supp(\mu)$, and so $\sigma \in H_\mu$. 
(Indeed, letting $p=\tp(\sigma(\bar{m})/M)$, we have $\mu * p =(\hat{\mu} * \hat{p})|_M$, so $\supp(\mu*p) =(\supp(\hat{\mu} * \hat{p}))|_M$ which equals $(\supp(\hat{\mu}) * \hat{p})|_M$ by Corollary \ref{cor:support}. On the other hand, by Remark \ref{fact:extend}, $\supp(\hat{\mu})=\{\hat{q}: q \in \supp(\mu)\}$. Thus, $\supp(\mu * p) = \supp(\mu) * p$.)

We now show that $\stab_{r}(\mu) \supseteq H_\mu$. 
Since from Proposition \ref{prop:relative} we know that $\gen(\supp(\mu))=\supp(\mu)$, by Theorem \ref{theorem: counterpart of Newelski's theorem}, 
it suffices to show that $\stab_{r}(\mu)$ is a relatively $\bar{m}$-type-definable over $M$ subgroup of $\aut(\FC)$ which contains $A_{\supp(\mu)}$. 
By Lemma \ref{lemma: relatively-type-definable}, $\stab_{r}(\mu)$ is relatively $\bar{m}$-type-definable over $M$ subgroup of $\aut(\FC)$.
We now argue that $\stab_{r}(\mu)$ contains $A_{\supp(\mu)}$. 
Indeed, suppose that $\sigma \in A_{\supp(\mu)}$; then $\sigma(\bar{m}) \models p \in \supp(\mu)$. 
We have $\mu * p=(\hat{\mu}*\hat{p})|_M$ and $\hat{p} \in \supp(\hat{\mu})$ by Remark \ref{fact:extend}.
Now, by Propositions \ref{proposition: profiniteness of supp} and \ref{prop:inv}, $\hat{\mu}*\hat{p}=\hat{\mu}$, so $\mu * p = \mu $. Therefore, $\sigma \in \stab_{r}(\mu)$.

Now, let us show that  $\mu = \mu_{\stab_{r}(\mu)}$. The measure $\mu$ is clearly $\stab_{r}(\mu)$-$*$-right-invariant. On the other hand, since $A_{\supp(\mu)} \subseteq \stab_{r}(\mu)$, we have $\supp(\mu) \subseteq \widetilde{\stab_{r}(\mu)}$.
Thus, by Proposition \ref{prop:existence/uniqueness}, we conclude that $\mu = \mu_{\stab_{r}(\mu)}$.

Proof of (2).
By Proposition \ref{prop:relative}, we know that $(\supp(\mu),*)$ is a profinite group, and so it suffices to prove that $\mu|_{\supp(\mu)}$ is the normalized Haar measure. 
Since $(\supp(\mu),*)$ is a topological group, it is enough to prove that for every formula $\varphi(\bar{x};\bar{y}) \in \mathcal{L}$
and every $p\in\supp(\mu)$,
we have that $\mu([\varphi(\bar{m};\bar{y})]_{\supp}) \leqslant \mu([\varphi(\bar{m};\bar{y})]_{\supp} * p)$, 
where $[\varphi(\bar{m};\bar{y})]_{\supp} := \{q \in \supp(\mu)\;\colon\; \varphi(\bar{m};\bar{y}) \in q\}$.

Fix $\varphi(\bar{x};\bar{y}) \in \mathcal{L}$, $p \in \supp(\mu)$.
 Choose $\sigma \in \aut( \mathfrak{C})$ such that $p = \tp(\sigma(\bar{m})/M)$.
 By Proposition \ref{prop:relative}, $\supp(\mu)\ast p=\supp(\mu)$,
 so $\sigma\in H_{\mu}=\stab_r(\mu)$.
By Lemma \ref{lemma: relatively-type-definable}, we have that $\stab_{r}(\mu) = \stab(\hat{\mu})$, 
thus $\sigma\in \stab(\hat{\mu})$ and 
$\sigma^{-1}\in \stab(\hat{\mu})$.
Hence, 
\begin{equation*}
    \mu(\varphi(\bar{m};\bar{y})) 
    = \hat{\mu}(\varphi(\bar{m};\bar{y})) 
    = (\sigma^{-1} \cdot \hat{\mu}) (\varphi(\bar{m};\bar{y})) 
    = \hat{\mu}(\varphi(\sigma(\bar{m});\bar{y})). 
\end{equation*}
On the other hand, by the formula for $\ast$-product for types over $M$ in stable case (discussed at the beginning of Subsection \ref{subsec: group chunk 2}), we have that 
\begin{equation*}
    \mu([\varphi(\bar{m};\bar{y})]_{\supp} * p) = 
    \mu\Big(\big\{\sigma(\hat{p})|_M\;\colon\;p\in[\varphi(\bar{m};\bar{y})]_{\supp}   \big\}\Big).
\end{equation*}
Since by Proposition \ref{prop:relative} we know that $\supp(\mu) = \Gen(\widetilde{H_\mu})$, we obtain that 
the set $\big\{\sigma(\hat{p})\;\colon\;p\in[\varphi(\bar{m};\bar{y})]_{\supp}   \big\}$ is contained in the set $S_{\bar{m}}^{\mathrm{nf}}(\mathfrak{C},M)$ of all types  in $S_{\bar{m}}(\FC)$ which do not fork over $M$
(as in the proof of Lemma \ref{lemma: relatively-type-definable}).
 As $\hat{\mu}$ is concentrated on $S_{\bar{m}}^{\mathrm{nf}}(\mathfrak{C},M)$, we conclude that 
\begin{align*}
    \mu\Big(\big\{\sigma(\hat{p})|_M\;\colon\;p\in[\varphi(\bar{m};\bar{y})]_{\supp}   \big\}\Big) &=
    \hat{\mu}\Big(\big\{\sigma(\hat{p})\;\colon\;p\in[\varphi(\bar{m};\bar{y})]_{\supp}   \big\}\Big) \\
    &\leqslant  \hat{\mu}(\varphi(\sigma(\bar{m});\bar{y})). 
\end{align*}
Finally, we obtain
\begin{equation*}
    \mu([\varphi(\bar{m};\bar{y})]_{\supp} * p) \leqslant
    \hat{\mu}\big(\varphi(\sigma(\bar{m});\bar{y})\big) = \mu(\varphi(\bar{m};\bar{y})) = \mu([\varphi(\bar{m};\bar{y})]_{\supp}),
\end{equation*}
which concludes the proof. 
\end{proof}

\begin{cor}\label{cor: stable idempotent vs unique}
Let $\mu \in \mathfrak{M}_{\bar{m}}(M)$. 
Then the following are equivalent: 
\begin{enumerate}
    \item $\mu$ is idempotent,
    \item $\mu$ is the unique $\stab_{r}(\mu)$-$*$-right (and also the unique $\stab_{r}(\mu)$-$*$-left) invariant measure which concentrates on $\widetilde{\stab_{r}(\mu)}$. 
\end{enumerate}
As consequence, there is a one-to-one correspondence between relatively $\bar m$-type definable over $M$ subgroups of $\aut(\FC)$ and idempotent Keisler measures in $\mathfrak{M}_{\bar{m}}(M)$. 
\end{cor}

\begin{proof} $(1) \Rightarrow (2)$ follows from Propositions \ref{prop:idempotent}(1) and \ref{prop:unique}.
On the other hand, $(2) \Rightarrow (1)$ follows from Claim 1 of Lemma \ref{lemma:implies-inverse} with $\nu=\mu$.

We demonstrate the correspondence. Let $J$ be the set of all idempotent measures in $\mathfrak{M}_{\bar{m}}(M)$ and $\mathcal{H}$ be the set of all relatively $\bar{m}$-type-definable subgroups of $\aut(\FC)$. 
The claimed correspondence is given by the maps $\Phi:J \to \mathcal{H}$ and $\Psi: \mathcal{H} \to J$ defined by $\Phi(\mu) := \stab_{r}(\mu)$ and $\Psi(H) := \mu_{H}$. We need to show that $\Psi \circ \Phi = \id_{J}$ and $\Phi \circ \Psi = \id_{\mathcal{H}}$. 

Notice that $\Psi(\Phi(\mu)) =\mu_{\stab_{r}(\mu)} = \mu$ by (1) of Proposition \ref{prop:idempotent}. Now, for fixed $H \in \mathcal{H}$, we have that $\Phi(\Psi(H)) = \Phi(\mu_{H}) = \stab_{r}(\mu_{H})$. We want to show that $\stab_{r}(\mu_{H}) = H$. The inclusion $H \subseteq \stab_{r}(\mu_{H})$ is immediate from the $H$-$*$-right invariance of $\mu_{H}$. For the opposite inclusion, suppose for a contradiction that $\mu_{H} * \tp(\sigma(\bar{m})/M) = \mu_{H}$ for some $\sigma \not \in H$. Then for every $\mathcal{L}$-formula $\varphi(\bar{x};\bar{y})$ we have that 
\begin{equation*} 
    \mu_{H}(\varphi(\bar{m};\bar{y}))
    = \Big(\mu_{H} * \tp\big(\sigma(\bar{m})/M\big)\Big)\big(\varphi(\bar{m};\bar{y})\big)
    = \widehat{\mu_{H}}\big(\varphi(\sigma^{-1}(\bar{m});\bar{y}\big). \tag{$\diamondsuit$}
\end{equation*}

Take $p \in \supp(\mu_{H}) \subseteq \widetilde{H}$. Choose $\tau' \in \aut(\mathfrak{C}')$ such that $\tau'(\bar{m}) \models \hat{p}$; 
then $\tau' \in H_{\mathfrak{C}'}$. Extend $\sigma$ to $\sigma' \in \aut(\mathfrak{C}')$. Then $\sigma' \not \in H_{\mathfrak{C}'}$ so $\sigma' \tau' \not \in H_{\mathfrak{C}'}$. 
Therefore, $\sigma(\hat{p}) = \tp(\sigma'(\tau'(\bar{m}))/\mathfrak{C}) \not \in \widetilde{H}_{\mathfrak{C},\bar{m}}$ and so $\sigma(\hat{p})|_{M} \not \in \widetilde{H}$.
We can find a formula $\varphi(\bar{x};\bar{y}) \in \mathcal{L}$ such that $\varphi(\sigma^{-1}(\bar{m}); \bar{y}) \in \hat{p}$ and $[\varphi(\bar{m};\bar{y})] \cap \widetilde{H} = \emptyset$. 

As $\supp(\mu_{H}) \subseteq \widetilde{H}$, we conclude that $\mu_{H}(\varphi(\bar{m};\bar{y})) = 0$.
Hence, by equation $(\diamondsuit)$, $\widehat{\mu_{H}}(\varphi(\sigma^{-1}(\bar{m});\bar{y}) = 0$.
 However, this is impossible since $p\in \supp(\mu)$ implies that $\hat{p} \in \supp(\widehat{\mu_{H}})$ (see Remark \ref{fact:extend}).  
\end{proof}

As a conclusion, we obtain Conjecture \ref{conjecture: main conjecture} in the stable case.

\begin{cor}[Stable case]\label{cor: 0.7 for stable}
    Let $\mu \in \mathfrak{M}^{\inv}_{\bar m}(\FC,M)$. 
We know that $\stab(\mu)=G_{\pi,\FC}$ for some partial type $\pi(\bar x;\bar y) \vdash \bar x \equiv \bar y$.
Then the following are equivalent:
\begin{enumerate}
\item $\mu$ is an idempotent.
\item $\mu$ is the unique (left) $G_{\pi,\FC}$-invariant measure in $\mathfrak{M}^{\inv}_{\pi(\bar{m};\bar{y})}(\FC,M)$.
\end{enumerate}
In particular, there is a correspondence between idempotent measures in $\mathfrak{M}^{\inv}_{\bar m}(\FC,M)$ and relatively $\bar m$-type-definable over $M$ subgroups of $\aut(\FC)$.
\end{cor}

\begin{proof}
    The implication $(2)\Rightarrow(1)$ follows by Corollary \ref{cor: two to one}. We argue for $(1)\Rightarrow(2)$.

    Since $\mu$ is idempotent, also $\mu|_M$ is idempotent.
    Corollary \ref{cor: stable idempotent vs unique} implies that
    $\mu|_M(\widetilde{\stab_r(\mu|_M)})=1$.
    Because the restriction map $r:S^{\inv}_{\bar{m}}(\FC,M)\to S_{\bar{m}}(M)$ is a homeomorphism and $r_\ast(\mu)=\mu|_M$,
    and $\stab_r(\mu|_M)=\stab(\mu)$ (by Lemma \ref{lemma: relatively-type-definable}),
    we see that
    $$1=\mu\Big(r^{-1}[\widetilde{\stab_r(\mu|_M)}] \Big)=
    \mu\Big(r^{-1}[\widetilde{\stab(\mu)}] \Big).$$
    By the choice of $\pi(\bar x;\bar y)$, $\widetilde{\stab(\mu)}=[\pi(\bar{m};\bar{y})]\subseteq S_{\bar{m}}(M)$, and so $r^{-1}[\widetilde{\stab(\mu)}]=[\pi(\bar{m};\bar{y})]\cap S^{\inv}_{\bar{m}}(\FC,M)$. Hence, $\mu\in\mathfrak{M}^{\inv}_{\pi(\bar{m};\bar{y})}(\FC,M)$. Naturally $\mu$ is $G_{\pi,\FC}$-invariant, and the uniqueness in $(2)$ follows by 
Corollary \ref{cor: uniqueness for measures}(1).

The correspondence follows from Corollary \ref{cor: uniqueness for measures}(1) and Remark \ref{remark: correspondence from positive answer}.
\end{proof}

\printbibliography

\end{document}